\documentclass[oneside,english,reqno]{amsart}
\usepackage[T1]{fontenc}
\usepackage[latin9]{inputenc}
\usepackage{geometry}
\usepackage{babel}
\usepackage{units}
\usepackage{mathrsfs}
\usepackage{enumitem}
\usepackage{amsthm}
\usepackage{amstext}
\usepackage{amssymb}
\usepackage{esint}

\makeatletter
\numberwithin{equation}{section}
\numberwithin{figure}{section}
\usepackage{enumitem}		
\theoremstyle{plain}
\newtheorem{thm}{\protect\theoremname}[section]
  \theoremstyle{definition}
  \newtheorem{defn}[thm]{\protect\definitionname}
  \theoremstyle{remark}
  \newtheorem*{rem*}{\protect\remarkname}
  \theoremstyle{plain}
  \newtheorem{lem}[thm]{\protect\lemmaname}
  \theoremstyle{plain}
  \newtheorem{cor}[thm]{\protect\corollaryname}
 \newlist{casenv}{enumerate}{4}
 \setlist[casenv]{leftmargin=*,align=left,widest={iiii}}
 \setlist[casenv,1]{label={{\itshape\ \casename} \arabic*.},ref=\arabic*}
 \setlist[casenv,2]{label={{\itshape\ \casename} \roman*.},ref=\roman*}
 \setlist[casenv,3]{label={{\itshape\ \casename\ \alph*.}},ref=\alph*}
 \setlist[casenv,4]{label={{\itshape\ \casename} \arabic*.},ref=\arabic*}
  \theoremstyle{remark}
  \newtheorem*{claim*}{\protect\claimname}
  \theoremstyle{remark}
  \newtheorem{claim}[thm]{\protect\claimname}
  \theoremstyle{definition}
  \newtheorem{example}[thm]{\protect\examplename}
  \theoremstyle{remark}
  \newtheorem{rem}[thm]{\protect\remarkname}

\usepackage[T1]{fontenc}

\theoremstyle{remark}
\newtheorem{claimx}[thm]{Claim}
\renewenvironment{claim}
  {\pushQED{\qed}\claimx}
  {\popQED\endclaimx}

\theoremstyle{remark}
\newtheorem*{claimxx}{Claim}
\renewenvironment{claim*}
  {\pushQED{\qed}\claimxx}
  {\popQED\endclaimxx}

\theoremstyle{remark}
\newtheorem{remx}[thm]{Remark}
\renewenvironment{rem}
  {\pushQED{\qed}\remx}
  {\popQED\endremx}

\theoremstyle{remark}
\newtheorem*{remxx}{Remark}
\renewenvironment{rem*}
  {\pushQED{\qed}\remxx}
  {\popQED\endremxx}

\theoremstyle{definition}
\newtheorem{examplex}[thm]{Example}
\renewenvironment{example}
  {\pushQED{\qed}\examplex}
  {\popQED\endexamplex}

\theoremstyle{definition}
\newtheorem{defnx}[thm]{Definition}
\renewenvironment{defn}
  {\pushQED{\qed}\defnx}
  {\popQED\enddefnx}

\theoremstyle{plain}
\newtheorem{lemx}[thm]{Lemma}
\renewenvironment{lem}
  {\pushQED{\qed}\lemx}
  {\popQED\endlemx}

\theoremstyle{plain}
\newtheorem{corx}[thm]{Corollary}
\renewenvironment{cor}
  {\pushQED{\qed}\corx}
  {\popQED\endcorx}

\theoremstyle{definition}
\newtheorem{thmx}[thm]{Theorem}
\renewenvironment{thm}
  {\pushQED{\qed}\thmx}
  {\popQED\endthmx}

\let\emptyset\varnothing

\makeatother

  \providecommand{\claimname}{Claim}
  \providecommand{\corollaryname}{Corollary}
  \providecommand{\definitionname}{Definition}
  \providecommand{\examplename}{Example}
  \providecommand{\lemmaname}{Lemma}
  \providecommand{\remarkname}{Remark}
 \providecommand{\casename}{Case}
\providecommand{\theoremname}{Theorem}

\begin{document}

\title{Embeddings of Decomposition Spaces into Sobolev and BV Spaces}

\author{Felix Voigtlaender}

\keywords{Decomposition spaces, Sobolev spaces, Coorbit spaces, Smoothness
spaces, Embeddings, Shearlets, $\alpha$-modulation spaces, Besov
spaces, BV spaces}
\begin{abstract}
In the present paper, we investigate whether an embedding of a decomposition
space $\mathcal{D}\left(\mathcal{Q},L^{p},Y\right)$ into a given
Sobolev space $W^{k,q}\left(\mathbb{R}^{d}\right)$ exists. As special
cases, this includes embeddings into Sobolev spaces of (homogeneous
and inhomogeneous) Besov spaces, ($\alpha$)-modulation spaces, shearlet
smoothness spaces and also of a large class of wavelet coorbit spaces,
in particular of shearlet-type coorbit spaces.

Precisely, we will show that under extremely mild assumptions on the
covering $\mathcal{Q}=\left(Q_{i}\right)_{i\in I}$, we have $\mathcal{D}\left(\mathcal{Q},L^{p},Y\right)\hookrightarrow W^{k,q}\left(\mathbb{R}^{d}\right)$
as soon as $p\leq q$ and $Y\hookrightarrow\ell_{u^{\left(k,p,q\right)}}^{q^{\triangledown}}\left(I\right)$
hold. Here, $q^{\triangledown}=\min\left\{ q,q'\right\} $ and the
weight $u^{\left(k,p,q\right)}$ can be easily computed, only based
on the covering $\mathcal{Q}$ and on the parameters $k,p,q$.

Conversely, a necessary condition for existence of the embedding is
that $p\leq q$ and $Y\cap\ell_{0}\left(I\right)\hookrightarrow\ell_{u^{\left(k,p,q\right)}}^{q}\left(I\right)$
hold, where $\ell_{0}\left(I\right)$ denotes the space of finitely
supported sequences on $I$. Thus, even though our approach applies
to (almost) \emph{arbitrary} decomposition spaces, we can \emph{completely
characterize} existence of the embedding $\mathcal{D}\left(\mathcal{Q},L^{p},Y\right)\hookrightarrow W^{k,q}\left(\mathbb{R}^{d}\right)$
for the range $q\in\left(0,2\right]$, since this implies $q=q^{\triangledown}$.
Furthermore, this complete characterization also remains valid for
the important special case $q=\infty$. In this case, we will even
see that the embedding $\mathcal{D}\left(\mathcal{Q},L^{p},Y\right)\hookrightarrow W^{k,\infty}\left(\mathbb{R}^{d}\right)$
is equivalent to an embedding into the space $C_{b}^{k}\left(\mathbb{R}^{d}\right)$.

As a further result, we show that a decomposition space embeds into
${\rm BV}$ if and only if it embeds into $W^{1,1}\left(\mathbb{R}^{d}\right)$.
Hence, our results also yield a \emph{complete characterization} of
embeddings of decomposition spaces into ${\rm BV}$.

Finally, for the most important special case of a weighted Lebesgue
sequence space $Y=\ell_{v}^{r}\left(I\right)$, we simplify the conditions
from above even further: In this case, $Y\hookrightarrow\ell_{u}^{s}\left(I\right)$
is equivalent to $\frac{u}{v}\in\ell^{s\cdot\left(r/s\right)'}\left(I\right)$.
As a consequence, verification of the previously mentioned necessary
or sufficient conditions reduces to an exercise in calculus which
can be solved without requiring any knowledge of Fourier analysis.

To indicate the power and simplicity of our criteria, we apply them
to all of the classes of decomposition spaces mentioned in the first
paragraph above.
\end{abstract}
\maketitle
\global\long\def\with{\,\middle|\,}
\global\long\def\esssup{\operatorname*{ess\, sup}}

\section{Introduction}

\label{sec:Introduction}Decomposition spaces were first introduced
in full generality by Feichtinger and Gröbner \cite{DecompositionSpaces1,DecompositionSpaces2}
in the 80s. They were then used by Gröbner in his PhD thesis\cite{GroebnerAlphaModulationSpaces}
to define the so-called $\alpha$-modulation spaces.

Recently, decomposition spaces received increased attention. This
began with the work of Borup and Nielsen \cite{BorupNielsenDecomposition},
who constructed Banach frames for certain decomposition spaces. Note,
however, that Borup and Nielsen consider a more restricted class of
decomposition spaces than Feichtinger and Gröbner. It is this class
of decomposition spaces that we will consider in the present paper.
Roughly, these decomposition spaces are defined as follows: One fixes
a covering $\mathcal{Q}=\left(Q_{i}\right)_{i\in I}$ of a subset
$\mathcal{O}$ of the \emph{frequency space} $\mathbb{R}^{d}$. Then,
given a suitable partition of unity $\Phi=\left(\varphi_{i}\right)_{i\in I}$
subordinate to $\mathcal{Q}$, the decomposition space (quasi)-norm
is defined as
\[
\left\Vert f\right\Vert _{\mathcal{D}\left(\mathcal{Q},L^{p},Y\right)}:=\left\Vert \left(\left\Vert \mathcal{F}^{-1}\left(\smash{\varphi_{i}\widehat{f}}\right)\right\Vert _{L^{p}}\right)_{i\in I}\right\Vert _{Y},
\]
where $Y\leq\mathbb{C}^{I}$ is a suitable sequence space. Thus, these
spaces are defined analogously to Besov spaces, in contrast Triebel
Lizorkin spaces, where the ``order'' of the $L^{p}$ and $Y$ norms
would be reversed.

Recent papers which study concrete examples of decomposition spaces
are \cite{HanWangAlphaModulationEmbeddings,ToftWahlbergAlphaModulationEmbeddings},
in which the respective authors consider embeddings between $\alpha$-modulation
spaces for different values of $\alpha$. Furthermore, we would like
to mention \cite{Labate_et_al_Shearlet}, in which Labate et al.\@
use the construction of decomposition spaces to introduce a new class
of function spaces, the so-called shearlet smoothness spaces. In the
same paper, the authors establish certain embeddings between Besov
spaces and shearlet smoothness spaces.

As a more abstract application of decomposition spaces, we mention
the paper \cite{FuehrVoigtlaenderCoorbitSpacesAsDecompositionSpaces},
in which Hartmut Führ and the present author showed that a large class
of wavelet coorbit spaces are naturally isomorphic to certain decomposition
spaces, i.e., we have
\begin{equation}
{\rm Co}\left(L_{v}^{p,q}\left(\mathbb{R}^{d}\rtimes H\right)\right)\cong\mathcal{D}\left(\mathcal{Q}_{H},L^{p},\ell_{u}^{q}\right)\text{ for a certain weight }u=u_{\mathcal{Q},v,q}.\label{eq:CoorbitDecompositionIsomorphism}
\end{equation}
Here, the covering $\mathcal{Q}_{H}$ used to define the decomposition
spaces (which is called an \textbf{induced covering}) depends crucially
on the dilation group $H\leq{\rm GL}\left(\mathbb{R}^{d}\right)$
through which the wavelet coorbit space under consideration is defined.
This decomposition space view on coorbit spaces is often superior
to the original description as coorbit spaces; for example, the decomposition
space view makes it possible to consider embeddings between wavelet
coorbit spaces ``living'' on different groups.

Indeed, in my PhD thesis \cite{VoigtlaenderPhDThesis} and in the
upcoming paper \cite{DecompositionEmbeddings}, I developed a general
theory of embeddings between different decomposition spaces, i.e.\@
embeddings of the form
\[
\mathcal{D}\left(\mathcal{Q},L^{p_{1}},\ell_{u}^{q_{1}}\right)\hookrightarrow\mathcal{D}\left(\mathcal{P},L^{p_{2}},\ell_{v}^{q_{2}}\right)
\]
for (possibly) different coverings $\mathcal{Q},\mathcal{P}$. Here,
we write $Y\hookrightarrow X$ if%
\footnote{Actually, in the present generality, the inclusion $Y\subset X$ is
sometimes not satisfied. In this case, we still write $Y\hookrightarrow X$
if there is a bounded linear map $\iota:Y\to X$ which (in some sense)
deserves to be called an embedding. This is made more precise in Definition
\ref{def:EmbeddingDefinition}. As a concrete example, note that \emph{homogeneous}
Besov spaces are usually defined as subspaces $\dot{\mathcal{B}}_{s}^{p,q}\leq\nicefrac{\mathcal{S}'}{\mathcal{P}}$,
where $\mathcal{P}$ is the space of polynomials. Thus, the inclusion
$\dot{\mathcal{B}}_{s_{1}}^{p_{1},q_{1}}\subset\mathcal{B}_{s_{2}}^{p_{2},q_{2}}\subset\mathcal{S}'$
for an inhomogeneous Besov space $\mathcal{B}_{s_{2}}^{p_{2},q_{2}}$
can \emph{never} be satisfied. Nevertheless, a (bounded linear) map
$\iota:\dot{\mathcal{B}}_{s_{1}}^{p_{1},q_{1}}\to\mathcal{B}_{s_{2}}^{p_{2},q_{2}}$
``deserves'' to be called an embedding if $\iota f+\mathcal{P}=f$
for all $f\in\dot{\mathcal{B}}_{s_{1}}^{p_{1},q_{1}}$.%
} $Y\subset X$ and $\left\Vert x\right\Vert _{X}\leq C\cdot\left\Vert y\right\Vert _{Y}$
for all $y\in Y$.

Using this embedding theory and the interpretation of wavelet coorbit
spaces as decomposition spaces, the existence of embeddings between
coorbit spaces with respect to the shearlet-type dilation groups
\begin{equation}
H^{\left(c\right)}:=\left\{ \varepsilon\left(\begin{matrix}a & b\\
0 & a^{c}
\end{matrix}\right)\with a\in\left(0,\infty\right),\varepsilon>0,\, b\in\mathbb{R}\right\} \label{eq:ShearletTypeGroup}
\end{equation}
for different values of $c\in\mathbb{R}$ could be characterized,
cf.\@ \cite[Theorem 6.38]{VoigtlaenderPhDThesis}. Furthermore, I
considered embeddings between shearlet-type coorbit spaces and (inhomogeneous)
Besov spaces, achieving a complete characterization in many cases\cite[Theorem 6.3.11]{VoigtlaenderPhDThesis}.

As further applications, the embedding results for $\alpha$-modulation
spaces given in \cite{HanWangAlphaModulationEmbeddings} can be obtained
(and even generalized) using the theory developed in \cite{VoigtlaenderPhDThesis},
cf.\@ \cite[Theorem 6.1.7]{VoigtlaenderPhDThesis}. Likewise, the
results in \cite{Labate_et_al_Shearlet} for embeddings between shearlet
smoothness spaces and Besov spaces can be improved to obtain a complete
characterization of the existence of such embeddings, cf.\@ \cite[Theorem 6.4.3]{VoigtlaenderPhDThesis}.

Despite these results, the thesis \cite{VoigtlaenderPhDThesis} only
considers embeddings between different decomposition spaces. Thus,
embeddings of decomposition spaces into Sobolev spaces are not covered.
In the present paper, we will initiate the development of such an
embedding theory. Indeed, we will completely characterize the existence
of an embedding $\mathcal{D}\left(\mathcal{Q},L^{p},Y\right)\hookrightarrow W^{k,q}\left(\mathbb{R}^{d}\right)$
for $q\in\left(0,2\right]\cup\left\{ \infty\right\} $. For the remaining
range $q\in\left(2,\infty\right)$, we also obtain sufficient criteria
and necessary criteria, but these two types of conditions do not coincide.
Investigating this gap is a valuable topic for future research.

\subsection{Our results}

The precise formulation of the main result of this paper needs the
notion of a \textbf{regular covering} $\mathcal{Q}=\left(Q_{i}\right)_{i\in I}=\left(T_{i}Q_{i}'+b_{i}\right)_{i\in I}$,
which will be fully explained in Section \ref{sec:DecompositionSpaces}.
At the moment, we only remark that this means that the sets $Q_{i}=T_{i}Q_{i}'+b_{i}$
are obtained from the \textbf{normalized sets} $Q_{i}'$ by means
of the (invertible) affine maps $x\mapsto T_{i}x+b_{i}$. Furthermore,
we require the normalized sets $Q_{i}'$ to be uniformly bounded,
since without an assumption of this kind there would not be any meaningful
relationship between the set $Q_{i}$ and the affine map $x\mapsto T_{i}x+b_{i}$.

In addition to the properties from the previous paragraph, a regular
covering has to satisfy certain additional technical assumptions,
which are explained in detail in Section \ref{sec:DecompositionSpaces}.
These assumptions, however, are fulfilled for any reasonable covering
occuring in practice, in particular for the coverings used to define
(homogeneous and inhomogeneous) Besov spaces, ($\alpha$)-modulation
spaces, Shearlet smoothness spaces and all induced coverings $\mathcal{Q}_{H}$
which are used when interpreting wavelet coorbit spaces as decomposition
spaces (as in equation (\ref{eq:CoorbitDecompositionIsomorphism})).

Finally, the following theorem uses the notion of a \textbf{$\mathcal{Q}$-regular
sequence space} $Y$, which is also introduced in Section \ref{sec:DecompositionSpaces}.
For simplicity, the reader might simply think of the case of a weighted
Lebesgue sequence space $Y=\ell_{v}^{r}\left(I\right)$, as defined
in Subsection \ref{sub:Notation}. Using these notions, our main result
reads (slightly simplified) as follows:
\begin{thm}
\label{thm:IntroductionMainTheorem}Let $\mathcal{Q}=\left(Q_{i}\right)_{i\in I}=\left(T_{i}Q_{i}'+b_{i}\right)_{i\in I}$
be a regular covering of the open set $\emptyset\neq\mathcal{O}\subset\mathbb{R}^{d}$.
Let $k\in\mathbb{N}_{0}$ and $q\in\left(0,\infty\right]$ and let
$Y\leq\mathbb{C}^{I}$ be a $\mathcal{Q}$-regular sequence space
on $I$. Define the weight $u^{\left(k,p,q\right)}$ by
\[
u_{i}^{\left(k,p,q\right)}:=\left|\det T_{i}\right|^{\frac{1}{p}-\frac{1}{q}}\cdot\left(1+\left|b_{i}\right|^{k}+\left\Vert T_{i}\right\Vert ^{k}\right)\qquad\text{ for }i\in I
\]
and let $q^{\triangledown}:=\min\left\{ q,q'\right\} $. Then the
following hold:
\begin{itemize}
\item If $p\leq q$ and if
\[
Y\hookrightarrow\ell_{u^{\left(k,p,q\right)}}^{q^{\triangledown}}\left(I\right),
\]
then
\[
\mathcal{D}\left(\mathcal{Q},L^{p},Y\right)\hookrightarrow W^{k,q}\left(\mathbb{R}^{d}\right).
\]

\item Conversely, if
\[
\mathcal{D}\left(\mathcal{Q},L^{p},Y\right)\hookrightarrow W^{k,q}\left(\mathbb{R}^{d}\right)
\]
holds, then we have $p\leq q$ and
\[
Y\cap\ell_{0}\left(I\right)\hookrightarrow\ell_{u^{\left(k,p,q\right)}}^{q}\left(I\right),
\]
where $\ell_{0}\left(I\right)$ denotes the space of finitely supported
sequences on $I$. In case of $q=\infty$, we also get $Y\cap\ell_{0}\left(I\right)\hookrightarrow\ell_{u^{\left(k,p,q\right)}}^{1}\left(I\right)=\ell_{u^{\left(k,p,q\right)}}^{q^{\triangledown}}\left(I\right)$.\qedhere
\end{itemize}
\end{thm}
As noted above, for $q\in\left(0,2\right]\cup\left\{ \infty\right\} $,
we thus obtain a \emph{complete characterization} of the existence
of an embedding of the decomposition space $\mathcal{D}\left(\mathcal{Q},L^{p},Y\right)$
into the Sobolev space $W^{k,q}\left(\mathbb{R}^{d}\right)$, at least
if we ignore the slight difference between the spaces $Y$ and $Y\cap\ell_{0}\left(I\right)$.
Note that the restriction $q\in\left(0,2\right]\cup\left\{ \infty\right\} $
(which is crucial for sharpness) is a restriction on the ``target
space'' $W^{k,q}\left(\mathbb{R}^{d}\right)$, not on the ``source
space'' $\mathcal{D}\left(\mathcal{Q},L^{p},Y\right)$.

As a further simplification, for the case of a weighted Lebesgue space
$Y=\ell_{v}^{r}\left(I\right)$, we establish the equivalence 
\begin{align}
 & \ell_{v}^{r}\left(I\right)\cap\ell_{0}\left(I\right)\hookrightarrow\ell_{u}^{s}\left(I\right)\nonumber \\
\Longleftrightarrow & \ell_{v}^{r}\left(I\right)\hookrightarrow\ell_{u}^{s}\left(I\right)\nonumber \\
\Longleftrightarrow & \left(u_{i}/v_{i}\right)_{i\in I}\in\ell^{s\cdot\left(r/s\right)'}\left(I\right).\label{eq:IntroductionSequenceEmbeddingCharacterization}
\end{align}
Thus, verification of the embedding $Y=\ell_{v}^{r}\left(I\right)\hookrightarrow\ell_{u}^{q}\left(I\right)$
reduces to verifying finiteness of a certain $\ell^{t}\left(I\right)$-norm
\emph{of a single sequence}.

Finally, we remark that to prove the above result, we actually establish
a stronger statement which might be of independent interest: Given
$n\in\mathbb{N}_{0}$, we give sufficient criteria and also necessary
criteria for boundedness of the family of all (suitably defined) partial
derivative operators $\partial_{\ast}^{\alpha}:\mathcal{D}\left(\mathcal{Q},L^{p},Y\right)\to L^{q}\left(\mathbb{R}^{d}\right)$
with $\left|\alpha\right|=n$. We will see that a sufficient criterion
is $p\leq q$ and $Y\hookrightarrow\ell_{v^{\left(n,p,q\right)}}^{q^{\triangledown}}\left(I\right)$
with
\[
v_{i}^{\left(n,p,q\right)}=\left|\det T_{i}\right|^{\frac{1}{p}-\frac{1}{q}}\cdot\left(\left|b_{i}\right|^{n}+\left\Vert T_{i}\right\Vert ^{n}\right).
\]
Similarly, a necessary criterion is $p\leq q$ and $Y\cap\ell_{0}\left(I\right)\hookrightarrow\ell_{v^{\left(n,p,q\right)}}^{q}\left(I\right)$.
The theorem above is then a corollary by applying these criteria simultaneously
for $n=0,1,\dots,k$.

\subsection{Comparison with earlier results}

\label{sub:ComparisonAndBesovDetour}There are only two types of published
results of which I am aware in which embeddings into Sobolev spaces
of some kind of decomposition space are considered.

The first result is the \emph{complete characterization} of the embeddings
$M_{s,0}^{p,q}\left(\mathbb{R}^{d}\right)\hookrightarrow W^{k,p}\left(\mathbb{R}^{d}\right)$
of modulation spaces into Sobolev spaces (or more precisely, the
Bessel potential spaces $L_{s}^{p}\left(\mathbb{R}^{d}\right)$ which
coincide with $W^{k,p}\left(\mathbb{R}^{d}\right)$ for $p\in\left(1,\infty\right)$
and $s=k\in\mathbb{N}_{0}$) given in \cite{KobayashiSugimotoModulationSobolevInclusion}
by Kobayashi and Sugimoto. Compared to the present results, we note
the following differences:
\begin{itemize}
\item Kobayashi and Sugimoto obtain a \emph{complete characterization},
even for $p\in\left(2,\infty\right)$. Using our results (cf.\@ also
example \ref{exa:AlphaModulationSpaces} below), we obtain some sufficient
and some necessary conditions, but these only yield a complete characterization
for $p\in\left(0,2\right]\cup\left\{ \infty\right\} $. Furthermore,
the approach in \cite{KobayashiSugimotoModulationSobolevInclusion}
also characterizes embeddings into the spaces $L_{s}^{p}\left(\mathbb{R}^{d}\right)$
for non-integer values $s\in\mathbb{R}\setminus\mathbb{N}_{0}$.
\item Our results apply to arbitrary embeddings $M_{s,0}^{p,q}\left(\mathbb{R}^{d}\right)\hookrightarrow W^{k,r}\left(\mathbb{R}^{d}\right)$,
while in \cite{KobayashiSugimotoModulationSobolevInclusion}, only
the case $r=p$ is considered. Furthermore, the present results can
be applied to general $\alpha$-modulation spaces $M_{s,\alpha}^{p,q}\left(\mathbb{R}^{d}\right)$,
not only to modulation spaces (i.e.\@ not only for $\alpha=0$).
\end{itemize}
The second known result of which I am aware are the embeddings of
(inhomogeneous) Besov spaces $\mathcal{B}_{s}^{p,q}\left(\mathbb{R}^{d}\right)$
into the Triebel-Lizorkin spaces $F_{s}^{p,q}\left(\mathbb{R}^{d}\right)$.
This yields embeddings into Sobolev spaces as follows: As shown in
\cite[Section 2.3.5, equation (2)]{TriebelTheoryOfFunctionSpaces},
we have
\[
H^{s,p}\left(\mathbb{R}^{d}\right)=F_{s}^{p,2}\left(\mathbb{R}^{d}\right)\text{ for }p\in\left(1,\infty\right)\text{ and }s\in\mathbb{R},
\]
where $H^{s,p}\left(\mathbb{R}^{d}\right)$ is a Bessel potential
space. Furthermore, because of $H^{s,p}\left(\mathbb{R}^{d}\right)=W^{s,p}\left(\mathbb{R}^{d}\right)$
for $s\in\mathbb{N}_{0}$ (cf.\@ \cite[Section 2.3.1, Definition 1(7) and Remark 2]{TriebelInterpolation}),
we finally get $W^{s,p}\left(\mathbb{R}^{d}\right)=F_{s}^{p,2}\left(\mathbb{R}^{d}\right)$
for $p\in\left(1,\infty\right)$ and $s\in\mathbb{N}_{0}$.

Now, in \cite[Section 2.3.2, Proposition 2(iii)]{TriebelTheoryOfFunctionSpaces},
the embedding
\[
\mathcal{B}_{s}^{p,\min\left\{ p,q\right\} }\left(\mathbb{R}^{d}\right)\hookrightarrow F_{s}^{p,q}\left(\mathbb{R}^{d}\right)\hookrightarrow\mathcal{B}_{s}^{p,\max\left\{ p,q\right\} }\left(\mathbb{R}^{d}\right)
\]
is established for $p\in\left(0,\infty\right)$, $q\in\left(0,\infty\right]$
and $s\in\mathbb{R}$. Hence,
\begin{equation}
\mathcal{B}_{s}^{p,\min\left\{ p,2\right\} }\left(\mathbb{R}^{d}\right)\hookrightarrow W^{s,p}\left(\mathbb{R}^{d}\right)=F_{s}^{p,2}\left(\mathbb{R}^{d}\right)\hookrightarrow\mathcal{B}_{s}^{p,\max\left\{ p,2\right\} }\left(\mathbb{R}^{d}\right)\label{eq:SobolevBesovInclusionHard}
\end{equation}
for $p\in\left(1,\infty\right)$ and $s\in\mathbb{N}_{0}$. Finally,
\cite[Section 2.3, Proposition 2(ii)]{TriebelTheoryOfFunctionSpaces}
yields
\[
\mathcal{B}_{s+\varepsilon}^{p,q_{0}}\left(\mathbb{R}^{d}\right)\hookrightarrow\mathcal{B}_{s}^{p,q_{1}}\left(\mathbb{R}^{d}\right)\text{ for all }s\in\mathbb{R},\:\varepsilon>0\text{ and }p,q_{0},q_{1}\in\left(0,\infty\right],
\]
so that we also get
\begin{equation}
\mathcal{B}_{s+\varepsilon}^{p,q}\left(\mathbb{R}^{d}\right)\hookrightarrow\mathcal{B}_{s}^{p,\min\left\{ p,2\right\} }\left(\mathbb{R}^{d}\right)\hookrightarrow W^{s,p}\left(\mathbb{R}^{d}\right)\hookrightarrow\mathcal{B}_{s}^{p,\max\left\{ p,2\right\} }\left(\mathbb{R}^{d}\right)\hookrightarrow\mathcal{B}_{s-\varepsilon}^{p,q}\left(\mathbb{R}^{d}\right)\label{eq:SobolevBesovInclusionSoft}
\end{equation}
for arbitrary $p\in\left(1,\infty\right)$, $s\in\mathbb{N}_{0}$,
$\varepsilon>0$ and $q\in\left(0,\infty\right]$.

Apart from establishing inclusion relations between Besov spaces
and Sobolev spaces, the above embeddings can also be used to derive
sufficient and necessary criteria for embeddings of decomposition
spaces into Sobolev spaces. Indeed, if $\mathcal{D}\left(\mathcal{Q},L^{p},Y\right)\hookrightarrow W^{s,q}\left(\mathbb{R}^{d}\right)$
for some $q\in\left(1,\infty\right)$, equation (\ref{eq:SobolevBesovInclusionSoft})
yields
\[
\mathcal{D}\left(\mathcal{Q},L^{p},Y\right)\hookrightarrow\mathcal{B}_{s-\varepsilon}^{q,r}\left(\mathbb{R}^{d}\right)\text{ for all }r\in\left(0,\infty\right]\text{ and }\varepsilon>0.
\]
Conversely, if $\mathcal{D}\left(\mathcal{Q},L^{p},Y\right)\hookrightarrow\mathcal{B}_{s+\varepsilon}^{q,r}\left(\mathbb{R}^{d}\right)$
holds for some $s\in\mathbb{N}_{0}$, $q\in\left(1,\infty\right)$,
$r\in\left(0,\infty\right]$ and $\varepsilon>0$, then also $\mathcal{D}\left(\mathcal{Q},L^{p},Y\right)\hookrightarrow W^{s,q}\left(\mathbb{R}^{d}\right)$.

In particular, since the Besov spaces are decomposition spaces, i.e.\@
$\mathcal{B}_{s}^{p,q}\left(\mathbb{R}^{d}\right)=\mathcal{D}\left(\mathcal{P},L^{p},\ell_{u^{\left(s\right)}}^{q}\right)$
for a certain dyadic covering $\mathcal{P}=\left(P_{j}\right)_{j\in\mathbb{N}_{0}}$
of $\mathbb{R}^{d}$ (cf. example \ref{exa:InhomogeneousBesovSpaces}),
many of the results for embeddings of the form
\[
\mathcal{D}\left(\mathcal{Q},L^{p_{1}},\ell_{u}^{q_{1}}\right)\hookrightarrow\mathcal{D}\left(\mathcal{P},L^{p_{2}},\ell_{u^{\left(s\right)}}^{q_{2}}\right)=\mathcal{B}_{s}^{p_{2},q_{2}}\left(\mathbb{R}^{d}\right)
\]
from \cite{VoigtlaenderPhDThesis} can be used to establish embeddings
between decomposition spaces and Sobolev spaces. In the following,
we will summarize arguments of this type under the term ``\textbf{Besov
detour}'', since to establish necessary/sufficient conditions for
embedding into Sobolev spaces, we are making a ``detour'' through
Besov spaces.

In comparison to the results in this paper, the ``Besov detour''
differs as follows:
\begin{itemize}
\item The identity $F_{k}^{q,2}\left(\mathbb{R}^{d}\right)=W^{k,q}\left(\mathbb{R}^{d}\right)$
can fail for $q\notin\left(1,\infty\right)$, so that the ``Besov
detour'' is not applicable.
\item As seen above, the ``detour'' via Besov spaces can yield sufficient
criteria and necessary criteria for existence of the embedding $\mathcal{D}\left(\mathcal{Q},L^{p},Y\right)\hookrightarrow W^{k,q}\left(\mathbb{R}^{d}\right)$,
but these criteria will not be sharp in general. For example, usage
of equation (\ref{eq:SobolevBesovInclusionSoft}) causes a loss of
an (arbitrary) $\varepsilon>0$ in the smoothness parameter $s$.
Thus, embedding results obtained in this way will -- at least for
$q\in\left(0,2\right]\cup\left\{ \infty\right\} $ -- be inferior
to those obtained using the results in the present paper. For $q\in\left(2,\infty\right)$,
however, we will see some examples where the ``Besov detour'' yields
better results than those from the present paper.
\item For the results from \cite{VoigtlaenderPhDThesis} to be applicable,
the covering $\mathcal{Q}$ has to fulfill certain geometric properties.
Roughly speaking, $\mathcal{Q}$ has to be finer than the dyadic covering
$\mathcal{P}$ (or vice versa). More precisely, $\mathcal{Q}$ has
to be almost subordinate to $\mathcal{P}$ or vice versa, see \cite[Definition 3.3.1]{VoigtlaenderPhDThesis}.
But this is not fulfilled in all cases: For example, the covering
$\mathcal{Q}_{H^{\left(c\right)}}$ -- which is induced by the shearlet-type
group $H^{\left(c\right)}$ from equation (\ref{eq:ShearletTypeGroup})
-- is \emph{not} almost subordinate to the dyadic covering $\mathcal{P}$
for $c\in\mathbb{R}\setminus\left[0,1\right]$. Since $\mathcal{P}$
is also not almost subordinate to $\mathcal{Q}_{H^{\left(c\right)}}$
for this range of $c$, the embedding results from \cite{VoigtlaenderPhDThesis}
are not applicable in this case.

In contrast to the results from \cite{VoigtlaenderPhDThesis}, in
the present paper, no ``compatibility'' between the covering $\mathcal{Q}$
and the dyadic covering $\mathcal{P}$ is required.

\item Even if $\mathcal{Q}=\left(Q_{i}\right)_{i\in I}$ is almost subordinate
to the dyadic covering $\mathcal{P}$, for the verification of the
criteria given in \cite{VoigtlaenderPhDThesis}, one first has to
compute the so-called \textbf{intersection sets}
\[
I_{j}=\left\{ i\in I\with Q_{i}\cap P_{j}\neq\emptyset\right\} 
\]
for each $j\in\mathbb{N}_{0}$. Then, one has to check finiteness
of an expression of the form
\[
\left\Vert \left(\left\Vert \left(w_{i}\right)_{i\in I_{j}}\right\Vert _{\ell^{t}\left(I_{j}\right)}\right)_{j\in J}\right\Vert _{\ell^{r}\left(\mathbb{N}_{0}\right)}\text{ for certain }r,t\in\left(0,\infty\right]\text{ and a certain weight }\left(w_{i}\right)_{i\in I}.
\]
In comparison, the criteria developed in the present paper will be
\emph{much} more convenient to verify.
\end{itemize}

\subsection{Structure of the paper}

We begin our exposition in Section \ref{sec:DecompositionSpaces}
by clarifying the assumptions on the covering $\mathcal{Q}$ and reviewing
the definitions of decomposition spaces and Sobolev spaces. In particular,
we formally introduce the class of \textbf{regular coverings} which
was already used in Theorem \ref{thm:IntroductionMainTheorem}. We
remark that for the Quasi-Banach regime $q\in\left(0,1\right)$, it
seems that there is no general consensus on how the Sobolev spaces
$W^{k,q}\left(\mathbb{R}^{d}\right)$ (for $k\geq1$) should be defined.
Thus, readers interested in this case should pay close attention
to the definition of these spaces which we adopt.

The development of criteria for embeddings of decomposition spaces
into Sobolev spaces begins in Section \ref{sec:SufficientConditions},
where we show that the embedding $Y\hookrightarrow\ell_{u^{\left(k,p,q\right)}}^{q^{\triangledown}}\left(I\right)$
indeed suffices for the existence of such an embedding. Necessity
of the (slightly different) embedding $Y\cap\ell_{0}\left(I\right)\hookrightarrow\ell_{u^{\left(k,p,q\right)}}^{q}\left(I\right)$
is established in Section \ref{sec:NecessaryConditions}.

In view of these criteria, it is desirable to have a painless way
of deciding whether an embedding of the form $Y\hookrightarrow\ell_{u}^{s}\left(I\right)$
is actually true. For the case of a weighted Lebesgue sequence space
$Y=\ell_{v}^{r}\left(I\right)$, this problem is solved completely
in Section \ref{sec:SimplifiedConditions}, where we show (cf.\@
also equation (\ref{eq:IntroductionSequenceEmbeddingCharacterization})
above) that all one has to check is finiteness of a certain $\ell^{t}\left(I\right)$-norm
of a \emph{single} sequence.

We complete our abstract results in Section \ref{sec:EmbeddingsIntoBV},
where we show that a decomposition space embeds into a BV space if
and only if it embeds into $W^{1,1}\left(\mathbb{R}^{d}\right)$.
Thus, our previous criteria yield a complete characterization of when
this is true.

Finally, we illustrate our results by considering embeddings into
Sobolev spaces of several classes of decomposition spaces, namely
of (homogeneous and inhomogeneous) Besov spaces, ($\alpha$)-modulation
spaces, Shearlet smoothness spaces, Shearlet-type coorbit spaces and
coorbit spaces of the diagonal group.

\subsection{Notation}

\label{sub:Notation}In this paper, we use the convention
\[
\mathcal{F}f\left(\xi\right):=\widehat{f}\left(\xi\right):=\int_{\mathbb{R}^{d}}f\left(x\right)\cdot e^{-2\pi i\left\langle x,\xi\right\rangle }\,{\rm d}x
\]
for the \textbf{Fourier transform} of a function $f\in L^{1}\left(\mathbb{R}^{d}\right)$.
As is well known (see \cite[Theorem 8.29]{FollandRA}), with this
normalization, the Fourier transform extends to a unitary automorphism
of $L^{2}\left(\mathbb{R}^{d}\right)$, where the inverse is the unique
extension to $L^{2}\left(\mathbb{R}^{d}\right)$ of the inverse Fourier
transform given by
\[
\mathcal{F}^{-1}f\left(x\right):=f^{\vee}\left(x\right)=\widehat{f}\left(-x\right)
\]
for $f\in L^{1}\left(\mathbb{R}^{d}\right)$.

For $n\in\mathbb{N}_{0}$, we write $\underline{n}:=\left\{ k\in\mathbb{N}\with k\leq n\right\} $.
In particular, $\underline{0}=\emptyset$. We denote the usual standard
basis of $\mathbb{R}^{d}$ by $e_{1},\dots,e_{d}$. For a matrix $A\in\mathbb{R}^{d\times d}$,
we write
\[
\left\Vert A\right\Vert :=\max_{\left|x\right|=1}\left|Ax\right|,
\]
where (as in the remainder of the paper), we write $\left|x\right|$
for the usual euclidean norm of a vector $x\in\mathbb{R}^{d}$.

For an integrability exponent $p\in\left(0,\infty\right]$, we define
its \textbf{conjugate exponent} $p'\in\left[1,\infty\right]$ by
\[
p':=\begin{cases}
p', & \text{if }p\in\left[1,\infty\right],\\
\infty, & \text{if }p\in\left(0,1\right),
\end{cases}
\]
where for $p\in\left[1,\infty\right]$, $p'$ satisfies $\frac{1}{p}+\frac{1}{p'}=1$.
Furthermore, we define the \textbf{lower conjugate exponent} of $p$
by
\[
p^{\triangledown}:=\min\left\{ p,p'\right\} .
\]

For a function $f:\mathbb{R}^{d}\to\mathbb{C}$ and $x,\omega\in\mathbb{R}^{d}$,
we define
\[
\left\Vert f\right\Vert _{\sup}=\sup_{y\in\mathbb{R}^{d}}\left|f\left(y\right)\right|
\]
and
\begin{align*}
L_{x}f & :\mathbb{R}^{d}\to\mathbb{C},y\mapsto f\left(y-x\right),\\
M_{\omega}f & :\mathbb{R}^{d}\to\mathbb{C},y\mapsto e^{2\pi i\left\langle \omega,y\right\rangle }\cdot f\left(y\right).
\end{align*}

We denote by $\mathcal{S}\left(\mathbb{R}^{d}\right)$ the space of
\textbf{Schwartz functions} on $\mathbb{R}^{d}$. Its topological
dual space $\mathcal{S}'\left(\mathbb{R}^{d}\right)$ is the space
of \textbf{tempered distributions} on $\mathbb{R}^{d}$. Furthermore,
we denote by $\mathcal{D}\left(\mathcal{O}\right):=C_{c}^{\infty}\left(\mathcal{O}\right)$
the space of $C^{\infty}$ functions $f:\mathcal{O}\to\mathbb{C}$
with compact support in the open set $\emptyset\neq\mathcal{O}\subset\mathbb{R}^{d}$.
With a suitable topology (cf.\@ \cite[Definition 6.3]{RudinFA}),
this space becomes a locally convex topological vector space such
that its topological dual coincides with the space of \textbf{distributions}
$\mathcal{D}'\left(\mathcal{O}\right)$ on $\mathcal{O}$ as defined
(without introducing the topology just mentioned) for example in \cite[Section 9.1]{FollandRA}.
We generally equip $\mathcal{D}'\left(\mathcal{O}\right)$ with the
weak-$\ast$-topology, i.e.\@ with the topology of pointwise convergence
on $\mathcal{D}\left(\mathcal{O}\right)$.

Finally, if $I\neq\emptyset$ is an index set, $q\in\left(0,\infty\right]$
and if $u=\left(u_{i}\right)_{i\in I}$ with $u_{i}>0$ for all $i\in I$
is a weight on $I$, we define the \textbf{weighted Lebesgue sequence
space} $\ell_{u}^{q}\left(I\right)$ as
\[
\ell_{u}^{q}\left(I\right):=\left\{ \left(x_{i}\right)_{i\in I}\in\mathbb{C}^{I}\with\left(u_{i}x_{i}\right)_{i\in I}\in\ell^{q}\left(I\right)\right\} ,
\]
with $\left\Vert \left(x_{i}\right)_{i\in I}\right\Vert _{\ell_{u}^{q}}=\left\Vert \left(u_{i}x_{i}\right)_{i\in I}\right\Vert _{\ell^{q}}$.

\section{Decomposition Spaces and Sobolev spaces}

\label{sec:DecompositionSpaces}In this section, we introduce the
notion of decomposition spaces and Sobolev spaces that we will use.
In Subsection \ref{sub:SobolevSpaces}, we begin by explaining our
convention regarding the Sobolev spaces $W^{k,q}\left(\mathbb{R}^{d}\right)$.
For $q\in\left[1,\infty\right]$, this definition is entirely standard,
but for $q\in\left(0,1\right)$, the situation changes dramatically.
Readers only interested in the case $q\in\left[1,\infty\right]$ can
safely skip this first subsection if they are familiar with the usual
definition of Sobolev spaces.

As we will see in Subsection \ref{sub:DecompositionSpaces}, in order
to obtain well-defined decomposition spaces $\mathcal{D}\left(\mathcal{Q},L^{p},Y\right)$,
we have to impose certain assumptions on the covering $\mathcal{Q}$.
These different assumptions are discussed in detail in Subsection
\ref{sub:Coverings}. In particular, we introduce the new notion of
\textbf{regular coverings} with which we will mainly work in the remainder
of the paper.

As mentioned above, in Subsection \ref{sub:DecompositionSpaces},
we recall the definition of the decomposition space $\mathcal{D}\left(\mathcal{Q},L^{p},Y\right)$.
Our definition -- which is based on \cite{VoigtlaenderPhDThesis}
and \cite{DecompositionEmbeddings} -- is slightly different from
the usual one, e.g.\@ as in \cite{BorupNielsenDecomposition}. The
main difference is that we use a reservoir different from the space
$\mathcal{S}'\left(\mathbb{R}^{d}\right)$ to define our decomposition
spaces. The reason for this is twofold: First, we want to allow coverings
$\mathcal{Q}$ which cover a proper subset $\mathcal{O}\subsetneq\mathbb{R}^{d}$
and second, with the usual definition, it can happen that the resulting
decomposition space is \emph{not} complete.

In the final subsection, we recall from \cite{VoigtlaenderPhDThesis,DecompositionEmbeddings}
and \cite{TriebelTheoryOfFunctionSpaces} some results concerning
convolution in the Quasi-Banach regime $q\in\left(0,1\right)$ which
we will need. In a nutshell, the problem is that Young's convolution
relation $L^{1}\ast L^{q}\hookrightarrow L^{q}$ fails completely
for $q\in\left(0,1\right)$. Instead, we get an estimate of the form
\[
\left\Vert f\ast g\right\Vert _{L^{q}}\leq C\cdot\left\Vert f\right\Vert _{L^{q}}\cdot\left\Vert g\right\Vert _{L^{q}},
\]
but only under the additional assumption that the Fourier supports
${\rm supp}\,\widehat{f}\subset Q_{1}$ and ${\rm supp}\,\widehat{g}\subset Q_{2}$
are compact. Furthermore, the constant $C$ will depend in a nontrivial
way on the sets $Q_{1},Q_{2}$. Again, readers who are only interested
in the case $q\in\left[1,\infty\right]$ may safely skip this subsection,
apart from Corollary \ref{cor:BandlimitedLpEmbeddingSemiStructured}.
Note that for $p\in\left[1,\infty\right]$, the proof of this Corollary
is independent of the rest of Subsection \ref{sub:QuasiBanachConvolution}.

\subsection{The Sobolev spaces $W^{k,q}\left(\mathbb{R}^{d}\right)$}

\label{sub:SobolevSpaces}The definition of the Sobolev spaces $W^{k,q}\left(\mathbb{R}^{d}\right)$
for $q\in\left[1,\infty\right]$ is entirely standard, i.e.\@ we
define
\[
W^{k,q}\left(\mathbb{R}^{d}\right):=\left\{ f\in L^{q}\left(\mathbb{R}^{d}\right)\with\partial^{\alpha}f\in L^{q}\left(\mathbb{R}^{d}\right)\text{ for all }\alpha\in\mathbb{N}_{0}^{d}\text{ with }\left|\alpha\right|\leq k\right\} .
\]
Here, the partial derivative $\partial^{\alpha}f$ denotes the distributional
derivative of $f$, which is well-defined, since every $f\in L^{q}\left(\mathbb{R}^{d}\right)\subset\mathcal{D}'\left(\mathbb{R}^{d}\right)$
defines a distribution -- in fact even a tempered distribution. Note
that we crucially use $q\in\left[1,\infty\right]$ for the inclusion
$L^{q}\left(\mathbb{R}^{d}\right)\subset\mathcal{D}'\left(\mathbb{R}^{d}\right)$
to be true. To be sure, $\partial^{\alpha}f\in L^{q}\left(\mathbb{R}^{d}\right)$
means that there is a (uniquely determined) function $f_{\alpha}\in L^{q}\left(\mathbb{R}^{d}\right)$
such that
\begin{align*}
\left(-1\right)^{\left|\alpha\right|}\int_{\mathbb{R}^{d}}f\left(x\right)\cdot\partial^{\alpha}\varphi\left(x\right)\,{\rm d}x & =\left(-1\right)^{\left|\alpha\right|}\left\langle f,\partial^{\alpha}\varphi\right\rangle \\
 & =\left\langle \partial^{\alpha}f,\varphi\right\rangle \\
 & =\int_{\mathbb{R}^{d}}f_{\alpha}\left(x\right)\cdot\varphi\left(x\right)\,{\rm d}x
\end{align*}
holds for all $\varphi\in C_{c}^{\infty}\left(\mathbb{R}^{d}\right)$.
In this case, we simply write $\partial^{\alpha}f$ instead of $f_{\alpha}$.

Finally, we equip $W^{k,q}\left(\mathbb{R}^{d}\right)$ with the norm
\[
\left\Vert f\right\Vert _{W^{k,q}}:=\sum_{\substack{\alpha\in\mathbb{N}_{0}^{d}\\
\left|\alpha\right|\leq k
}
}\left\Vert \partial^{\alpha}f\right\Vert _{L^{q}}
\]
which makes it a Banach space. Note that since the weak partial derivatives
$\partial^{\alpha}f$ are uniquely determined by $f$, the inclusion
map
\[
\iota:W^{k,q}\left(\mathbb{R}^{d}\right)\to L^{q}\left(\mathbb{R}^{d}\right),f\mapsto f
\]
is injective.

In case of $q\in\left(0,1\right)$, we can \emph{not} proceed as above,
since in this case $L^{q}\left(\mathbb{R}^{d}\right)\nsubseteq L_{{\rm loc}}^{1}\left(\mathbb{R}^{d}\right)$,
so that a function $f\in L^{q}\left(\mathbb{R}^{d}\right)$ does not
define a distribution in general. Hence, we cannot define the (weak)
derivative using distribution theory. Instead, we proceed as follows:
For $q\in\left(0,1\right)$, we define $W^{k,q}\left(\mathbb{R}^{d}\right)$
as the closure of
\[
W_{\ast}^{k,q}\left(\mathbb{R}^{d}\right):=\left\{ \left(\partial^{\alpha}f\right)_{\alpha\in\mathbb{N}_{0}^{d},\left|\alpha\right|\leq k}\with f\in C^{\infty}\left(\mathbb{R}^{d}\right)\text{ with }\partial^{\alpha}f\in L^{q}\left(\mathbb{R}^{d}\right)\text{ for all }\alpha\in\mathbb{N}_{0}^{d}\text{ with }\left|\alpha\right|\leq k\right\} 
\]
in the product $\prod_{\alpha\in\mathbb{N}_{0}^{d},\left|\alpha\right|\leq k}L^{q}\left(\mathbb{R}^{d}\right)$.
Here, $\partial^{\alpha}f$ denotes the classical derivative of $f\in C^{\infty}\left(\mathbb{R}^{d}\right)$.
Note that if we were to adopt the same definition of $W^{k,q}\left(\mathbb{R}^{d}\right)$
also for $q\in\left[1,\infty\right)$, we would obtain the same spaces
as defined above (up to obvious identifications).

We finally remark that the space $W^{k,q}\left(\mathbb{R}^{d}\right)$
for $q\in\left(0,1\right)$ -- as defined above -- behaves quite pathologically
in some respects. Indeed, in \cite{PeetreSobolevQuasiBanach} and
in the related paper \cite{SobolevQuasiBanachCorrectionToPeetre},
it is shown that the ``inclusion'' map
\[
W^{k,q}\left(\mathbb{R}^{d}\right)\to L^{q}\left(\mathbb{R}^{d}\right),\left(f_{\alpha}\right)_{\alpha\in\mathbb{N}_{0}^{d},\left|\alpha\right|\leq k}\mapsto f_{0}
\]
is \emph{not} injective for $k\geq1$. Furthermore, the dual space
of $W^{k,q}\left(\mathbb{R}^{d}\right)$ is trivial. There are other
(\emph{nonequivalent}) conventions for defining $W^{k,q}\left(\mathbb{R}^{d}\right)$
for $q\in\left(0,1\right)$, but the present formulation will turn
out to be most convenient for the results in this paper.

In any case, for $k=0$, we have $W^{k,q}\left(\mathbb{R}^{d}\right)=L^{q}\left(\mathbb{R}^{d}\right)$
for all $q\in\left(0,\infty\right]$. For $q\in\left[1,\infty\right]$,
this is clear and for $q\in\left(0,1\right)$, we use that the set
of simple functions of the form $\sum_{i=1}^{n}\alpha_{i}\chi_{A_{i}}$
with measurable, bounded sets $A_{i}\subset\mathbb{R}^{d}$ is dense
in $L^{q}\left(\mathbb{R}^{d}\right)$. Now, for any indicator function
$\chi_{A}$ with measurable bounded $A\subset\mathbb{R}^{d}$, we
can find (e.g.\@ by density of $C_{c}^{\infty}\left(\mathbb{R}^{d}\right)$
in $L^{1}\left(\mathbb{R}^{d}\right)$) a sequence $\left(f_{n}\right)_{n\in\mathbb{N}}$
in $C_{c}^{\infty}\left(\mathbb{R}^{d}\right)$ with $f_{n}\left(x\right)\to\chi_{A}\left(x\right)$
almost everywhere and such that $-10\cdot\chi_{B}\leq f_{n}\leq10\cdot\chi_{B}$
for all $n\in\mathbb{N}$ and some fixed bounded set $B$. Using the
dominated convergence theorem, this implies $f_{n}\to\chi_{A}$ in
$L^{q}\left(\mathbb{R}^{d}\right)$.

We had to use this somewhat involved argument, since Young's convolution
relation $L^{1}\ast L^{q}\hookrightarrow L^{q}$ fails for $q\in\left(0,1\right)$.
Hence, approximating $L^{q}$ functions by convolution with an approximate
identity fails for $q\in\left(0,1\right)$. A more detailed discussion
of the failure of Young's inequality for $q\in\left(0,1\right)$ will
be given in Subsection \ref{sub:QuasiBanachConvolution}.

\subsection{Structured, semi-structured and regular coverings}

\label{sub:Coverings}In this subsection, we introduce several different
classe of coverings. All our coverings will always be of the form
\[
\mathcal{Q}=\left(Q_{i}\right)_{i\in I}=\left(T_{i}Q_{i}'+b_{i}\right)_{i\in I}
\]
for suitable subsets $Q_{i}'\subset\mathbb{R}^{d}$, invertible matrices
$T_{i}\in{\rm GL}\left(\mathbb{R}^{d}\right)$ and shifts $b_{i}\in\mathbb{R}^{d}$.
Furthermore, we always assume that the set $\mathcal{O}=\bigcup_{i\in I}Q_{i}\subset\mathbb{R}^{d}$
is fixed, i.e.\@ $\mathcal{Q}$ will always be a covering of the
set $\mathcal{O}$. For the sake of brevity, we will not repeat these
assumptions every time.

It is most important to keep in mind that the covering $\mathcal{Q}$
is a covering of (the subset $\mathcal{O}$ of) the \emph{frequency
domain} $\mathbb{R}^{d}$, and not of the space domain $\mathbb{R}^{d}$.

The type of covering which is easiest to understand is that of a \textbf{structured
admissible covering}, essentially as introduced by Borup and Nielsen
in \cite{BorupNielsenDecomposition}. The only difference between
their definition and ours is that we allow coverings of proper subsets
$\mathcal{O}\subsetneq\mathbb{R}^{d}$, whereas Borup and Nielsen
only consider coverings of the whole frequency space $\mathbb{R}^{d}$.
\begin{defn}
\label{def:StructuredAdmissibleCovering}The covering $\mathcal{Q}$
is called an \textbf{admissible covering} of $\mathcal{O}$ if $Q_{i}\neq\emptyset$
for all $i\in I$ and if the constant
\[
N_{\mathcal{Q}}:=\sup_{i\in I}\left|i^{\ast}\right|
\]
is finite, where
\[
i^{\ast}:=\left\{ j\in I\with Q_{i}\cap Q_{j}\neq\emptyset\right\} 
\]
denotes the \textbf{set of $\mathcal{Q}$-neighbors} of the index
$i\in I$.

An admissible covering $\mathcal{Q}$ is called a \textbf{structured
admissible covering} of $\mathcal{O}$, if the following hold:
\begin{enumerate}
\item We have $Q_{i}'=Q$ for all $i\in I$, where $Q\subset\mathbb{R}^{d}$
is a fixed open, bounded set.
\item There is a an open set $P\subset Q$ with $\overline{P}\subset Q$
and with
\[
\bigcup_{i\in I}\left(T_{i}P+b_{i}\right)=\mathcal{O}.
\]

\item We have
\begin{equation}
C_{\mathcal{Q}}:=\sup_{i\in I}\sup_{j\in i^{\ast}}\left\Vert T_{i}^{-1}T_{j}\right\Vert <\infty.\qedhere\label{eq:NormalizationConstantDefinition}
\end{equation}

\end{enumerate}
\end{defn}
\begin{rem*}
The following notation related to the set of $\mathcal{Q}$-neighbors
will be frequently convenient: For $M\subset I$, we define
\[
M^{\ast}:=\bigcup_{\ell\in M}\ell^{\ast}\subset I.
\]
Now, we inductively define $M^{0\ast}:=M$ and $M^{\left(n+1\right)\ast}:=\left(M^{n\ast}\right)^{\ast}$
for $n\in\mathbb{N}_{0}$. Finally, we set $i^{n\ast}:=\left\{ i\right\} ^{n\ast}$
and 
\[
Q_{i}^{n\ast}:=\bigcup_{\ell\in i^{n\ast}}Q_{\ell}
\]
for $i\in I$ and $n\in\mathbb{N}_{0}$.

Now, in words, admissibility of a covering means that the number of
neighbors of a set $Q_{i}$ of the covering $\mathcal{Q}$ is uniformly
bounded. For a structured admissible covering, we additionally assume
all sets $Q_{i}$ to be ``of a similar shape/form'', in the sense
that every set $Q_{i}$ is of the form $Q_{i}=T_{i}Q+b_{i}$ for a
fixed set $Q\subset\mathbb{R}^{d}$. Furthermore, we assume that we
can shrink the set $Q$ slightly, while still covering all of $\mathcal{O}$.
This assumption -- together with the technical condition $C_{\mathcal{Q}}<\infty$
-- ensures existence of suitable partitions of unity subordinate to
$\mathcal{Q}$, see Theorem \ref{thm:StructuredAdmissibleCoveringsAreRegular}
below.
\end{rem*}
In some cases, the notion of a structured admissible covering turns
out to be too restrictive. Thus, in \cite[Definition 3.8]{VoigtlaenderPhDThesis},
I introduced the notion of a \emph{semi}-structured admissible covering
for which the assumption $Q_{i}'=Q$ for all $i\in I$ is dropped:
\begin{defn}
\label{def:SemiStructuredCovering}The covering $\mathcal{Q}$ is
called a \textbf{semi-structured covering} of $\mathcal{O}$ if the
following conditions hold:
\begin{enumerate}
\item $\mathcal{Q}$ is admissible,
\item The set $\bigcup_{i\in I}Q_{i}'\subset\mathbb{R}^{d}$ is bounded,
\item The constant $C_{\mathcal{Q}}$ as defined in equation (\ref{eq:NormalizationConstantDefinition})
is finite.
\end{enumerate}
Finally, we say that the semi-structured covering $\mathcal{Q}$ is
\textbf{tight} if we additionally have the following:
\begin{enumerate}[resume]
\item There is some $\varepsilon>0$ such that for each $i\in I$, there
is some $c_{i}\in\mathbb{R}^{d}$ with $B_{\varepsilon}\left(c_{i}\right)\subset Q_{i}'$.\qedhere
\end{enumerate}
\end{defn}
\begin{rem*}
Note that every structured admissible covering is a tight semi-structured
admissible covering.
\end{rem*}
As noted above, the definition of a structured admissible covering
ensures existence of certain partitions of unity subordinate to the
covering. The special type of partitions of unity which we introduce
now will turn out to be suitable for defining decomposition spaces,
cf.\@ Subsection \ref{sub:DecompositionSpaces}.
\begin{defn}
\label{def:LpBAPUs}(cf.\@ \cite[Definition 2]{BorupNielsenDecomposition})

Let $\mathcal{Q}$ be a semi-structured admissible covering of $\mathcal{O}$.
We say that $\Phi=\left(\varphi_{i}\right)_{i\in I}$ is a partition
of unity subordinate to $\mathcal{Q}$ if the following hold:
\begin{enumerate}
\item $\varphi_{i}\in C_{c}^{\infty}\left(\mathcal{O}\right)$ for all $i\in I$,
\item $\varphi_{i}\equiv0$ on $\mathbb{R}^{d}\setminus Q_{i}$ for all
$i\in I$,
\item $\sum_{i\in I}\varphi_{i}\equiv1$ on $\mathcal{O}$.
\end{enumerate}
Furthermore, for $p\in\left[1,\infty\right]$, we say that $\Phi$
is an \textbf{$L^{p}$-BAPU} (bounded admissible partition of unity)
for $\mathcal{Q}$, if $\Phi$ is a partition of unity subordinate
to $\mathcal{O}$ for which the constant
\[
C_{\Phi,p}:=\sup_{i\in I}\left\Vert \mathcal{F}^{-1}\varphi_{i}\right\Vert _{L^{1}}
\]
is finite. For $p\in\left(0,1\right)$, we instead require finiteness
of
\[
C_{\Phi,p}:=\sup_{i\in I}\left|\det T_{i}\right|^{\frac{1}{p}-1}\cdot\left\Vert \mathcal{F}^{-1}\varphi_{i}\right\Vert _{L^{p}}.
\]

Finally, we say that $\mathcal{Q}$ is an \textbf{$L^{p}$-decomposition
covering} if there is an $L^{p}$-BAPU for $\mathcal{Q}$.\end{defn}
\begin{rem*}
The term ``$L^{p}$-BAPU'' does \emph{not} refer to the fact that
the $L^{p}$ norm of $\varphi_{i}$ or $\mathcal{F}^{-1}\varphi_{i}$
is uniformly bounded. Instead, the point is that the $\left(\varphi_{i}\right)_{i\in I}$
define a uniformly bounded family of $L^{p}$ Fourier multipliers.
For $p\in\left[1,\infty\right]$, this is a direct consequence of
Young's inequality, whereas for $p\in\left(0,1\right)$, this statement
needs to be taken with a grain of salt, as explained in Subsection
\ref{sub:QuasiBanachConvolution}.
\end{rem*}
While existence of an $L^{p}$-BAPU is sufficient for obtaining well-defined
decomposition spaces (cf.\@ Subsection \ref{sub:DecompositionSpaces}),
we will need to impose more restrictive conditions on the partition
of unity $\Phi$ in order to establish embeddings into Sobolev spaces.
Our next definition explains exactly which properties $\Phi$ needs
to have.
\begin{defn}
\label{def:RegularPartitionOfUnity}Let $\mathcal{Q}$ be a semi-structured
covering of $\mathcal{O}$ and let $\Phi=\left(\varphi_{i}\right)_{i\in I}$
be a partition of unity subordinate to $\mathcal{Q}$. For $i\in I$,
define the \textbf{normalized version} of $\varphi_{i}$ by
\[
\varphi_{i}^{\#}:\mathbb{R}^{d}\to\mathbb{C},\xi\mapsto\varphi_{i}\left(T_{i}\xi+b_{i}\right).
\]

We say that $\Phi$ is a \textbf{regular partition of unity} subordinate
to $\mathcal{Q}$ if $\varphi_{i}\in C_{c}^{\infty}\left(Q_{i}\right)$
for all $i\in I$ and if additionaly
\[
C_{\Phi,\alpha}:=\sup_{i\in I}\left\Vert \partial^{\alpha}\varphi_{i}^{\#}\right\Vert _{\sup}
\]
is finite for \emph{all} $\alpha\in\mathbb{N}_{0}^{d}$.

The covering $\mathcal{Q}$ is called a \textbf{regular covering}
of $\mathcal{O}$ if there exists a regular partition of unity $\Phi$
subordinate to $\mathcal{Q}$.
\end{defn}
We will now show that every regular partition of unity is also an
$L^{p}$-BAPU for all $p\in\left(0,\infty\right]$. In fact, we establish
a slightly stronger claim.
\begin{lem}
\label{lem:SufficientConditionDerivativeEstimate}Let $\mathcal{Q}$
be a semi-structured covering of $\mathcal{O}\subset\mathbb{R}^{d}$
and let $\left(\gamma_{i}\right)_{i\in I}$ be a family in $C_{c}^{\infty}\left(\mathcal{O}\right)$
with $\gamma_{i}\equiv0$ on $\mathcal{O}\setminus Q_{i}$ for all
$i\in I$ and so that the normalized family $\left(\smash{\gamma_{i}^{\#}}\right)_{i\in I}$
given by
\[
\gamma_{i}^{\#}:\mathbb{R}^{d}\to\mathbb{C},\xi\mapsto\gamma_{i}\left(T_{i}\xi+b_{i}\right)
\]
satisfies
\begin{equation}
C_{\alpha}:=\sup_{i\in I}\left\Vert \partial^{\alpha}\gamma_{i}^{\#}\right\Vert _{\sup}<\infty\label{eq:UniformControlOverDerivativesConstant}
\end{equation}
for all $\alpha\in\mathbb{N}_{0}^{d}$.

Then we have
\[
\left\Vert \partial^{\alpha}\left[\mathcal{F}^{-1}\gamma_{i}\right]\right\Vert _{L^{p}}\leq K_{\alpha}\cdot\left|\det T_{i}\right|^{1-\frac{1}{p}}\cdot\left(\left\Vert T_{i}\right\Vert ^{\left|\alpha\right|}+\left|b_{i}\right|^{\left|\alpha\right|}\right)
\]
for all $\alpha\in\mathbb{N}_{0}^{d}$, $p\in\left(0,\infty\right]$
and $i\in I$, where the constant $K_{\alpha}=K_{\alpha}\left(d,p,\mathcal{Q},\left(\gamma_{i}\right)_{i\in I}\right)$
is independent of $i\in I$.\end{lem}
\begin{rem*}
Actually, the proof establishes the stronger estimate
\[
\left|\left(\partial^{\alpha}\left[\mathcal{F}^{-1}\gamma_{i}\right]\right)\left(x\right)\right|\leq C_{\alpha,N}\cdot\left|\det T_{i}\right|\left(1+\left|T_{i}^{T}x\right|\right)^{-N}\cdot\left(\left|b_{i}\right|^{\left|\alpha\right|}+\left\Vert T_{i}\right\Vert ^{\left|\alpha\right|}\right)\qquad\text{ for all }x\in\mathbb{R}^{d}
\]
for all $i\in I$ and arbitrary $N\in\mathbb{N}$ for some constant
$C_{\alpha,N}=C_{\alpha,N}\left(\mathcal{Q},\left(\gamma_{i}\right)_{i\in I}\right)$
which is independent of $i\in I$.
\end{rem*}

\begin{proof}
For brevity, we define $\varrho_{i}:=\mathcal{F}^{-1}\gamma_{i}^{\#}$
and $\theta_{i}:=\mathcal{F}^{-1}\gamma_{i}$ for $i\in I$. We begin
with showing that for every $N\in\mathbb{N}$ and $\gamma\in\mathbb{N}_{0}^{d}$,
there is a constant $K_{\gamma,N}=K_{\gamma,N}\left(\mathcal{Q},\left(\gamma_{i}\right)_{i\in I}\right)>0$
with
\begin{equation}
\left|\partial^{\gamma}\varrho_{i}\left(x\right)\right|\leq K_{\gamma,N}\cdot\left(1+\left|x\right|\right)^{-N}\text{ for all }x\in\mathbb{R}^{d}\text{ and }i\in I.\label{eq:NormalizedBAPUWithDerivativesOnSpaceSide}
\end{equation}
Here it is crucial that the constant $K_{\gamma,N}$ is independent
of $i\in I$ and $x\in\mathbb{R}^{d}$. A high-level proof of this
estimate is as follows: Since ${\rm supp}\,\gamma_{i}\subset Q_{i}=T_{i}Q_{i}'+b_{i}$,
we get the uniform inclusion ${\rm supp}\,\gamma_{i}^{\#}\subset Q_{i}'\subset B_{R}\left(0\right)$
for some fixed $R=R\left(\mathcal{Q}\right)>0$ and all $i\in I$.
In conjunction  with the prerequisite from equation (\ref{eq:UniformControlOverDerivativesConstant}),
we conclude that the family $\left(\smash{\gamma_{i}^{\#}}\right)_{i\in I}$
is bounded with respect to each of the norms
\[
\varrho_{N}\left(f\right):=\max_{\substack{\alpha\in\mathbb{N}_{0}^{d}\\
\left|\alpha\right|\leq N
}
}\sup_{x\in\mathbb{R}^{d}}\left[\left(1+\left|x\right|\right)^{N}\cdot\left|\partial^{\alpha}f\left(x\right)\right|\right].
\]
Thus, $\left\{ \smash{\gamma_{i}^{\#}}\with i\in I\right\} \subset\mathcal{S}\left(\mathbb{R}^{d}\right)$
is a bounded subset of the topological vector space $\mathcal{S}\left(\mathbb{R}^{d}\right)$
(see \cite[Section 1.6]{RudinFA} for the relevant definition and
also \cite[Theorem 1.37]{RudinFA} for the equivalent characterization
which we use here). Since $\mathcal{F}:\mathcal{S}\left(\mathbb{R}^{d}\right)\to\mathcal{S}\left(\mathbb{R}^{d}\right)$
is a homeomorphism, we conclude that 
\[
\left\{ \varrho_{i}\with i\in I\right\} =\left\{ \mathcal{F}^{-1}\smash{\gamma_{i}^{\#}}\with i\in I\right\} \subset\mathcal{S}\left(\mathbb{R}^{d}\right)\hookrightarrow L^{p}\left(\mathbb{R}^{d}\right)
\]
is also bounded for each $p\in\left(0,\infty\right]$. In the next
two paragraphs (about $1+\frac{1}{2}$ pages), we provide a more direct
proof of estimate (\ref{eq:NormalizedBAPUWithDerivativesOnSpaceSide}).
Readers who are willing to take this estimate on faith or who are
satisfied with the preceding abstract argument should thus skip these
two paragraphs.

For the explicit proof of equation (\ref{eq:NormalizedBAPUWithDerivativesOnSpaceSide}),
recall that standard properties of the Fourier transform (cf.\@ e.g.\@
\cite[Theorem 8.22]{FollandRA}) yield for each multiindex $\kappa\in\mathbb{N}_{0}^{d}$
that
\begin{align*}
x^{\kappa}\cdot\partial^{\gamma}\varrho_{i}\left(x\right) & =x^{\kappa}\cdot\partial^{\gamma}\left[\mathcal{F}^{-1}\gamma_{i}^{\#}\right]\left(x\right)\\
 & =x^{\kappa}\cdot\mathcal{F}^{-1}\left[\xi\mapsto\left(2\pi j\xi\right)^{\gamma}\cdot\gamma_{i}^{\#}\left(\xi\right)\right]\left(x\right)\\
 & =\left(\frac{j}{2\pi}\right)^{\kappa}\cdot\mathcal{F}^{-1}\left(\xi\mapsto\partial_{\xi}^{\kappa}\left[\left(2\pi j\xi\right)^{\gamma}\cdot\gamma_{i}^{\#}\left(\xi\right)\right]\right)\left(x\right).
\end{align*}
Here, we used the notation $j$ instead of $i$ for the imaginary
unit to avoid possible confusion with the index $i\in I$. Since
$\mathcal{F}^{-1}:L^{1}\left(\mathbb{R}^{d}\right)\to C_{0}\left(\mathbb{R}^{d}\right)$
is bounded, we conclude
\begin{align*}
\left\Vert x^{\kappa}\cdot\partial^{\gamma}\varrho_{i}\right\Vert _{\sup} & \leq\left\Vert \xi\mapsto\partial_{\xi}^{\kappa}\left[\left(2\pi j\xi\right)^{\gamma}\cdot\gamma_{i}^{\#}\left(\xi\right)\right]\right\Vert _{L^{1}\left(\mathbb{R}^{d}\right)}\\
\left(\text{by Leibniz's formula}\right) & =\left\Vert \xi\mapsto\sum_{\substack{\lambda\in\mathbb{N}_{0}^{d}\\
\lambda\leq\kappa
}
}\binom{\kappa}{\lambda}\cdot\left[\partial_{\xi}^{\lambda}\left(2\pi j\xi\right)^{\gamma}\cdot\left(\partial^{\kappa-\lambda}\gamma_{i}^{\#}\right)\left(\xi\right)\right]\right\Vert _{L^{1}\left(\mathbb{R}^{d}\right)}.
\end{align*}

Now, note that we have $\gamma_{i}\left(\xi\right)=0$ for all $\xi\in\mathbb{R}^{d}\setminus Q_{i}=\mathbb{R}^{d}\setminus\left[T_{i}Q_{i}'+b_{i}\right]$
and hence 
\[
\gamma_{i}^{\#}\left(\xi\right)=\gamma_{i}\left(T_{i}\xi+b_{i}\right)=0\text{ for all }\xi\in\mathbb{R}^{d}\setminus Q_{i}'.
\]
Here, we implicitly used that $T_{i}\in{\rm GL}\left(\mathbb{R}^{d}\right)$
is invertible. Since $\mathcal{Q}$ is a semi-structured covering,
there is some $R\geq1$ with $Q_{i}'\subset B_{R}\left(0\right)$
for all $i\in I$ and we conclude ${\rm supp}\,\gamma_{i}^{\#}\subset\overline{B_{R}}\left(0\right)$
and thus also ${\rm supp}\,\partial^{\gamma}\gamma_{i}^{\#}\subset\overline{B_{R}}\left(0\right)$
for all $i\in I$. Now, observe
\[
\left|\partial_{\xi}^{\lambda}\left(2\pi j\xi\right)^{\gamma}\right|=\left|\left(2\pi j\right)^{\gamma}\cdot\partial_{\xi}^{\lambda}\xi^{\gamma}\right|=\begin{cases}
K_{\lambda,\gamma}\cdot\left|\xi^{\gamma-\lambda}\right|\leq K_{\lambda,\gamma}R^{\left|\gamma-\lambda\right|}\leq K_{\gamma}'\cdot R^{\left|\gamma\right|}, & \text{if }\lambda\leq\gamma,\\
0, & \text{else}
\end{cases}
\]
for $\left|\xi\right|\leq R$. Here, $K_{\gamma}'>0$ is an absolute
constant only depending on $\gamma\in\mathbb{N}_{0}^{d}$. Altogether,
we conclude
\begin{align*}
\left\Vert x^{\kappa}\cdot\partial^{\gamma}\varrho_{i}\right\Vert _{\sup} & \leq\left\Vert \xi\mapsto\sum_{\substack{\lambda\in\mathbb{N}_{0}^{d}\\
\lambda\leq\kappa
}
}\binom{\kappa}{\lambda}\cdot\left[\partial_{\xi}^{\lambda}\left(2\pi j\xi\right)^{\gamma}\cdot\left(\partial^{\kappa-\lambda}\gamma_{i}^{\#}\right)\left(\xi\right)\right]\right\Vert _{L^{1}\left(\mathbb{R}^{d}\right)}\\
 & \leq K_{\gamma}'\cdot R^{\left|\gamma\right|}\cdot\sum_{\substack{\lambda\in\mathbb{N}_{0}^{d}\\
\lambda\leq\kappa\text{ and }\lambda\leq\gamma
}
}\left[\binom{\kappa}{\lambda}\cdot\left\Vert \partial^{\kappa-\lambda}\gamma_{i}^{\#}\right\Vert _{L^{1}\left(\overline{B_{R}}\left(0\right)\right)}\right]\\
 & \leq K_{\gamma}'\cdot R^{\left|\gamma\right|}\cdot\sum_{\substack{\lambda\in\mathbb{N}_{0}^{d}\\
\lambda\leq\kappa\text{ and }\lambda\leq\gamma
}
}\left[\binom{\kappa}{\lambda}\cdot\lambda\left(\overline{B_{R}}\left(0\right)\right)\cdot\left\Vert \partial^{\kappa-\lambda}\gamma_{i}^{\#}\right\Vert _{\sup}\right]\\
\left(\text{cf. eq. }\eqref{eq:UniformControlOverDerivativesConstant}\right) & \leq K_{\gamma}'\cdot R^{\left|\gamma\right|}\cdot\lambda\left(\overline{B_{R}}\left(0\right)\right)\cdot\sum_{\substack{\lambda\in\mathbb{N}_{0}^{d}\\
\lambda\leq\kappa\text{ and }\lambda\leq\gamma
}
}\left[\binom{\kappa}{\lambda}\cdot C_{\kappa-\lambda}\right]\\
 & \leq K_{\gamma,\kappa,R},
\end{align*}
where the constant $K_{\gamma,\kappa,R}=K\left(\gamma,\kappa,\mathcal{Q},\left(\gamma_{i}\right)_{i\in I}\right)>0$
is \emph{independent} of $i\in I$. But because of
\begin{align*}
\left(1+\left|x\right|\right)^{N} & \leq\left(1+d\left\Vert x\right\Vert _{\infty}\right)^{N}\\
\left(\text{with suitable }\ell=\ell\left(x\right)\in\underline{d}\right) & \leq d^{N}\cdot\left(1+\left|x_{\ell}\right|\right)^{N}\\
 & =d^{N}\cdot\sum_{m=0}^{N}\binom{N}{m}\left|x_{\ell}\right|^{m}\\
 & =d^{N}\cdot\sum_{m=0}^{N}\binom{N}{m}\left|x^{\ell e_{m}}\right|\\
 & \leq d^{N}\cdot\sum_{m=0}^{N}\binom{N}{m}\sum_{t=1}^{d}\left|x^{te_{m}}\right|,
\end{align*}
this finally implies
\begin{align*}
\left\Vert \left(1+\left|x\right|\right)^{N}\cdot\partial^{\gamma}\varrho_{i}\right\Vert _{\sup} & \leq d^{N}\cdot\sum_{m=0}^{N}\binom{N}{m}\sum_{t=1}^{d}\left\Vert x^{te_{m}}\cdot\partial^{\gamma}\varrho_{i}\right\Vert _{\sup}\\
 & \leq d^{N}\cdot\sum_{m=0}^{N}\binom{N}{m}\sum_{t=1}^{d}K_{\gamma,te_{m},R}\\
 & =:K_{\gamma,N}
\end{align*}
with $K_{\gamma,N}=K\left(\gamma,N,\mathcal{Q},\left(\gamma_{i}\right)_{i\in I}\right)$.
This completes the explicit proof of equation (\ref{eq:NormalizedBAPUWithDerivativesOnSpaceSide}).

Now, note
\[
\gamma_{i}\left(\xi\right)=\gamma_{i}^{\#}\left(T_{i}^{-1}\left(\xi-b_{i}\right)\right)=\left[L_{b_{i}}\left(\gamma_{i}^{\#}\circ T_{i}^{-1}\right)\right]\left(\xi\right)
\]
for all $\xi\in\mathbb{R}^{d}$, where $L_{x}f\left(y\right)=f\left(y-x\right)$
denotes the left-translation of $f$ by $x$.

By standard properties of the Fourier transform (cf.\@ e.g.\@ \cite[Theorem 8.22]{FollandRA}),
this implies
\begin{align}
\theta_{i}=\mathcal{F}^{-1}\gamma_{i} & =M_{b_{i}}\left[\mathcal{F}^{-1}\left(\gamma_{i}^{\#}\circ T_{i}^{-1}\right)\right]\nonumber \\
 & =\left|\det T_{i}\right|\cdot M_{b_{i}}\left[\left(\mathcal{F}^{-1}\gamma_{i}^{\#}\right)\circ T_{i}^{T}\right]\nonumber \\
 & =\left|\det T_{i}\right|\cdot M_{b_{i}}\left(\varrho_{i}\circ T_{i}^{T}\right),\label{eq:OriginalBAPUIntoNormalizedVersionSpaceSide}
\end{align}
where we  recall $\varrho_{i}=\mathcal{F}^{-1}\gamma_{i}^{\#}$ 
and  where $M_{b}f\left(y\right)=e^{2\pi j\left\langle b,y\right\rangle }\cdot f\left(y\right)$
denotes the modulation of $f$ by $b$. Here, we used as above the
notation $j$ for the imaginary unit. Using Leibniz's formula, we
derive
\begin{align}
\partial^{\alpha}\theta_{i} & =\left|\det T_{i}\right|\cdot\partial^{\alpha}\left(e^{2\pi j\left\langle b_{i},\cdot\right\rangle }\cdot\left[\varrho_{i}\circ T_{i}^{T}\right]\right)\nonumber \\
 & =\left|\det T_{i}\right|\cdot\sum_{\beta\leq\alpha}\left[\binom{\alpha}{\beta}\cdot\left(2\pi j\cdot b_{i}\right)^{\alpha-\beta}\cdot e^{2\pi j\left\langle b_{i},\cdot\right\rangle }\cdot\partial^{\beta}\left[\varrho_{i}\circ T_{i}^{T}\right]\right].\label{eq:BAPUDerivativeLpEstimateLeibnizApplication}
\end{align}
We will now estimate each of the summands individually.

To this end, fix $\beta\leq\alpha$. We first consider the case $k:=\left|\beta\right|>0$.
In this case, there are $i_{1},\dots,i_{k}\in\underline{d}$ with
$\beta=\sum_{\ell=1}^{k}e_{i_{\ell}}$, so that formula (\ref{eq:ChainRuleForLinearTransformation})
from Lemma \ref{lem:ChainRuleForLinearTransformations} below yields
\begin{align}
\left|\partial^{\beta}\left[\varrho_{i}\circ T_{i}^{T}\right]\right| & =\left|\sum_{\ell_{1},\dots,\ell_{k}\in\underline{d}}\left(T_{i}^{T}\right)_{\ell_{1},i_{1}}\cdots\left(T_{i}^{T}\right)_{\ell_{k},i_{k}}\cdot\left(\partial_{\ell_{1}}\cdots\partial_{\ell_{k}}\varrho_{i}\right)\left(T_{i}^{T}\cdot\right)\right|\nonumber \\
 & \leq\sum_{\ell_{1},\dots,\ell_{k}\in\underline{d}}\left|\left(T_{i}\right)_{i_{1},\ell_{1}}\cdots\left(T_{i}\right)_{i_{k},\ell_{k}}\right|\cdot\left|\left(\partial_{\ell_{1}}\cdots\partial_{\ell_{k}}\varrho_{i}\right)\left(T_{i}^{T}\cdot\right)\right|\nonumber \\
\left(\text{by eq. }\eqref{eq:NormalizedBAPUWithDerivativesOnSpaceSide}\right) & \leq\sum_{\ell\in\underline{d}^{k}}\left\Vert T_{i}\right\Vert ^{k}\cdot C_{\ell,N}\cdot\left(1+\left|T_{i}^{T}\cdot\right|\right)^{-N}\nonumber \\
 & \leq C_{\left|\beta\right|,N}'\cdot\left\Vert T_{i}\right\Vert ^{\left|\beta\right|}\cdot\left(1+\left|T_{i}^{T}\cdot\right|\right)^{-N}\label{eq:BAPUDerivativeLpEstimateChainRuleApplication}
\end{align}
for some constant $C_{\left|\beta\right|,N}'=C_{\left|\beta\right|,N}'\left(\mathcal{Q},\left(\gamma_{i}\right)_{i\in I}\right)$
which is independent of $i\in I$.  In case of $k=\left|\beta\right|=0$,
i.e.\@ for $\beta=0$, an analogous argument yields
\begin{align*}
\left|\partial^{\beta}\left[\varrho_{i}\circ T_{i}^{T}\right]\right| & =\left|\varrho_{i}\circ T_{i}^{T}\right|\\
\left(\text{by eq. }\eqref{eq:NormalizedBAPUWithDerivativesOnSpaceSide}\right) & \leq C_{0,N}\cdot\left(1+\left|T_{i}^{T}\cdot\right|\right)^{-N}\\
 & =C_{0,N}'\cdot\left\Vert T_{i}\right\Vert ^{\left|\beta\right|}\cdot\left(1+\left|T_{i}^{T}\cdot\right|\right)^{-N},
\end{align*}
so that estimate (\ref{eq:BAPUDerivativeLpEstimateChainRuleApplication})
also holds for $\left|\beta\right|=0$.

In conjunction  with equation (\ref{eq:BAPUDerivativeLpEstimateLeibnizApplication}),
we conclude
\begin{align*}
\left|\partial^{\alpha}\theta_{i}\right| & \leq\left|\det T_{i}\right|\cdot\sum_{\beta\leq\alpha}\left[\binom{\alpha}{\beta}\cdot\left|\left(2\pi j\cdot b_{i}\right)^{\alpha-\beta}\right|\cdot\left|\partial^{\beta}\left[\varrho_{i}\circ T_{i}^{T}\right]\right|\right].\\
 & \leq\left(2\pi\right)^{\left|\alpha\right|}\cdot\left|\det T_{i}\right|\left(1+\left|T_{i}^{T}\cdot\right|\right)^{-N}\cdot\sum_{\beta\leq\alpha}\left[\binom{\alpha}{\beta}C_{\left|\beta\right|,N}'\cdot\left|b_{i}\right|^{\left|\alpha\right|-\left|\beta\right|}\left\Vert T_{i}\right\Vert ^{\left|\beta\right|}\right].
\end{align*}
Now, Young's inequality $ab\leq\frac{a^{r}}{r}+\frac{b^{s}}{s}$ for
$r,s>0$ with $\frac{1}{r}+\frac{1}{s}=1$ implies for $0\neq\beta<\alpha$
and $r=r_{\alpha,\beta}=\frac{\left|\alpha\right|}{\left|\alpha\right|-\left|\beta\right|}$,
as well as $s=s_{\alpha,\beta}=\frac{\left|\alpha\right|}{\left|\beta\right|}$
that
\[
\left|b_{i}\right|^{\left|\alpha\right|-\left|\beta\right|}\left\Vert T_{i}\right\Vert ^{\left|\beta\right|}\leq\frac{\left|b_{i}\right|^{\left|\alpha\right|}}{r}+\frac{\left\Vert T_{i}\right\Vert ^{\left|\alpha\right|}}{s}\leq C_{\alpha}'\cdot\left(\left|b_{i}\right|^{\left|\alpha\right|}+\left\Vert T_{i}\right\Vert ^{\left|\alpha\right|}\right),
\]
where the last step used that for each fixed $\alpha\in\mathbb{N}_{0}^{d}$,
there are only finitely many possible values of $r_{\alpha,\beta}$
and $s_{\alpha,\beta}$ where $\beta$ runs through all $0\neq\beta<\alpha$.

In case of $\beta=0$, we trivially have $\left|b_{i}\right|^{\left|\alpha\right|-\left|\beta\right|}\left\Vert T_{i}\right\Vert ^{\left|\beta\right|}=\left|b_{i}\right|^{\left|\alpha\right|}\leq\left|b_{i}\right|^{\left|\alpha\right|}+\left\Vert T_{i}\right\Vert ^{\left|\alpha\right|}$
and in case of $\beta=\alpha$, we also have $\left|b_{i}\right|^{\left|\alpha\right|-\left|\beta\right|}\left\Vert T_{i}\right\Vert ^{\left|\beta\right|}=\left\Vert T_{i}\right\Vert ^{\left|\alpha\right|}\leq\left|b_{i}\right|^{\left|\alpha\right|}+\left\Vert T_{i}\right\Vert ^{\left|\alpha\right|}$,
so that all in all, we arrive at
\[
\left|\partial^{\alpha}\theta_{i}\right|\leq C_{\alpha,N}''\cdot\left|\det T_{i}\right|\left(1+\left|T_{i}^{T}\cdot\right|\right)^{-N}\cdot\left(\left|b_{i}\right|^{\left|\alpha\right|}+\left\Vert T_{i}\right\Vert ^{\left|\alpha\right|}\right)
\]
for all $i\in I$, where $C_{\alpha,N}''=C_{\alpha,N}\left(\mathcal{Q},\left(\gamma_{i}\right)_{i\in I}\right)$
is independent of $i\in I$. 

We have thus established the estimate which was claimed in the remark.
It remains to show that this implies the desired estimate for the
$L^{p}$ norm claimed in the statement of the lemma. For $p=\infty$,
this is clear. For $p\in\left(0,\infty\right)$, choose $N=N\left(p,d\right)\geq\frac{d+1}{p}$.
Then, the change of variables formula yields
\begin{align*}
\left\Vert \left(1+\left|T_{i}^{T}\cdot\right|\right)^{-N}\right\Vert _{L^{p}}^{p} & =\left|\det T_{i}\right|^{-1}\cdot\int_{\mathbb{R}^{d}}\left(1+\left|T_{i}^{T}x\right|\right)^{-Np}\cdot\left|\det T_{i}^{T}\right|\,{\rm d}x\\
 & =\left|\det T_{i}\right|^{-1}\cdot\int_{\mathbb{R}^{d}}\left(1+\left|y\right|\right)^{-Np}\,{\rm d}y\\
 & \leq\left|\det T_{i}\right|^{-1}\cdot\int_{\mathbb{R}^{d}}\left(1+\left|y\right|\right)^{-\left(d+1\right)}\,{\rm d}y\\
 & =C_{d}\cdot\left|\det T_{i}\right|^{-1}
\end{align*}
and hence $\left\Vert \left(1+\left|T_{i}^{T}\cdot\right|\right)^{-N}\right\Vert _{L^{p}}\leq C_{p,d}\cdot\left|\det T_{i}\right|^{-1/p}$.
This yields
\[
\left\Vert \partial^{\alpha}\theta_{i}\right\Vert _{L^{p}}\leq C_{\alpha,N\left(p\right),d,p}\cdot\left|\det T_{i}\right|^{1-\frac{1}{p}}\cdot\left(\left|b_{i}\right|^{\left|\alpha\right|}+\left\Vert T_{i}\right\Vert ^{\left|\alpha\right|}\right),
\]
where $C_{\alpha,N\left(p\right),d,p}>0$ is independent of $i\in I$,
as desired.
\end{proof}
In the preceding proof, we used a special form of the chain rule to
compute higher derivatives of a function composed with a linear transformation.
The following lemma formally establishes this elementary result.
\begin{lem}
\label{lem:ChainRuleForLinearTransformations}Let $A\in\mathbb{R}^{d\times d}$
be arbitrary and $f\in C^{k}\left(\mathbb{R}^{d}\right)$ for some
$k\in\mathbb{N}$. Let $i_{1},\dots,i_{k}\in\underline{d}$ be arbitrary
and let $\alpha=\sum_{m=1}^{k}e_{i_{m}}\in\mathbb{N}_{0}^{d}$, where
$\left(e_{1},\dots,e_{d}\right)$ is the standard basis of $\mathbb{R}^{d}$.

Then $\left|\alpha\right|=k$ and
\begin{equation}
\left(\partial^{\alpha}\left[f\circ A\right]\right)\left(x\right)=\sum_{\ell_{1},\dots,\ell_{k}\in\underline{d}}\left[A_{\ell_{1},i_{1}}\cdots A_{\ell_{k},i_{k}}\cdot\left(\partial_{\ell_{1}}\cdots\partial_{\ell_{k}}f\right)\left(Ax\right)\right]\label{eq:ChainRuleForLinearTransformation}
\end{equation}
for every $x\in\mathbb{R}^{d}$.\end{lem}
\begin{proof}
We show the claim by induction on $k\in\mathbb{N}$. For $k=1$, we
have $\alpha=e_{i_{1}}$. Now, the chain rule implies
\begin{align*}
\left(\partial^{\alpha}\left[f\circ A\right]\right)\left(x\right) & =\left(D\left(f\circ A\right)\right)_{1,i_{1}}\left(x\right)\\
 & =\left(\left(Df\right)\left(Ax\right)\cdot A\right)_{1,i_{1}}\\
 & =\sum_{\ell=1}^{d}\left(Df\left(Ax\right)\right)_{1,\ell}\cdot A_{\ell,i_{1}}\\
 & =\sum_{\ell_{1}=1}^{d}\left[A_{\ell_{1},i_{1}}\cdot\left(\partial_{\ell_{1}}f\right)\left(Ax\right)\right],
\end{align*}
so that the claim holds for $k=1$.

Now, assume that the claim holds for some $k\in\mathbb{N}$ and let
$f\in C^{k+1}\left(\mathbb{R}^{d}\right)$ and $i_{1},\dots,i_{k+1}\in\underline{d}$.
By using the case $k=1$ and setting $\beta:=\sum_{m=1}^{k}e_{i_{m}}$,
we get
\begin{align*}
\left(\partial^{\alpha}\left[f\circ A\right]\right)\left(x\right) & =\left(\partial^{\beta}\left(\partial_{i_{k+1}}\left[f\circ A\right]\right)\right)\left(x\right)\\
\left(\text{case }k=1\right) & =\partial^{\beta}\left(\sum_{\ell_{k+1}=1}^{d}A_{\ell_{k+1},i_{k+1}}\cdot\left(\partial_{\ell_{k+1}}f\right)\left(Ax\right)\right)\\
\left(\text{with }f_{\ell}:=\partial_{\ell}f\in C^{k}\left(\mathbb{R}^{d}\right)\right) & =\sum_{\ell_{k+1}=1}^{d}\left(A_{\ell_{k+1},i_{k+1}}\cdot\left(\partial^{\beta}\left[f_{\ell_{k+1}}\circ A\right]\right)\left(x\right)\right)\\
\left(\text{by induction}\right) & =\sum_{\ell_{k+1}\in\underline{d}}\left(A_{\ell_{k+1},i_{k+1}}\sum_{\ell_{1},\dots,\ell_{k}\in\underline{d}}\left[A_{\ell_{1},i_{1}}\cdots A_{\ell_{k},i_{k}}\cdot\left(\partial_{\ell_{1}}\cdots\partial_{\ell_{k}}f_{\ell_{k+1}}\right)\left(Ax\right)\right]\right)\\
 & =\sum_{\ell_{1},\dots,\ell_{k+1}\in\underline{d}}\left[A_{\ell_{1}i_{1}}\cdots A_{\ell_{k+1},i_{k+1}}\cdot\left(\partial_{\ell_{1}}\cdots\partial_{\ell_{k+1}}f\right)\left(Ax\right)\right],
\end{align*}
so that the claim also holds for $k+1$ instead of $k$.
\end{proof}
As a consequence of Lemma \ref{lem:SufficientConditionDerivativeEstimate},
it is now straightforward to show that indeed every regular partition
of unity is an $L^{p}$-partition of unity, for every $p\in\left(0,\infty\right]$.
\begin{cor}
\label{cor:RegularPartitionsAreBAPUs}Let $\Phi=\left(\varphi_{i}\right)_{i\in I}$
be a regular partition of unity subordinate to $\mathcal{Q}$. Then
$\Phi$ is an $L^{p}$-BAPU for $\mathcal{Q}$ for every $p\in\left(0,\infty\right]$.

In particular, every regular covering is also an $L^{p}$-decomposition
covering for all $p\in\left(0,\infty\right]$.\end{cor}
\begin{proof}
Lemma \ref{lem:SufficientConditionDerivativeEstimate} (with $\alpha=0$)
yields
\[
\left\Vert \mathcal{F}^{-1}\varphi_{i}\right\Vert _{L^{q}}\leq C_{d,q,\mathcal{Q},\Phi}\cdot\left|\det T_{i}\right|^{1-\frac{1}{q}}
\]
for all $i\in I$ and $q\in\left(0,\infty\right]$. For $p\in\left[1,\infty\right]$,
taking $q=1$ shows that $\Phi$ is an $L^{p}$-decomposition covering.
For $p\in\left(0,1\right)$, we get the same conclusion by choosing
$q=p$.
\end{proof}
It is an important fact established by Borup and Nielsen in \cite[Proposition 1]{BorupNielsenDecomposition}that
every structured admissible covering is an $L^{p}$-decomposition
covering. In fact, their \emph{proof} even shows that every structured
admissible covering is a regular covering. Our next result slightly
generalizes this.
\begin{thm}
\label{thm:StructuredAdmissibleCoveringsAreRegular}Every structured
admissible covering is a regular covering.

More generally, the following holds: If $\mathcal{Q}=\left(Q_{i}\right)_{i\in I}=\left(T_{i}Q_{i}'+b_{i}\right)_{i\in I}$
is a semi-structured \emph{open} covering of the open set $\emptyset\neq\mathcal{O}\subset\mathbb{R}^{d}$,
and if for each $i\in I$, there is an open set $P_{i}'$ with $\overline{P_{i}'}\subset Q_{i}'$
and with $\mathcal{O}=\bigcup_{i\in I}\left(T_{i}P_{i}'+b_{i}\right)$
and such that the sets $\left\{ P_{i}'\with i\in I\right\} $ and
$\left\{ Q_{i}'\with i\in I\right\} $ are finite, then $\mathcal{Q}$
is a tight regular covering of $\mathcal{O}$. Furthermore, $I$ is
countably infinite.\end{thm}
\begin{rem*}
All in all, we have the following inclusions/implications between
the different types of coverings:
\[
\text{structured admissible}\subset\text{regular}\subset\text{semi-structured}.\qedhere
\]
\end{rem*}
\begin{proof}
Essentially, a proof of this result is contained in the proof (but
not in the statement) of \cite[Proposition 1]{BorupNielsenDecomposition}.
Here, we give a different proof which is (at least in spirit) close
to that of \cite[Theorem 3.2.17]{VoigtlaenderPhDThesis}.

For brevity, set $P_{i}:=T_{i}P_{i}'+b_{i}$. Note that $P_{i}'\subset Q_{i}'$
is bounded, so that $\overline{P_{i}}\subset Q_{i}\subset\mathcal{O}$
is compact. Thus, $I$ has to be infinite; indeed, if $I$ was finite,
then $\mathcal{O}$ would be compact, because of 
\[
\mathcal{O}=\bigcup_{i\in I}P_{i}\subset\bigcup_{i\in I}\overline{P_{i}}\subset\bigcup_{i\in I}Q_{i}=\mathcal{O},
\]
where the set $\bigcup_{i\in I}\overline{P_{i}}$ is compact if $I$
is finite. But this would imply that $\mathcal{O}$ is open and closed,
so that $\mathcal{O}\in\left\{ \emptyset,\mathbb{R}^{d}\right\} $,
which contradicts compactness of $\mathcal{O}\neq\emptyset$.

Furthermore, $I$ is countable. To see this, note that since $\mathcal{O}\subset\mathbb{R}^{d}$
is second countable and $\left(P_{i}\right)_{i\in I}$ is an open
cover of $\mathcal{O}$, there is a countable subcover $\left(P_{i_{n}}\right)_{n\in\mathbb{N}}$.
But this implies that $I=\bigcup_{n\in\mathbb{N}}i_{n}^{\ast}$ is
countable as a countable union of finite sets. To see the equality,
note that for $i\in I$ and arbitrary $x\in Q_{i}\subset\mathcal{O}$,
there is $n\in\mathbb{N}$ with $x\in P_{i_{n}}\subset Q_{i_{n}}$,
so that we get $Q_{i}\cap Q_{i_{n}}\neq\emptyset$ and hence $i\in i_{n}^{\ast}$.

Tightness of $\mathcal{Q}$ is immediate: Since all $Q_{i}'\subset\mathbb{R}^{d}$
are nonempty and open, and since $\left\{ Q_{i}'\with i\in I\right\} $
is \emph{finite}, there is some $\varepsilon>0$ and for each $i\in I$
some $c_{i}\in\mathbb{R}^{d}$ with $B_{\varepsilon}\left(c_{i}\right)\subset Q_{i}'$.

It remains to show that $\mathcal{Q}$ is regular. To this end, we
can assume $I=\mathbb{N}$, since $I$ is countably infinite. By assumption%
\footnote{The claim made here is not completely obvious. To see it, note that
if we set $P_{i}'':=\bigcup\left\{ \smash{P_{j}'}\with\smash{\overline{P_{j}'}}\subset Q_{i}'\right\} $,
then the set $\left\{ P_{i}''\with i\in I\right\} $ is finite, since
$\left\{ P_{i}'\with i\in I\right\} $ is. Furthermore, $P_{i}'\subset P_{i}''\subset\overline{P_{i}''}\subset Q_{i}'$.
By suitably numbering the $\left(Q_{i}'\right)_{i\in I}$ and $\left(P_{i}''\right)_{i\in I}$
as $U_{1},\dots,U_{N}$ and $V_{1},\dots,V_{N}$ (possibly with repititions),
we obtain the claim. %
}, there are certain open sets $U_{1},\dots,U_{N},V_{1},\dots,V_{N}\subset\mathbb{R}^{d}$
with $\overline{V_{m}}\subset U_{m}$ for all $m\in\underline{N}$
and such that for each $i\in I$, there is some $m_{i}\in\underline{N}$
with $P_{i}'\subset V_{m_{i}}\subset\overline{V_{m_{i}}}\subset U_{m_{i}}=Q_{i}'$.
Now, for every $m\in\underline{N}$, choose some $\psi_{m}\in C_{c}^{\infty}\left(U_{m}\right)$
with $\psi_{m}\equiv1$ on $V_{m}$.

Finally, for $i\in I$, define $\gamma_{i}:=L_{b_{i}}\left(\psi_{m_{i}}\circ T_{i}^{-1}\right)$,
i.e.\@
\[
\gamma_{i}:\mathbb{R}^{d}\to\mathbb{C},\xi\mapsto\psi_{m_{i}}\left(T_{i}^{-1}\left(\xi-b_{i}\right)\right)
\]
and note $\gamma_{i}\in C_{c}^{\infty}\left(Q_{i}\right)$ with $\gamma_{i}\equiv1$
on $P_{i}$. In view of our assumption $I=\mathbb{N}$ from above,
we can now define
\[
\varphi_{n}:=\gamma_{n}\cdot\prod_{j=1}^{n-1}\left(1-\gamma_{j}\right)\text{ for }n\in\mathbb{N}.
\]
With this definition, a straightforward induction yields $\sum_{j=1}^{n}\varphi_{j}=1-\prod_{j=1}^{n}\left(1-\gamma_{j}\right)$
for all $n\in\mathbb{N}$. Now, for arbitrary $x\in\mathcal{O}$,
there is some $n\in\mathbb{N}=I$ with $x\in P_{n}$. Hence, $1-\gamma_{n}\left(x\right)=0$,
so that we get
\[
\sum_{j=1}^{\infty}\varphi_{j}\left(x\right)=1-\prod_{j=1}^{\infty}\left(1-\gamma_{j}\left(x\right)\right)=1-0=1.
\]
All in all, we have shown that $\left(\varphi_{j}\right)_{j\in\mathbb{N}}$
is a smooth partition of unity on $\mathcal{O}$, subordinate to $\mathcal{Q}$.

It remains to show that for each $\alpha\in\mathbb{N}_{0}^{d}$, the
constant
\[
C_{\alpha}:=\sup_{n\in\mathbb{N}}\left\Vert \partial^{\alpha}\varphi_{n}^{\#}\right\Vert _{\sup}
\]
is finite, where $\varphi_{n}^{\#}\left(\xi\right)=\varphi_{n}\left(T_{n}\xi+b_{n}\right)$.
To see this, note that for arbitrary $n\in\mathbb{N}$ and $m\in\mathbb{N}\setminus n^{\ast}$,
we have
\[
Q_{n}\cap{\rm supp}\,\gamma_{m}\subset Q_{n}\cap Q_{m}=\emptyset
\]
and hence $1-\gamma_{m}\equiv1$ on $Q_{n}$. Because of ${\rm supp}\,\gamma_{n}\subset Q_{n}$,
this implies
\[
\varphi_{n}=\gamma_{n}\cdot\prod_{j=1}^{n-1}\left(1-\gamma_{j}\right)=\gamma_{n}\cdot\prod_{j\in\underline{n-1}\cap n^{\ast}}\left(1-\gamma_{j}\right).
\]
Note that (by admissibility of $\mathcal{Q}$) the number of factors
in the product is uniformly bounded, independent of $n\in\mathbb{N}$.
By the Leibniz rule, it thus suffices to show that the constant
\[
C_{\alpha}':=\sup_{n\in\mathbb{N}}\sup_{j\in n^{\ast}}\left\Vert \partial^{\alpha}\left(\xi\mapsto\gamma_{j}\left(T_{n}\xi+b_{n}\right)\right)\right\Vert _{\sup}
\]
is finite for all $\alpha\in\mathbb{N}_{0}^{d}$. But this easily
follows from the assumptions on a semi-structured admissible covering:
Indeed, for $n\in\mathbb{N}$ and $j\in n^{\ast}$, we have
\begin{align*}
\gamma_{j}\left(T_{n}\xi+b_{n}\right) & =\psi_{m_{j}}\left(T_{j}^{-1}\left(\left[T_{n}\xi+b_{n}\right]-b_{j}\right)\right)\\
 & =\psi_{m_{j}}\left(T_{j}^{-1}T_{n}\xi+T_{j}^{-1}\left(b_{n}-b_{j}\right)\right)
\end{align*}
and thus
\[
\left\Vert \partial^{\alpha}\left(\xi\mapsto\gamma_{j}\left(T_{n}\xi+b_{n}\right)\right)\right\Vert _{\sup}=\left\Vert \partial^{\alpha}\left(\xi\mapsto\psi_{m_{j}}\left(T_{j}^{-1}T_{n}\xi\right)\right)\right\Vert _{\sup}.
\]
But using Lemma \ref{lem:ChainRuleForLinearTransformations}, we easily
get
\begin{align*}
\left|\left[\partial^{\alpha}\left(\psi_{m_{j}}\circ T_{j}^{-1}T_{n}\right)\right]\left(\xi\right)\right| & \lesssim_{\alpha}\left\Vert T_{j}^{-1}T_{n}\right\Vert ^{\left|\alpha\right|}\cdot\max_{\substack{\beta\in\mathbb{N}_{0}^{d}\\
\left|\beta\right|=\left|\alpha\right|
}
}\left|\left(\partial^{\beta}\psi_{m_{j}}\right)\left(T_{j}^{-1}T_{n}\xi\right)\right|\\
 & \lesssim_{\alpha,\mathcal{Q}}\max_{m\in\underline{N}}\max_{\substack{\beta\in\mathbb{N}_{0}^{d}\\
\left|\beta\right|=\left|\alpha\right|
}
}\left\Vert \partial^{\beta}\psi_{m}\right\Vert _{\sup}
\end{align*}
for all $\xi\in\mathbb{R}^{d}$. Here, we used $\left\Vert T_{j}^{-1}T_{n}\right\Vert \leq C_{\mathcal{Q}}$,
because of $j\in n^{\ast}$.
\end{proof}

\subsection{Decomposition spaces}

\label{sub:DecompositionSpaces}Using the notion of $L^{p}$-decomposition
coverings from the previous subsection, we are almost ready to define
decomposition spaces. We only need one more definition which pertains
to the sequence space $Y$ used to define $\mathcal{D}\left(\mathcal{Q},L^{p},Y\right)$.
\begin{defn}
\label{def:QRegularSequenceSpace}(cf.\@ \cite[Definition 2.5]{DecompositionSpaces1})
Let $\mathcal{Q}=\left(Q_{i}\right)_{i\in I}$ be an admissible covering.
We say that a quasi-normed vector space $\left(Y,\left\Vert \cdot\right\Vert _{Y}\right)$
is a \textbf{$\mathcal{Q}$-regular sequence space} if the following
hold:
\begin{enumerate}
\item $Y$ is a \textbf{sequence space} over $I$, i.e.\@ $Y$ is a subspace
of the space $\mathbb{C}^{I}$ of all (complex) sequences over $I$.
\item $Y$ is a Quasi-Banach space, i.e.\@ $\left(Y,\left\Vert \cdot\right\Vert _{Y}\right)$
is complete.
\item $Y$ is \textbf{solid}, i.e.\@ if $x=\left(x_{i}\right)_{i\in I}\in\mathbb{C}^{I}$
is a sequence with $\left|x_{i}\right|\leq\left|y_{i}\right|$ for
all $i\in I$ and some sequence $y=\left(y_{i}\right)_{i\in I}\in Y$,
then $x\in Y$ with $\left\Vert x\right\Vert _{Y}\leq\left\Vert y\right\Vert _{Y}$.
\item $Y$ is \textbf{$\mathcal{Q}$-regular}, i.e., the \textbf{clustering
map}
\[
\Psi:Y\to Y,x=\left(x_{i}\right)_{i\in I}\mapsto x^{\ast}=\left(\sum_{\ell\in i^{\ast}}x_{\ell}\right)_{i\in I}
\]
is well-defined and bounded.\qedhere
\end{enumerate}
\end{defn}
\begin{rem*}
That $\left(Y,\left\Vert \cdot\right\Vert _{Y}\right)$ is a \textbf{quasi-normed
vector space} means that $\left\Vert \cdot\right\Vert _{Y}$ is a
norm on $Y$, with the exception that the usual triangle inequality
is replaced by
\[
\left\Vert x+y\right\Vert \leq C\cdot\left(\left\Vert x\right\Vert +\left\Vert y\right\Vert \right)
\]
for all $x,y\in Y$ and some fixed constant $C\geq1$.

The most important class of $\mathcal{Q}$-regular sequence spaces
is given by weighted Lebesgue spaces $Y=\ell_{u}^{r}\left(I\right)$,
if the weight $u$ is \textbf{$\mathcal{Q}$-moderate}, i.e.\@ if
\[
C_{u,\mathcal{Q}}:=\sup_{i\in I}\sup_{j\in i^{\ast}}\frac{u_{i}}{u_{j}}
\]
is finite. That this indeed yields a $\mathcal{Q}$-regular sequence
is shown e.g.\@ in \cite[Lemma 3.4.2]{VoigtlaenderPhDThesis}, see
also \cite[Lemma 3.2]{DecompositionSpaces1}.

Finally, using completeness of $Y$ and a variant of the closed graph
theorem (see e.g.\@ \cite[Theorem 2.1.5]{RudinFA}), one can show
that the clustering map $\Psi$ is bounded if and only if it is well-defined.
This claim uses admissibility of $\mathcal{Q}$, since this implies
that the set $i^{\ast}\subset I$ is finite for all $i\in I$, so
that $\Psi$ has a closed graph.
\end{rem*}
It turns out to be easiest to first define decomposition spaces \emph{on
the Fourier side} and to introduce the \emph{space-side} versions
of these spaces only afterwards.
\begin{defn}
\label{def:FourierSideDecompositionSpace}Let $p\in\left(0,\infty\right]$
and assume that $\mathcal{Q}=\left(Q_{i}\right)_{i\in I}=\left(T_{i}Q_{i}'+b_{i}\right)_{i\in I}$
is an $L^{p}$-decomposition covering of the open set $\emptyset\neq\mathcal{O}\subset\mathbb{R}^{d}$.
Finally, let $Y\leq\mathbb{C}^{I}$ be a $\mathcal{Q}$-regular sequence
space and let $\Phi=\left(\varphi_{i}\right)_{i\in I}$ be an $L^{p}$-BAPU
for $\mathcal{Q}$.

For a distribution $f\in\mathcal{D}'\left(\mathcal{O}\right)$, define
the \textbf{Fourier-side decomposition space (quasi)-norm} of $f$
(with respect to $\mathcal{Q},L^{p},Y$) as
\begin{equation}
\left\Vert f\right\Vert _{\mathcal{D}_{\mathcal{F}}\left(\mathcal{Q},L^{p},Y\right)}:=\left\Vert \left(\left\Vert \mathcal{F}^{-1}\left(\varphi_{i}f\right)\right\Vert _{L^{p}}\right)_{i\in I}\right\Vert _{Y}\;,\label{eq:FourierSideDecompositionSpaceNorm}
\end{equation}
with the convention that $\left\Vert \left(c_{i}\right)_{i\in I}\right\Vert _{Y}=\infty$
if $c_{i}=\infty$ for some $i\in I$ and furthermore $\left\Vert \mathcal{F}^{-1}\left(\varphi_{i}f\right)\right\Vert _{L^{p}}=\infty$
if $\mathcal{F}^{-1}\left(\varphi_{i}f\right)\notin L^{p}\left(\mathbb{R}^{d}\right)$.

Finally, define the \textbf{Fourier-side decomposition space }(with
respect to $\mathcal{Q},L^{p},Y$) as
\[
\mathcal{D}_{\mathcal{F}}\left(\mathcal{Q},L^{p},Y\right):=\left\{ f\in\mathcal{D}'\left(\mathcal{O}\right)\with\left\Vert f\right\Vert _{\mathcal{D}_{\mathcal{F}}\left(\mathcal{Q},L^{p},Y\right)}<\infty\right\} .\qedhere
\]
\end{defn}
\begin{rem*}
Note that $\varphi_{i}f$ is a distribution on $\mathcal{O}$ with
compact support in $\mathcal{O}$ because of $\varphi_{i}\in C_{c}^{\infty}\left(\mathcal{O}\right)$.
Thus, $\varphi_{i}f$ is actually a compactly supported (tempered)
distribution on all of $\mathbb{R}^{d}$. Hence, thanks to the Paley
Wiener theorem (cf.\@ \cite[Proposition 9.11]{FollandRA}), the inverse
Fourier transform $\mathcal{F}^{-1}\left(\varphi_{i}f\right)\in\mathcal{S}'\left(\mathbb{R}^{d}\right)$
is given by (integration against) a smooth function with polynomially
bounded derivatives of all orders. Thus, it makes sense to write $\left\Vert \mathcal{F}^{-1}\left(\varphi_{i}f\right)\right\Vert _{L^{p}}$,
with the caveat that we might have $\left\Vert \mathcal{F}^{-1}\left(\varphi_{i}f\right)\right\Vert _{L^{p}}=\infty$.
In the remainder of the paper, we will always identify the (tempered)
distribution $\mathcal{F}^{-1}\left(\varphi_{i}f\right)$ with the
smooth function obtained from the Paley Wiener theorem. Hence,
\[
\left[\mathcal{F}^{-1}\left(\varphi_{i}f\right)\right]\left(x\right)=\left\langle \varphi_{i}f,\, e^{2\pi i\left\langle x,\cdot\right\rangle }\right\rangle =\left\langle f,\,\varphi_{i}\cdot e^{2\pi i\left\langle x,\cdot\right\rangle }\right\rangle 
\]
for all $x\in\mathbb{R}^{d}$.

Finally, note that the right-hand side of equation (\ref{eq:FourierSideDecompositionSpaceNorm})
uses the $L^{p}$-BAPU $\Phi$, whereas $\Phi$ is not mentioned on
the left-hand side. This is justified by \cite[Corollary 3.4.11]{VoigtlaenderPhDThesis},
where it is shown that different choices of $\Phi$ yield the same
(Fourier side) decomposition space with equivalent quasi-norms. Furthermore,
\cite[Theorem 3.4.13]{VoigtlaenderPhDThesis} shows that the resulting
space $\mathcal{D}_{\mathcal{F}}\left(\mathcal{Q},L^{p},Y\right)$
is complete and satisfies $\mathcal{D}_{\mathcal{F}}\left(\mathcal{Q},L^{p},Y\right)\hookrightarrow\mathcal{D}'\left(\mathcal{O}\right)$.
\end{rem*}
Now that we have introduced the Fourier-side version of decomposition
spaces, we also want to define their space-side counterpart. To this
end, we first have to define a suitable reservoir which takes the
role of the space of distributions $\mathcal{D}'\left(\mathcal{O}\right)$
from above. We remark that our notation is heavily influenced by Triebel,
in particular by \cite{TriebelFourierAnalysisAndFunctionSpaces}.
\begin{defn}
\label{def:SpaceSideTriebelReservoir}For an open set $\emptyset\neq\mathcal{O}\subset\mathbb{R}^{d}$,
we set
\[
Z\left(\mathcal{O}\right):=\mathcal{F}\left(C_{c}^{\infty}\left(\mathcal{O}\right)\right)=\left\{ \widehat{f}\with f\in C_{c}^{\infty}\left(\mathcal{O}\right)\right\} \leq\mathcal{S}\left(\mathbb{R}^{d}\right)
\]
and endow this space with the unique topology which makes the Fourier
transform
\[
\mathcal{F}:C_{c}^{\infty}\left(\mathcal{O}\right)=\mathcal{D}\left(\mathcal{O}\right)\to Z\left(\mathcal{O}\right)
\]
a homeomorphism.

We equip the topological dual space $Z'\left(\mathcal{O}\right):=\left[Z\left(\mathcal{O}\right)\right]'$
of $Z\left(\mathcal{O}\right)$ with the weak-$\ast$-topology, i.e.\@
with the topology of pointwise convergence on $Z\left(\mathcal{O}\right)$.
Finally, as on the Schwartz space, we extend the Fourier transform
by duality to $Z'\left(\mathcal{O}\right)$, i.e.\@ we define
\[
\mathcal{F}:Z'\left(\mathcal{O}\right)\to\mathcal{D}'\left(\mathcal{O}\right),f\mapsto f\circ\mathcal{F}.
\]
For $f\in Z'\left(\mathcal{O}\right)$, we also write $\widehat{f}:=\mathcal{F}f\in\mathcal{D}'\left(\mathcal{O}\right)$.

Since the Fourier transform $\mathcal{F}:C_{c}^{\infty}\left(\mathcal{O}\right)\to Z\left(\mathcal{O}\right)$
is bijective (even a homeomorphism), it is easily seen that $\mathcal{F}:Z'\left(\mathcal{O}\right)\to\mathcal{D}'\left(\mathcal{O}\right)$
is also a homeomorphism.
\end{defn}
Using the reservoir $Z'\left(\mathcal{O}\right)$ that we just introduced,
we can finally define the space-side decomposition spaces.
\begin{defn}
\label{def:SpaceSideDecompositionSpaces}Under the general assumptions
of Definition \ref{def:FourierSideDecompositionSpace}, we define
the \textbf{(space-side) Decomposition space} (with respect to $\mathcal{Q},L^{p},Y$)
as
\[
\mathcal{D}\left(\mathcal{Q},L^{p},Y\right):=\left\{ f\in Z'\left(\mathcal{O}\right)\with\mathcal{F}f\in\mathcal{D}_{\mathcal{F}}\left(\mathcal{Q},L^{p},Y\right)\right\} ,
\]
with (quasi)-norm
\[
\left\Vert f\right\Vert _{\mathcal{D}\left(\mathcal{Q},L^{p},Y\right)}:=\left\Vert \smash{\widehat{f}}\right\Vert _{\mathcal{D}_{\mathcal{F}}\left(\mathcal{Q},L^{p},Y\right)}=\left\Vert \left(\left\Vert \mathcal{F}^{-1}\left(\smash{\varphi_{i}\widehat{f}}\right)\right\Vert _{L^{p}}\right)_{i\in I}\right\Vert _{Y}\:,
\]
where $\Phi=\left(\varphi_{i}\right)_{i\in I}$ is an $L^{p}$-BAPU
for $\mathcal{Q}$.\end{defn}
\begin{rem*}
From the properties of the Fourier-side decomposition spaces, it is
immediate that $\mathcal{D}\left(\mathcal{Q},L^{p},Y\right)$ is a
Quasi-Banach space with $\mathcal{D}\left(\mathcal{Q},L^{p},Y\right)\hookrightarrow Z'\left(\mathcal{O}\right)$
which is independent of the choice of $\Phi$, with equivalent quasi-norms
for different choices.

Readers familiar with the work of Borup and Nielsen (e.g.\@ \cite{BorupNielsenDecomposition})
might object to the seemingly overcomplicated choice of the reservoir
$Z'\left(\mathcal{O}\right)$, when one could simply use $\mathcal{S}'\left(\mathbb{R}^{d}\right)$.
This choice, however, has two serious limitations:
\begin{enumerate}
\item We want to allow for the case that $\mathcal{Q}$ only covers a proper
subset $\mathcal{O}\subsetneq\mathbb{R}^{d}$ of the frequency space
$\mathbb{R}^{d}$. In this case, the expression $\left\Vert \cdot\right\Vert _{\mathcal{D}\left(\mathcal{Q},L^{p},Y\right)}$
does \emph{not} define a (quasi)-norm on $\mathcal{S}'\left(\mathbb{R}^{d}\right)$,
since it is not positive definite. For example, the inverse Fourier
transform of any Dirac delta distribution, $f=\mathcal{F}^{-1}\delta_{x_{0}}$
with $x_{0}\in\mathbb{R}^{d}\setminus\mathcal{O}$, would satisfy
$\left\Vert f\right\Vert _{\mathcal{D}\left(\mathcal{Q},L^{p},Y\right)}=0$,
although $f\neq0$.

Thus, to obtain a proper norm, one would have to factor out a certain
subspace of $\mathcal{S}'\left(\mathbb{R}^{d}\right)$. For example,
the homogeneous Besov space $\dot{\mathcal{B}}_{s}^{p,q}\left(\mathbb{R}^{d}\right)$
is usually defined as a subspace of $\mathcal{S}'\left(\mathbb{R}^{d}\right)/\mathcal{P}$,
where $\mathcal{P}$ denotes the space of polynomials. Note that $\mathcal{P}$
is exactly the set of inverse Fourier transforms of distributions
supported at the origin. Instead of factoring out a subspace of $\mathcal{S}'\left(\mathbb{R}^{d}\right)$
depending on $\mathcal{O}$, we prefer to work with the space $Z'\left(\mathcal{O}\right)$.

\item Even in case of $\mathcal{O}=\mathbb{R}^{d}$, the space
\[
\mathcal{D}_{\mathcal{S}'}\left(\mathcal{Q},L^{p},Y\right):=\left\{ f\in\mathcal{S}'\left(\mathbb{R}^{d}\right)\with\left\Vert f\right\Vert _{\mathcal{D}\left(\mathcal{Q},L^{p},Y\right)}<\infty\right\} 
\]
is in general \emph{not} complete, as observed in \cite[Example 3.4.14]{VoigtlaenderPhDThesis}
and \cite[Example after Definition 21]{FuehrVoigtlaenderCoorbitSpacesAsDecompositionSpaces}.
In contrast, the space $\mathcal{D}\left(\mathcal{Q},L^{p},Y\right)$
defined above \emph{is} complete. Thus, the choice of the reservoir
$Z'\left(\mathcal{O}\right)$ turns out to be superior to $\mathcal{S}'\left(\mathbb{R}^{d}\right)$
even for $\mathcal{O}=\mathbb{R}^{d}$.\qedhere
\end{enumerate}
\end{rem*}
Having properly introduced the classes of Sobolev spaces and decomposition
spaces, the next section is dedicated to deriving sufficient conditions
for an embedding of the form $\mathcal{D}\left(\mathcal{Q},L^{p},Y\right)\hookrightarrow W^{k,q}\left(\mathbb{R}^{d}\right)$.
But beforehand, we briefly gather some facts concerning convolution
relations for the Lebesgue spaces $L^{q}\left(\mathbb{R}^{d}\right)$
in the Quasi-Banach regime $q\in\left(0,1\right)$.

\subsection{Convolution in the Quasi-Banach regime}

\label{sub:QuasiBanachConvolution}As observed above, the convolution
relation $L^{1}\ast L^{q}\subset L^{q}$, which is valid for $q\in\left[1,\infty\right]$,
fails for $q\in\left(0,1\right)$. In this subsection, we recall from
\cite{TriebelTheoryOfFunctionSpaces} and \cite{VoigtlaenderPhDThesis}
some alternative convolution relations which are valid in the Quasi-Banach
regime. All of these theorems, however, will require the ``factors''
of the convolution product to be bandlimited.

We begin with the following result, which shows that bandlimited functions
in $L^{q}$ are automatically contained in $L^{r}$ for all $r\geq q$.
Intuitively, this reflects the fact that bandlimited functions are
locally well behaved, so that the only obstruction to membership in
$L^{q}$ is insufficient decay at infinity. We first state the result
for $p\in\left(0,2\right]$. Afterwards, we present a  version 
for general $p\in\left(0,\infty\right]$ for the special case where
the frequency support ${\rm supp}\,\widehat{f}$ is contained in $Q_{i}$
for a semi-structured covering $\mathcal{Q}$.
\begin{lem}
\label{lem:BandlimitedLpEmbeddingGeneral}Let $\emptyset\neq\Omega\subset\mathbb{R}^{d}$
be compact and assume that $f\in\mathcal{S}'\left(\mathbb{R}^{d}\right)$
is a tempered distribution with compact Fourier support ${\rm supp}\,\widehat{f}\subset\Omega$.

If $f\in L^{p}\left(\mathbb{R}^{d}\right)$ for some $p\in\left(0,2\right]$,
then
\[
\left\Vert f\right\Vert _{L^{q}}\leq\left[\lambda\left(\Omega\right)\right]^{\frac{1}{p}-\frac{1}{q}}\cdot\left\Vert f\right\Vert _{L^{p}}
\]
holds for all $q\in\left[p,\infty\right]$.\end{lem}
\begin{proof}
For a proof, see \cite[Corollary 3.1.3]{VoigtlaenderPhDThesis}. The
proof given there is strongly based on that of \cite[1.4.1(3)]{TriebelTheoryOfFunctionSpaces},
where the same statement is shown, but with an unspecified constant
$C_{\Omega}$ instead of $\left[\lambda\left(\Omega\right)\right]^{\frac{1}{p}-\frac{1}{q}}$.
This constant -- which can be extracted from the proof -- will be
important for us below.
\end{proof}
Now, we specialize the above result to functions which are bandlimited
to sets $Q_{i}$ of a semi-structured covering $\mathcal{Q}$. Note
that the proof for $p\in\left[1,\infty\right]$ is self-contained,
independent of Lemma \ref{lem:BandlimitedLpEmbeddingGeneral}.
\begin{cor}
\label{cor:BandlimitedLpEmbeddingSemiStructured}Let $\mathcal{Q}=\left(Q_{i}\right)_{i\in I}=\left(T_{i}Q_{i}'+b_{i}\right)_{i\in I}$
be a semi-structured covering of the open set $\emptyset\neq\mathcal{O}\subset\mathbb{R}^{d}$.

Let $p,q\in\left(0,\infty\right]$ with $p\leq q$. Then there is
a constant $C=C\left(\mathcal{Q},d,p,q\right)$ such that for each
$i\in I$, we have
\[
\left\Vert \mathcal{F}^{-1}f\right\Vert _{L^{q}}\leq C\cdot\left|\det T_{i}\right|^{\frac{1}{p}-\frac{1}{q}}\cdot\left\Vert \mathcal{F}^{-1}f\right\Vert _{L^{p}}
\]
for all $f\in\mathcal{D}'\left(\mathcal{O}\right)$ with compact support
${\rm supp}\, f\subset Q_{i}$.\end{cor}
\begin{rem*}
Since ${\rm supp}\, f\subset Q_{i}\subset\mathcal{O}$ is compact,
$f$ is actually a (tempered) distribution on all of $\mathbb{R}^{d}$,
so that the tempered distribution $\mathcal{F}^{-1}f$ is given by
(integration against) a smooth function, by the Paley Wiener theorem.
In particular, Lemma \ref{lem:BandlimitedLpEmbeddingGeneral} is applicable
to $\mathcal{F}^{-1}f$.\end{rem*}
\begin{proof}
The proof given here is essentially that of \cite[Lemma 5.1.3]{VoigtlaenderPhDThesis}.

By definition of a semi-structured covering, there is some $R>0$
with $\overline{Q_{i}'}\subset B_{R}\left(0\right)$ for all $i\in I$.
Now, we distinguish the two cases $p\in\left(0,1\right)$ and $p\in\left[1,\infty\right]$.

For $p\in\left(0,1\right)$, we have 
\begin{align*}
\lambda\left(\overline{Q_{i}}\right) & \leq\lambda\left(T_{i}\overline{Q_{i}'}+b_{i}\right)\\
 & \leq\lambda\left(T_{i}\overline{B_{R}\left(0\right)}+b_{i}\right)\\
 & =\lambda\left(B_{R}\left(0\right)\right)\cdot\left|\det T_{i}\right|
\end{align*}
and hence
\[
\left[\lambda\left(\overline{Q_{i}}\right)\right]^{\frac{1}{p}-\frac{1}{q}}\leq C_{R,p,q}\cdot\left|\det T_{i}\right|^{\frac{1}{p}-\frac{1}{q}},
\]
since $p\leq q$ implies $\frac{1}{p}-\frac{1}{q}\geq0$. Thus, Lemma
\ref{lem:BandlimitedLpEmbeddingGeneral} yields $\mathcal{F}^{-1}f\in L^{q}\left(\mathbb{R}^{d}\right)$
with
\[
\left\Vert \mathcal{F}^{-1}f\right\Vert _{L^{q}}\leq\left[\lambda\left(\overline{Q_{i}}\right)\right]^{\frac{1}{p}-\frac{1}{q}}\cdot\left\Vert \mathcal{F}^{-1}f\right\Vert _{L^{p}}\leq C_{R,p,q}\cdot\left|\det T_{i}\right|^{\frac{1}{p}-\frac{1}{q}}\cdot\left\Vert \mathcal{F}^{-1}f\right\Vert _{L^{p}},
\]
as desired.

For $p\in\left[1,\infty\right]$, we give a self-contained proof.
Choose $\psi\in C_{c}^{\infty}\left(\mathbb{R}^{d}\right)$ with $\psi\equiv1$
on $B_{R}\left(0\right)$ and let $\psi_{i}:\mathbb{R}^{d}\to\mathbb{C},\xi\mapsto\psi\left(T_{i}^{-1}\left(\xi-b_{i}\right)\right)$,
i.e.\@ $\psi_{i}=L_{b_{i}}\left(\psi\circ T_{i}^{-1}\right)$. Note
$\psi_{i}\equiv1$ on a neighborhood of $Q_{i}$, so that $f=\psi_{i}f$,
because of ${\rm supp}\, f\subset Q_{i}$.

Using the general form of Young's inequality for convolutions (cf.\@
\cite[Proposition 8.9]{FollandRA}), we get
\begin{align*}
\left\Vert \mathcal{F}^{-1}f\right\Vert _{L^{q}} & =\left\Vert \mathcal{F}^{-1}\left(\psi_{i}f\right)\right\Vert _{L^{q}}\\
 & =\left\Vert \left(\mathcal{F}^{-1}\psi_{i}\right)\ast\left(\mathcal{F}^{-1}f\right)\right\Vert _{L^{q}}\\
 & \leq\left\Vert \mathcal{F}^{-1}\psi_{i}\right\Vert _{L^{r}}\cdot\left\Vert \mathcal{F}^{-1}f\right\Vert _{L^{p}},
\end{align*}
where $r\in\left[1,\infty\right]$ has to be chosen such that $1+\frac{1}{q}=\frac{1}{r}+\frac{1}{p}$,
i.e.\@ $\frac{1}{r}=1+\frac{1}{q}-\frac{1}{p}$. This is indeed possible,
since $1\leq p\leq q$, which yields $0\leq1-\frac{1}{p}\leq1+\frac{1}{q}-\frac{1}{p}\leq1$.
Finally, standard properties of the Fourier transform (cf.\@ \cite[Theorem 8.22]{FollandRA})
imply
\[
\mathcal{F}^{-1}\psi_{i}=\left|\det T_{i}\right|\cdot M_{b_{i}}\left[\left(\mathcal{F}^{-1}\psi\right)\circ T_{i}^{T}\right]
\]
and hence (using the change-of-variables formula)
\begin{align*}
\left\Vert \mathcal{F}^{-1}\psi_{i}\right\Vert _{L^{r}} & =\left|\det T_{i}\right|\cdot\left|\det T_{i}^{T}\right|^{-\frac{1}{r}}\cdot\left\Vert \mathcal{F}^{-1}\psi\right\Vert _{L^{r}}\\
 & =\left|\det T_{i}\right|^{1-\left(1+\frac{1}{q}-\frac{1}{p}\right)}\cdot\left\Vert \mathcal{F}^{-1}\psi\right\Vert _{L^{r}}\\
 & =C\cdot\left|\det T_{i}\right|^{\frac{1}{p}-\frac{1}{q}}.
\end{align*}
Altogether, this yields the desired estimate for the case $p\in\left[1,\infty\right]$.
\end{proof}
Now that we have established embeddings of bandlimited $L^{p}$-functions
into $L^{q}$ for $p\leq q$, we give an overview over the alternatives
to the convolution relation $L^{1}\ast L^{q}\hookrightarrow L^{q}$
in case of $q\in\left(0,1\right)$. As above, we first state a general
result and then specialize to the case in which the factors of the
convolution product are supported in sets $Q_{i}$ of a semi-structured
covering $\mathcal{Q}$.
\begin{thm}
\label{thm:QuasiBanachConvolution}(cf.\@ \cite[Proposition 1.5.1]{TriebelTheoryOfFunctionSpaces})

Let $Q_{1},Q_{2}\subset\mathbb{R}^{d}$ be compact and let $p\in\left(0,1\right]$.
Furthermore, assume $\psi\in L^{1}\left(\mathbb{R}^{d}\right)$ with
${\rm supp}\,\psi\subset Q_{1}$ and such that $\mathcal{F}^{-1}\psi\in L^{p}\left(\mathbb{R}^{d}\right)$.
Then, for $f\in L^{p}\left(\mathbb{R}^{d}\right)\cap\mathcal{S}'\left(\mathbb{R}^{d}\right)$
with ${\rm supp}\,\widehat{f}\subset Q_{2}$, we have
\[
\mathcal{F}^{-1}\left(\psi\cdot\widehat{f}\right)=\left(\mathcal{F}^{-1}\psi\right)\ast f\in L^{p}\left(\mathbb{R}^{d}\right)
\]
with
\[
\left\Vert \mathcal{F}^{-1}\left(\smash{\psi\cdot\widehat{f}}\right)\right\Vert _{L^{p}}\leq\left[\lambda\left(Q_{1}-Q_{2}\right)\right]^{\frac{1}{p}-1}\cdot\left\Vert \mathcal{F}^{-1}\psi\right\Vert _{L^{p}}\cdot\left\Vert f\right\Vert _{L^{p}},
\]
where $\lambda$ is the usual Lebesgue measure on $\mathbb{R}^{d}$
and where
\[
Q_{1}-Q_{2}:=\left\{ q_{1}-q_{2}\with q_{1}\in Q_{1},\, q_{2}\in Q_{2}\right\} 
\]
is the \textbf{difference set} of $Q_{1},Q_{2}$, which is compact
and hence measurable, since $Q_{1},Q_{2}$ are compact.\end{thm}
\begin{rem*}
Observe that $\psi\cdot\widehat{f}\in\mathcal{S}'\left(\mathbb{R}^{d}\right)$
is well-defined, even though $\psi$ might not be smooth. Indeed,
since $p\leq1$, Lemma \ref{lem:BandlimitedLpEmbeddingGeneral} shows
$f\in L^{1}\left(\mathbb{R}^{d}\right)$ and hence $\widehat{f}\in C_{0}\left(\mathbb{R}^{d}\right)\subset L^{\infty}\left(\mathbb{R}^{d}\right)$
by the Riemann-Lebesgue lemma, so that $\psi\cdot\widehat{f}\in L^{1}\left(\mathbb{R}^{d}\right)\subset\mathcal{S}'\left(\mathbb{R}^{d}\right)$,
because of $\psi\in L^{1}\left(\mathbb{R}^{d}\right)$.\end{rem*}
\begin{proof}
A proof of this statement can be found in the proof of \cite[Proposition 1.5.1]{TriebelTheoryOfFunctionSpaces}.
Note, however, that the constant $\left[\lambda\left(Q_{1}-Q_{2}\right)\right]^{\frac{1}{p}-1}$
becomes apparent from the proof, but is not stated explicitly in the
statement of \cite[Proposition 1.5.1]{TriebelTheoryOfFunctionSpaces}.

Another proof (stating the constant explicitly) can be found in \cite[Theorem 3.1.4]{VoigtlaenderPhDThesis}.
\end{proof}
We close this section by specializing the above result to a more convenient
version which applies if the sets $Q_{1},Q_{2}$ from above are in
fact members of a semi-structured covering $\mathcal{Q}$. We remark
that a version of the following convolution relation is implicitly
used repeatedly in \cite{BorupNielsenDecomposition} and \cite{HanWangAlphaModulationEmbeddings},
withouth stating it explicitly.
\begin{cor}
\label{cor:SemiStructuredQuasiBanachConvolution}Let $p\in\left(0,1\right]$
and let $\mathcal{Q}=\left(Q_{i}\right)_{i\in I}=\left(T_{i}Q_{i}'+b_{i}\right)_{i\in I}$
be an $L^{p}$-decomposition covering of the open set $\emptyset\neq\mathcal{O}\subset\mathbb{R}^{d}$.

For arbitrary $n\in\mathbb{N}$, there is a constant $C=C\left(\mathcal{Q},n,d,p\right)>0$
with the following property: If $i\in I$ and
\begin{itemize}
\item if $\psi\in L^{1}\left(\mathbb{R}^{d}\right)$ with ${\rm supp}\,\psi\subset Q_{i}^{n\ast}$
and $\mathcal{F}^{-1}\psi\in L^{p}\left(\mathbb{R}^{d}\right)$ and
\item if $f\in\mathcal{D}'\left(\mathcal{O}\right)$ with ${\rm supp}\, f\subset Q_{i}^{n\ast}$
and $\mathcal{F}^{-1}f\in L^{p}\left(\mathbb{R}^{d}\right)$,
\end{itemize}
then $\mathcal{F}^{-1}\left(\psi f\right)\in L^{p}\left(\mathbb{R}^{d}\right)$
with
\[
\left\Vert \mathcal{F}^{-1}\left(\psi f\right)\right\Vert _{L^{p}}\leq C\cdot\left|\det T_{i}\right|^{\frac{1}{p}-1}\cdot\left\Vert \mathcal{F}^{-1}\psi\right\Vert _{L^{p}}\cdot\left\Vert \mathcal{F}^{-1}f\right\Vert _{L^{p}}.\qedhere
\]
\end{cor}
\begin{proof}
A complete proof of this result is given in \cite[Corollary 3.2.15]{VoigtlaenderPhDThesis}.
Here, we use without proof the result of \cite[Lemma 3.2.13]{VoigtlaenderPhDThesis}
which implies that there is a constant $L=L\left(n,\mathcal{Q}\right)>0$
such that
\[
Q_{i}^{n\ast}\subset T_{i}\left(B_{L}\left(0\right)\right)+b_{i}
\]
holds for all $i\in I$. Hence, we can apply Theorem \ref{thm:QuasiBanachConvolution}
with $Q_{1}=Q_{2}=\overline{Q_{i}^{n\ast}}$, since we have
\[
Q_{1}-Q_{2}\subset T_{i}\left(\overline{B_{2L}\left(0\right)}\right)
\]
and hence $\lambda\left(Q_{1}-Q_{2}\right)\leq\left|\det T_{i}\right|\cdot\lambda\left(B_{2L}\left(0\right)\right)$.
\end{proof}
Now, we are equipped with a solid definition of the decomposition
space $\mathcal{D}\left(\mathcal{Q},L^{p},Y\right)$ and certain convolution
relations for $L^{p}$ in case of $p\in\left(0,1\right)$. These will
be put to use in the next section, where we derive sufficient conditions
for existence of an embedding of a decomposition space into a Sobolev
space. Actually, we will derive sufficient conditions for boundedness
of (suitable defined) derivative operators $\partial_{\ast}^{\alpha}:\mathcal{D}\left(\mathcal{Q},L^{p},Y\right)\to L^{q}\left(\mathbb{R}^{d}\right)$.

\section{Sufficient Conditions}

\label{sec:SufficientConditions}In this section, we will show that
the two conditions $p\leq q$ and $Y\hookrightarrow\ell_{u^{\left(k,p,q\right)}}^{q^{\triangledown}}\left(I\right)$
for a suitable weight $u^{\left(k,p,q\right)}$ are sufficient for
the existence of the embedding
\[
\mathcal{D}\left(\mathcal{Q},L^{p},Y\right)\hookrightarrow W^{k,q}\left(\mathbb{R}^{d}\right).
\]
The main ingredient for the proof of this result is the following
lemma which allows to sum a sequence of functions $f_{i}$, each bandlimited
to the set $Q_{i}$. The important fact is that the $L^{p}$-norm
of the sum $\sum_{i\in I}f_{i}$ can be controlled by the $\ell^{p^{\triangledown}}$-norm
of the individual norms $\left\Vert f_{i}\right\Vert _{L^{p}}$. In
most cases, this is a huge improvement over the obvious estimate obtained
by the triangle inequality, which would yield an estimate in terms
of the $\ell^{1}$-norm of the individual functions. In a second step,
we will then estimate the $L^{q}$-norms of the derivatives of the
pieces $f_{i}$ in terms of the norms $\left\Vert f_{i}\right\Vert _{L^{p}}$.
This is possible, since each of the ``pieces'' $f_{i}$ is bandlimited
to the set $Q_{i}$, cf.\@ Lemma \ref{lem:LocalDerivativeEstimate}.

We remark that the proof heavily relies on Plancherel's theorem (for
the case $p=2$) and interpolation. A different (more complicated)
proof of a very similar result was given in \cite[Lemma 5.1.2]{VoigtlaenderPhDThesis}.
\begin{lem}
\label{lem:SufficiencyWithoutDerivativesMainLemma}Let $\mathcal{Q}=\left(Q_{i}\right)_{i\in I}=\left(T_{i}Q_{i}'+b_{i}\right)_{i\in I}$
be an $L^{1}$-decomposition covering of the open set $\emptyset\neq\mathcal{O}\subset\mathbb{R}^{d}$.
Furthermore, let $p\in\left(0,\infty\right]$ and $k\in\mathbb{N}_{0}$
and assume that for each $i\in I$, we are given $f_{i}\in\mathcal{S}'\left(\mathbb{R}^{d}\right)\cap L^{p}\left(\mathbb{R}^{d}\right)$
with Fourier support ${\rm supp}\,\widehat{f_{i}}\subset Q_{i}^{k\ast}$
and such that
\[
\left\Vert \left(\left\Vert f_{i}\right\Vert _{L^{p}}\right)_{i\in I}\right\Vert _{\ell^{p^{\triangledown}}}<\infty,
\]
where $p^{\triangledown}:=\min\left\{ p,p'\right\} $.

Then $\sum_{i\in I}f_{i}\in L^{p}\left(\mathbb{R}^{d}\right)$ with
unconditional convergence of the series in $L^{p}\left(\mathbb{R}^{d}\right)$.
Here, we identify $f_{i}$ with its continuous (even smooth) version
which exists by the Paley-Wiener theorem (cf. \cite[Proposition 9.11]{FollandRA}).
Finally, we have
\begin{equation}
\left\Vert \sum_{i\in I}f_{i}\right\Vert _{L^{p}\left(\mathbb{R}^{d}\right)}\leq C\cdot\left\Vert \left(\left\Vert f_{i}\right\Vert _{L^{p}}\right)_{i\in I}\right\Vert _{\ell^{p^{\triangledown}}}\label{eq:SufficientConditionWithoutDerivativesMainEstimate}
\end{equation}
for some constant $C=C\left(\mathcal{Q},k\right)>0$.\end{lem}
\begin{proof}
Note that we always have $p^{\triangledown}\leq2<\infty$, as can
be seen by distinguishing the cases $p\leq2$ and $p\geq2$. In particular,
since $\left\Vert \left(\left\Vert f_{i}\right\Vert _{L^{p}}\right)_{i\in I}\right\Vert _{\ell^{p^{\triangledown}}}<\infty$,
we easily get $\left\Vert f_{i}\right\Vert _{L^{p}}=0$ for all $i\in I\setminus I_{0}$
for an at most countable set $I_{0}\subset I$. Now, since $f_{i}$
is continuous for each $i\in I$ by assumption, $\left\Vert f_{i}\right\Vert _{L^{p}}=0$
actually implies $f_{i}\equiv0$ everywhere (and not only almost everywhere).
Thus, we can assume in the following that $I=I_{0}$ is countable%
\footnote{Actually, if each $Q_{i}$ is open, then $I$ is necessary countable,
as seen in the proof of Theorem \ref{thm:StructuredAdmissibleCoveringsAreRegular}.
But using the argument from the present proof, we can avoid assuming
openness of the $Q_{i}$.%
}.

Let us first handle the (easier) case $p\in\left(0,1\right]$. In
this case, we have $p^{\triangledown}=p$. Furthermore, it is well-known
that $\left\Vert \cdot\right\Vert _{L^{p}}$ is a $p$-norm, i.e.\@
we have $\left\Vert f+g\right\Vert _{L^{p}}^{p}\leq\left\Vert f\right\Vert _{L^{p}}^{p}+\left\Vert g\right\Vert _{L^{p}}^{p}$
for all measurable $f,g$. Indeed, this is an immediate consequence
of the estimate 
\[
\left|a+b\right|^{p}\leq\left(\left|a\right|+\left|b\right|\right)^{p}\leq\left|a\right|^{p}+\left|b\right|^{p}
\]
which is valid for $a,b\in\mathbb{C}$ since $p\in\left(0,1\right]$.

In view of the above and since $I=I_{0}$ is countable, the monotone
convergence theorem yields
\[
\left\Vert \sum_{i\in I}\left|f_{i}\right|\right\Vert _{L^{p}}^{p}\leq\sum_{i\in I}\left\Vert f_{i}\right\Vert _{L^{p}}^{p}=\left\Vert \left(\left\Vert f_{i}\right\Vert _{L^{p}}\right)_{i\in I}\right\Vert _{\ell^{p}}^{p}=\left\Vert \left(\left\Vert f_{i}\right\Vert _{L^{p}}\right)_{i\in I}\right\Vert _{\ell^{p^{\triangledown}}}^{p}<\infty.
\]
By solidity of $L^{p}\left(\mathbb{R}^{d}\right)$, this implies the
claim. In addition to unconditional convergence in $L^{p}$, we even
get absolute convergence almost everywhere of the series $\sum_{i\in I}f_{i}$.

Now, we handle the case $p\in\left[1,2\right]$. To this end, let
$\Phi=\left(\varphi_{i}\right)_{i\in I}$ be an $L^{1}$-BAPU for
$\mathcal{Q}$. Such a family exists by assumption. For $i\in I$,
let $\varphi_{i}^{\left(k+1\right)\ast}:=\sum_{\ell\in i^{\left(k+1\right)\ast}}\varphi_{\ell}$.
Now, for $p\in\left[1,2\right]$, define the map
\[
\Phi_{p}:\ell^{p}\left(I_{0};L^{p}\left(\mathbb{R}^{d}\right)\right)\to L^{p}\left(\mathbb{R}^{d}\right),\left(g_{i}\right)_{i\in I_{0}}\mapsto\sum_{i\in I_{0}}\mathcal{F}^{-1}\left(\varphi_{i}^{\left(k+1\right)\ast}\widehat{g_{i}}\right).
\]
We will show that this map is well-defined and bounded (with unconditional
convergence of the series in $L^{p}\left(\mathbb{R}^{d}\right)$)
for $p=1$ and $p=2$. By complex interpolation (for vector-valued
$L^{p}$-spaces, cf.\@ \cite[Theorems 5.1.1 and 5.1.2]{BerghLoefstroemInterpolationSpaces}),
it then follows%
\footnote{Complex interpolation shows at least that the series $\Phi_{p}\left(g\right)=\sum_{i\in I_{0}}\mathcal{F}^{-1}\left(\varphi_{i}^{\left(k+1\right)\ast}\widehat{g_{i}}\right)$
is a well-defined element of $L^{1}+L^{2}$ for $g=\left(g_{i}\right)_{i\in I_{0}}\in\ell^{p}\left(I_{0};L^{p}\left(\mathbb{R}^{d}\right)\right)$.
But we also get $\left\Vert \Phi_{p}\left(g\right)\right\Vert _{L^{p}}\lesssim\left\Vert \left(\left\Vert g_{i}\right\Vert _{L^{p}}\right)_{i\in I_{0}}\right\Vert _{\ell^{p}}$.
Because of $p<\infty$, this easily yields unconditional convergence
of the series, since for $\varepsilon>0$, there is a finite set $J_{\varepsilon}\subset I_{0}$
with $\left\Vert \left(\left\Vert g_{i}\right\Vert _{L^{p}}\right)_{i\in I_{0}\setminus J_{\varepsilon}}\right\Vert _{\ell^{p}}<\varepsilon$.%
} that this indeed holds for all $p\in\left[1,2\right]$. Note that
each summand of the series defining $\Phi_{p}\left(\left(g_{i}\right)_{i\in I_{0}}\right)$
is a well-defined (even smooth) function in $L^{p}\left(\mathbb{R}^{d}\right)$,
since $\varphi_{i}^{\left(k+1\right)\ast}\in\mathcal{F}L^{1}\left(\mathbb{R}^{d}\right)$
is an $L^{p}$-Fourier multiplier by Young's inequality (cf.\@ \cite[Theorem 8.7]{FollandRA}).
Here, we used that $\mathcal{F}^{-1}\varphi_{i}\in L^{1}\left(\mathbb{R}^{d}\right)$
for all $i\in I$, by definition of an $L^{p}$-BAPU for $p\in\left[1,\infty\right]$.

For $p=1$, boundedness of $\Phi_{p}$ is easy: By definition of an
$L^{1}$-BAPU, the constant $K:=\sup_{i\in I}\left\Vert \mathcal{F}^{-1}\varphi_{i}\right\Vert _{L^{1}}$
is finite, as is $N:=\sup_{i\in I}\left|i^{\left(k+2\right)\ast}\right|$.
Hence, $\left\Vert \mathcal{F}^{-1}\varphi_{i}^{\left(k+1\right)\ast}\right\Vert _{L^{1}}\leq\sum_{\ell\in i^{\left(k+1\right)\ast}}\left\Vert \mathcal{F}^{-1}\varphi_{\ell}\right\Vert _{L^{1}}\leq NK$
for all $i\in I$, so that we get
\begin{align*}
\left\Vert \sum_{i\in I_{0}}\mathcal{F}^{-1}\left(\varphi_{i}^{\left(k+1\right)\ast}\widehat{g_{i}}\right)\right\Vert _{L^{1}} & \leq\sum_{i\in I_{0}}\left\Vert \mathcal{F}^{-1}\left(\varphi_{i}^{\left(k+1\right)\ast}\widehat{g_{i}}\right)\right\Vert _{L^{1}}\\
 & \leq\sum_{i\in I_{0}}\left\Vert \mathcal{F}^{-1}\varphi_{i}^{\left(k+1\right)\ast}\right\Vert _{L^{1}}\left\Vert g_{i}\right\Vert _{L^{1}}\\
 & \leq NK\cdot\sum_{i\in I_{0}}\left\Vert g_{i}\right\Vert _{L^{1}}\\
 & =NK\cdot\left\Vert \left(g_{i}\right)_{i\in I_{0}}\right\Vert _{\ell^{1}\left(I_{0};L^{1}\left(\mathbb{R}^{d}\right)\right)}.
\end{align*}
This even establishes ``absolute'' -- and hence unconditional --
convergence of the series.

For $p=2$, we employ Plancherel's theorem to get for arbitrary finite
subsets $J\subset I_{0}$
\begin{align*}
\left\Vert \sum_{i\in J}\mathcal{F}^{-1}\left(\varphi_{i}^{\left(k+1\right)\ast}\widehat{g_{i}}\right)\right\Vert _{L^{2}}^{2} & =\left\Vert \sum_{i\in J}\varphi_{i}^{\left(k+1\right)\ast}\widehat{g_{i}}\right\Vert _{L^{2}}^{2}\\
 & =\int_{\mathbb{R}^{d}}\left|\sum_{i\in J}\varphi_{i}^{\left(k+1\right)\ast}\left(\xi\right)\cdot\widehat{g_{i}}\left(\xi\right)\right|^{2}\,{\rm d}\xi.
\end{align*}
But for arbitrary $\xi\in\mathcal{O}$, there is some $i_{\xi}\in I$
with $\xi\in Q_{i_{\xi}}$. For arbitrary $i\in J$ with $\varphi_{i}^{\left(k+1\right)\ast}\left(\xi\right)\neq0$,
this implies $\xi\in Q_{i}^{\left(k+1\right)\ast}\cap Q_{i_{\xi}}\neq\emptyset$
and hence $i\in J\cap i_{\xi}^{\left(k+2\right)\ast}$. By Cauchy-Schwarz,
we conclude
\begin{align*}
\left|\sum_{i\in J}\varphi_{i}^{\left(k+1\right)\ast}\left(\xi\right)\cdot\widehat{g_{i}}\left(\xi\right)\right|^{2} & \leq\sum_{i\in J\cap i_{\xi}^{\left(k+2\right)\ast}}\left|\varphi_{i}^{\left(k+1\right)\ast}\left(\xi\right)\right|^{2}\cdot\sum_{i\in J\cap i_{\xi}^{\left(k+2\right)\ast}}\left|\widehat{g_{i}}\left(\xi\right)\right|^{2}\\
 & \leq\left|J\cap i_{\xi}^{\left(k+2\right)\ast}\right|\cdot\left\Vert \varphi_{i}^{\left(k+1\right)\ast}\right\Vert _{\sup}^{2}\cdot\sum_{i\in J}\left|\widehat{g_{i}}\left(\xi\right)\right|^{2}\\
 & \leq N^{3}K^{2}\cdot\sum_{i\in J}\left|\widehat{g_{i}}\left(\xi\right)\right|^{2}.
\end{align*}
Here, the last step used the easily verifiable estimates $\left|i_{\xi}^{\left(k+2\right)\ast}\right|\leq\sup_{i\in I}\left|i^{\left(k+2\right)\ast}\right|=N$
and 
\[
\left\Vert \varphi_{i}^{\left(k+1\right)\ast}\right\Vert _{\sup}=\left\Vert \mathcal{F}\mathcal{F}^{-1}\varphi_{i}^{\left(k+1\right)\ast}\right\Vert _{\sup}\leq\left\Vert \mathcal{F}^{-1}\varphi_{i}^{\left(k+1\right)\ast}\right\Vert _{L^{1}}\leq NK.
\]

If $\xi\in\mathbb{R}^{d}\setminus\mathcal{O}$, then $\varphi_{i}^{\left(k+1\right)\ast}\left(\xi\right)=0$,
so that the above estimate trivially holds in this case. Altogether,
another application of Plancherel's theorem yields
\begin{align*}
\left\Vert \sum_{i\in J}\mathcal{F}^{-1}\left(\varphi_{i}^{\left(k+1\right)\ast}\widehat{g_{i}}\right)\right\Vert _{L^{2}}^{2} & \leq N^{3}K^{2}\cdot\sum_{i\in J}\int_{\mathbb{R}^{d}}\left|\widehat{g_{i}}\left(\xi\right)\right|^{2}\,{\rm d}\xi\\
 & \leq\left(N^{2}K\right)^{2}\cdot\sum_{i\in J}\left\Vert g_{i}\right\Vert _{L^{2}}^{2}\\
 & =\left(N^{2}K\right)^{2}\cdot\left\Vert \left(g_{i}\cdot\chi_{J}\right)_{i\in I_{0}}\right\Vert _{\ell^{2}\left(I_{0};L^{2}\left(\mathbb{R}^{d}\right)\right)}^{2}\\
 & \leq\left(N^{2}K\right)^{2}\cdot\left\Vert \left(g_{i}\right)_{i\in I_{0}}\right\Vert _{\ell^{2}\left(I_{0};L^{2}\left(\mathbb{R}^{d}\right)\right)}^{2}.
\end{align*}
Now, since $\left(g_{i}\right)_{i\in I_{0}}\in\ell^{2}\left(I_{0};L^{2}\left(\mathbb{R}^{d}\right)\right)$,
we can choose for arbitrary $\varepsilon>0$ a finite subset $J_{0}\subset I_{0}$
with $\left\Vert \left(g_{i}\cdot\chi_{J_{0}^{c}}\right)_{i\in I_{0}}\right\Vert _{\ell^{2}\left(I_{0};L^{2}\left(\mathbb{R}^{d}\right)\right)}<\varepsilon$.
Together with the above estimate and since $I_{0}$ is countable,
this easily entails that the series $\Phi_{2}\left(\left(g_{i}\right)_{i\in I_{0}}\right)=\sum_{i\in I_{0}}\mathcal{F}^{-1}\left(\varphi_{i}^{\left(k+1\right)\ast}\widehat{g_{i}}\right)$
converges unconditionally in $L^{2}\left(\mathbb{R}^{d}\right)$,
with 
\[
\left\Vert \Phi_{2}\left(\left(g_{i}\right)_{i\in I_{0}}\right)\right\Vert _{L^{2}\left(\mathbb{R}^{d}\right)}\leq N^{2}K\cdot\left\Vert \left(g_{i}\right)_{i\in I_{0}}\right\Vert _{\ell^{2}\left(I_{0};L^{2}\left(\mathbb{R}^{d}\right)\right)}.
\]
Note that the constant $N^{2}K$ only depends on $N=N\left(\mathcal{Q},k\right)$
and on $K=K\left(\left(\varphi_{i}\right)_{i\in I}\right)=K\left(\mathcal{Q}\right)$.

Because of $\left\Vert \Phi_{1}\right\Vert \leq NK\leq N^{2}K$ and
$\left\Vert \Phi_{2}\right\Vert \leq N^{2}K$, complex interpolation
implies that each map $\Phi_{p}$ is well-defined and bounded with
$\left\Vert \Phi_{p}\right\Vert \leq N^{2}K=:C'=C'\left(\mathcal{Q},k\right)$
for $p\in\left[1,2\right]$. To complete the proof for $p\in\left[1,2\right]$,
it remains to show that boundedness of $\Phi_{p}$ implies validity
of equation (\ref{eq:SufficientConditionWithoutDerivativesMainEstimate})
(with unconditional convergence of the series). But since ${\rm supp}\,\widehat{f_{i}}\subset Q_{i}^{k\ast}$
holds for all $i\in I$ by assumption and because of $\varphi_{i}^{\left(k+1\right)\ast}\equiv1$
on $Q_{i}^{k\ast}$, we get 
\[
f_{i}=\mathcal{F}^{-1}\widehat{f_{i}}=\mathcal{F}^{-1}\left(\varphi_{i}^{\left(k+1\right)\ast}\widehat{f_{i}}\right)
\]
for all $i\in I$, so that 
\[
\sum_{i\in I}f_{i}=\sum_{i\in I_{0}}f_{i}=\sum_{i\in I_{0}}\mathcal{F}^{-1}\left(\varphi_{i}^{\left(k+1\right)\ast}\widehat{f_{i}}\right)=\Phi_{p}\left(\left(f_{i}\right)_{i\in I_{0}}\right).
\]
Hence,
\[
\left\Vert \sum_{i\in I}f_{i}\right\Vert _{L^{p}\left(\mathbb{R}^{d}\right)}\leq\left\Vert \Phi_{p}\right\Vert \cdot\left\Vert \left(f_{i}\right)_{i\in I_{0}}\right\Vert _{\ell^{p}\left(I_{0};L^{p}\left(\mathbb{R}^{d}\right)\right)}\leq C'\cdot\left\Vert \left(f_{i}\right)_{i\in I}\right\Vert _{\ell^{p^{\triangledown}}\left(I;L^{p}\left(\mathbb{R}^{d}\right)\right)}<\infty,
\]
since $p\in\left[1,2\right]$ implies $p^{\triangledown}=p$.

Thus, it remains to consider the case $p\in\left[2,\infty\right]$.
Here, instead of the map $\Phi_{p}$ from above, we consider
\[
\Psi_{p}:\ell^{p'}\left(I_{0};L^{p}\left(\mathbb{R}^{d}\right)\right)\to L^{p}\left(\mathbb{R}^{d}\right),\left(g_{i}\right)_{i\in I_{0}}\mapsto\sum_{i\in I_{0}}\mathcal{F}^{-1}\left(\varphi_{i}^{\left(k+1\right)\ast}\widehat{g_{i}}\right).
\]
Note that for $p=2$, we have $p'=p$, so that $\Psi_{2}=\Phi_{2}$
is bounded with unconditional convergence of the defining series.
Thus, another complex interpolation argument shows that it suffices
to prove boundedness of $\Psi_{\infty}$ (with unconditional convergence
of the series). Once this is done, validity of (\ref{eq:SufficientConditionWithoutDerivativesMainEstimate})
for $p\in\left[2,\infty\right]$ follows exactly as for $p\in\left[1,2\right]$,
since we have $p^{\triangledown}=p'<\infty$ for $p\in\left[2,\infty\right]$.
Note that the complex interpolation argument uses that taking the
conjugate exponent ``commutes'' with interpolation.

To show boundedness of $\Psi_{\infty}$, note that $\infty'=1$. Hence,
we can argue as for $p=1$. Indeed,
\begin{align*}
\left\Vert \sum_{i\in I_{0}}\mathcal{F}^{-1}\left(\varphi_{i}^{\left(k+1\right)\ast}\widehat{g_{i}}\right)\right\Vert _{L^{\infty}} & \leq\sum_{i\in I_{0}}\left\Vert \mathcal{F}^{-1}\left(\varphi_{i}^{\left(k+1\right)\ast}\widehat{g_{i}}\right)\right\Vert _{L^{\infty}}\\
 & \leq\sum_{i\in I_{0}}\left\Vert \mathcal{F}^{-1}\varphi_{i}^{\left(k+1\right)\ast}\right\Vert _{L^{1}}\left\Vert g_{i}\right\Vert _{L^{\infty}}\\
 & \leq NK\cdot\sum_{i\in I_{0}}\left\Vert g_{i}\right\Vert _{L^{\infty}}\\
 & =NK\cdot\left\Vert \left(g_{i}\right)_{i\in I_{0}}\right\Vert _{\ell^{\infty'}\left(I_{0};L^{\infty}\left(\mathbb{R}^{d}\right)\right)}.
\end{align*}
Here, we get ``absolute'' and hence unconditional convergence of
the series. This completes the proof.
\end{proof}
Now, to derive a sufficient condition for boundedness of the derivative
operator 
\[
\partial_{\ast}^{\alpha}:\mathcal{D}\left(\mathcal{Q},L^{p},Y\right)\to L^{q}\left(\mathbb{R}^{d}\right),
\]
we need a way to estimate $\left\Vert \partial^{\alpha}\left(\mathcal{F}^{-1}\left(\varphi_{i}\widehat{g}\right)\right)\right\Vert _{L^{q}}$
in terms of $\left\Vert \mathcal{F}^{-1}\left(\varphi_{i}\widehat{g}\right)\right\Vert _{L^{p}}$.
Such an estimate is established in our next lemma.
\begin{lem}
\label{lem:LocalDerivativeEstimate}Let $\mathcal{Q}=\left(Q_{i}\right)_{i\in I}=\left(T_{i}Q_{i}'+b_{i}\right)_{i\in I}$
be a regular covering of the open set $\emptyset\neq\mathcal{O}\subset\mathbb{R}^{d}$.
Let $p\in\left(0,\infty\right]$ and assume that $\Phi=\left(\varphi_{i}\right)_{i\in I}$
is a regular partition of unity subordinate to $\mathcal{Q}$.

Then
\begin{equation}
\left\Vert \partial^{\alpha}\left[\mathcal{F}^{-1}\left(\varphi_{i}f\right)\right]\right\Vert _{L^{p}}\leq C_{\alpha,p,\mathcal{Q},\Phi}\cdot\left(\left|b_{i}\right|^{\left|\alpha\right|}+\left\Vert T_{i}\right\Vert ^{\left|\alpha\right|}\right)\cdot\left\Vert \mathcal{F}^{-1}\left[\varphi_{i}^{\ast}f\right]\right\Vert _{L^{p}}\label{eq:LocalDerivativeLpEstimate}
\end{equation}
holds for each $f\in\mathcal{D}'\left(\mathcal{O}\right)$ and each
$i\in I$.\end{lem}
\begin{proof}
Let us first handle the case $p\in\left[1,\infty\right]$. Clearly,
we may suppose that the right-hand side of equation (\ref{eq:LocalDerivativeLpEstimate})
is finite. But $\varphi_{i}^{\ast}\equiv1$ on $Q_{i}$, whereas $\varphi_{i}$
vanishes outside of $Q_{i}$. Hence, $\varphi_{i}=\varphi_{i}^{\ast}\varphi_{i}$
and thus
\begin{align*}
\left\Vert \partial^{\alpha}\left[\mathcal{F}^{-1}\left(\varphi_{i}f\right)\right]\right\Vert _{L^{p}} & =\left\Vert \partial^{\alpha}\left[\mathcal{F}^{-1}\left(\varphi_{i}\varphi_{i}^{\ast}f\right)\right]\right\Vert _{L^{p}}\\
 & =\left\Vert \left[\partial^{\alpha}\left(\mathcal{F}^{-1}\varphi_{i}\right)\right]\ast\mathcal{F}^{-1}\left(\varphi_{i}^{\ast}f\right)\right\Vert _{L^{p}}\\
\left(\text{by Young's inequality}\right) & \leq\left\Vert \partial^{\alpha}\left(\mathcal{F}^{-1}\varphi_{i}\right)\right\Vert _{L^{1}}\cdot\left\Vert \mathcal{F}^{-1}\left(\varphi_{i}^{\ast}f\right)\right\Vert _{L^{p}}\\
\left(\text{by Lemma }\ref{lem:SufficientConditionDerivativeEstimate}\text{ with }p=1\right) & \leq C_{\alpha}\cdot\left(\left\Vert T_{i}\right\Vert ^{\left|\alpha\right|}+\left|b_{i}\right|^{\left|\alpha\right|}\right)\cdot\left\Vert \mathcal{F}^{-1}\left(\varphi_{i}^{\ast}f\right)\right\Vert _{L^{p}}
\end{align*}
for a constant $C_{\alpha}=C_{\alpha}\left(\mathcal{Q},\Phi,d\right)>0$
which is independent of $i\in I$. Here, it is worth noting that $\varphi_{i}^{\ast}\in C_{c}^{\infty}\left(\mathcal{O}\right)$,
so that $\varphi_{i}^{\ast}f$ is a well-defined (tempered) distribution
on $\mathbb{R}^{d}$ with compact support.

It remains to consider $p\in\left(0,1\right)$. The argument for
proving equation (\ref{eq:LocalDerivativeLpEstimate}) is analogous
to the one before, with the important exception that Young's inequality
$L^{1}\ast L^{p}\hookrightarrow L^{p}$ fails for $p\in\left(0,1\right)$.
But Corollary \ref{cor:SemiStructuredQuasiBanachConvolution} shows
that for $f\in C_{c}^{\infty}\left(\mathbb{R}^{d}\right)$ with ${\rm supp}\, f\subset Q_{i}^{k\ast}$
and a tempered distribution $g\in\mathcal{S}'\left(\mathbb{R}^{d}\right)$
with ${\rm supp}\, g\subset Q_{i}^{k\ast}$ and with $\mathcal{F}^{-1}g\in L^{p}\left(\mathbb{R}^{d}\right)$,
we have
\begin{equation}
\left\Vert \mathcal{F}^{-1}\left(fg\right)\right\Vert _{L^{p}}\lesssim_{k,p,\mathcal{Q}}\left|\det T_{i}\right|^{\frac{1}{p}-1}\cdot\left\Vert \mathcal{F}^{-1}f\right\Vert _{L^{p}}\cdot\left\Vert \mathcal{F}^{-1}g\right\Vert _{L^{p}}.\label{eq:LpConvolutionRelation}
\end{equation}
Note that this means that the Fourier multiplier $f$ acts boundedly
on those $L^{p}$ functions with Fourier support in $Q_{i}^{k\ast}$
for each fixed $k\in\mathbb{N}_{0}$, but not in general on arbitrary
$L^{p}$ functions. Furthermore, the operator norm of the resulting
Fourier multiplier depends in a nontrivial way on $\left|\det T_{i}\right|\approx\lambda\left(Q_{i}\right)$.

In the present case, this convolution relation implies (similar to
the argument above)
\begin{align*}
\left\Vert \partial^{\alpha}\left[\mathcal{F}^{-1}\left(\varphi_{i}f\right)\right]\right\Vert _{L^{p}} & =\left\Vert \left[\partial^{\alpha}\left(\mathcal{F}^{-1}\varphi_{i}\right)\right]\ast\mathcal{F}^{-1}\left(\varphi_{i}^{\ast}f\right)\right\Vert _{L^{p}}\\
 & \leq C_{\mathcal{Q},p}\cdot\left|\det T_{i}\right|^{\frac{1}{p}-1}\cdot\left\Vert \partial^{\alpha}\left(\mathcal{F}^{-1}\varphi_{i}\right)\right\Vert _{L^{p}}\cdot\left\Vert \mathcal{F}^{-1}\left(\varphi_{i}^{\ast}f\right)\right\Vert _{L^{p}}\\
\left(\text{by Lemma }\ref{lem:SufficientConditionDerivativeEstimate}\right) & \leq C_{\alpha,p,\mathcal{Q},\Phi}\cdot\left(\left\Vert T_{i}\right\Vert ^{\left|\alpha\right|}+\left|b_{i}\right|^{\left|\alpha\right|}\right)\cdot\left\Vert \mathcal{F}^{-1}\left(\varphi_{i}^{\ast}f\right)\right\Vert _{L^{p}},
\end{align*}
as desired. Here, we used ${\rm supp}\left(\varphi_{i}^{\ast}f\right)\subset Q_{i}^{\ast}$
and ${\rm supp}\,\varphi_{i}\subset Q_{i}\subset Q_{i}^{\ast}$ to
justify application of the convolution relation from equation (\ref{eq:LpConvolutionRelation}).
\end{proof}
Before we continue the development of our sufficient condition for
boundedness of certain partial derivative operators, we first introduce
a convenient notation.
\begin{defn}
\label{def:EmbeddingDefinition}Let $\mathcal{Q}=\left(Q_{i}\right)_{i\in I}$
be an $L^{p}$-decomposition covering of an open set $\emptyset\neq\mathcal{O}\subset\mathbb{R}^{d}$
and let $Y\leq\mathbb{C}^{I}$ be a $\mathcal{Q}$-regular sequence
space. We define
\[
\mathcal{S}_{\mathcal{O}}^{p,Y}\left(\mathbb{R}^{d}\right):=\mathcal{S}_{\mathcal{O}}\left(\mathbb{R}^{d}\right)\cap\mathcal{D}\left(\mathcal{Q},L^{p},Y\right),
\]
where we use the notation
\[
\mathcal{S}_{\mathcal{O}}\left(\mathbb{R}^{d}\right):=\left\{ f\in\mathcal{S}\left(\mathbb{R}^{d}\right)\with\widehat{f}\in C_{c}^{\infty}\left(\mathcal{O}\right)\right\} .
\]

We say that the decomposition space $\mathcal{D}\left(\mathcal{Q},L^{p},Y\right)$
embeds into a function space $Z$ on $\mathbb{R}^{d}$, written
\[
\mathcal{D}\left(\mathcal{Q},L^{p},Y\right)\hookrightarrow Z,
\]
if there is a bounded linear map $\iota:\mathcal{D}\left(\mathcal{Q},L^{p},Y\right)\to Z$
which satisfies $\iota f=f$ for all $f\in\mathcal{S}_{\mathcal{O}}^{p,Y}\left(\mathbb{R}^{d}\right)$.
In case of $Z=W^{k,q}\left(\mathbb{R}^{d}\right)$ with $q\in\left(0,1\right)$,
we instead require $\iota f=\left(\partial^{\alpha}f\right)_{\left|\alpha\right|\leq k}$
for all $f\in\mathcal{S}_{\mathcal{O}}^{p,Y}\left(\mathbb{R}^{d}\right)$.

We say that $\mathcal{D}\left(\mathcal{Q},L^{p},Y\right)$ embeds
injectively into $Z$, written $\mathcal{D}\left(\mathcal{Q},L^{p},Y\right)\overset{{\rm inj.}}{\hookrightarrow}Z$,
if the map $\iota$ can be chosen to be injective.
\end{defn}
Now, we finally state and prove our sufficient conditions for boundedness
of the partial derivative maps $\partial_{\ast}^{\alpha}:\mathcal{D}\left(\mathcal{Q},L^{p},Y\right)\to L^{q}\left(\mathbb{R}^{d}\right)$.
We use the notation $\partial_{\ast}^{\alpha}$ instead of $\partial^{\alpha}$
to distinguish the map which we define in the following theorem from
the usual (weak) partial derivatives $\partial^{\alpha}$.
\begin{thm}
\label{thm:SufficientConditionsForBoundednessOfDerivative}Let $\mathcal{Q}=\left(Q_{i}\right)_{i\in I}=\left(T_{i}Q_{i}'+b_{i}\right)_{i\in I}$
be a regular covering of the open set $\emptyset\neq\mathcal{O}\subset\mathbb{R}^{d}$
and let $p,q\in\left(0,\infty\right]$ with $p\leq q$ and $k\in\mathbb{N}_{0}$.
Define the weight
\[
u^{\left(k,p,q\right)}:=\left(\left|\det T_{i}\right|^{\frac{1}{p}-\frac{1}{q}}\cdot\left(\left|b_{i}\right|^{k}+\left\Vert T_{i}\right\Vert ^{k}\right)\right)_{i\in I}.
\]

Assume that $Y\leq\mathbb{C}^{I}$ is a $\mathcal{Q}$-regular sequence
space on $I$ satisfying $Y\hookrightarrow\ell_{u^{\left(k,p,q\right)}}^{q^{\triangledown}}\left(I\right)$
with $q^{\triangledown}=\min\left\{ q,q'\right\} $. Let $\Phi=\left(\varphi_{i}\right)_{i\in I}$
be a regular partition of unity for $\mathcal{Q}$. Then for each
$\alpha\in\mathbb{N}_{0}^{d}$ with $\left|\alpha\right|=k$, the
map
\begin{eqnarray*}
\partial_{\ast}^{\alpha}: & \mathcal{D}\left(\mathcal{Q},L^{p},Y\right) & \to L^{q}\left(\mathbb{R}^{d}\right),\\
 & f & \mapsto\sum_{i\in I}\partial^{\alpha}\left[\mathcal{F}^{-1}\left(\varphi_{i}\cdot\smash{\widehat{f}}\:\right)\right]
\end{eqnarray*}
is well-defined and bounded, with unconditional convergence of the
series in $L^{q}\left(\mathbb{R}^{d}\right)$.

Furthermore, we have $\partial_{\ast}^{\alpha}f=\partial^{\alpha}f$
for all $f\in\mathcal{S}_{\mathcal{O}}^{p,Y}\left(\mathbb{R}^{d}\right)$.\end{thm}
\begin{rem*}
We will see in Lemma \ref{lem:ShiftedBallsFitIntoSets} that 
\[
\left|b_{i}\right|+\left\Vert T_{i}\right\Vert \asymp\sup_{x\in Q_{i}}\left|x\right|
\]
is (for a given family $\mathcal{Q}=\left(Q_{i}\right)_{i\in I}$)
(asymptotically) independent of the specific choice of $T_{i},b_{i},Q_{i}'$,
as long as the resulting semi-structured covering $\mathcal{Q}=\left(T_{i}Q_{i}'+b_{i}\right)_{i\in I}$
is tight.\end{rem*}
\begin{proof}
Since $\mathcal{Q}$ is admissible, the constant $N:=\sup_{i\in I}\left|i^{\ast}\right|$
is finite. By Corollary \ref{cor:RegularPartitionsAreBAPUs}, $\Phi$
is an $L^{p}$-BAPU for $\mathcal{Q}$. Fix $\alpha\in\mathbb{N}_{0}^{d}$
with $\left|\alpha\right|=k$.

Let $f\in\mathcal{D}\left(\mathcal{Q},L^{p},Y\right)$ and note $g:=\widehat{f}\in\mathcal{D}_{\mathcal{F}}\left(\mathcal{Q},L^{p},Y\right)\subset\mathcal{D}'\left(\mathcal{O}\right)$.
Set $g_{i}:=\partial^{\alpha}\left[\mathcal{F}^{-1}\left(\varphi_{i}g\right)\right]$
for $i\in I$.

Note that $g_{i}\in C^{\infty}\left(\mathbb{R}^{d}\right)\cap\mathcal{S}'\left(\mathbb{R}^{d}\right)$
is a well-defined polynomially bounded function by the Paley Wiener
theorem (cf.\@ \cite[Proposition 9.11]{FollandRA}), since $\varphi_{i}g$
is compactly supported. Furthermore, Lemma \ref{lem:LocalDerivativeEstimate}
(with $q$ instead of $p$) implies -- with $u_{i}^{\left(k\right)}:=\left|b_{i}\right|^{k}+\left\Vert T_{i}\right\Vert ^{k}$
for $i\in I$ -- that
\begin{align}
\left\Vert g_{i}\right\Vert _{L^{q}} & =\left\Vert \partial^{\alpha}\left[\mathcal{F}^{-1}\left(\varphi_{i}g\right)\right]\right\Vert _{L^{q}}\nonumber \\
\left(\text{by Lemma }\ref{lem:LocalDerivativeEstimate}\right) & \leq C_{\alpha,p,\mathcal{Q},\Phi}\cdot\left(\left|b_{i}\right|^{\left|\alpha\right|}+\left\Vert T_{i}\right\Vert ^{\left|\alpha\right|}\right)\cdot\left\Vert \mathcal{F}^{-1}\left[\varphi_{i}^{\ast}g\right]\right\Vert _{L^{q}}\nonumber \\
\left(\text{by Corollary }\ref{cor:BandlimitedLpEmbeddingSemiStructured}\text{ since }p\leq q\right) & \leq C_{k,p,q,\mathcal{Q},\Phi}\cdot\left|\det T_{i}\right|^{\frac{1}{p}-\frac{1}{q}}\cdot u_{i}^{\left(k\right)}\cdot\left\Vert \mathcal{F}^{-1}\left[\varphi_{i}^{\ast}g\right]\right\Vert _{L^{p}}\nonumber \\
 & =C_{k,p,q,\mathcal{Q},\Phi}\cdot u_{i}^{\left(k,p,q\right)}\cdot\left\Vert \mathcal{F}^{-1}\left[\varphi_{i}^{\ast}g\right]\right\Vert _{L^{p}}\label{eq:DerivativeBoundedness}
\end{align}
for all $i\in I$.

Note that taking partial derivatives can not increase the Fourier
support, so that
\[
{\rm supp}\,\widehat{g_{i}}={\rm supp}\left[\mathcal{F}\left(\partial^{\alpha}\left[\mathcal{F}^{-1}\left(\varphi_{i}g\right)\right]\right)\right]\subset{\rm supp}\left(\mathcal{F}\left[\mathcal{F}^{-1}\left(\varphi_{i}g\right)\right]\right)\subset Q_{i}.
\]
Hence, Lemma \ref{lem:SufficiencyWithoutDerivativesMainLemma} yields
unconditional convergence in $L^{q}\left(\mathbb{R}^{d}\right)$ of
the series
\[
\sum_{i\in I}g_{i}=\sum_{i\in I}\partial^{\alpha}\left[\mathcal{F}^{-1}\left(\varphi_{i}\cdot\smash{\widehat{f}}\:\right)\right]=\partial_{\ast}^{\alpha}f,
\]
with
\begin{align*}
\left\Vert \partial_{\ast}^{\alpha}f\right\Vert _{L^{q}} & =\left\Vert \sum_{i\in I}g_{i}\right\Vert _{L^{q}}\\
 & \lesssim\left\Vert \left(\left\Vert g_{i}\right\Vert _{L^{q}}\right)_{i\in I}\right\Vert _{\ell^{q^{\triangledown}}}\\
\left(\text{by equation }\eqref{eq:DerivativeBoundedness}\right) & \leq C_{k,p,q,\mathcal{Q},\Phi}\cdot\left\Vert \left(u_{i}^{\left(k,p,q\right)}\cdot\left\Vert \mathcal{F}^{-1}\left[\varphi_{i}^{\ast}g\right]\right\Vert _{L^{p}}\right)_{i\in I}\right\Vert _{\ell^{q^{\triangledown}}}\\
\left(\text{since }\left|i^{\ast}\right|\leq N\text{ and }L^{p}\text{ is quasi-normed}\right) & \lesssim\left\Vert \left(u_{i}^{\left(k,p,q\right)}\cdot\sum_{\ell\in i^{\ast}}\left\Vert \mathcal{F}^{-1}\left[\varphi_{\ell}g\right]\right\Vert _{L^{p}}\right)_{i\in I}\right\Vert _{\ell^{q^{\triangledown}}}\\
 & =\left\Vert \left(\sum_{\ell\in i^{\ast}}\cdot\left\Vert \mathcal{F}^{-1}\left[\varphi_{\ell}g\right]\right\Vert _{L^{p}}\right)_{i\in I}\right\Vert _{\ell_{u^{\left(k,p,q\right)}}^{q^{\triangledown}}}\\
\left(\text{since }Y\hookrightarrow\ell_{u^{\left(k,p,q\right)}}^{q^{\triangledown}}\left(I\right)\right) & \lesssim\left\Vert \left(\sum_{\ell\in i^{\ast}}\cdot\left\Vert \mathcal{F}^{-1}\left[\varphi_{\ell}g\right]\right\Vert _{L^{p}}\right)_{i\in I}\right\Vert _{Y}\\
\left(\text{since }Y\text{ is }\mathcal{Q}\text{-regular}\right) & \lesssim\left\Vert \left(\left\Vert \mathcal{F}^{-1}\left[\varphi_{i}g\right]\right\Vert _{L^{p}}\right)_{i\in I}\right\Vert _{Y}\\
 & =\left\Vert g\right\Vert _{\mathcal{D}_{\mathcal{F}}\left(\mathcal{Q},L^{p},Y\right)}=\left\Vert f\right\Vert _{\mathcal{D}\left(\mathcal{Q},L^{p},Y\right)}<\infty.
\end{align*}

It remains to show $\partial_{\ast}^{\alpha}f=\partial^{\alpha}f$
for $f\in\mathcal{S}_{\mathcal{O}}^{p,Y}\left(\mathbb{R}^{d}\right)$.
Note that $\sum_{i\in I}\varphi_{i}\equiv1$ on $\mathcal{O}$ implies
$\mathcal{O}\subset\bigcup_{i\in I}U_{i}$ with $U_{i}:=\varphi_{i}^{-1}\left(\mathbb{C}\setminus\left\{ 0\right\} \right)$.
Since $K:={\rm supp}\,\widehat{f}\subset\mathcal{O}$ is compact,
there is thus a finite subset $I_{0}\subset I$ such that $U:=\bigcup_{i\in I_{0}}U_{i}\supset K$.
Now, $\varphi_{i}\equiv0$ on $U\supset K$ if $i\notin I_{0}^{\ast}$.
Hence, $\varphi_{I_{0}^{\ast}}:=\sum_{\ell\in I_{0}^{\ast}}\varphi_{\ell}$
satisfies $\varphi_{I_{0}^{\ast}}\equiv1$ on $U\supset K={\rm supp}\,\widehat{f}$
and hence $\widehat{f}=\varphi_{I_{0}^{\ast}}\cdot\widehat{f}$. Furthermore,
$\varphi_{i}\widehat{f}\equiv0$ for all $i\in I\setminus I_{0}^{\ast}$.
All in all, this yields
\begin{align*}
\partial_{\ast}^{\alpha}f & =\sum_{i\in I}\partial^{\alpha}\left[\mathcal{F}^{-1}\left(\varphi_{i}\cdot\smash{\widehat{f}}\,\right)\right]\\
 & =\sum_{i\in I_{0}^{\ast}}\partial^{\alpha}\left[\mathcal{F}^{-1}\left(\varphi_{i}\cdot\smash{\widehat{f}}\,\right)\right]\\
 & =\partial^{\alpha}\left[\mathcal{F}^{-1}\left(\varphi_{I_{0}^{\ast}}\cdot\smash{\widehat{f}}\,\right)\right]=\partial^{\alpha}\left[\mathcal{F}^{-1}\widehat{f}\right]=\partial^{\alpha}f,
\end{align*}
as claimed. Note that this calculation is justified, since $I_{0}^{\ast}\subset I$
is finite.
\end{proof}
As a corollary of the preceding theorem, we derive sufficient conditions
for embeddings of decomposition spaces into Sobolev spaces.
\begin{cor}
\label{cor:SufficientConditionsForSobolevEmbeddings}Let $\mathcal{Q}=\left(Q_{i}\right)_{i\in I}=\left(T_{i}Q_{i}'+b_{i}\right)_{i\in I}$
be a regular covering of the open set $\emptyset\neq\mathcal{O}\subset\mathbb{R}^{d}$
and let $p,q\in\left(0,\infty\right]$ and $k\in\mathbb{N}_{0}$.

For $i\in I$, define
\begin{align*}
v_{i} & :=\left|\det T_{i}\right|^{\frac{1}{p}-\frac{1}{q}},\\
w_{i} & :=\left|\det T_{i}\right|^{\frac{1}{p}-\frac{1}{q}}\left(\left|b_{i}\right|^{k}+\left\Vert T_{i}\right\Vert ^{k}\right).
\end{align*}
If we have $p\leq q$ and if the $\mathcal{Q}$-regular sequence space
$Y\leq\mathbb{C}^{I}$ satisfies 
\[
Y\hookrightarrow\ell_{v}^{q^{\triangledown}}\left(I\right)\qquad\text{ and }\qquad Y\hookrightarrow\ell_{w}^{q^{\triangledown}}\left(I\right),
\]
then $\mathcal{D}\left(\mathcal{Q},L^{p},Y\right)\hookrightarrow W^{k,q}\left(\mathbb{R}^{d}\right)$.

More precisely, we have the following:
\begin{enumerate}
\item For $q\geq1$, the map $\iota_{q}^{\left(k\right)}:=\partial_{\ast}^{0}$
from Theorem \ref{thm:SufficientConditionsForBoundednessOfDerivative}
is injective and bounded as a map
\[
\iota_{q}^{\left(k\right)}:\mathcal{D}\left(\mathcal{Q},L^{p},Y\right)\to W^{k,q}\left(\mathbb{R}^{d}\right).
\]
Furthermore, we have the following:

\begin{enumerate}
\item $\iota_{q}^{\left(k\right)}f=f$ for all $f\in\mathcal{S}_{\mathcal{O}}^{p,Y}\left(\mathbb{R}^{d}\right)$
and
\item $\partial^{\alpha}\left(\iota_{q}^{\left(k\right)}f\right)=\partial_{\ast}^{\alpha}f$
(with $\partial_{\ast}^{\alpha}$ as in Theorem \ref{thm:SufficientConditionsForBoundednessOfDerivative})
for all $f\in\mathcal{D}\left(\mathcal{Q},L^{p},Y\right)$ and $\left|\alpha\right|\leq k$.
\item If $q=\infty$, then $\iota_{\infty}^{\left(k\right)}$ is even well-defined
and bounded as a map into
\[
C_{b}^{k}\left(\mathbb{R}^{d}\right):=\left\{ f\in C^{k}\left(\mathbb{R}^{d}\right)\with\left\Vert f\right\Vert _{C_{b}^{k}}:=\sum_{\left|\alpha\right|\leq k}\left\Vert \partial^{\alpha}f\right\Vert _{\sup}<\infty\right\} .
\]

\end{enumerate}
\item For $q<1$, the map 
\[
\iota_{q}^{\left(k\right)}:\mathcal{D}\left(\mathcal{Q},L^{p},Y\right)\to W^{k,q}\left(\mathbb{R}^{d}\right),f\mapsto\left(\partial_{\ast}^{\alpha}f\right)_{\left|\alpha\right|\leq k}
\]
is well-defined and bounded, with the $\partial_{\ast}^{\alpha}$
as in Theorem \ref{thm:SufficientConditionsForBoundednessOfDerivative}.
Furthermore, $\iota_{q}^{\left(k\right)}f=\left(\partial^{\alpha}f\right)_{\left|\alpha\right|\leq k}$
for all $f\in\mathcal{S}_{\mathcal{O}}^{p,Y}\left(\mathbb{R}^{d}\right)$.
\item If $q<1$ and if there are $r\geq1$ and $\ell\in\mathbb{N}_{0}$
such that $\iota_{r}^{\left(\ell\right)}$ is bounded (with unconditional
convergence of the series), then $\iota_{q}^{\left(k\right)}$ is
injective with 
\[
\iota_{q}^{\left(k\right)}f\bigg|_{\left|\alpha\right|\leq\min\left\{ k,\ell\right\} }=\left(\partial^{\alpha}\left(\iota_{r}^{\left(\ell\right)}f\right)\right)_{\left|\alpha\right|\leq\min\left\{ k,\ell\right\} }\qquad\text{ for all }f\in\mathcal{D}\left(\mathcal{Q},L^{p},Y\right).\qedhere
\]

\end{enumerate}
\end{cor}
\begin{proof}
Let $\Phi=\left(\varphi_{i}\right)_{i\in I}$ be the regular partition
of unity subordinate to $\mathcal{Q}$ which is used in Theorem \ref{thm:SufficientConditionsForBoundednessOfDerivative}
to define the maps $\partial_{\ast}^{\alpha}$. Using the weights
$u^{\left(n,p,q\right)}$ for $n\in\mathbb{N}_{0}$ which were defined
in Theorem \ref{thm:SufficientConditionsForBoundednessOfDerivative},
we have $v=u^{\left(0,p,q\right)}$, so that the map 
\[
\iota_{q}^{\left(0\right)}=\partial_{\ast}^{0}:\mathcal{D}\left(\mathcal{Q},L^{p},Y\right)\to L^{q}\left(\mathbb{R}^{d}\right),f\mapsto\sum_{i\in I}\mathcal{F}^{-1}\left(\varphi_{i}\cdot\widehat{f}\,\right)
\]
is well-defined with unconditional convergence of the series in $L^{q}\left(\mathbb{R}^{d}\right)$
and with $\iota_{q}^{\left(0\right)}f=f$ for $f\in\mathcal{S}_{\mathcal{O}}^{p,Y}\left(\mathbb{R}^{d}\right)$.

We first show that $\iota_{q}^{\left(0\right)}$ is injective for
$q\geq1$. To this end, let $f\in\mathcal{D}\left(\mathcal{Q},L^{p},Y\right)$
with $\iota_{q}^{\left(0\right)}f=0$. Since convolution with the
Schwartz function $\mathcal{F}^{-1}\varphi_{j}$ for arbitrary $j\in I$
is a bounded linear operator on $L^{q}\left(\mathbb{R}^{d}\right)$
(this crucially uses $q\geq1$), we conclude (using unconditional
convergence of the series) that
\begin{align*}
0=\left(\mathcal{F}^{-1}\varphi_{j}\right)\ast\left(\iota_{q}^{\left(0\right)}f\right) & =\sum_{i\in I}\left[\left(\mathcal{F}^{-1}\varphi_{j}\right)\ast\mathcal{F}^{-1}\left(\varphi_{i}\widehat{f}\right)\right]\\
 & =\sum_{i\in I}\mathcal{F}^{-1}\left(\varphi_{j}\cdot\varphi_{i}\widehat{f}\right)\\
 & =\sum_{i\in j^{\ast}}\mathcal{F}^{-1}\left(\varphi_{j}\cdot\varphi_{i}\widehat{f}\right)\\
 & =\mathcal{F}^{-1}\left(\left[\sum_{i\in j^{\ast}}\varphi_{i}\right]\cdot\varphi_{j}\cdot\widehat{f}\right)\\
 & =\mathcal{F}^{-1}\left(\varphi_{j}\widehat{f}\right),
\end{align*}
where the last step used $\sum_{i\in j^{\ast}}\varphi_{i}\equiv1$
on $Q_{j}\supset{\rm supp}\,\varphi_{j}$. But this easily entails
$\left\Vert f\right\Vert _{\mathcal{D}\left(\mathcal{Q},L^{p},Y\right)}=0$
and thus $f=0$, so that $\iota_{q}^{\left(0\right)}$ is injective.

Now, we simultaneously prove well-definedness and boundedness of $\iota_{q}^{\left(k\right)}$
for arbitrary $q\in\left(0,\infty\right]$ and also establish (for
$q\geq1$) that $\partial_{\ast}^{\alpha}f=\partial^{\alpha}\left(\iota_{q}^{\left(0\right)}f\right)$
for $f\in\mathcal{D}\left(\mathcal{Q},L^{p},Y\right)$ and $\alpha\in\mathbb{N}_{0}^{d}$
with $\left|\alpha\right|\leq k$. To this end, let $\alpha\in\mathbb{N}_{0}^{d}\setminus\left\{ 0\right\} $
be arbitrary with $\ell:=\left|\alpha\right|\leq k$. Now, for $a\geq0$,
there are two cases:
\begin{casenv}
\item If $a\leq1$, then $a^{\ell}\leq1\leq1+a^{k}$.
\item If $a\geq1$, then $a^{\ell}\leq a^{k}\leq1+a^{k}$.
\end{casenv}
Together, these considerations show  that $u_{i}^{\left(\ell,p,q\right)}\leq2v_{i}+w_{i}$
for all $i\in I$ and hence for $\left(x_{i}\right)_{i\in I}\in Y$
that
\begin{align*}
\left\Vert \left(x_{i}\right)_{i\in I}\right\Vert _{\ell_{u^{\left(\ell,p,q\right)}}^{q^{\triangledown}}} & \lesssim\left\Vert \left(x_{i}\right)_{i\in I}\right\Vert _{\ell_{v+w}^{q^{\triangledown}}}\\
 & =\left\Vert \left(v_{i}\cdot x_{i}\right)_{i\in I}+\left(w_{i}\cdot x_{i}\right)_{i\in I}\right\Vert _{\ell^{q^{\triangledown}}}\\
\left(\text{by the quasi triangle inequality}\right) & \lesssim\left\Vert \left(v_{i}\cdot x_{i}\right)_{i\in I}\right\Vert _{\ell^{q^{\triangledown}}}+\left\Vert \left(w_{i}\cdot x_{i}\right)_{i\in I}\right\Vert _{\ell^{q^{\triangledown}}}\\
\left(\text{since }Y\hookrightarrow\ell_{v}^{q^{\triangledown}}\text{ and }Y\hookrightarrow\ell_{w}^{q^{\triangledown}}\right) & \lesssim\left\Vert \left(x_{i}\right)_{i\in I}\right\Vert _{Y}+\left\Vert \left(x_{i}\right)_{i\in I}\right\Vert _{Y}\\
 & \lesssim\left\Vert \left(x_{i}\right)_{i\in I}\right\Vert _{Y}<\infty.
\end{align*}
Thus, an application of Theorem \ref{thm:SufficientConditionsForBoundednessOfDerivative}
shows that the map
\[
\partial_{\ast}^{\alpha}:\mathcal{D}\left(\mathcal{Q},L^{p},Y\right)\to L^{q}\left(\mathbb{R}^{d}\right),f\mapsto\sum_{i\in I}\partial^{\alpha}\left(\mathcal{F}^{-1}\left[\varphi_{i}\widehat{f}\right]\right)
\]
is well-defined and bounded with unconditional convergence of the
series in $L^{q}\left(\mathbb{R}^{d}\right)$. Since $\alpha\in\mathbb{N}_{0}^{d}\setminus\left\{ 0\right\} $
with $\left|\alpha\right|\leq k$ was arbitrary, we get unconditional
convergence of the series $\left(\sum_{i\in I}\partial^{\alpha}\left(\mathcal{F}^{-1}\left[\varphi_{i}\widehat{f}\right]\right)\right)_{\left|\alpha\right|\leq k}$
in $\prod_{\left|\alpha\right|\leq k}L^{q}\left(\mathbb{R}^{d}\right)$.
Note that the partial sums of the series $\sum_{i\in I}\mathcal{F}^{-1}\left[\varphi_{i}\widehat{f}\right]$
are elements of $L^{q}\left(\mathbb{R}^{d}\right)$ as well as of
$C^{\infty}\left(\mathbb{R}^{d}\right)$, by the Paley Wiener theorem.
In conclusion, for $q\in\left(0,1\right)$, we get
\[
\iota_{q}^{\left(k\right)}f=\left(\partial_{\ast}^{\alpha}f\right)_{\left|\alpha\right|\leq k}=\left(\sum_{i\in I}\partial^{\alpha}\left(\mathcal{F}^{-1}\left[\varphi_{i}\widehat{f}\right]\right)\right)_{\left|\alpha\right|\leq k}\in W^{k,q}\left(\mathbb{R}^{d}\right)
\]
for every $f\in\mathcal{D}\left(\mathcal{Q},L^{p},Y\right)$, and
also a quasi-norm estimate of the form
\[
\left\Vert \iota_{q}^{\left(k\right)}f\right\Vert _{W^{k,q}}\asymp\sum_{\left|\alpha\right|\leq k}\left\Vert \partial_{\ast}^{\alpha}f\right\Vert _{L^{q}}\leq\left(\sum_{\left|\alpha\right|\leq k}\left\Vert \partial_{\ast}^{\alpha}\right\Vert \right)\cdot\left\Vert f\right\Vert _{\mathcal{D}\left(\mathcal{Q},L^{p},Y\right)}.
\]
Furthermore, for $f\in\mathcal{S}_{\mathcal{O}}^{p,Y}\left(\mathbb{R}^{d}\right)$,
Theorem \ref{thm:SufficientConditionsForBoundednessOfDerivative}
shows $\iota_{q}^{\left(k\right)}f=\left(\partial_{\ast}^{\alpha}f\right)_{\left|\alpha\right|\leq k}=\left(\partial^{\alpha}f\right)_{\left|\alpha\right|\leq k}$,
as claimed.

Now, assume $q\geq1$ and let $\phi\in C_{c}^{\infty}\left(\mathbb{R}^{d}\right)$
be arbitrary. We have $\partial^{\beta}\phi\in L^{q'}\left(\mathbb{R}^{d}\right)\subset\left[L^{q}\left(\mathbb{R}^{d}\right)\right]'$
for all $\beta\in\mathbb{N}_{0}^{d}$, since $q\geq1$. Thus, using
the unconditional convergence in $L^{q}\left(\mathbb{R}^{d}\right)$
of the series defining $\iota_{q}^{\left(0\right)}f$ and $\partial_{\ast}^{\alpha}f$,
we conclude for $\alpha\in\mathbb{N}_{0}^{d}$ with $\left|\alpha\right|\leq k$
that
\begin{align*}
\int_{\mathbb{R}^{d}}\partial^{\alpha}\phi\cdot\iota_{q}^{\left(0\right)}f\,{\rm d}x & =\sum_{i\in I}\int_{\mathbb{R}^{d}}\partial^{\alpha}\phi\cdot\mathcal{F}^{-1}\left(\varphi_{i}\widehat{f}\right)\,{\rm d}x\\
 & =\left(-1\right)^{\left|\alpha\right|}\cdot\sum_{i\in I}\int_{\mathbb{R}^{d}}\phi\cdot\partial^{\alpha}\left[\mathcal{F}^{-1}\left(\varphi_{i}\widehat{f}\right)\right]\,{\rm d}x\\
 & =\left(-1\right)^{\left|\alpha\right|}\cdot\int_{\mathbb{R}^{d}}\phi\cdot\partial_{\ast}^{\alpha}f\,{\rm d}x,
\end{align*}
so that the weak derivative $\partial^{\alpha}\left(\iota_{q}^{\left(0\right)}f\right)$
is given by $\partial_{\ast}^{\alpha}f\in L^{q}\left(\mathbb{R}^{d}\right)$.
Since this holds for all $\alpha\in\mathbb{N}_{0}^{d}\setminus\left\{ 0\right\} $
with $\left|\alpha\right|\leq k$, we finally get $\iota_{q}^{\left(k\right)}f=\iota_{q}^{\left(0\right)}f\in W^{k,q}\left(\mathbb{R}^{d}\right)$
with
\[
\left\Vert \iota_{q}f\right\Vert _{W^{k,q}}\lesssim\left\Vert \iota_{q}^{\left(0\right)}f\right\Vert _{L^{q}}+\sum_{\substack{\alpha\in\mathbb{N}_{0}^{d}\setminus\left\{ 0\right\} \\
\left|\alpha\right|\leq k
}
}\left\Vert \partial_{\ast}^{\alpha}f\right\Vert _{L^{q}}\lesssim\left\Vert f\right\Vert _{\mathcal{D}\left(\mathcal{Q},L^{p},Y\right)}<\infty.
\]

Now, we show that for $q=\infty$, the map $\iota_{q}^{\left(k\right)}$
is actually well-defined and bounded as a map into $C_{b}^{k}\left(\mathbb{R}^{d}\right)$.
To see this, simply note that by the Paley-Wiener theorem, each of
the functions $\mathcal{F}^{-1}\left(\varphi_{i}\smash{\widehat{f}}\,\right)$
is smooth, so that the (finite) partial sums of the series $\sum_{i\in I}\mathcal{F}^{-1}\left(\varphi_{i}\smash{\widehat{f}}\,\right)=\iota_{q}^{\left(k\right)}f$
lie in $W^{k,\infty}\left(\mathbb{R}^{d}\right)\cap C^{\infty}\left(\mathbb{R}^{d}\right)\subset C_{b}^{k}\left(\mathbb{R}^{d}\right)$.
Since the series converges unconditionally in $W^{k,\infty}\left(\mathbb{R}^{d}\right)$
and since $C_{b}^{k}\left(\mathbb{R}^{d}\right)\leq W^{k,\infty}\left(\mathbb{R}^{d}\right)$
is closed (with $\left\Vert \cdot\right\Vert _{C_{b}^{k}}\asymp\left\Vert \cdot\right\Vert _{W^{k,\infty}}$
on $C_{b}^{k}\left(\mathbb{R}^{d}\right)$), we easily see that $\iota_{q}^{\left(k\right)}:\mathcal{D}\left(\mathcal{Q},L^{p},Y\right)\to C_{b}^{k}\left(\mathbb{R}^{d}\right)$
is indeed well-defined and bounded.

It remains to prove the third statement of the corollary. By assumption,
the series $\sum_{i\in I}\mathcal{F}^{-1}\left(\varphi_{i}\cdot\smash{\widehat{f}}\,\right)$
converges to $\left(\iota_{q}^{\left(k\right)}f\right)_{0}$ in $L^{q}\left(\mathbb{R}^{d}\right)$
and to $\iota_{r}^{\left(0\right)}f$ in $L^{r}\left(\mathbb{R}^{d}\right)$.
Since convergence in $L^{s}\left(\mathbb{R}^{d}\right)$ for arbitrary
$s\in\left(0,\infty\right]$ implies convergence in measure, we get
$\left(\iota_{q}^{\left(k\right)}f\right)_{0}=\iota_{r}^{\left(0\right)}f$
(almost everywhere). In particular, we see that $\iota_{q}^{\left(k\right)}f=0$
yields $\iota_{r}^{\left(0\right)}f=\left(\iota_{q}^{\left(k\right)}f\right)_{0}=0$
and hence $f=0$ (by injectivity of $\iota_{r}^{\left(0\right)}$),
so that $\iota_{q}^{\left(k\right)}$ is injective. A completely analogous
argument using unconditional convergence of the series $\sum_{i\in I}\partial^{\alpha}\left[\mathcal{F}^{-1}\left(\varphi_{i}\cdot\smash{\widehat{f}}\,\right)\right]$
shows $\iota_{q}^{\left(k\right)}f\bigg|_{\left|\alpha\right|\leq\min\left\{ k,\ell\right\} }=\left(\partial^{\alpha}\left(\iota_{r}^{\left(\ell\right)}f\right)\right)_{\left|\alpha\right|\leq\min\left\{ k,\ell\right\} }$
for all $f\in\mathcal{D}\left(\mathcal{Q},L^{p},Y\right)$, as claimed.
\end{proof}
Now that we have obtained sufficient criteria for the embedding $\mathcal{D}\left(\mathcal{Q},L^{p},Y\right)\hookrightarrow W^{k,q}\left(\mathbb{R}^{d}\right)$,
it is natural to ask whether these criteria are sharp. This is the
goal of the next section.

\section{Necessary Conditions}

\label{sec:NecessaryConditions}In this section, we will assume that
we have the following, for some $k\in\mathbb{N}_{0}$: For every multiindex
$\alpha\in\mathbb{N}_{0}^{d}$ with $\left|\alpha\right|=k$, the
map
\[
\iota_{\alpha}:\mathcal{S}_{\mathcal{O}}^{p,Y}\left(\mathbb{R}^{d}\right)\to L^{q}\left(\mathbb{R}^{d}\right),f\mapsto\partial^{\alpha}f
\]
is bounded, where the space $\mathcal{S}_{\mathcal{O}}^{p,Y}\left(\mathbb{R}^{d}\right)$
is as in Definition \ref{def:EmbeddingDefinition}. Our general goal
is to show that this implies conditions very similar to the sufficient
conditions given in Theorem \ref{thm:SufficientConditionsForBoundednessOfDerivative}.
In case of $q\in\left(0,2\right]\cup\left\{ \infty\right\} $, we
will even see that the necessary and sufficient criteria will coincide.

Our first aim is to show that the assumption $p\leq q$ from Theorem
\ref{thm:SufficientConditionsForBoundednessOfDerivative} is necessary.

The proofs of \emph{all} our necessary conditions (not only of $p\leq q$)
will be based on the following -- relatively elementary -- \emph{arbitrage}
result\cite{TerryTaoAmplificationArbitrageAndTensorPower}. Since
it is so central to our approach, we provide a proof, even though
the result is probably well known.
\begin{lem}
\label{lem:AsymptoticTranslationNorm}Let $n\in\mathbb{N}$, $p\in\left(0,\infty\right]$
and let $f_{1},\dots,f_{n}\in L^{p}\left(\mathbb{R}^{d}\right)$.
For $p=\infty$, assume additionally that
\[
f_{i}\in\overline{\left\{ {\displaystyle f\in L^{\infty}\left(\mathbb{R}^{d}\right)}\with{\rm supp}\, f\text{ compact}\right\} }\qquad\text{ for all }i\in\underline{n}\:,
\]
where the closure is taken in $L^{\infty}\left(\mathbb{R}^{d}\right)$.

Then we have
\[
\left\Vert \sum_{i=1}^{n}L_{x_{i}}f_{i}\right\Vert _{L^{p}}\xrightarrow[\min_{i\neq j}\left|x_{i}-x_{j}\right|\to\infty]{}\left\Vert \left(\left\Vert f_{i}\right\Vert _{L^{p}}\right)_{i\in\underline{n}}\right\Vert _{\ell^{p}}.
\]
In particular, there exists $R=R\left(p,\smash{\left(f_{i}\right)_{i\in\underline{n}}}\right)>0$
with
\[
\frac{1}{2}\cdot\left\Vert \left(\left\Vert f_{i}\right\Vert _{L^{p}}\right)_{i\in\underline{n}}\right\Vert _{\ell^{p}}\leq\left\Vert \sum_{i=1}^{n}L_{x_{i}}f_{i}\right\Vert _{L^{p}}\leq2\cdot\left\Vert \left(\left\Vert f_{i}\right\Vert _{L^{p}}\right)_{i\in\underline{n}}\right\Vert _{\ell^{p}}
\]
for all $x_{1},\dots,x_{n}\in\mathbb{R}^{d}$ with $\left|x_{i}-x_{j}\right|\geq R$
for all $i,j\in\underline{n}$ with $i\neq j$.\end{lem}
\begin{proof}
If $\left\Vert \left(\left\Vert f_{i}\right\Vert _{L^{p}}\right)_{i\in\underline{n}}\right\Vert _{\ell^{p}}>0$,
then the second part of the lemma is a trivial consequence of the
first part. If $\left\Vert \left(\left\Vert f_{i}\right\Vert _{L^{p}}\right)_{i\in\underline{n}}\right\Vert _{\ell^{p}}=0$,
then all quantities in the last part of the lemma vanish, so that
the claim is trivial. Thus, it remains to prove the first part.

To this end, let us first assume that $f_{1},\dots,f_{n}$ are compactly
supported. Since there are only finitely many $f_{i}$, there is some
$R>0$ with ${\rm supp}\, f_{i}\subset B_{R/2}\left(0\right)$ for
all $i\in I$. For $x_{1},\dots,x_{n}\in\mathbb{R}^{d}$ with $\left|x_{i}-x_{j}\right|\geq R$
for $i\neq j$, we then have
\[
{\rm supp}\left(L_{x_{i}}f_{i}\right)\cap{\rm supp}\left(L_{x_{j}}f_{j}\right)\subset B_{R/2}\left(x_{i}\right)\cap B_{R/2}\left(x_{j}\right)=\emptyset\qquad\text{ if }i\neq j\,.
\]
For $p\in\left(0,\infty\right)$, this implies
\begin{align*}
\left\Vert \sum_{i=1}^{n}L_{x_{i}}f_{i}\right\Vert _{L^{p}}^{p} & =\int_{\mathbb{R}^{d}}\left|\sum_{i=1}^{n}\left(L_{x_{i}}f_{i}\right)\left(x\right)\right|^{p}\,{\rm d}x\\
 & \overset{\left(\ast\right)}{=}\int_{\mathbb{R}^{d}}\sum_{i=1}^{n}\left|\left(L_{x_{i}}f_{i}\right)\left(x\right)\right|^{p}\,{\rm d}x\\
 & =\left\Vert \left(\left\Vert f_{i}\right\Vert _{L^{p}}\right)_{i\in\underline{n}}\right\Vert _{\ell^{p}}^{p},
\end{align*}
where the step marked with $\left(\ast\right)$ used that at most
one summand of the sum does not vanish for each $x\in\mathbb{R}^{d}$.
For $p=\infty$, we similarly have
\[
\left\Vert \sum_{i=1}^{n}L_{x_{i}}f_{i}\right\Vert _{L^{\infty}}=\esssup_{x\in\mathbb{R}^{d}}\left|\sum_{i=1}^{n}\left(L_{x_{i}}f_{i}\right)\left(x\right)\right|=\max_{i\in\underline{n}}\esssup_{x\in\mathbb{R}^{d}}\left|\left(L_{x_{i}}f_{i}\right)\left(x\right)\right|=\left\Vert \left(\left\Vert f_{i}\right\Vert _{L^{\infty}}\right)_{i\in\underline{n}}\right\Vert _{\ell^{\infty}}
\]
for $x_{1},\dots,x_{n}\in\mathbb{R}^{d}$ with $\left|x_{i}-x_{j}\right|\geq R$
for $i\neq j$, since for each $x\in\mathbb{R}^{d}$, at most one
summand of the sum does not vanish.

For the general case, we use an approximation argument. Let us first
consider the case $p\geq1$. For $\varepsilon>0$, there are compactly
supported $g_{1},\dots,g_{n}\in L^{p}\left(\mathbb{R}^{d}\right)$
with $\left\Vert f_{i}-g_{i}\right\Vert _{L^{p}}<\varepsilon$ for
all $i\in\underline{n}$. For $p<\infty$, this follows by density
of $C_{c}\left(\mathbb{R}^{d}\right)$ in $L^{p}\left(\mathbb{R}^{d}\right)$;
for $p=\infty$, it is a consequence of our additional assumptions
on the $f_{i}$. Hence, for $R>0$ large enough and $\left|x_{i}-x_{j}\right|\geq R$
for $i\neq j$, the considerations from above yield
\[
\left\Vert \sum_{i=1}^{n}L_{x_{i}}g_{i}\right\Vert _{L^{p}}=\left\Vert \left(\left\Vert g_{i}\right\Vert _{L^{p}}\right)_{i\in\underline{n}}\right\Vert _{\ell^{p}}.
\]
Using the (second) triangle inequality for $L^{p}$ and $\ell^{p}$,
this yields
\begin{align*}
 & \left|\left\Vert \sum_{i=1}^{n}L_{x_{i}}f_{i}\right\Vert _{L^{p}}-\left\Vert \left(\left\Vert f_{i}\right\Vert _{L^{p}}\right)_{i\in\underline{n}}\right\Vert _{\ell^{p}}\right|\\
 & \leq\left|\left\Vert \sum_{i=1}^{n}L_{x_{i}}f_{i}\right\Vert _{L^{p}}-\left\Vert \sum_{i=1}^{n}L_{x_{i}}g_{i}\right\Vert _{L^{p}}\right|+\left|\left\Vert \left(\left\Vert g_{i}\right\Vert _{L^{p}}\right)_{i\in\underline{n}}\right\Vert _{\ell^{p}}-\left\Vert \left(\left\Vert f_{i}\right\Vert _{L^{p}}\right)_{i\in\underline{n}}\right\Vert _{\ell^{p}}\right|\\
 & \leq\left\Vert \sum_{i=1}^{n}L_{x_{i}}\left(f_{i}-g_{i}\right)\right\Vert _{L^{p}}+\left\Vert \left(\left\Vert g_{i}\right\Vert _{L^{p}}-\left\Vert f_{i}\right\Vert _{L^{p}}\right)_{i\in\underline{n}}\right\Vert _{\ell^{p}}\\
 & \leq\sum_{i=1}^{n}\left\Vert L_{x_{i}}\left(f_{i}-g_{i}\right)\right\Vert _{L^{p}}+\left\Vert \left(\varepsilon\right)_{i\in\underline{n}}\right\Vert _{\ell^{p}}\\
 & \leq2n\varepsilon.
\end{align*}
Here, the last step used translation invariance of $\left\Vert \cdot\right\Vert _{L^{p}}$
and the embedding $\ell^{1}\hookrightarrow\ell^{p}$.

For the case $p\in\left(0,1\right)$, the above argument is invalid,
since $\left\Vert \cdot\right\Vert _{L^{p}}$ and $\left\Vert \cdot\right\Vert _{\ell^{p}}$
do not satisfy the triangle inequality anymore. Instead, we note that
we have the $p$-triangle inequality $\left\Vert f+g\right\Vert _{L^{p}}^{p}\leq\left\Vert f\right\Vert _{L^{p}}^{p}+\left\Vert g\right\Vert _{L^{p}}^{p}$
which implies that $d\left(f,g\right):=\left\Vert f-g\right\Vert _{L^{p}}^{p}$
defines a metric on $L^{p}\left(\mathbb{R}^{d}\right)$. The same
also holds for $\ell^{p}$. In particular, $\left\Vert \cdot\right\Vert _{L^{p}}$
is continuous. By density of $C_{c}\left(\mathbb{R}^{d}\right)\subset L^{p}\left(\mathbb{R}^{d}\right)$,
this allows us to choose $g_{1},\dots,g_{n}\in C_{c}\left(\mathbb{R}^{d}\right)$
with $\left|\left\Vert f_{i}\right\Vert _{L^{p}}-\left\Vert g_{i}\right\Vert _{L^{p}}\right|<\varepsilon$
and $\left\Vert f_{i}-g_{i}\right\Vert _{L^{p}}<\varepsilon$ for
all $i\in\underline{n}$. For $R>0$ large enough and $\left|x_{i}-x_{j}\right|\geq R$
for $i\neq j$, this yields
\begin{align*}
 & \left|\left\Vert \sum_{i=1}^{n}L_{x_{i}}f_{i}\right\Vert _{L^{p}}^{p}-\left\Vert \left(\left\Vert f_{i}\right\Vert _{L^{p}}\right)_{i\in\underline{n}}\right\Vert _{\ell^{p}}^{p}\right|\\
 & \leq\left|\left\Vert \sum_{i=1}^{n}L_{x_{i}}f_{i}\right\Vert _{L^{p}}^{p}-\left\Vert \sum_{i=1}^{n}L_{x_{i}}g_{i}\right\Vert _{L^{p}}^{p}\right|+\left|\left\Vert \left(\left\Vert g_{i}\right\Vert _{L^{p}}\right)_{i\in\underline{n}}\right\Vert _{\ell^{p}}^{p}-\left\Vert \left(\left\Vert f_{i}\right\Vert _{L^{p}}\right)_{i\in\underline{n}}\right\Vert _{\ell^{p}}^{p}\right|\\
 & \leq\left\Vert \sum_{i=1}^{n}L_{x_{i}}\left(f_{i}-g_{i}\right)\right\Vert _{L^{p}}^{p}+\left\Vert \left(\left\Vert g_{i}\right\Vert _{L^{p}}-\left\Vert f_{i}\right\Vert _{L^{p}}\right)_{i\in\underline{n}}\right\Vert _{\ell^{p}}^{p}\\
 & \leq\sum_{i=1}^{n}\left\Vert L_{x_{i}}\left(f_{i}-g_{i}\right)\right\Vert _{L^{p}}^{p}+\left\Vert \left(\varepsilon\right)_{i\in\underline{n}}\right\Vert _{\ell^{p}}^{p}\\
 & \leq2n\cdot\varepsilon^{p}.
\end{align*}
Since $\varepsilon>0$ was arbitrary and because the map $\left[0,\infty\right)\to\left[0,\infty\right),x\mapsto x^{1/p}$
is continuous, this implies the claim also for $p\in\left(0,1\right)$.
\end{proof}
As a consequence, we obtain the following corollary which provides
a generalization of the fact that bounded, translation invariant operators
from $L^{p}\left(\mathbb{R}^{d}\right)\to L^{q}\left(\mathbb{R}^{d}\right)$
can only be nontrivial if $p\leq q$. Note that this is decidedly
false if $\mathbb{R}^{d}$ is replaced by any compact topological
group. No originality regarding this result is claimed. A slightly
less general form of the present result can be found in \cite[around equation (11)]{TerryTaoAmplificationArbitrageAndTensorPower}.
\begin{cor}
\label{cor:SmallExponentsOnTheLeftForTranslationInvariantOperators}Let
$p,q\in\left(0,\infty\right]$ and $x_{0}\in\mathbb{R}^{d}\setminus\left\{ 0\right\} $.
Assume that the (not necessarily closed) subspace $V\leq L^{p}\left(\mathbb{R}^{d}\right)$
is invariant under the translation $L_{x_{0}}$ and that
\[
T:\left(V,\left\Vert \cdot\right\Vert _{L^{p}}\right)\to L^{q}\left(\mathbb{R}^{d}\right)
\]
is linear and bounded with $T\left(L_{x_{0}}f\right)=L_{x_{0}}\left(Tf\right)$
for all $f\in V$.

Finally, assume that
\begin{equation}
\left\{ f\in V\with Tf\neq0\right\} \cap\overline{\left\{ f\in{\displaystyle L^{p}\left(\mathbb{R}^{d}\right)}\with{\rm supp}\, f\text{ compact}\right\} }\neq\emptyset,\label{eq:NonDegeneratenessAssumptionTranslationInvariantOp}
\end{equation}
where the closure is taken in $L^{p}\left(\mathbb{R}^{d}\right)$.
Then we have $p\leq q$.\end{cor}
\begin{rem*}
Note that assumption (\ref{eq:NonDegeneratenessAssumptionTranslationInvariantOp})
is always satisfied for $p<\infty$ if $T\not\equiv0$, since in this
case, we have 
\[
\overline{\left\{ f\in{\displaystyle L^{p}\left(\mathbb{R}^{d}\right)}\with{\rm supp}\, f\text{ compact}\right\} }=L^{p}\left(\mathbb{R}^{d}\right).\qedhere
\]
\end{rem*}
\begin{proof}
In case of $q=\infty$, the claim is trivial, so that we can assume
$q<\infty$ in the following.

By assumption, there is some 
\[
f\in\overline{\left\{ g\in{\displaystyle L^{p}\left(\mathbb{R}^{d}\right)}\with{\rm supp}\, f\text{ compact}\right\} }\cap V
\]
with $Tf\neq0$. In particular, $f\neq0$.

Let $n\in\mathbb{N}$ be arbitrary. Applying Lemma \ref{lem:AsymptoticTranslationNorm}
to the family $\left(f_{1},\dots,f_{n}\right):=\left(f,\dots,f\right)$,
we obtain some $R_{1}>0$ such that
\[
\left\Vert \sum_{i=1}^{n}L_{x_{i}}f\right\Vert _{L^{p}}\leq2\cdot\left\Vert \left(\left\Vert f\right\Vert _{L^{p}}\right)_{i\in\underline{n}}\right\Vert _{\ell^{p}}=2\left\Vert f\right\Vert _{L^{p}}\cdot n^{1/p}
\]
for all $x_{1},\dots,x_{n}\in\mathbb{R}^{d}$ satisfying $\left|x_{i}-x_{j}\right|\geq R_{1}$
for all $i,j\in\underline{n}$ with $i\neq j$.

Note that because of $q<\infty$, we have 
\[
g:=Tf\in\overline{\left\{ {\displaystyle h\in L^{q}\left(\mathbb{R}^{d}\right)}\with{\rm supp}\, h\text{ compact}\right\} },
\]
so that we can apply Lemma \ref{lem:AsymptoticTranslationNorm} to
the family $\left(g_{1},\dots,g_{n}\right)=\left(g,\dots,g\right)$.
This yields $R_{2}>0$ such that
\[
\left\Vert \sum_{i=1}^{n}L_{x_{i}}g\right\Vert _{L^{q}}\geq\frac{1}{2}\cdot\left\Vert \left(\left\Vert g\right\Vert _{L^{q}}\right)_{i\in\underline{n}}\right\Vert _{\ell^{q}}=\frac{\left\Vert g\right\Vert _{L^{q}}}{2}\cdot n^{1/q}
\]
holds for all $x_{1},\dots,x_{n}\in\mathbb{R}^{d}$ satisfying $\left|x_{i}-x_{j}\right|\geq R_{2}$
for all $i,j\in\underline{n}$ with $i\neq j$.

Let $R:=\max\left\{ R_{1},R_{2}\right\} >0$. Because of $x_{0}\neq0$,
there is $N\in\mathbb{N}$ with $\left|Nx_{0}\right|>R$. If we set
$x_{\ell}:=\ell\cdot Nx_{0}$ for $\ell\in\underline{n}$, this implies
\[
\left|x_{i}-x_{j}\right|=\left|\left(i-j\right)\cdot Nx_{0}\right|=\left|i-j\right|\cdot\left|Nx_{0}\right|\geq\left|Nx_{0}\right|\geq R\geq R_{k}
\]
for $k\in\left\{ 1,2\right\} $ and $i,j\in\underline{n}$ with $i\neq j$.
Hence,
\begin{align*}
\frac{\left\Vert g\right\Vert _{L^{q}}}{2}\cdot n^{1/q} & \leq\left\Vert \sum_{i=1}^{n}L_{x_{i}}g\right\Vert _{L^{q}}\\
 & =\left\Vert \sum_{i=1}^{n}L_{\ell N\cdot x_{0}}\left(Tf\right)\right\Vert _{L^{q}}\\
 & =\left\Vert T\left(\sum_{i=1}^{n}L_{\ell N\cdot x_{0}}f\right)\right\Vert _{L^{q}}\\
 & =\left\Vert T\left(\sum_{i=1}^{n}L_{x_{i}}f\right)\right\Vert _{L^{q}}\\
 & \leq\left\Vert T\right\Vert \cdot\left\Vert \sum_{i=1}^{n}L_{x_{i}}f\right\Vert _{L^{p}}\\
 & \leq2\left\Vert T\right\Vert \left\Vert f\right\Vert _{L^{p}}\cdot n^{1/p},
\end{align*}
where we used that $T$ commutes with $L_{nx_{0}}$ for all $n\in\mathbb{N}$,
since it commutes with $L_{x_{0}}$. All in all, we get
\[
n^{\frac{1}{q}-\frac{1}{p}}\leq4\left\Vert T\right\Vert \cdot\frac{\left\Vert f\right\Vert _{L^{p}}}{\left\Vert g\right\Vert _{L^{q}}}
\]
for all $n\in\mathbb{N}$, where the right-hand side is independent(!)
of $n\in\mathbb{N}$. Thus, $\frac{1}{q}-\frac{1}{p}\leq0$ which
easily implies the claim $p\leq q$.
\end{proof}
For later use, we will also need the following fact about ``richness''
of $C_{c}^{\infty}\left(U\right)$ for arbitrary open $U\subset\mathbb{R}^{d}$.
\begin{lem}
\label{lem:RichnessOfDerivativeOfFunctionsWithFourierSupport}Let
$\emptyset\neq U\subset\mathbb{R}^{d}$ be open and let $k\in\mathbb{N}$
be arbitrary. Then, for each $\alpha\in\mathbb{N}_{0}^{d}$ with $\left|\alpha\right|\leq k$,
there is a function $f_{\alpha,k}\in C_{c}^{\infty}\left(U\right)$
with
\[
\left(\partial^{\beta}\left[\mathcal{F}^{-1}f_{\alpha,k}\right]\right)\left(0\right)=\delta_{\alpha,\beta}\qquad\text{ for all }\beta\in\mathbb{N}_{0}^{d}\text{ with }\left|\beta\right|\leq k.\qedhere
\]
\end{lem}
\begin{proof}
We first show the following intermediate result:
\begin{claim*}
Let $\emptyset\neq I\subset\mathbb{R}$ be an open, bounded interval
and let $N\in\mathbb{N}$ be arbitrary. Then the map
\[
\Phi:C_{c}^{\infty}\left(I\right)\to\mathbb{C}^{N},f\mapsto\left(\left[\partial^{\ell-1}\left(\mathcal{F}^{-1}f\right)\right]\left(0\right)\right)_{\ell\in\underline{N}}
\]
is surjective.\end{claim*}
\begin{proof}
Assume that the claim fails, so that $V:=\Phi\left(C_{c}^{\infty}\left(I\right)\right)\lneq\mathbb{C}^{N}$
is a strict subspace. This yields some $a\in\mathbb{C}^{N}\setminus\left\{ 0\right\} $
with $\left\langle \Phi\left(f\right),a\right\rangle =0$ for all
$f\in C_{c}^{\infty}\left(I\right)$.

But by standard properties of the Fourier transform (see \cite[Theorem 8.22]{FollandRA}),
we have for $f\in C_{c}^{\infty}\left(I\right)$ that
\[
\left[\partial^{\ell-1}\left(\mathcal{F}^{-1}f\right)\right]\left(0\right)=\int_{\mathbb{R}}f\left(\xi\right)\cdot\left(2\pi i\xi\right)^{\ell-1}\,{\rm d}\xi,
\]
so that we get
\begin{align*}
0 & =\int_{\mathbb{R}}f\left(\xi\right)\cdot\sum_{\ell=1}^{N}\overline{a_{\ell}}\left(2\pi i\xi\right)^{\ell-1}\,{\rm d}\xi\\
 & =\int_{I}f\left(\xi\right)\cdot\sum_{\ell=1}^{N}\overline{a_{\ell}}\left(2\pi i\xi\right)^{\ell-1}\,{\rm d}\xi
\end{align*}
for all $f\in C_{c}^{\infty}\left(I\right)$. Since $I$ is a bounded
interval, we have $g\in L^{2}\left(I\right)$ for 
\[
g:I\to\mathbb{C},\xi\mapsto\sum_{\ell=1}^{N}\left(2\pi i\right)^{\ell-1}\overline{a_{\ell}}\cdot\xi^{\ell-1}.
\]
But since $C_{c}^{\infty}\left(I\right)\subset L^{2}\left(I\right)$
is dense, we conclude $g=0$ as an element of $L^{2}\left(I\right)$.
By continuity of $g$, we get $g\equiv0$ and hence (by uniqueness
of the coefficients of a polynomial) $\left(2\pi i\right)^{\ell-1}\overline{a_{\ell}}=0$
for all $\ell\in\underline{N}$, in contradiction to $a\neq0$. This
proves the claim.
\end{proof}
Now, we return to the actual proof of the lemma. Since $U\neq\emptyset$
is open, there is for each $\ell\in\underline{d}$ an open, bounded
interval $I_{\ell}\subset\mathbb{R}$ with $\prod_{\ell=1}^{d}I_{\ell}\subset U$.
Now, for each $\ell\in\underline{d}$, we have $\alpha_{\ell}\in\left\{ 0\right\} \cup\underline{k}$,
so that the claim yields a function $f_{\ell}\in C_{c}^{\infty}\left(I_{\ell}\right)$
with
\[
\left[\partial^{m}\left(\mathcal{F}^{-1}f_{\ell}\right)\right]\left(0\right)=\delta_{m,\alpha_{\ell}}\qquad\text{ for all }m\in\left\{ 0\right\} \cup\underline{k}.
\]
We have $f:=f_{1}\otimes\cdots\otimes f_{d}\in C_{c}^{\infty}\left(U\right)$
for
\[
\left(f_{1}\otimes\cdots\otimes f_{d}\right)\left(x\right)=f_{1}\left(x_{1}\right)\cdots f_{d}\left(x_{d}\right)\qquad\text{ where }x=\left(x_{1},\dots,x_{d}\right).
\]
But it is easy to see that the Fourier transform commutes with tensor
products, i.e.\@
\[
\mathcal{F}^{-1}f=\left(\mathcal{F}^{-1}f_{1}\right)\otimes\cdots\otimes\left(\mathcal{F}^{-1}f_{d}\right).
\]
From this, we easily get for $\beta\in\mathbb{N}_{0}^{d}$ with $\left|\beta\right|\leq k$
(which implies $\beta_{\ell}\in\left\{ 0\right\} \cup\underline{k}$
for all $\ell\in\underline{d}$) that
\begin{align*}
\left[\partial^{\beta}\left(\mathcal{F}^{-1}f\right)\right]\left(0\right) & =\left[\partial^{\beta_{1}}\left(\mathcal{F}^{-1}f_{1}\right)\right]\left(0\right)\cdots\left[\partial^{\beta_{d}}\left(\mathcal{F}^{-1}f_{d}\right)\right]\left(0\right)\\
 & =\delta_{\beta_{1},\alpha_{1}}\cdots\delta_{\beta_{d},\alpha_{d}}\\
 & =\delta_{\beta,\alpha}.
\end{align*}
This completes the proof.
\end{proof}
We can now finally show that $p\leq q$ is a necessary condition for
boundedness of $\partial^{\alpha}:\mathcal{S}_{\mathcal{O}}^{p,Y}\left(\mathbb{R}^{d}\right)\to L^{q}\left(\mathbb{R}^{d}\right)$.
\begin{thm}
\label{thm:NecessaryExponentRelation}Let $p,q\in\left(0,\infty\right]$
and let $\mathcal{Q}=\left(Q_{i}\right)_{i\in I}$ be an $L^{p}$-decomposition
covering of the open set $\emptyset\neq\mathcal{O}\subset\mathbb{R}^{d}$.
Assume that $\left\{ 0\right\} \neq Y\leq\mathbb{C}^{I}$ is a $\mathcal{Q}$-regular
sequence space on $I$.

Finally, assume that
\[
\iota_{\alpha}:\mathcal{S}_{\mathcal{O}}^{p,Y}\left(\mathbb{R}^{d}\right)\to L^{q}\left(\mathbb{R}^{d}\right),f\mapsto\partial^{\alpha}f
\]
is bounded for some $\alpha\in\mathbb{N}_{0}^{d}$. Then $p\leq q$.\end{thm}
\begin{proof}
By assumption, $Y\neq\left\{ 0\right\} $ is nontrivial. Since $Y\leq\mathbb{C}^{I}$
is solid, there is thus some $i_{0}\in I$ with $\delta_{i_{0}}\in Y$.
Since $Y$ is $\mathcal{Q}$-regular, this even implies $\chi_{i_{0}^{\ast}}=\delta_{i_{0}}^{\ast}\in Y$
and thus also $\sum_{\ell\in i_{0}^{\ast}}\chi_{\ell^{\ast}}\in Y$,
which -- by solidity -- shows $\chi_{i_{0}^{2\ast}}\in Y$.

Let $\Phi=\left(\varphi_{i}\right)_{i\in I}$ be an $L^{p}$-BAPU
for $\mathcal{Q}$. By admissibility of $\mathcal{Q}$, we have $Q_{i_{0}}\neq\emptyset$.
Since $\varphi_{i_{0}}^{\ast}\equiv1$ on $Q_{i_{0}}$, this implies
that the open(!) set $U:=\left\{ x\in\mathbb{R}^{d}\with\varphi_{i_{0}}^{\ast}\left(x\right)\neq0\right\} \subset Q_{i_{0}}^{\ast}\subset\mathcal{O}$
is nonempty. 

Now, define
\[
V:=\left\{ f\in\mathcal{S}\left(\mathbb{R}^{d}\right)\with{\rm supp}\,\widehat{f}\subset U\right\} \subset\mathcal{S}_{\mathcal{O}}\left(\mathbb{R}^{d}\right).
\]
We claim that the operator
\[
T:\left(V,\left\Vert \cdot\right\Vert _{L^{p}}\right)\to L^{q}\left(\mathbb{R}^{d}\right),f\mapsto\partial^{\alpha}f
\]
is bounded. Indeed, let $f\in V\subset\mathcal{S}_{\mathcal{O}}\left(\mathbb{R}^{d}\right)$
be arbitrary. Because of ${\rm supp}\,\widehat{f}\subset U\subset Q_{i_{0}}^{\ast}$,
we have $\varphi_{i}\widehat{f}\equiv0$ for all $i\in I\setminus i_{0}^{2\ast}$.
But for $i\in i_{0}^{2\ast}$, we have
\[
\left\Vert \mathcal{F}^{-1}\left(\varphi_{i}\cdot\smash{\widehat{f}}\:\right)\right\Vert _{L^{p}}\lesssim\left\Vert \mathcal{F}^{-1}\smash{\widehat{f}}\right\Vert _{L^{p}}=\left\Vert f\right\Vert _{L^{p}}.
\]
For $p\in\left[1,\infty\right]$, this follows from Young's inequality,
since $\left\Vert \mathcal{F}^{-1}\varphi_{i}\right\Vert _{L^{1}}\leq C$
by definition of an $L^{p}$-BAPU. In case of $p\in\left(0,1\right)$,
we use again the definition of an $L^{p}$-BAPU and the support restrictions
${\rm supp}\,\widehat{f}\subset Q_{i_{0}}^{2\ast}\subset Q_{i}^{4\ast}$,
as well as ${\rm supp}\,\varphi_{i}\subset\overline{Q_{i}}\subset Q_{i}^{4\ast}$,
together with Corollary \ref{cor:SemiStructuredQuasiBanachConvolution}.
Note that for $p\in\left(0,1\right)$, the above estimate would fail
in general without an a priori support estimate for $\widehat{f}$.

Altogether, we have shown
\[
\left\Vert \mathcal{F}^{-1}\left(\varphi_{i}\cdot\smash{\widehat{f}}\:\right)\right\Vert _{L^{p}}\leq C\cdot\left\Vert f\right\Vert _{L^{p}}\cdot\chi_{i_{0}^{2\ast}}\left(i\right)\qquad\text{ for all }i\in I,
\]
where $C=C\left(\Phi,p\right)$ is an absolute constant. Thus, solidity
of $Y$ yields
\[
\left\Vert f\right\Vert _{\mathcal{D}\left(\mathcal{Q},L^{p},Y\right)}=\left\Vert \left(\left\Vert \mathcal{F}^{-1}\left(\varphi_{i}\cdot\smash{\widehat{f}}\:\right)\right\Vert _{L^{p}}\right)_{i\in I}\right\Vert _{Y}\leq C\left\Vert \chi_{i_{0}^{2\ast}}\right\Vert _{Y}\cdot\left\Vert f\right\Vert _{L^{p}}<\infty
\]
and hence $f\in\mathcal{S}_{\mathcal{O}}^{p,Y}\left(\mathbb{R}^{d}\right)$.
Since $\iota_{\alpha}$ is bounded by assumption, we conclude
\[
\left\Vert Tf\right\Vert _{L^{q}}=\left\Vert \partial^{\alpha}f\right\Vert _{L^{q}}=\left\Vert \iota_{\alpha}f\right\Vert _{L^{q}}\leq\left\Vert \iota_{\alpha}\right\Vert \left\Vert f\right\Vert _{\mathcal{D}\left(\mathcal{Q},L^{p},Y\right)}\leq C\left\Vert \iota_{\alpha}\right\Vert \left\Vert \chi_{i_{0}^{2\ast}}\right\Vert _{Y}\cdot\left\Vert f\right\Vert _{L^{p}},
\]
so that $T$ is indeed bounded.

Since translation corresponds to modulation on the Fourier side, i.e.\@
$\widehat{L_{x}f}=M_{-x}\widehat{f}$ for $f\in\mathcal{S}\left(\mathbb{R}^{d}\right)$
and since modulation does not change the support, we see that $V$
is invariant under arbitrary translations. Furthermore, it is clear
that $T$ commutes with arbitrary translations. Now, since $U\neq\emptyset$
is open, Lemma \ref{lem:RichnessOfDerivativeOfFunctionsWithFourierSupport}
shows that there is some $g\in C_{c}^{\infty}\left(U\right)$ such
that $f:=\mathcal{F}^{-1}g\in V$ satisfies $\partial^{\alpha}f\left(0\right)=\partial^{\alpha}\left[\mathcal{F}^{-1}g\right]\left(0\right)\neq0$.
Since $\partial^{\alpha}f$ is continuous, this implies $Tf=\partial^{\alpha}f\neq0$
in $L^{q}\left(\mathbb{R}^{d}\right)$. Finally, we have 
\[
f\in\mathcal{S}\left(\mathbb{R}^{d}\right)\subset\overline{\left\{ {\displaystyle h\in L^{p}\left(\mathbb{R}^{d}\right)}\with{\rm supp}\, h\text{ compact}\right\} },
\]
so that all prerequisites of Corollary \ref{cor:SmallExponentsOnTheLeftForTranslationInvariantOperators}
are satisfied. But this corollary implies $p\leq q$ as desired.
\end{proof}
Our next major goal is to show that boundedness of all $\iota_{\alpha}$
for $\left|\alpha\right|=k$ already implies boundedness of the embedding
$Y\cap\ell_{0}\left(I\right)\hookrightarrow\ell_{u^{\left(k,p,q\right)}}^{q}\left(I\right)$,
where $u^{\left(k,p,q\right)}$ is defined as in Theorem \ref{thm:SufficientConditionsForBoundednessOfDerivative}.
Note that (ignoring the intersection with $\ell_{0}\left(I\right)$)
this condition coincides with the one given in Theorem \ref{thm:SufficientConditionsForBoundednessOfDerivative}
if $q=q^{\triangledown}=\min\left\{ q,q'\right\} $, i.e.\@ if $q\leq2$.
 In general, for $q\in\left(2,\infty\right)$, the embedding $Y\cap\ell_{0}\left(I\right)\hookrightarrow\ell_{u^{\left(k,p,q\right)}}^{q^{\triangledown}}\left(I\right)$
is  \emph{not} a necessary condition for boundedness of all $\iota_{\alpha}$
with $\left|\alpha\right|=k$, as we will see in Section \ref{sec:Applications}
using the concrete examples of modulation spaces and Besov spaces.
This raises the question whether one can find less restrictive \emph{sufficient}
conditions for boundedness of the $\iota_{\alpha}$. For modulation
spaces\cite{KobayashiSugimotoModulationSobolevInclusion} and Besov
spaces\cite{TriebelTheoryOfFunctionSpaces}, this is indeed the case.
Nevertheless, I do not know of any results which apply in the full
generality considered here. This is a valuable topic for further research.

To prove necessity of the embedding $Y\cap\ell_{0}\left(I\right)\hookrightarrow\ell_{u^{\left(k,p,q\right)}}^{q}\left(I\right)$,
we will make the technical assumption that $\mathcal{Q}=\left(Q_{i}\right)_{i\in I}=\left(T_{i}Q_{i}'+b_{i}\right)_{i\in I}$
is tight, i.e.\@ that there is some $\varepsilon>0$ such that for
each $i\in I$, some ball $B_{\varepsilon}\left(c_{i}\right)$ is
contained in the ``normalized'' set $Q_{i}'$. In principle, the
points $\left(c_{i}\right)_{i\in I}$ are allowed to vary wildly with
$i\in I$. But the next lemma shows that we can actually choose the
$c_{i}$ only from a fixed finite set.

This technical fact will become important for constructing suitable
functions to ``test'' boundedness of the maps $\iota_{\alpha}$
in order to derive boundedness of the embedding $Y\cap\ell_{0}\left(I\right)\hookrightarrow\ell_{u^{\left(k,p,q\right)}}^{q}\left(I\right)$.
\begin{lem}
\label{lem:FinitelyManyNormalizationsSuffice}Let $\mathcal{Q}=\left(T_{i}Q_{i}'+b_{i}\right)_{i\in I}$
be a tight semi-structured covering of the open set $\emptyset\neq\mathcal{O}\subset\mathbb{R}^{d}$.

Then there is some $\varepsilon>0$ and finitely many points $a_{1},\dots,a_{N}\in\mathbb{R}^{d}$,
such that for each $i\in I$, there is some $\ell_{i}\in\underline{N}$
with $B_{\varepsilon}\left(a_{\ell_{i}}\right)\subset Q_{i}'$.\end{lem}
\begin{proof}
Fix some $\varepsilon>0$ and for each $i\in I$ some $c_{i}\in\mathbb{R}^{d}$
with $B_{\varepsilon}\left(c_{i}\right)\subset Q_{i}'$. Existence
is ensured by tightness of $\mathcal{Q}$. Furthermore, since $\mathcal{Q}$
is semi-structured, there is some $R>0$ with $Q_{i}'\subset B_{R}\left(0\right)$
for all $i\in I$.

Now, using Zorn's Lemma (or in fact using an induction, since $I$
is countable), we can choose a maximal subset $J\subset I$ with the
property $B_{\varepsilon/3}\left(c_{i}\right)\cap B_{\varepsilon/3}\left(c_{j}\right)=\emptyset$
for all $i,j\in J$ with $i\neq j$. For an arbitrary \emph{finite}
subset $L\subset J$, we now have -- because the sets $B_{\varepsilon/3}\left(c_{i}\right)\subset B_{\varepsilon}\left(c_{i}\right)\subset Q_{i}'\subset B_{R}\left(0\right)$
for $i\in L$ are pairwise disjoint -- that
\begin{align*}
\left|L\right|\cdot\lambda\left(B_{\varepsilon/3}\left(0\right)\right) & =\sum_{i\in L}\lambda\left(B_{\varepsilon/3}\left(c_{i}\right)\right)\\
 & =\lambda\left(\biguplus_{i\in L}B_{\varepsilon/3}\left(c_{i}\right)\right)\\
 & \leq\lambda\left(B_{R}\left(0\right)\right)
\end{align*}
and hence $\left|L\right|\leq\lambda\left(B_{R}\left(0\right)\right)/\lambda\left(B_{\varepsilon/3}\left(0\right)\right)$.
Since this holds for every finite subset $L\subset J$, $J$ must
be finite, say $J=\left\{ i_{1},\dots,i_{N}\right\} $. Define $a_{\ell}:=c_{i_{\ell}}$
for $\ell\in\underline{N}$.

Now, let $i\in I$ be arbitrary. If $i\in J$, i.e.\@ $i=i_{\ell}$
for some $\ell\in\underline{N}$, then 
\[
B_{\varepsilon/3}\left(a_{\ell}\right)=B_{\varepsilon/3}\left(c_{i_{\ell}}\right)\subset B_{\varepsilon}\left(c_{i_{\ell}}\right)\subset Q_{i_{\ell}}'=Q_{i}'.
\]
Otherwise, if $i\notin J$, then by maximality of $J$, we have $B_{\varepsilon/3}\left(c_{i}\right)\cap B_{\varepsilon/3}\left(c_{j}\right)\neq\emptyset$
for some $j\in J$, say $j=i_{\ell}$ for some $\ell\in\underline{N}$.
For $x\in B_{\varepsilon/3}\left(c_{j}\right)$, we now have
\[
\left|x-c_{i}\right|\leq\left|x-c_{j}\right|+\left|c_{j}-y\right|+\left|y-c_{i}\right|<\frac{\varepsilon}{3}+\frac{\varepsilon}{3}+\frac{\varepsilon}{3}=\varepsilon,
\]
where $y\in B_{\varepsilon/3}\left(c_{i}\right)\cap B_{\varepsilon/3}\left(c_{j}\right)\neq\emptyset$
was chosen arbitrarily. This means 
\[
B_{\varepsilon/3}\left(a_{\ell}\right)=B_{\varepsilon/3}\left(c_{i_{\ell}}\right)=B_{\varepsilon/3}\left(c_{j}\right)\subset B_{\varepsilon}\left(c_{i}\right)\subset Q_{i}'.
\]
We have thus shown that for every $i\in I$, there is some $\ell\in\underline{N}$
with $B_{\varepsilon/3}\left(a_{\ell}\right)\subset Q_{i}'$, as desired.
\end{proof}
As a further technical result, we need the following:
\begin{lem}
\label{lem:ShiftedBallsFitIntoSets}Let $\mathcal{Q}=\left(T_{i}Q_{i}'+b_{i}\right)_{i\in I}$
be a tight semi-structured covering. Then the following hold:
\begin{enumerate}
\item We have $\left|b_{i}\right|+\left\Vert T_{i}\right\Vert \asymp\sup_{x\in Q_{i}}\left|x\right|$
uniformly in $i\in I$.
\item If $I^{\left(0\right)}\subset I$ and $M>0$ satisfy $\left\Vert T_{i}^{-1}\right\Vert \leq M$
for all $i\in I^{\left(0\right)}$, then there is some $\delta>0$
and for each $i\in I^{\left(0\right)}$ some $y_{i}\in\mathbb{R}^{d}$
with $B_{\delta}\left(y_{i}\right)\subset Q_{i}$ and with
\[
\left|y_{i}\right|\asymp\left|b_{i}\right|+\left\Vert T_{i}\right\Vert .\qedhere
\]

\end{enumerate}
\end{lem}
\begin{proof}
By definition of a semi-structured covering, there is some $R>0$
with $Q_{i}'\subset B_{R}\left(0\right)$ for all $i\in I$. Furthermore,
by tightness, there is some $\varepsilon>0$ and for each $i\in I$
some $x_{i}\in\mathbb{R}^{d}$ with $B_{\varepsilon}\left(x_{i}\right)\subset Q_{i}'$.
We clearly have
\[
\sup_{x\in Q_{i}}\left|x\right|=\sup_{x\in Q_{i}'}\left|T_{i}x+b_{i}\right|\leq\left|b_{i}\right|+\left\Vert T_{i}\right\Vert \cdot\sup_{x\in Q_{i}'}\left|x\right|\leq\left|b_{i}\right|+R\cdot\left\Vert T_{i}\right\Vert ,
\]
so that for the first part of the lemma, it remains to show ``$\lesssim$''.

Now, let $i\in I$ be arbitrary. We distinguish two cases:
\begin{casenv}
\item We have $\left|b_{i}\right|>2R\cdot\left\Vert T_{i}\right\Vert $.
In this case, we choose $y_{i}:=b_{i}+T_{i}x_{i}\in Q_{i}$. Note
$\left|x_{i}\right|\leq R$ because of $x_{i}\in B_{\varepsilon}\left(x_{i}\right)\subset Q_{i}'\subset B_{R}\left(0\right)$.
Hence,
\begin{align*}
\left|b_{i}\right|+\left\Vert T_{i}\right\Vert  & \lesssim\frac{\left|b_{i}\right|}{4}+\frac{R}{2}\left\Vert T_{i}\right\Vert \\
 & \leq\frac{\left|b_{i}\right|}{2}\\
 & \leq\left|b_{i}\right|-R\left\Vert T_{i}\right\Vert \\
 & \leq\left|b_{i}\right|-\left\Vert T_{i}\right\Vert \left|x_{i}\right|\\
 & \leq\left|y_{i}\right|\\
 & \leq\left|b_{i}\right|+\left\Vert T_{i}\right\Vert \left|x_{i}\right|\\
 & \leq\left|b_{i}\right|+R\left\Vert T_{i}\right\Vert \\
 & \lesssim\left|b_{i}\right|+\left\Vert T_{i}\right\Vert .
\end{align*}
Because of $y_{i}\in Q_{i}$, this shows $\sup_{x\in Q_{i}}\left|x\right|\geq\left|y_{i}\right|\gtrsim\left|b_{i}\right|+\left\Vert T_{i}\right\Vert $,
as desired.

Now, assume that $i\in I^{\left(0\right)}$, so that $\left\Vert T_{i}^{-1}\right\Vert \leq M$.
For $\delta<\frac{\varepsilon}{M}$ and $x\in B_{\delta}\left(y_{i}\right)$,
we then have
\begin{align*}
\left|T_{i}^{-1}\left(x-T_{i}x_{i}-b_{i}\right)\right| & \leq\left\Vert T_{i}^{-1}\right\Vert \cdot\left|x-T_{i}x_{i}-b_{i}\right|\\
 & \leq M\cdot\left|x-y_{i}\right|\\
 & <M\cdot\delta<\varepsilon.
\end{align*}
Thus, 
\begin{align*}
x & =T_{i}x_{i}+b_{i}+T_{i}\left[T_{i}^{-1}\left(x-T_{i}x_{i}-b_{i}\right)\right]\\
 & \in T_{i}x_{i}+b_{i}+T_{i}\left(B_{\varepsilon}\left(0\right)\right)\\
 & =T_{i}\left(B_{\varepsilon}\left(x_{i}\right)\right)+b_{i}\\
 & \subset T_{i}Q_{i}'+b_{i}=Q_{i}.
\end{align*}
so that we get $B_{\delta}\left(y_{i}\right)\subset Q_{i}$ as claimed.

\item We have $\left|b_{i}\right|\leq2R\cdot\left\Vert T_{i}\right\Vert $.
Let $r:=\frac{\varepsilon}{2}>0$ and choose $z\in\mathbb{R}^{d}$
with $\left|z\right|=r$ and $\left|T_{i}z\right|=r\cdot\left\Vert T_{i}\right\Vert $.
Then we have
\begin{align*}
{\rm diam}\left(T_{i}\left(B_{r}\left(x_{i}\right)\right)+b_{i}\right) & ={\rm diam}\left(T_{i}\left(B_{r}\left(0\right)\right)\right)\\
 & ={\rm diam}\left(T_{i}\left(\overline{B_{r}}\left(0\right)\right)\right)\\
 & \geq\left|T_{i}z-T_{i}\left(-z\right)\right|=2r\cdot\left\Vert T_{i}\right\Vert .
\end{align*}
Now, if ${\rm diam}\left(A\right)>\alpha>0$ for some $A\subset\mathbb{R}^{d}$,
then there is some $x\in A$ with $\left|x\right|\geq\frac{\alpha}{2}$,
since otherwise we would have $\left|x-y\right|\leq\left|x\right|+\left|y\right|\leq\alpha$
for all $x,y\in A$ and hence ${\rm diam}\left(A\right)\leq\alpha$.

Thus, by the above calculation, there is some $z_{i}\in B_{r}\left(x_{i}\right)$
with $\left|T_{i}z_{i}+b_{i}\right|\geq\frac{r}{2}\cdot\left\Vert T_{i}\right\Vert $.
Recall $B_{r}\left(x_{i}\right)\subset B_{\varepsilon}\left(x_{i}\right)\subset Q_{i}'\subset B_{R}\left(0\right)$.
For $y_{i}:=T_{i}z_{i}+b_{i}\in Q_{i}$, we thus have
\[
\qquad\qquad\left|b_{i}\right|+\left\Vert T_{i}\right\Vert \lesssim\frac{r}{4}\cdot\left(\left\Vert T_{i}\right\Vert +\frac{\left|b_{i}\right|}{2R}\right)\leq\frac{r}{2}\cdot\left\Vert T_{i}\right\Vert \leq\left|y_{i}\right|\leq\left\Vert T_{i}\right\Vert \left|z_{i}\right|+\left|b_{i}\right|\lesssim\left|b_{i}\right|+\left\Vert T_{i}\right\Vert .
\]
As above, this shows $\sup_{x\in Q_{i}}\left|x\right|\gtrsim\left|b_{i}\right|+\left\Vert T_{i}\right\Vert $.

Now, assume again that $i\in I^{\left(0\right)}$, so that $\left\Vert T_{i}^{-1}\right\Vert \leq M$.
For $\delta<\frac{\varepsilon}{2M}$ and arbitrary $x\in B_{\delta}\left(y_{i}\right)$,
this yields
\begin{align*}
\left|T_{i}^{-1}\left(x-T_{i}x_{i}-b_{i}\right)\right| & \leq\left|T_{i}^{-1}\left(x-y_{i}\right)\right|+\left|T_{i}^{-1}\left(y_{i}-T_{i}x_{i}-b_{i}\right)\right|\\
 & \leq M\cdot\left|x-y_{i}\right|+\left|T_{i}^{-1}\left(T_{i}\left(z_{i}-x_{i}\right)\right)\right|\\
 & \leq M\cdot\left|x-y_{i}\right|+\left|z_{i}-x_{i}\right|\\
 & <M\delta+r\\
 & <\varepsilon.
\end{align*}
As in the first case, this implies $B_{\delta}\left(y_{i}\right)\subset T_{i}\left(B_{\varepsilon}\left(x_{i}\right)\right)+b_{i}\subset Q_{i}$.

\end{casenv}
We have thus shown that any $\delta>0$ with $\delta<\frac{\varepsilon}{2M}$
makes the second part of the lemma true.
\end{proof}
Now, we can finally prove necessity of the embedding $Y\cap\ell_{0}\left(I\right)\hookrightarrow\ell_{u^{\left(k,p,q\right)}}^{q}\left(I\right)$.
Note that we will see in the next section that the restriction to
$\ell_{0}\left(I\right)$ is superfluous if $Y$ is a weighted Lebesgue
space.
\begin{thm}
\label{thm:SequenceSpaceEmbeddingIsNecessary}Let $p,q\in\left(0,\infty\right]$
and let $\mathcal{Q}=\left(Q_{i}\right)_{i\in I}=\left(T_{i}Q_{i}'+b_{i}\right)_{i\in I}$
be a tight semi-structured $L^{p}$-decomposition covering of the
open set $\emptyset\neq\mathcal{O}\subset\mathbb{R}^{d}$. Assume
that $Y\leq\mathbb{C}^{I}$ is a $\mathcal{Q}$-regular sequence space.

Finally, let $k\in\mathbb{N}_{0}$ and assume that
\[
\iota_{\alpha}:\mathcal{S}_{\mathcal{O}}^{p,Y}\left(\mathbb{R}^{d}\right)\to L^{q}\left(\mathbb{R}^{d}\right),f\mapsto\partial^{\alpha}f
\]
is bounded for all $\alpha\in\mathbb{N}_{0}^{d}$ with $\left|\alpha\right|=k$.
Define
\[
u^{\left(k,p,q\right)}:=\left(\left|\det T_{i}\right|^{\frac{1}{p}-\frac{1}{q}}\cdot\left(\left|b_{i}\right|^{k}+\left\Vert T_{i}\right\Vert ^{k}\right)\right)_{i\in I}
\]
as in Theorem \ref{thm:SufficientConditionsForBoundednessOfDerivative}.

Then the following hold:
\begin{enumerate}
\item The map
\[
Y\cap\ell_{0}\left(I\right)\hookrightarrow\ell_{u^{\left(k,p,q\right)}}^{q}\left(I\right),\left(c_{i}\right)_{i\in I}\mapsto\left(c_{i}\right)_{i\in I}
\]
is well-defined and bounded, where $\ell_{0}\left(I\right)$ denotes
the space of finitely supported sequences on $I$.
\item If $q=\infty$, then the map
\[
Y\cap\ell_{0}\left(I\right)\hookrightarrow\ell_{u^{\left(k,p,q\right)}}^{q^{\triangledown}}\left(I\right)=\ell_{u^{\left(k,p,q\right)}}^{1}\left(I\right),\left(c_{i}\right)_{i\in I}\mapsto\left(c_{i}\right)_{i\in I}
\]
is well-defined and bounded.
\item For each subset $I^{\left(0\right)}\subset I$ and each $M>0$ with
$\left\Vert T_{i}^{-1}\right\Vert \leq M$ for all $i\in I^{\left(0\right)}$,
the following hold:

\begin{enumerate}
\item If $q\in\left[2,\infty\right)$, then
\[
Y\cap\ell_{0}\left(\smash{I^{\left(0\right)}}\right)\hookrightarrow\ell_{u^{\left(k,p,2\right)}}^{2}\left(I\right).
\]

\item If $q\in\left(0,\infty\right)$, then
\[
Y\cap\ell_{0}\left(\smash{I^{\left(0\right)}}\right)\hookrightarrow\ell_{u^{\left(k,p,p\right)}}^{2}\left(I\right).\qedhere
\]

\end{enumerate}
\end{enumerate}
\end{thm}
\begin{proof}
Our argument is based on the following claim, whose proof we postpone
to the end of the current proof.
\begin{claim}
\label{claim:NecessaryConditionCentralClaim}Let 
\begin{align*}
I_{1} & :=\left\{ i\in I\with\left\Vert T_{i}\right\Vert ^{k}>\left(2\pi\left|b_{i}\right|\right)^{k}\right\} ,\\
I_{2} & :=\left\{ i\in I\with\left\Vert T_{i}\right\Vert ^{k}\leq\left(2\pi\left|b_{i}\right|\right)^{k}\right\} .
\end{align*}
Then there are a \emph{finite} family of functions $\mathscr{F}\subset C_{c}^{\infty}\left(\mathbb{R}^{d}\right)$,
some $\delta>0$ and $C_{1},C_{2}>0$ such that the following hold:
\begin{enumerate}
\item For each $i\in I_{1}$, there is some $\alpha^{\left(i\right)}\in\mathbb{N}_{0}^{d}$
with $\left|\alpha^{\left(i\right)}\right|=k$ and a function $\gamma_{i}\in C_{c}^{\infty}\left(Q_{i}\right)$
with
\begin{equation}
\left|\left[\partial^{\alpha^{\left(i\right)}}\left(\mathcal{F}^{-1}\gamma_{i}\right)\right]\left(x\right)\right|\geq C_{1}\cdot\left|\det T_{i}\right|\cdot\left\Vert T_{i}\right\Vert ^{k}\cdot\chi_{B_{\delta}\left(0\right)}\left(T_{i}^{T}x\right)\qquad\text{ for all }x\in\mathbb{R}^{d}.\label{eq:MatrixDominatedTestFunctionEstimateFromBelow}
\end{equation}

\item For each $i\in I_{2}$, there is some $\alpha^{\left(i\right)}\in\mathbb{N}_{0}^{d}$
with $\left|\alpha^{\left(i\right)}\right|=k$ and a function $\gamma_{i}\in C_{c}^{\infty}\left(Q_{i}\right)$
with
\begin{equation}
\left|\left[\partial^{\alpha^{\left(i\right)}}\left(\mathcal{F}^{-1}\gamma_{i}\right)\right]\left(x\right)\right|\geq C_{2}\cdot\left|\det T_{i}\right|\cdot\left|b_{i}\right|^{k}\cdot\chi_{B_{\delta}\left(0\right)}\left(T_{i}^{T}x\right)\qquad\text{ for all }x\in\mathbb{R}^{d}.\label{eq:VectorDominatedTestFunctionEstimateFromBelow}
\end{equation}

\end{enumerate}
Furthermore, each $\gamma_{i}$ (for $i\in I_{1}$ as well as for
$i\in I_{2}$) is of the form $\gamma_{i}=L_{b_{i}}\left(f_{i}\circ T_{i}^{-1}\right)$,
i.e.\@ 
\begin{equation}
\gamma_{i}\left(\xi\right)=f_{i}\left(T_{i}^{-1}\left(\xi-b_{i}\right)\right)\label{eq:MatrixDominatedTestFunctionNormalization}
\end{equation}
for all $\xi\in\mathbb{R}^{d}$ and some $f_{i}\in\mathscr{F}$.
\end{claim}
Note that equations (\ref{eq:MatrixDominatedTestFunctionEstimateFromBelow})
and (\ref{eq:VectorDominatedTestFunctionEstimateFromBelow}) -- together
with the definitions of $I_{1},I_{2}$ -- imply
\begin{equation}
\left|\left[\partial^{\alpha^{\left(i\right)}}\left(\mathcal{F}^{-1}\gamma_{i}\right)\right]\left(x\right)\right|\geq C\cdot\left|\det T_{i}\right|\cdot\left(\left|b_{i}\right|^{k}+\left\Vert T_{i}\right\Vert ^{k}\right)\cdot\chi_{B_{\delta}\left(0\right)}\left(T_{i}^{T}x\right)\label{eq:TestFunctionsMainEstimate}
\end{equation}
for all $x\in\mathbb{R}^{d}$ and $i\in I$ with $C:=\frac{1}{2}\cdot\min\left\{ C_{1},\left(2\pi\right)^{-k}\cdot C_{2}\right\} >0$.

Now, let $c=\left(c_{i}\right)_{i\in I}\in Y\cap\ell_{0}\left(I\right)$
be arbitrary and let $I_{0}:={\rm supp}\, c\subset I$. For $\alpha\in\mathbb{N}_{0}^{d}$
with $\left|\alpha\right|=k$, set
\[
I^{\left(\alpha\right)}:=\left\{ i\in I\with\alpha^{\left(i\right)}=\alpha\right\} .
\]
Finally, let $S^{1}:=\left\{ z\in\mathbb{C}\with\left|z\right|=1\right\} $
and for arbitrary sequences $x=\left(x_{i}\right)_{i\in I}\in\left(\mathbb{R}^{d}\right)^{I}$,
$\varepsilon=\left(\varepsilon_{i}\right)_{i\in I}\in\left(S^{1}\right)^{I}$
and an arbitrary multiindex $\alpha\in\mathbb{N}_{0}^{d}$ with $\left|\alpha\right|=k$,
define
\begin{align*}
g_{x,\varepsilon,\alpha} & :=\sum_{i\in I^{\left(\alpha\right)}}\left(\varepsilon_{i}c_{i}\left|\det T_{i}\right|^{\frac{1}{p}-1}\cdot M_{-x_{i}}\gamma_{i}\right)\in C_{c}^{\infty}\left(\mathcal{O}\right),\\
f_{x,\varepsilon,\alpha} & :=\mathcal{F}^{-1}g_{x,\varepsilon,\alpha}=\sum_{i\in I^{\left(\alpha\right)}}\left(\varepsilon_{i}c_{i}\left|\det T_{i}\right|^{\frac{1}{p}-1}\cdot L_{x_{i}}\left[\mathcal{F}^{-1}\gamma_{i}\right]\right)\in\mathcal{F}^{-1}\left(C_{c}^{\infty}\left(\mathcal{O}\right)\right)=\mathcal{S}_{\mathcal{O}}\left(\mathbb{R}^{d}\right).
\end{align*}
Let us first show $g_{x,\varepsilon,\alpha}\in\mathcal{D}_{\mathcal{F}}\left(\mathcal{Q},L^{p},Y\right)$
and provide an estimate for the corresponding norm. To this end, we
note that equation (\ref{eq:MatrixDominatedTestFunctionNormalization})
-- together with elementary properties of the Fourier transform --
implies
\begin{align}
\left\Vert \mathcal{F}^{-1}\gamma_{i}\right\Vert _{L^{p}} & =\left|\det T_{i}\right|\cdot\left\Vert M_{b_{i}}\left[\left(\mathcal{F}^{-1}f_{i}\right)\circ T_{i}^{T}\right]\right\Vert _{L^{p}}\nonumber \\
 & =\left|\det T_{i}\right|\cdot\left\Vert \left(\mathcal{F}^{-1}f_{i}\right)\circ T_{i}^{T}\right\Vert _{L^{p}}\nonumber \\
 & =\left|\det T_{i}\right|^{1-\frac{1}{p}}\cdot\left\Vert \mathcal{F}^{-1}f_{i}\right\Vert _{L^{p}}\nonumber \\
 & \leq K\cdot\left|\det T_{i}\right|^{1-\frac{1}{p}},\label{eq:SpecialTestFunctionsLpEstimate}
\end{align}
for the finite(!) constant $K:=\sup_{f\in\mathscr{F}}\left\Vert \mathcal{F}^{-1}f\right\Vert _{L^{p}}$.
For finiteness of $K$, we used that $\mathscr{F}\subset C_{c}^{\infty}\left(\mathbb{R}^{d}\right)$
is a finite set.

Now, let $\Phi=\left(\varphi_{i}\right)_{i\in I}$ be an $L^{p}$-BAPU
for $\mathcal{Q}$ and set $N:=\sup_{i\in I}\left|i^{\ast}\right|$.
For arbitrary $\ell\in I$, we have
\begin{equation}
\varphi_{\ell}\cdot g_{x,\varepsilon,\alpha}=\sum_{i\in I^{\left(\alpha\right)}\cap\ell^{\ast}}\varepsilon_{i}c_{i}\left|\det T_{i}\right|^{\frac{1}{p}-1}\cdot M_{-x_{i}}\left(\varphi_{\ell}\gamma_{i}\right),\label{eq:TestFunctionLocalization}
\end{equation}
since multiplication by $\varphi_{\ell}$ and modulation by $-x_{i}$
commute and because of $\gamma_{i}\in C_{c}^{\infty}\left(Q_{i}\right)$.
Thus, the (quasi)-triangle inequality for $L^{p}\left(\mathbb{R}^{d}\right)$,
together with the uniform estimate $\left|I^{\left(\alpha\right)}\cap\ell^{\ast}\right|\leq\left|\ell^{\ast}\right|\leq N$,
yields
\begin{align}
\left\Vert \mathcal{F}^{-1}\left(\varphi_{\ell}\cdot g_{x,\varepsilon,\alpha}\right)\right\Vert _{L^{p}} & \lesssim\sum_{i\in I^{\left(\alpha\right)}\cap\ell^{\ast}}\left|c_{i}\right|\left|\det T_{i}\right|^{\frac{1}{p}-1}\cdot\left\Vert \mathcal{F}^{-1}\left(M_{-x_{i}}\left(\varphi_{\ell}\gamma_{i}\right)\right)\right\Vert _{L^{p}}\nonumber \\
 & =\sum_{i\in I^{\left(\alpha\right)}\cap\ell^{\ast}}\left|c_{i}\right|\left|\det T_{i}\right|^{\frac{1}{p}-1}\cdot\left\Vert \mathcal{F}^{-1}\left(\varphi_{\ell}\gamma_{i}\right)\right\Vert _{L^{p}}\nonumber \\
 & \overset{\left(\ast\right)}{\lesssim}\sum_{i\in I^{\left(\alpha\right)}\cap\ell^{\ast}}\left|c_{i}\right|\left|\det T_{i}\right|^{\frac{1}{p}-1}\cdot\left\Vert \mathcal{F}^{-1}\gamma_{i}\right\Vert _{L^{p}}\nonumber \\
\left(\text{by equation }\eqref{eq:SpecialTestFunctionsLpEstimate}\right) & \lesssim\sum_{i\in I^{\left(\alpha\right)}\cap\ell^{\ast}}\left|c_{i}\right|\nonumber \\
 & \leq\left(\left|c\right|^{\ast}\right)_{\ell}\label{eq:TestFunctionLocalizedNormEstimate}
\end{align}
for all $\ell\in I$. Here, the sequence $\left|c\right|\in Y$ is
defined by $\left|c\right|_{i}=\left|c_{i}\right|$ and furthermore
$\left(d^{\ast}\right)_{i}=\sum_{\ell\in i^{\ast}}d_{\ell}$ as usual.
Above, the step marked with $\left(\ast\right)$ can be justified
as follows: For $p\in\left[1,\infty\right]$, it is a direct consequence
of Young's convolution inequality $L^{1}\ast L^{p}\hookrightarrow L^{p}$,
since the definition of an $L^{p}$-BAPU implies finiteness of $\sup_{i\in I}\left\Vert \mathcal{F}^{-1}\varphi_{i}\right\Vert _{L^{1}}$.
In case of $p\in\left(0,1\right)$, it is a consequence of Corollary
\ref{cor:SemiStructuredQuasiBanachConvolution} and the definition
of an $L^{p}$-BAPU, since we have the support estimate ${\rm supp}\,\gamma_{i}\subset Q_{i}\subset Q_{\ell}^{\ast}$
for $i\in\ell^{\ast}$ and also ${\rm supp}\,\varphi_{\ell}\subset\overline{Q_{\ell}}\subset Q_{\ell}^{\ast}$.

All in all, we conclude $g_{x,\varepsilon,\alpha}\in\mathcal{D}_{\mathcal{F}}\left(\mathcal{Q},L^{p},Y\right)$,
since the previous estimate yields -- by solidity and $\mathcal{Q}$-regularity
of $Y$ -- that
\begin{align*}
\left\Vert g_{x,\varepsilon,\alpha}\right\Vert _{\mathcal{D}_{\mathcal{F}}\left(\mathcal{Q},L^{p},Y\right)} & =\left\Vert \left(\left\Vert \mathcal{F}^{-1}\left(\varphi_{\ell}\cdot g_{x,\varepsilon,\alpha}\right)\right\Vert _{L^{p}}\right)_{\ell\in I}\right\Vert _{Y}\\
 & \lesssim\left\Vert \left|c\right|^{\ast}\right\Vert _{Y}\lesssim\left\Vert \left|c\right|\right\Vert _{Y}=\left\Vert c\right\Vert _{Y}<\infty.
\end{align*}
Hence, $f_{x,\varepsilon,\alpha}\in\mathcal{S}_{\mathcal{O}}^{p,Y}$.
Recall that we assume $\iota_{\alpha}:\mathcal{S}_{\mathcal{O}}^{p,Y}\to L^{q}\left(\mathbb{R}^{d}\right),f\mapsto\partial^{\alpha}f$
to be bounded. Thus,
\begin{equation}
\left\Vert \partial^{\alpha}f_{x,\varepsilon,\alpha}\right\Vert _{L^{q}}=\left\Vert \iota_{\alpha}f_{x,\varepsilon,\alpha}\right\Vert _{L^{q}}\lesssim\left\Vert f_{x,\varepsilon,\alpha}\right\Vert _{\mathcal{D}\left(\mathcal{Q},L^{p},Y\right)}=\left\Vert g_{x,\varepsilon,\alpha}\right\Vert _{\mathcal{D}_{\mathcal{F}}\left(\mathcal{Q},L^{p},Y\right)}\lesssim\left\Vert c\right\Vert _{Y},\label{eq:NecessaryConditionBoundednessOfDerivativeApplied}
\end{equation}
where the implied constants are independent of $x,\varepsilon$ and
of $c$.

Finally, we will obtain a lower bound on $\left\Vert \partial^{\alpha}f_{x,\varepsilon,\alpha}\right\Vert _{L^{q}}$
which will then imply the claim. Indeed, we have
\begin{align*}
\left\Vert \partial^{\alpha}f_{x,\varepsilon,\alpha}\right\Vert _{L^{q}} & =\left\Vert \sum_{i\in I^{\left(\alpha\right)}}\left(\varepsilon_{i}c_{i}\left|\det T_{i}\right|^{\frac{1}{p}-1}\cdot L_{x_{i}}\left(\partial^{\alpha}\left[\mathcal{F}^{-1}\gamma_{i}\right]\right)\right)\right\Vert _{L^{q}}\\
 & =\left\Vert \sum_{i\in I^{\left(\alpha\right)}\cap I_{0}}L_{x_{i}}\left(\varepsilon_{i}c_{i}\left|\det T_{i}\right|^{\frac{1}{p}-1}\cdot\partial^{\alpha}\left(\mathcal{F}^{-1}\gamma_{i}\right)\right)\right\Vert _{L^{q}}\\
\left(\text{by Lemma }\ref{lem:AsymptoticTranslationNorm}\text{ for suitable }x=\left(x_{i}\right)_{i\in I}\in\left(\mathbb{R}^{d}\right)^{I}\right) & \geq\frac{1}{2}\cdot\left\Vert \left(\left\Vert \varepsilon_{i}c_{i}\left|\det T_{i}\right|^{\frac{1}{p}-1}\cdot\partial^{\alpha}\left(\mathcal{F}^{-1}\gamma_{i}\right)\right\Vert _{L^{q}}\right)_{i\in I^{\left(\alpha\right)}\cap I_{0}}\right\Vert _{\ell^{q}}\\
\left(\alpha=\alpha^{\left(i\right)}\text{ for }i\in I^{\left(\alpha\right)}\text{ and }c_{i}=0\text{ for }i\notin I_{0}\right) & =\frac{1}{2}\cdot\left\Vert \left(\left|c_{i}\right|\left|\det T_{i}\right|^{\frac{1}{p}-1}\cdot\left\Vert \partial^{\alpha^{\left(i\right)}}\left(\mathcal{F}^{-1}\gamma_{i}\right)\right\Vert _{L^{q}}\right)_{i\in I^{\left(\alpha\right)}}\right\Vert _{\ell^{q}}.
\end{align*}
Note that Lemma \ref{lem:AsymptoticTranslationNorm} is indeed applicable
(even for $q=\infty$), since $I^{\left(\alpha\right)}\cap I_{0}\subset I_{0}$
is finite and because of $\partial^{\alpha}\left(\mathcal{F}^{-1}\gamma_{i}\right)\in\mathcal{S}\left(\mathbb{R}^{d}\right)\subset C_{0}\left(\mathbb{R}^{d}\right)$.

Now, we finally employ estimate (\ref{eq:TestFunctionsMainEstimate})
to conclude -- for $q\in\left(0,\infty\right)$ -- that
\begin{align}
\left\Vert \partial^{\alpha^{\left(i\right)}}\left(\mathcal{F}^{-1}\gamma_{i}\right)\right\Vert _{L^{q}} & =\left[\int_{\mathbb{R}^{d}}\left|\left(\partial^{\alpha^{\left(i\right)}}\left[\mathcal{F}^{-1}\gamma_{i}\right]\right)\left(x\right)\right|^{q}\,{\rm d}x\right]^{1/q}\nonumber \\
 & \geq C\cdot\left|\det T_{i}\right|\cdot\left(\left|b_{i}\right|^{k}+\left\Vert T_{i}\right\Vert ^{k}\right)\cdot\left(\int_{\mathbb{R}^{d}}\left[\chi_{B_{\delta}\left(0\right)}\left(T_{i}^{T}x\right)\right]^{q}\,{\rm d}x\right)^{1/q}\nonumber \\
\left(y=T_{i}^{T}x\right) & =C\cdot\left|\det T_{i}\right|^{1-\frac{1}{q}}\cdot\left(\left|b_{i}\right|^{k}+\left\Vert T_{i}\right\Vert ^{k}\right)\cdot\left(\int_{\mathbb{R}^{d}}\left[\chi_{B_{\delta}\left(0\right)}\left(y\right)\right]^{q}\,{\rm d}y\right)^{1/q}\nonumber \\
 & =C\cdot\left[\lambda\left(B_{\delta}\left(0\right)\right)\right]^{1/q}\cdot\left|\det T_{i}\right|^{1-\frac{1}{q}}\cdot\left(\left|b_{i}\right|^{k}+\left\Vert T_{i}\right\Vert ^{k}\right),\label{eq:TestFunctionSpaceSideDerivativeLowerLqEstimate}
\end{align}
where $\lambda$ denotes the Lebesgue measure. For $q=\infty$, equation
(\ref{eq:TestFunctionsMainEstimate}) easily shows that equation (\ref{eq:TestFunctionSpaceSideDerivativeLowerLqEstimate})
still holds. By putting everything together (recall equation (\ref{eq:NecessaryConditionBoundednessOfDerivativeApplied})),
we arrive at
\begin{align*}
\left\Vert c\right\Vert _{Y} & \gtrsim\left\Vert \partial^{\alpha}f_{x,\varepsilon,\alpha}\right\Vert _{L^{q}}\\
\left(\text{for suitable }x=\left(x_{i}\right)_{i\in I}\in\left(\mathbb{R}^{d}\right)^{I}\right) & \gtrsim\left\Vert \left(\left|c_{i}\right|\left|\det T_{i}\right|^{\frac{1}{p}-1}\cdot\left\Vert \partial^{\alpha^{\left(i\right)}}\left(\mathcal{F}^{-1}\gamma_{i}\right)\right\Vert _{L^{q}}\right)_{i\in I^{\left(\alpha\right)}}\right\Vert _{\ell^{q}}\\
 & \gtrsim\left\Vert \left(\left|c_{i}\right|\cdot\left|\det T_{i}\right|^{\frac{1}{p}-\frac{1}{q}}\cdot\left(\left|b_{i}\right|^{k}+\left\Vert T_{i}\right\Vert ^{k}\right)\right)_{i\in I^{\left(\alpha\right)}}\right\Vert _{\ell^{q}}\\
 & =\left\Vert \left(c_{i}\cdot\chi_{I^{\left(\alpha\right)}}\right)_{i\in I}\right\Vert _{\ell_{u^{\left(k,p,q\right)}}^{q}},
\end{align*}
where the implied constants are independent of $c$. Furthermore,
because of 
\[
I=\biguplus_{\alpha\in\mathbb{N}_{0}^{d},\left|\alpha\right|=k}I^{\left(\alpha\right)},
\]
the (quasi)-triangle inequality for $\ell_{u^{\left(k,p,q\right)}}^{q}\left(I\right)$
finally implies 
\[
\left\Vert c\right\Vert _{\ell_{u^{\left(k,p,q\right)}}^{q}}\lesssim\sum_{\alpha\in\mathbb{N}_{0}^{d},\left|\alpha\right|=k}\left\Vert \left(c_{i}\cdot\chi_{I^{\left(\alpha\right)}}\right)_{i\in I}\right\Vert _{\ell_{u^{\left(k,p,q\right)}}^{q}}\lesssim\left\Vert c\right\Vert _{Y},
\]
where the implied constants are independent of $c\in Y\cap\ell_{0}\left(I\right)$.
 This establishes the first part of the theorem.

For the second part, we use $x_{i}=0$ for all $i\in I$ and we choose
each $\varepsilon_{i}\in S^{1}$ in such a way that 
\begin{align*}
\varepsilon_{i}c_{i}\cdot\left[\partial^{\alpha^{\left(i\right)}}\left(\mathcal{F}^{-1}\gamma_{i}\right)\right]\left(0\right) & =\left|c_{i}\cdot\left[\partial^{\alpha^{\left(i\right)}}\left(\mathcal{F}^{-1}\gamma_{i}\right)\right]\left(0\right)\right|\\
\left(\text{by equation }\eqref{eq:TestFunctionsMainEstimate}\right) & \geq C\cdot\left|c_{i}\right|\cdot\left|\det T_{i}\right|\cdot\left(\left|b_{i}\right|^{k}+\left\Vert T_{i}\right\Vert ^{k}\right).
\end{align*}
Now, let $\alpha\in\mathbb{N}_{0}^{d}$ with $\left|\alpha\right|=k$
be arbitrary. By continuity of $\partial^{\alpha}f_{x,\varepsilon,\alpha}\in\mathcal{S}\left(\mathbb{R}^{d}\right)$,
the $L^{\infty}$ norm of the partial derivative $\partial^{\alpha}f_{x,\varepsilon,\alpha}$
coincides with the genuine supremum norm. Hence,
\begin{align*}
\left\Vert \partial^{\alpha}f_{x,\varepsilon,\alpha}\right\Vert _{L^{\infty}} & \geq\left|\partial^{\alpha}f_{x,\varepsilon,\alpha}\left(0\right)\right|\\
\left(x_{i}=0\text{ for all }i\in I\right) & =\left|\left[\partial^{\alpha}\sum_{i\in I^{\left(\alpha\right)}}\left(\varepsilon_{i}c_{i}\left|\det T_{i}\right|^{\frac{1}{p}-1}\cdot\mathcal{F}^{-1}\gamma_{i}\right)\right]\left(0\right)\right|\\
\left(\alpha=\alpha^{\left(i\right)}\text{ for }i\in I^{\left(\alpha\right)}\right) & =\left|\sum_{i\in I^{\left(\alpha\right)}}\varepsilon_{i}c_{i}\left|\det T_{i}\right|^{\frac{1}{p}-1}\cdot\left[\partial^{\alpha^{\left(i\right)}}\left(\mathcal{F}^{-1}\gamma_{i}\right)\right]\left(0\right)\right|\\
\left(\text{all summands nonnegative}\right) & \geq C\cdot\sum_{i\in I^{\left(\alpha\right)}}\left|c_{i}\right|\left|\det T_{i}\right|^{\frac{1}{p}}\cdot\left(\left|b_{i}\right|^{k}+\left\Vert T_{i}\right\Vert ^{k}\right).
\end{align*}
But as seen in equation (\ref{eq:NecessaryConditionBoundednessOfDerivativeApplied})
above, we have $\left\Vert \partial^{\alpha}f_{x,\varepsilon,\alpha}\right\Vert _{L^{q}}\lesssim\left\Vert c\right\Vert _{Y}$,
where the implied constant does not depend on $\varepsilon,x$ and
$c$. Thus, we can sum over all $\alpha\in\mathbb{N}_{0}^{d}$ with
$\left|\alpha\right|=k$, to obtain (recalling that $q=\infty$)
\begin{align*}
\left\Vert c\right\Vert _{Y} & \gtrsim C^{-1}\cdot\sum_{\left|\alpha\right|=k}\left\Vert \partial^{\alpha}f_{x,\varepsilon,\alpha}\right\Vert _{L^{\infty}}\\
 & \geq\sum_{i\in I}\left|c_{i}\right|\left|\det T_{i}\right|^{\frac{1}{p}}\cdot\left(\left|b_{i}\right|^{k}+\left\Vert T_{i}\right\Vert ^{k}\right)\\
\left(\text{since }\frac{1}{q}=0\right) & =\left\Vert \left(c_{i}\right)_{i\in I}\right\Vert _{\ell_{u^{\left(k,p,q\right)}}^{1}}.
\end{align*}
Since $\left(c_{i}\right)_{i\in I}\in Y\cap\ell_{0}\left(I\right)$
was arbitrary, this completes the proof of the second part of the
theorem.

For the first statement in the last part of the theorem, we assume
$2\leq q<\infty$ and we let $M>0$ with $\left\Vert T_{i}^{-1}\right\Vert \leq M$
for all $i\in I^{\left(0\right)}$. Furthermore, we use $x_{i}=0$
for all $i\in I$ and we assume $c=\left(c_{i}\right)_{i\in I}\in Y\cap\ell_{0}\left(I^{\left(0\right)}\right)$,
i.e.\@ $c_{i}=0$ for all $i\in I\setminus I^{\left(0\right)}$ and
hence $I_{0}={\rm supp}\, c\subset I^{\left(0\right)}$. Let $\alpha\in\mathbb{N}_{0}^{d}$
with $\left|\alpha\right|=k$. We want to apply Khintchine's inequality
(cf.\@ \cite[Proposition 4.5]{WolffHarmonicAnalysis}), so that we
consider $\varepsilon=\left(\varepsilon_{i}\right)_{i\in I^{\left(\alpha\right)}\cap I_{0}}$
as a random vector, with $\varepsilon_{i}$ independent and identically
distributed Rademacher random variables. As noted above in equation
(\ref{eq:NecessaryConditionBoundednessOfDerivativeApplied}), we have
\[
\left\Vert \partial^{\alpha}f_{x,\varepsilon,\alpha}\right\Vert _{L^{q}}\lesssim\left\Vert c\right\Vert _{Y},
\]
where the implied constant does not depend on $\varepsilon,x$ and
$c$. By raising this to the $q$-th power and taking expectations
(with respect to $\varepsilon$), we conclude (cf.\@ the definition
of $f_{x,\varepsilon,\alpha}$ from above)
\begin{align*}
\left\Vert c\right\Vert _{Y}^{q} & \gtrsim\mathbb{E}_{\varepsilon}\left\Vert \partial^{\alpha}f_{x,\varepsilon,\alpha}\right\Vert _{L^{q}}^{q}\\
\left(x_{i}=0\text{ and }\alpha=\alpha^{\left(i\right)}\,\forall i\in I^{\left(\alpha\right)}\right) & =\int_{\mathbb{R}^{d}}\mathbb{E}_{\varepsilon}\left|\sum_{i\in I^{\left(\alpha\right)}\cap I_{0}}\varepsilon_{i}c_{i}\left|\det T_{i}\right|^{\frac{1}{p}-1}\cdot\partial^{\alpha^{\left(i\right)}}\left[\mathcal{F}^{-1}\gamma_{i}\right]\left(x\right)\right|^{q}\,{\rm d}x\\
\left(\text{by Khintchine's inequality}\right) & \gtrsim\int_{\mathbb{R}^{d}}\left(\sum_{i\in I^{\left(\alpha\right)}\cap I_{0}}\left|c_{i}\left|\det T_{i}\right|^{\frac{1}{p}-1}\cdot\partial^{\alpha^{\left(i\right)}}\left[\mathcal{F}^{-1}\gamma_{i}\right]\left(x\right)\right|^{2}\right)^{q/2}\,{\rm d}x\\
\left(\text{by eq. }\eqref{eq:TestFunctionsMainEstimate}\right) & \gtrsim\int_{\mathbb{R}^{d}}\left[\sum_{i\in I^{\left(\alpha\right)}\cap I_{0}}\!\left|c_{i}\left|\det T_{i}\right|^{\frac{1}{p}-1}\cdot\left|\det T_{i}\right|\left(\left|b_{i}\right|^{k}+\left\Vert T_{i}\right\Vert ^{k}\right)\cdot\chi_{B_{\delta}\left(0\right)}\left(T_{i}^{T}x\right)\right|^{2}\right]^{\frac{q}{2}}\!\!\!\!{\rm d}x\\
 & =\int_{\mathbb{R}^{d}}\left(\sum_{i\in I^{\left(\alpha\right)}\cap I_{0}}\left|c_{i}\left|\det T_{i}\right|^{\frac{1}{p}}\left(\left|b_{i}\right|^{k}+\left\Vert T_{i}\right\Vert ^{k}\right)\cdot\chi_{B_{\delta}\left(0\right)}\left(T_{i}^{T}x\right)\right|^{2}\right)^{q/2}\,{\rm d}x.
\end{align*}
Now, note that if $\chi_{B_{\delta}\left(0\right)}\left(T_{i}^{T}x\right)\neq0$
for some $i\in I_{0}\subset I^{\left(0\right)}$, we get $T_{i}^{T}x\in B_{\delta}\left(0\right)$,
i.e.\@ $x\in T_{i}^{-T}\left(B_{\delta}\left(0\right)\right)$ and
hence $\left|x\right|\leq\delta\left\Vert T_{i}^{-T}\right\Vert =\delta\left\Vert T_{i}^{-1}\right\Vert \leq M\delta$.
Thus, we get
\begin{align*}
\left\Vert c\right\Vert _{Y}^{q} & \gtrsim\int_{B_{M\delta}\left(0\right)}\left(\sum_{i\in I^{\left(\alpha\right)}\cap I_{0}}\left|c_{i}\left|\det T_{i}\right|^{\frac{1}{p}}\left(\left|b_{i}\right|^{k}+\left\Vert T_{i}\right\Vert ^{k}\right)\cdot\chi_{B_{\delta}\left(0\right)}\left(T_{i}^{T}x\right)\right|^{2}\right)^{q/2}\,{\rm d}x\\
\left(\text{by Jensen's ineq., since }\frac{q}{2}\geq1\right) & \gtrsim\left[\int_{B_{M\delta}\left(0\right)}\sum_{i\in I^{\left(\alpha\right)}\cap I_{0}}\left|c_{i}\left|\det T_{i}\right|^{\frac{1}{p}}\left(\left|b_{i}\right|^{k}+\left\Vert T_{i}\right\Vert ^{k}\right)\cdot\chi_{B_{\delta}\left(0\right)}\left(T_{i}^{T}x\right)\right|^{2}\,{\rm d}x\right]^{q/2}\\
 & \overset{\left(\dagger\right)}{=}\left[\sum_{i\in I^{\left(\alpha\right)}\cap I_{0}}\left|c_{i}\left|\det T_{i}\right|^{\frac{1}{p}}\left(\left|b_{i}\right|^{k}+\left\Vert T_{i}\right\Vert ^{k}\right)\right|^{2}\cdot\int_{\mathbb{R}^{d}}\left|\chi_{B_{\delta}\left(0\right)}\left(T_{i}^{T}x\right)\right|^{2}\,{\rm d}x\right]^{q/2}\\
\left(\text{since }\chi_{B_{\delta}\left(0\right)}\left(T_{i}^{T}\cdot\right)=\chi_{T_{i}^{-T}\left(B_{\delta}\left(0\right)\right)}\right) & =\left[\sum_{i\in I^{\left(\alpha\right)}\cap I_{0}}\left|c_{i}\left|\det T_{i}\right|^{\frac{1}{p}}\left(\left|b_{i}\right|^{k}+\left\Vert T_{i}\right\Vert ^{k}\right)\right|^{2}\cdot\left|\det T_{i}^{-T}\right|\cdot\lambda\left(B_{\delta}\left(0\right)\right)\right]^{q/2}\\
 & \asymp\left[\sum_{i\in I^{\left(\alpha\right)}\cap I_{0}}\left|c_{i}\left|\det T_{i}\right|^{\frac{1}{p}-\frac{1}{2}}\left(\left|b_{i}\right|^{k}+\left\Vert T_{i}\right\Vert ^{k}\right)\right|^{2}\right]^{q/2}\\
 & =\left\Vert \left(c_{i}\cdot u_{i}^{\left(k,p,2\right)}\right)_{i\in I^{\left(\alpha\right)}}\right\Vert _{\ell^{2}}^{q}.
\end{align*}
Now, the same steps as in the proof of the first part of the theorem
show $Y\cap\ell_{0}\left(I^{\left(0\right)}\right)\hookrightarrow\ell_{u^{\left(k,p,2\right)}}^{2}\left(I\right)$,
as desired. The step marked with $\left(\dagger\right)$ in the estimate
above used exactly the same argument as before the displayed equation
to justify changing the domain of integration from $B_{M\delta}\left(0\right)$
to all of $\mathbb{R}^{d}$.

Now, let us prove the second embedding in the last part of the theorem.
To this end, recall from Lemma \ref{lem:ShiftedBallsFitIntoSets}
that there is some $\delta>0$ and for each $i\in I^{\left(0\right)}$
some $y_{i}\in\mathbb{R}^{d}$ with $\left|y_{i}\right|\asymp\left|b_{i}\right|+\left\Vert T_{i}\right\Vert $
and with $B_{\delta}\left(y_{i}\right)\subset Q_{i}$. Choose $j^{\left(i\right)}\in\underline{d}$
with 
\[
L_{i}:=\left|\left(y_{i}\right)_{j^{\left(i\right)}}\right|=\max_{j\in\underline{d}}\left|\left(y_{i}\right)_{j}\right|\geq\frac{\left|y_{i}\right|}{d}\asymp\left|b_{i}\right|+\left\Vert T_{i}\right\Vert 
\]
and define $\alpha^{\left(i\right)}:=k\cdot e_{j^{\left(i\right)}}\in\mathbb{N}_{0}^{d}$
with $\left|\alpha^{\left(i\right)}\right|=k$. To make the above
estimate precise, let $L>0$ satisy $L_{i}\geq L\cdot\left(\left|b_{i}\right|+\left\Vert T_{i}\right\Vert \right)$
for all $i\in I^{\left(0\right)}$. Note that we have $1=\left\Vert T_{i}^{-1}T_{i}\right\Vert \leq M\cdot\left\Vert T_{i}\right\Vert $
and thus 
\begin{equation}
L_{i}\geq L\cdot\left(\left|b_{i}\right|+\left\Vert T_{i}\right\Vert \right)\geq L\cdot\left\Vert T_{i}\right\Vert \geq\frac{L}{M}\qquad\forall i\in I^{\left(0\right)}.\label{eq:KhinchinWithoutWeightLiDefinition}
\end{equation}

Now, for $\alpha\in\mathbb{N}_{0}^{d}$ with $\left|\alpha\right|=k$,
let 
\[
I_{\alpha}:=\left\{ i\in I^{\left(0\right)}\with\alpha^{\left(i\right)}=\alpha\right\} .
\]
Fix $\alpha\in\mathbb{N}_{0}^{d}$ with $\left|\alpha\right|=k$.
By Lemma \ref{lem:RichnessOfDerivativeOfFunctionsWithFourierSupport},
there is a function $\varphi\in C_{c}^{\infty}\left(B_{\delta}\left(0\right)\right)$
with
\begin{equation}
\left(\partial^{\beta}\psi\right)\left(0\right)=\delta_{0,\beta}\qquad\text{ for all }\beta\in\mathbb{N}_{0}^{d}\text{ with }\left|\beta\right|\leq k,\text{ where }\psi:=\mathcal{F}^{-1}\varphi.\label{eq:KhinchinSpecialTestFunctionConstruction}
\end{equation}
Note $L_{y_{i}}\varphi\in C_{c}^{\infty}\left(B_{\delta}\left(y_{i}\right)\right)\subset C_{c}^{\infty}\left(Q_{i}\right)\subset C_{c}^{\infty}\left(\mathcal{O}\right)$
for each $i\in I^{\left(0\right)}$. Similar to the construction above,
let $c=\left(c_{i}\right)_{i\in I}\in Y\cap\ell_{0}\left(I^{\left(0\right)}\right)$
and $\varepsilon=\left(\varepsilon_{i}\right)_{i\in I}\in\left\{ \pm1\right\} ^{I}$
be arbitrary and define
\begin{align*}
g_{\varepsilon,\alpha} & :=\sum_{i\in I_{\alpha}}\left(\varepsilon_{i}c_{i}\cdot L_{y_{i}}\varphi\right)\in C_{c}^{\infty}\left(\mathcal{O}\right),\\
f_{\varepsilon,\alpha} & :=\mathcal{F}^{-1}g_{\varepsilon,\alpha}=\sum_{i\in I_{\alpha}}\left(\varepsilon_{i}c_{i}\cdot M_{y_{i}}\psi\right)\in\mathcal{S}_{\mathcal{O}}\left(\mathbb{R}^{d}\right).
\end{align*}

Using the same arguments as in equations (\ref{eq:TestFunctionLocalization})
and (\ref{eq:TestFunctionLocalizedNormEstimate}) above, we see
\begin{align*}
\left\Vert \mathcal{F}^{-1}\left(\varphi_{\ell}\cdot g_{\varepsilon,\alpha}\right)\right\Vert _{L^{p}} & \lesssim\sum_{i\in I_{\alpha}\cap\ell^{\ast}}\left|c_{i}\right|\cdot\left\Vert \mathcal{F}^{-1}\left(\varphi_{\ell}\cdot L_{y_{i}}\varphi\right)\right\Vert _{L^{p}}\\
 & \lesssim\sum_{i\in I_{\alpha}\cap\ell^{\ast}}\left|c_{i}\right|\cdot\left\Vert \mathcal{F}^{-1}\left(L_{y_{i}}\varphi\right)\right\Vert _{L^{p}}\\
 & \leq\left\Vert \mathcal{F}^{-1}\varphi\right\Vert _{L^{p}}\cdot\sum_{i\in\ell^{\ast}}\left|c_{i}\right|\\
 & \lesssim\left(\left|c\right|^{\ast}\right)_{\ell}
\end{align*}
for all $\ell\in I$ and thus $f_{\varepsilon,\alpha}\in\mathcal{S}_{\mathcal{O}}^{p,Y}\left(\mathbb{R}^{d}\right)$
with $\left\Vert f_{\varepsilon,\alpha}\right\Vert _{\mathcal{D}\left(\mathcal{Q},L^{p},Y\right)}\lesssim\left\Vert c\right\Vert _{Y}$.

Now, by continuity and by equation (\ref{eq:KhinchinSpecialTestFunctionConstruction}),
there is some $\delta>0$ with
\[
\left|\psi\left(x\right)\right|\geq\frac{1}{2}\text{ and }\left|\partial^{\beta}\psi\left(x\right)\right|\leq\frac{1}{4}\cdot\left(\frac{M}{L}+2\pi\right)^{-k}\qquad\forall\left|x\right|\leq\delta\,\forall\beta\in\mathbb{N}_{0}^{d}\text{ with }0<\left|\beta\right|\leq k.
\]
Using the Leibniz rule, this implies for $\left|x\right|\leq\delta$
that
\begin{align*}
\left|\left[\partial^{\alpha^{\left(i\right)}}\left(M_{y_{i}}\psi\right)\right]\left(x\right)\right| & =\left|\sum_{\beta\leq\alpha^{\left(i\right)}}\binom{\alpha^{\left(i\right)}}{\beta}\cdot\left(2\pi jy_{i}\right)^{\alpha^{\left(i\right)}-\beta}e^{2\pi j\left\langle y_{i},x\right\rangle }\cdot\left(\partial^{\beta}\psi\right)\left(x\right)\right|\\
 & \geq\left|\left(2\pi y_{i}\right)^{\alpha^{\left(i\right)}}\psi\left(x\right)\right|-\sum_{0\neq\beta\leq\alpha^{\left(i\right)}}\left|\binom{\alpha^{\left(i\right)}}{\beta}\left(2\pi jy_{i}\right)^{\alpha^{\left(i\right)}-\beta}\cdot\left(\partial^{\beta}\psi\right)\left(x\right)\right|\\
\left(\alpha^{\left(i\right)}=k\cdot e_{j^{\left(i\right)}}\text{ and }L_{i}\geq\left|\left(y_{i}\right)_{j}\right|\text{ for all }j\right) & \geq\frac{\left(2\pi L_{i}\right)^{k}}{2}-\frac{1}{4}\left(\frac{M}{L}+2\pi\right)^{-k}\cdot\sum_{\beta\leq\alpha}\binom{\alpha^{\left(i\right)}}{\beta}\cdot\left(2\pi L_{i}\right)^{\alpha^{\left(i\right)}-\beta}\\
\left(\text{multinomial theorem}\right) & =\frac{\left(2\pi L_{i}\right)^{k}}{2}-\frac{1}{4}\left(\frac{M}{L}+2\pi\right)^{-k}\cdot\left(\left(1,\dots,1\right)+\left(2\pi L_{i},\dots,2\pi L_{i}\right)\right)^{\alpha^{\left(i\right)}}\\
 & =\frac{\left(2\pi L_{i}\right)^{k}}{2}-\frac{1}{4}\left(\frac{M}{L}+2\pi\right)^{-k}\cdot\left(1+2\pi L_{i}\right)^{k}\\
\left(\text{since }L_{i}\geq L/M\text{ by eq. }\eqref{eq:KhinchinWithoutWeightLiDefinition}\right) & \geq\left(2\pi L_{i}\right)^{k}\cdot\frac{1}{2}-\frac{1}{4}\cdot\left(\frac{M}{L}+2\pi\right)^{-k}\cdot\left(\left(\frac{M}{L}+2\pi\right)L_{i}\right)^{k}\\
 & \geq L_{i}^{k}/4.
\end{align*}

But because of $\alpha^{\left(i\right)}=\alpha$ for $i\in I_{\alpha}$,
we get (with another application of Khintchine's inequality) 
\begin{align*}
\mathbb{E}_{\varepsilon}\left\Vert \partial^{\alpha}f_{\varepsilon,\alpha}\right\Vert _{L^{q}}^{q} & =\int_{\mathbb{R}^{d}}\mathbb{E}_{\varepsilon}\left|\sum_{i\in I_{\alpha}}\varepsilon_{i}c_{i}\cdot\left[\partial^{\alpha^{\left(i\right)}}\left(M_{y_{i}}\psi\right)\right]\left(x\right)\right|^{q}\,{\rm d}x\\
 & \asymp\int_{\mathbb{R}^{d}}\left(\sum_{i\in I_{\alpha}}\left|c_{i}\cdot\left[\partial^{\alpha^{\left(i\right)}}\left(M_{y_{i}}\psi\right)\right]\left(x\right)\right|^{2}\right)^{q/2}\,{\rm d}x\\
 & \ge\int_{B_{\delta}\left(0\right)}\left(\sum_{i\in I_{\alpha}}\left|c_{i}\cdot\left[\partial^{\alpha^{\left(i\right)}}\left(M_{y_{i}}\psi\right)\right]\left(x\right)\right|^{2}\right)^{q/2}\,{\rm d}x\\
 & \geq\int_{B_{\delta}\left(0\right)}\left(\sum_{i\in I_{\alpha}}\left|c_{i}\cdot\frac{L_{i}^{k}}{4}\right|^{2}\right)^{q/2}\,{\rm d}x\\
 & \gtrsim\left(\sum_{i\in I_{\alpha}}\left|c_{i}\cdot\left(\left|b_{i}\right|+\left\Vert T_{i}\right\Vert \right)^{k}\right|^{2}\right)^{q/2}.
\end{align*}
All in all, we finally get
\[
\left\Vert \left(c_{i}\cdot u_{i}^{\left(k,p,p\right)}\right)_{i\in I^{\alpha}}\right\Vert _{\ell^{2}}^{q}\lesssim\mathbb{E}_{\varepsilon}\left\Vert \partial^{\alpha}f_{\varepsilon,\alpha}\right\Vert _{L^{q}}^{q}\leq\left\Vert \iota_{\alpha}\right\Vert ^{q}\cdot\mathbb{E}_{\varepsilon}\left\Vert f_{\varepsilon,\alpha}\right\Vert _{\mathcal{D}\left(\mathcal{Q},L^{p},Y\right)}^{q}\lesssim\left\Vert \iota_{\alpha}\right\Vert ^{q}\cdot\left\Vert c\right\Vert _{Y}^{q},
\]
which yields the desired embedding $Y\cap\ell_{0}\left(I^{\left(0\right)}\right)\hookrightarrow\ell_{u^{\left(k,p,p\right)}}^{2}\left(I\right)$
by summing over $\alpha\in\mathbb{N}_{0}^{d}$ with $\left|\alpha\right|=k$,
since $I^{\left(0\right)}=\bigcup_{\left|\alpha\right|=k}I_{\alpha}$.

This completes the proof of Theorem \ref{thm:SequenceSpaceEmbeddingIsNecessary},
with the caveat that we postponed the proof of Claim \ref{claim:NecessaryConditionCentralClaim},
which we now give.
\end{proof}

\begin{proof}[Proof of Claim \ref{claim:NecessaryConditionCentralClaim}]
Let $\varepsilon>0$, $N\in\mathbb{N}$ and $a_{1},\dots,a_{N}\in\mathbb{R}^{d}$
as in Lemma \ref{lem:FinitelyManyNormalizationsSuffice}, i.e.\@
such that for each $i\in I$, there is some $m^{\left(i\right)}\in\underline{N}$
with $B_{\varepsilon}\left(a_{m^{\left(i\right)}}\right)\subset Q_{i}'$.
Lemma \ref{lem:RichnessOfDerivativeOfFunctionsWithFourierSupport}
shows that for each $\alpha\in\mathbb{N}_{0}^{d}$ with $\left|\alpha\right|\leq k$
and each $m\in\underline{N}$, there is some function $f_{\alpha,m}\in C_{c}^{\infty}\left(B_{\varepsilon}\left(a_{m}\right)\right)$
with
\[
\left[\partial^{\beta}\left(\mathcal{F}^{-1}f_{\alpha,m}\right)\right]\left(0\right)=\delta_{\alpha,\beta}\text{ for all }\beta\in\mathbb{N}_{0}^{d}\text{ with }\left|\beta\right|\leq k.
\]
Let 
\[
\mathscr{F}:=\left\{ f_{\alpha,m}\with m\in\underline{N}\text{ and }\alpha\in\mathbb{N}_{0}^{d}\text{ with }\left|\alpha\right|\leq k\right\} .
\]
Note that $\mathscr{F}\subset C_{c}^{\infty}\left(\mathbb{R}^{d}\right)$
is indeed finite.

We let $\eta=\eta\left(d,k\right)>0$ be some unspecified constant,
the precise value of which we will choose below%
\footnote{We will see that the choice $\eta=\eta\left(d,k\right):=\frac{1}{4\left(d^{k}+2^{k}d^{3k}\right)}$
is suitable.%
}. Now, by continuity, there is some $\delta>0$ such that
\begin{equation}
\left|\left[\partial^{\alpha}\left(\mathcal{F}^{-1}f_{\alpha,m}\right)\right]\left(x\right)\right|>\frac{1}{2}\quad\text{ and }\quad\left|\left[\partial^{\beta}\left(\mathcal{F}^{-1}f_{\alpha,m}\right)\right]\left(x\right)\right|<\eta\text{ for all }\beta\in\mathbb{N}_{0}^{d}\setminus\left\{ \alpha\right\} \text{ with }\left|\beta\right|\leq k\label{eq:MatrixDominatedTestFunctionConditionsOnDelta}
\end{equation}
holds for all $x\in\mathbb{R}^{d}$ with $\left|x\right|<\delta$
and all $\alpha\in\mathbb{N}_{0}^{d}$ with $\left|\alpha\right|\leq k$
and $m\in\underline{N}$.

Now, for $i\in I$ and $\left|\alpha\right|\leq k$, define $\varrho_{i,\alpha}:=\mathcal{F}^{-1}f_{\alpha,m^{\left(i\right)}}$,
as well as 
\[
\gamma_{i,\alpha}:\mathbb{R}^{d}\to\mathbb{C},\xi\mapsto f_{\alpha,m^{\left(i\right)}}\left(T_{i}^{-1}\left(\xi-b_{i}\right)\right)
\]
and $\theta_{i,\alpha}:=\mathcal{F}^{-1}\gamma_{i,\alpha}$. Precisely
the same calculation as in equation (\ref{eq:OriginalBAPUIntoNormalizedVersionSpaceSide})
yields
\begin{equation}
\theta_{i,\alpha}=\mathcal{F}^{-1}\gamma_{i,\alpha}=\left|\det T_{i}\right|\cdot M_{b_{i}}\left(\varrho_{i,\alpha}\circ T_{i}^{T}\right).\label{eq:MatrixDominatedTestFunctionNormalizationSpaceSide}
\end{equation}

We first indicate the construction in case of $i\in I_{1}$. Note
that $i\in I_{1}$ implies $k>0$ and $\left\Vert T_{i}\right\Vert >2\pi\left|b_{i}\right|$.
Now, for each $i\in I_{1}$, choose $k^{\left(i\right)},\ell^{\left(i\right)}\in\underline{d}$
satisfying 
\begin{equation}
K_{i}:=\left|\left(T_{i}\right)_{k^{\left(i\right)},\ell^{\left(i\right)}}\right|=\max_{k,\ell\in\underline{d}}\left|\left(T_{i}\right)_{k,\ell}\right|\geq\frac{\left\Vert T_{i}\right\Vert }{d^{2}}.\label{eq:MatrixDominatedTestFunctionKiDefinition}
\end{equation}
Here, the final (suboptimal) estimate is an easy consequence of the
definitions. Define $\alpha^{\left(i\right)}:=k\cdot e_{k^{\left(i\right)}}\in\mathbb{N}_{0}^{d}$
and $\beta^{\left(i\right)}:=k\cdot e_{\ell^{\left(i\right)}}$, where
$\left(e_{1},\dots,e_{d}\right)$ denotes the standard basis of $\mathbb{R}^{d}$.
Note $\left|\alpha^{\left(i\right)}\right|=k=\left|\beta^{\left(i\right)}\right|$.
Finally, set $\gamma_{i}:=\gamma_{i,\beta^{\left(i\right)}}$ and
$\varrho_{i}:=\varrho_{i,\beta^{\left(i\right)}}$, as well as $\theta_{i}:=\theta_{i,\beta^{\left(i\right)}}$.
Note by definition of $\gamma_{i,\alpha}$ that $\gamma_{i}$ indeed
satisfies equation (\ref{eq:MatrixDominatedTestFunctionNormalization})
with $f_{i}=f_{\beta^{\left(i\right)},m^{\left(i\right)}}\in\mathscr{F}$.
Furthermore, note that we have 
\begin{align*}
{\rm supp}\,\gamma_{i} & ={\rm supp}\,\gamma_{i,\beta^{\left(i\right)}}\\
 & =T_{i}\left({\rm supp}\, f_{\beta^{\left(i\right)},m^{\left(i\right)}}\right)+b_{i}\\
 & \subset T_{i}\left(B_{\varepsilon}\left(a_{m^{\left(i\right)}}\right)\right)+b_{i}\\
 & \subset T_{i}Q_{i}'+b_{i}=Q_{i}
\end{align*}
and thus $\gamma_{i}\in C_{c}^{\infty}\left(Q_{i}\right)$ as desired.
It thus remains to establish estimate (\ref{eq:MatrixDominatedTestFunctionEstimateFromBelow}).

For $i\in I_{2}$, we proceed similarly. We choose some $j^{\left(i\right)}\in\underline{d}$
with 
\[
L_{i}:=\left|\left(b_{i}\right)_{j^{\left(i\right)}}\right|=\max_{j\in\underline{d}}\left|\left(b_{i}\right)_{j}\right|\geq\frac{\left|b_{i}\right|}{d}.
\]
Furthermore, we define $\alpha^{\left(i\right)}:=k\cdot e_{j^{\left(i\right)}}\in\mathbb{N}_{0}^{d}$
with $\left|\alpha^{\left(i\right)}\right|=k$ and we choose $\gamma_{i}:=\gamma_{i,0}$
and $\varrho_{i}:=\varrho_{i,0}$, as well as $\theta_{i}:=\theta_{i,0}$.
Precisely as above, we see that $\gamma_{i}$ indeed satisfies equation
(\ref{eq:MatrixDominatedTestFunctionNormalization}) with $f_{i}=f_{0,m^{\left(i\right)}}\in\mathscr{F}$
and that $\gamma_{i}\in C_{c}^{\infty}\left(Q_{i}\right)$. Thus,
for $i\in I_{2}$, it remains to establish estimate (\ref{eq:VectorDominatedTestFunctionEstimateFromBelow}).

We start by establishing equation (\ref{eq:MatrixDominatedTestFunctionEstimateFromBelow})
for $i\in I_{1}$. To this end, first note that this estimate is trivial
in case of $T_{i}^{T}x\notin B_{\delta}\left(0\right)$. Hence, we
can (and will) assume for the following considerations that $T_{i}^{T}x\in B_{\delta}\left(0\right)$.
Next, note as above that $i\in I_{1}$ implies $k>0$ and hence $\alpha^{\left(i\right)}\neq0$,
so that Lemma \ref{lem:ChainRuleForLinearTransformations} (with $i_{1}=\dots=i_{k}=k^{\left(i\right)}$)
is applicable and yields
\begin{align}
 & \hphantom{=\,\,}\left|\left(\partial^{\alpha^{\left(i\right)}}\left[\varrho_{i}\circ T_{i}^{T}\right]\right)\left(x\right)\right|\nonumber \\
 & =\left|\sum_{\ell_{1},\dots,\ell_{k}\in\underline{d}}\left(T_{i}\right)_{k^{\left(i\right)},\ell_{1}}\cdots\left(T_{i}\right)_{k^{\left(i\right)},\ell_{k}}\cdot\left(\partial_{\ell_{1}}\cdots\partial_{\ell_{k}}\varrho_{i}\right)\left(T_{i}^{T}x\right)\right|\nonumber \\
 & \geq\left|\left(T_{i}\right)_{k^{\left(i\right)},\ell^{\left(i\right)}}^{k}\cdot\left(\partial_{\ell^{\left(i\right)}}^{k}\varrho_{i}\right)\left(T_{i}^{T}x\right)\right|-\sum_{\substack{\ell_{1},\dots,\ell_{k}\in\underline{d}\\
\left(\ell_{1},\dots,\ell_{k}\right)\neq\left(\ell^{\left(i\right)},\dots,\ell^{\left(i\right)}\right)
}
}\left(\left|\left(T_{i}\right)_{k^{\left(i\right)},\ell_{1}}\cdots\left(T_{i}\right)_{k^{\left(i\right)},\ell_{k}}\right|\cdot\left|\left(\partial_{\ell_{1}}\cdots\partial_{\ell_{k}}\varrho_{i}\right)\left(T_{i}^{T}x\right)\right|\right)\nonumber \\
 & \overset{\left(\ast\right)}{\geq}K_{i}^{k}\cdot\left[\left|\left(\partial^{\beta^{\left(i\right)}}\varrho_{i,\beta^{\left(i\right)}}\right)\left(T_{i}^{T}x\right)\right|-\sum_{\substack{\ell_{1},\dots,\ell_{k}\in\underline{d}\\
\left(\ell_{1},\dots,\ell_{k}\right)\neq\left(\ell^{\left(i\right)},\dots,\ell^{\left(i\right)}\right)
}
}\left|\left(\partial_{\ell_{1}}\cdots\partial_{\ell_{k}}\varrho_{i,\beta^{\left(i\right)}}\right)\left(T_{i}^{T}x\right)\right|\right]\nonumber \\
 & \overset{\left(\dagger\right)}{\geq}K_{i}^{k}\left(\frac{1}{2}-\eta\cdot d^{k}\right).\label{eq:MatrixDominatedTestFunctionMainTermFromBelow}
\end{align}
Here, the step marked with $\left(\ast\right)$ used the definition
of $K_{i}$ (cf.\@ equation (\ref{eq:MatrixDominatedTestFunctionKiDefinition}))
and of $\beta^{\left(i\right)}=k\cdot e_{\ell^{\left(i\right)}}$,
as well as of $\varrho_{i}=\varrho_{i,\beta^{\left(i\right)}}$. Finally,
the step marked with $\left(\dagger\right)$ made use $T_{i}^{T}x\in B_{\delta}\left(0\right)$,
of estimate (\ref{eq:MatrixDominatedTestFunctionConditionsOnDelta})
and of $\partial_{\ell_{1}}\cdots\partial_{\ell_{k}}\neq\partial^{\beta^{\left(i\right)}}$
for those indices $\left(\ell_{1},\dots,\ell_{k}\right)$ over which
the sum is taken. Furthermore, this step also used that the sum has
less than $d^{k}$ terms. The estimate given in equation (\ref{eq:MatrixDominatedTestFunctionMainTermFromBelow})
is nontrivial as soon as $\eta=\eta\left(d,k\right)$ satisfies $\eta<\frac{1}{2}d^{-k}$.
We will assume this for the rest of the proof.

Now, we note that equation (\ref{eq:MatrixDominatedTestFunctionNormalizationSpaceSide})
and an application of Leibniz's rule yield -- as in equation (\ref{eq:BAPUDerivativeLpEstimateLeibnizApplication})
-- that
\begin{align}
\left[\partial^{\alpha^{\left(i\right)}}\left(\mathcal{F}^{-1}\gamma_{i}\right)\right]\left(x\right) & =\left(\partial^{\alpha^{\left(i\right)}}\theta_{i}\right)\left(x\right)\nonumber \\
 & =\left|\det T_{i}\right|\cdot\sum_{\beta\leq\alpha^{\left(i\right)}}\left[\binom{\alpha^{\left(i\right)}}{\beta}\cdot\left(2\pi j\cdot b_{i}\right)^{\alpha^{\left(i\right)}-\beta}\cdot e^{2\pi j\left\langle b_{i},x\right\rangle }\cdot\left(\partial^{\beta}\left[\varrho_{i}\circ T_{i}^{T}\right]\right)\left(x\right)\right],\label{eq:MatrixDominatedTestFunctionLeibnizApplication}
\end{align}
where we have written $j$ for the imaginary unit to avoid confusion
with the index $i\in I_{1}$. We will keep this convention for the
remainder of the proof. In view of this identity and of equation (\ref{eq:MatrixDominatedTestFunctionMainTermFromBelow}),
our next step is to estimate for $\beta\leq\alpha^{\left(i\right)}$
with $\alpha^{\left(i\right)}\neq\beta$ the term
\begin{align}
 & \hphantom{\leq\,\,}\left|\left(2\pi j\cdot b_{i}\right)^{\alpha^{\left(i\right)}-\beta}\cdot e^{2\pi j\left\langle b_{i},x\right\rangle }\cdot\left(\partial^{\beta}\left[\varrho_{i}\circ T_{i}^{T}\right]\right)\left(x\right)\right|\nonumber \\
\left(2\pi\left|b_{i}\right|\leq\left\Vert T_{i}\right\Vert \leq d^{2}K_{i}\text{ and }\left|\smash{\alpha^{\left(i\right)}}\right|=k\right) & \leq d^{2k}\cdot K_{i}^{k-\left|\beta\right|}\cdot\left|\left(\partial^{\beta}\left[\varrho_{i}\circ T_{i}^{T}\right]\right)\left(x\right)\right|\nonumber \\
\left(\text{Lemma }\ref{lem:ChainRuleForLinearTransformations}\text{ with }\beta=\sum_{m=1}^{\left|\beta\right|}e_{i_{m}}\right) & \overset{\left(\dagger\right)}{\leq}d^{2k}K_{i}^{k-\left|\beta\right|}\sum_{\ell_{1},\dots,\ell_{\left|\beta\right|}\in\underline{d}}\left|\left(T_{i}\right)_{i_{1},\ell_{i}}\cdots\left(T_{i}\right)_{i_{\left|\beta\right|},\ell_{\left|\beta\right|}}\cdot\left(\partial_{\ell_{1}}\cdots\partial_{\ell_{\left|\beta\right|}}\varrho_{i}\right)\left(T_{i}^{T}x\right)\right|\nonumber \\
\left(\text{by definition of }K_{i}\right) & \leq d^{2k}\cdot K_{i}^{k}\cdot\sum_{\ell_{1},\dots,\ell_{\left|\beta\right|}\in\underline{d}}\left|\left(\partial_{\ell_{1}}\cdots\partial_{\ell_{\left|\beta\right|}}\varrho_{i,\beta^{\left(i\right)}}\right)\left(T_{i}^{T}x\right)\right|\nonumber \\
 & \overset{\left(\ast\right)}{\leq}d^{2k}\cdot K_{i}^{k}\cdot d^{\left|\beta\right|}\eta\nonumber \\
\left(\text{since }\left|\beta\right|\leq\left|\smash{\alpha^{\left(i\right)}}\right|=k\right) & \leq K_{i}^{k}\cdot d^{3k}\eta.\label{eq:MatrixDominatedTestFunctionSideTerms}
\end{align}
Here, the step marked with $\left(\ast\right)$ used that $\beta\leq\alpha^{\left(i\right)}$
together with $\beta\neq\alpha^{\left(i\right)}$ implies $\left|\beta\right|<\left|\alpha^{\left(i\right)}\right|=\left|\beta^{\left(i\right)}\right|$
and hence $\partial_{\ell_{1}}\cdots\partial_{\ell_{\left|\beta\right|}}\neq\partial^{\beta^{\left(i\right)}}$
for all $\ell_{1},\dots,\ell_{\left|\beta\right|}\in\underline{d}$,
so that equation (\ref{eq:MatrixDominatedTestFunctionConditionsOnDelta})
-- together with $T_{i}^{T}x\in B_{\delta}\left(0\right)$ -- justifies
the marked step. Note though, that the step marked with $\left(\dagger\right)$
is only justified for $\left|\beta\right|\in\mathbb{N}$, i.e.\@
for $\beta\neq0$ (cf.\@ Lemma \ref{lem:ChainRuleForLinearTransformations}).
But in case of $\beta=0$, we simply have
\begin{align*}
 & \hphantom{\leq\,\,}\left|\left(2\pi j\cdot b_{i}\right)^{\alpha^{\left(i\right)}-\beta}\cdot e^{2\pi j\left\langle b_{i},x\right\rangle }\cdot\left(\partial^{\beta}\left[\varrho_{i}\circ T_{i}^{T}\right]\right)\left(x\right)\right|\\
\left(\text{since }\varrho_{i}=\varrho_{i,\beta^{\left(i\right)}}=\mathcal{F}^{-1}f_{\beta^{\left(i\right)},m^{\left(i\right)}}\right) & =\left|\left(2\pi j\cdot b_{i}\right)^{\alpha^{\left(i\right)}}\cdot\left(\mathcal{F}^{-1}f_{\beta^{\left(i\right)},m^{\left(i\right)}}\right)\left(T_{i}^{T}x\right)\right|\\
\left(2\pi\left|b_{i}\right|\leq\left\Vert T_{i}\right\Vert \leq d^{2}K_{i}\text{ and }\left|\smash{\alpha^{\left(i\right)}}\right|=k\right) & \leq d^{2k}\cdot K_{i}^{k}\cdot\left|\left(\mathcal{F}^{-1}f_{\beta^{\left(i\right)},m^{\left(i\right)}}\right)\left(T_{i}^{T}x\right)\right|\\
\left(\text{by eq. }\eqref{eq:MatrixDominatedTestFunctionConditionsOnDelta}\text{ since }\left|\smash{\beta^{\left(i\right)}}\right|=k>0\right) & \overset{\left(\ast\right)}{\leq}d^{2k}\cdot K_{i}^{k}\cdot\eta\\
 & \leq K_{i}^{k}\cdot d^{3k}\eta.
\end{align*}
Thus, equation (\ref{eq:MatrixDominatedTestFunctionSideTerms}) also
holds for $\beta\neq0$. At $\left(\ast\right)$, we used that $k>0$,
since $i\in I_{1}$.

Altogether, equation (\ref{eq:MatrixDominatedTestFunctionLeibnizApplication})
and estimates (\ref{eq:MatrixDominatedTestFunctionMainTermFromBelow})
and (\ref{eq:MatrixDominatedTestFunctionSideTerms}) yield
\begin{align*}
\left|\left[\partial^{\alpha^{\left(i\right)}}\left(\mathcal{F}^{-1}\gamma_{i}\right)\right]\left(x\right)\right| & \geq\left|\det T_{i}\right|\cdot\left[K_{i}^{k}\left(\frac{1}{2}-d^{k}\eta\right)-\sum_{\substack{\beta\leq\alpha^{\left(i\right)}\\
\beta\neq\alpha^{\left(i\right)}
}
}\binom{\alpha^{\left(i\right)}}{\beta}K_{i}^{k}d^{3k}\eta\right]\\
 & \geq\left|\det T_{i}\right|K_{i}^{k}\cdot\left[\frac{1}{2}-\eta\left(d^{k}+d^{3k}\sum_{\beta\leq\alpha^{\left(i\right)}}\binom{\alpha^{\left(i\right)}}{\beta}\right)\right]\\
\left(\text{by the multi-binomial theorem}\right) & =\left|\det T_{i}\right|K_{i}^{k}\cdot\left[\frac{1}{2}-\eta\left(d^{k}+d^{3k}\left(\left(1,\dots,1\right)+\left(1,\dots,1\right)\right)^{\alpha^{\left(i\right)}}\right)\right]\\
 & =\left|\det T_{i}\right|K_{i}^{k}\cdot\left[\frac{1}{2}-\eta\left(d^{k}+d^{3k}2^{\left|\alpha^{\left(i\right)}\right|}\right)\right]\\
\left(\text{since }\left|\smash{\alpha^{\left(i\right)}}\right|=k\right) & =\left|\det T_{i}\right|K_{i}^{k}\cdot\left[\frac{1}{2}-\eta\left(d^{k}+d^{3k}2^{k}\right)\right].
\end{align*}
All in all, we see that if we choose%
\footnote{Note that this choice of $\eta$ also makes the assumption $\eta<\frac{1}{2}d^{-k}$
from above true.%
} $\eta=\eta\left(d,k\right):=\frac{1}{4\left(d^{k}+2^{k}d^{3k}\right)}$,
then
\[
\left|\left[\partial^{\alpha^{\left(i\right)}}\left(\mathcal{F}^{-1}\gamma_{i}\right)\right]\left(x\right)\right|\geq\frac{\left|\det T_{i}\right|\cdot K_{i}^{k}}{4}\geq\frac{\left|\det T_{i}\right|\cdot\left(\left\Vert T_{i}\right\Vert /d^{2}\right)^{k}}{4}=\left(4d^{2k}\right)^{-1}\cdot\left|\det T_{i}\right|\cdot\left\Vert T_{i}\right\Vert ^{k}
\]
holds for all $x\in\mathbb{R}^{d}$ with $T_{i}^{T}x\in B_{\delta}\left(0\right)$.
This finally establishes equation (\ref{eq:MatrixDominatedTestFunctionEstimateFromBelow})
with $C_{1}=\left(4d^{2k}\right)^{-1}$.

To complete the proof, we establish equation (\ref{eq:VectorDominatedTestFunctionEstimateFromBelow})
for $i\in I_{2}$. As above, we see that this estimate trivially holds
if $T_{i}^{T}x\notin B_{\delta}\left(0\right)$. Thus, in the following,
we will assume $T_{i}^{T}x\in B_{\delta}\left(0\right)$.

Now, let $\beta\in\mathbb{N}_{0}^{d}$ with $0\neq\beta\leq\alpha^{\left(i\right)}$
be arbitrary. Note that existence of such a multiindex $\beta$ implies
$0<\left|\beta\right|\leq\left|\alpha^{\left(i\right)}\right|=k$
and thus $2\pi\left|b_{i}\right|\geq\left\Vert T_{i}\right\Vert $,
since $i\in I_{2}$. Choose $i_{1},\dots,i_{\left|\beta\right|}\in\underline{d}$
with $\beta=\sum_{m=1}^{\left|\beta\right|}e_{i_{m}}$ and apply Lemma
\ref{lem:ChainRuleForLinearTransformations} to conclude
\begin{align}
\left|\left(\partial^{\beta}\left[\varrho_{i}\circ T_{i}^{T}\right]\right)\left(x\right)\right| & =\left|\sum_{\ell_{1},\dots,\ell_{\left|\beta\right|}\in\underline{d}}\left[\left(T_{i}\right)_{i_{1},\ell_{1}}\cdots\left(T_{i}\right)_{i_{\left|\beta\right|},\ell_{\left|\beta\right|}}\cdot\left(\partial_{\ell_{1}}\cdots\partial_{\ell_{\left|\beta\right|}}\varrho_{i}\right)\left(T_{i}^{T}x\right)\right]\right|\nonumber \\
\left(\text{since }\left|\smash{\left(T_{i}\right)_{j,\ell}}\right|\leq\left\Vert T_{i}\right\Vert \leq2\pi\left|b_{i}\right|\leq2\pi dL_{i}\right) & \leq\left(2\pi dL_{i}\right)^{\left|\beta\right|}\cdot\sum_{\ell_{1},\dots,\ell_{\left|\beta\right|}\in\underline{d}}\left|\left(\partial_{\ell_{1}}\cdots\partial_{\ell_{\left|\beta\right|}}\varrho_{i,0}\right)\left(T_{i}^{T}x\right)\right|\nonumber \\
 & \overset{\left(\ast\right)}{\leq}\left(2\pi dL_{i}\right)^{\left|\beta\right|}\cdot\sum_{\ell_{1},\dots,\ell_{\left|\beta\right|}\in\underline{d}}\eta\nonumber \\
 & =\left(2\pi d^{2}\cdot L_{i}\right)^{\left|\beta\right|}\cdot\eta.\label{eq:VectorDominatedTestFunctionSideTerms}
\end{align}
Here, the step marked with $\left(\ast\right)$ used equation (\ref{eq:MatrixDominatedTestFunctionConditionsOnDelta}),
which is applicable since we have $T_{i}^{T}x\in B_{\delta}\left(0\right)$
and $\varrho_{i,0}=\mathcal{F}^{-1}f_{0,m^{\left(i\right)}}$, as
well as $0<\left|\beta\right|\leq k$.

But for $\beta=0$, we have
\begin{align*}
\left|\left(\partial^{0}\left[\varrho_{i}\circ T_{i}^{T}\right]\right)\left(x\right)\right| & =\left|\varrho_{i,0}\left(T_{i}^{T}x\right)\right|\\
\left(\text{since }\varrho_{i,0}=\mathcal{F}^{-1}f_{0,m^{\left(i\right)}}\right) & =\left|\left(\mathcal{F}^{-1}f_{0,m^{\left(i\right)}}\right)\left(T_{i}^{T}x\right)\right|\\
\left(\text{by eq. }\eqref{eq:MatrixDominatedTestFunctionConditionsOnDelta}\text{ since }T_{i}^{T}x\in B_{\delta}\left(0\right)\right) & \geq\frac{1}{2}.
\end{align*}
Altogether, since equation (\ref{eq:MatrixDominatedTestFunctionLeibnizApplication})
also holds for $i\in I_{2}$, we conclude
\begin{align*}
 & \hphantom{\leq\,\,}\left|\left[\partial^{\alpha^{\left(i\right)}}\left(\mathcal{F}^{-1}\gamma_{i}\right)\right]\left(x\right)\right|\\
 & =\left|\left|\det T_{i}\right|\cdot\sum_{\beta\leq\alpha^{\left(i\right)}}\left[\binom{\alpha^{\left(i\right)}}{\beta}\cdot\left(2\pi j\cdot b_{i}\right)^{\alpha^{\left(i\right)}-\beta}\cdot e^{2\pi j\left\langle b_{i},x\right\rangle }\cdot\left(\partial^{\beta}\left[\varrho_{i}\circ T_{i}^{T}\right]\right)\left(x\right)\right]\right|\\
 & \geq\left|\det T_{i}\right|\!\cdot\!\left[\left|\left(2\pi jb_{i}\right)^{\alpha^{\left(i\right)}}\partial^{0}\!\left[\varrho_{i}\circ T_{i}^{T}\right]\!\left(x\right)\right|-\!\!\!\!\sum_{0\neq\beta\leq\alpha^{\left(i\right)}}\!\!\!\binom{\alpha^{\left(i\right)}}{\beta}\left|\left(2\pi jb_{i}\right)^{\alpha^{\left(i\right)}-\beta}\cdot\partial^{\beta}\!\left[\varrho_{i}\circ T_{i}^{T}\right]\!\left(x\right)\right|\right]\\
 & \overset{\left(\dagger\right)}{\geq}\left|\det T_{i}\right|\cdot\left[\left|2\pi\left(b_{i}\right)_{j^{\left(i\right)}}\right|^{k}\cdot\frac{1}{2}-\sum_{0\neq\beta\leq\alpha^{\left(i\right)}}\binom{\alpha^{\left(i\right)}}{\beta}\cdot\left(2\pi L_{i}\right)^{\left|\alpha^{\left(i\right)}\right|-\left|\beta\right|}\cdot\left(2\pi d^{2}\cdot L_{i}\right)^{\left|\beta\right|}\cdot\eta\right]\\
\left(\text{since }\left|\smash{\alpha^{\left(i\right)}}\right|=k\right) & \geq\left|\det T_{i}\right|\left(2\pi L_{i}\right)^{k}\cdot\left[\frac{1}{2}-\eta\cdot\sum_{\beta\leq\alpha^{\left(i\right)}}\binom{\alpha^{\left(i\right)}}{\beta}\left(d^{2}\right)^{\left|\beta\right|}\right]\\
\left(\left|\beta\right|\leq\left|\smash{\alpha^{\left(i\right)}}\right|=k\right) & \geq\left|\det T_{i}\right|\left(2\pi L_{i}\right)^{k}\cdot\left[\frac{1}{2}-\eta\cdot d^{2k}\sum_{\beta\leq\alpha^{\left(i\right)}}\binom{\alpha^{\left(i\right)}}{\beta}\right]\\
\left(\text{multi-binomial th.}\right) & =\left|\det T_{i}\right|\left(2\pi L_{i}\right)^{k}\cdot\left[\frac{1}{2}-\eta\cdot d^{2k}\left(\left(1,\dots,1\right)+\left(1,\dots,1\right)\right)^{\alpha^{\left(i\right)}}\right]\\
\left(\text{since }\left|\smash{\alpha^{\left(i\right)}}\right|=k\right) & =\left|\det T_{i}\right|\left(2\pi L_{i}\right)^{k}\cdot\left[\frac{1}{2}-\eta\cdot d^{2k}2^{k}\right].
\end{align*}
Here, the step marked with $\left(\dagger\right)$ is justified by
equation (\ref{eq:VectorDominatedTestFunctionSideTerms}) and because
of $\left|\smash{\left(b_{i}\right)_{j}}\right|\leq L_{i}$ for all
$j\in\underline{d}$.

Now, recall 
\[
\eta=\eta\left(d,k\right)=\frac{1}{4\left(d^{k}+2^{k}d^{3k}\right)}<\frac{1}{4\cdot2^{k}d^{3k}}\leq\frac{1}{4\cdot2^{k}d^{2k}}
\]
and $L_{i}\geq\frac{\left|b_{i}\right|}{d}$, so that we finally get
\begin{align*}
\left|\left[\partial^{\alpha^{\left(i\right)}}\left(\mathcal{F}^{-1}\gamma_{i}\right)\right]\left(x\right)\right| & \geq\left|\det T_{i}\right|\left(2\pi L_{i}\right)^{k}\cdot\left[\frac{1}{2}-\eta\cdot d^{2k}2^{k}\right]\\
 & \geq\frac{1}{4}\left|\det T_{i}\right|\left(2\pi L_{i}\right)^{k}\\
 & \geq\frac{1}{4d^{k}}\cdot\left|\det T_{i}\right|\cdot\left|b_{i}\right|^{k}
\end{align*}
for all $x\in\mathbb{R}^{d}$ with $T_{i}^{T}x\in B_{\delta}\left(0\right)$.
As observed above, this establishes equation (\ref{eq:VectorDominatedTestFunctionEstimateFromBelow})
with $C_{2}=\frac{1}{4d^{k}}$.
\end{proof}
As a corollary, we get the following -- somewhat surprising -- result.
\begin{cor}
Let $\mathcal{Q}=\left(Q_{i}\right)_{i\in I}=\left(T_{i}Q_{i}'+b_{i}\right)_{i\in I}$
be a tight regular covering of the open set $\emptyset\neq\mathcal{O}\subset\mathbb{R}^{d}$
such that $\sup_{i\in I}\left\Vert T_{i}^{-1}\right\Vert <\infty$
.

Let $\ell\in\mathbb{N}_{0}$ and $p\in\left(0,\infty\right]$ and
let $Y\leq\mathbb{C}^{I}$ be a $\mathcal{Q}$-regular sequence space
for which $Y\cap\ell_{0}\left(I\right)\subset Y$ is dense.

If $\mathcal{D}\left(\mathcal{Q},L^{p},Y\right)\hookrightarrow W^{\ell,q}\left(\mathbb{R}^{d}\right)$
for some $q\in\left(2,\infty\right)$ and if $p\leq2$, then $\mathcal{D}\left(\mathcal{Q},L^{p},Y\right)\overset{{\rm inj.}}{\hookrightarrow}W^{\ell,2}\left(\mathbb{R}^{d}\right)$.\end{cor}
\begin{rem*}
This result is slightly surprising, since there is no embedding $W^{\ell,q}\left(\mathbb{R}^{d}\right)\hookrightarrow W^{\ell,2}\left(\mathbb{R}^{d}\right)$
for $q\neq2$.\end{rem*}
\begin{proof}
The assumptions easily imply that the prerequisites of the last part
of Theorem \ref{thm:SequenceSpaceEmbeddingIsNecessary} are satisfied
for each $0\leq k\leq\ell$ with $I^{\left(0\right)}=I$, so that
we get
\[
Y\cap\ell_{0}\left(I\right)\hookrightarrow\ell_{u^{\left(k,p,2\right)}}^{2}\left(I\right)
\]
for all $0\leq k\leq\ell$.

Since $Y\cap\ell_{0}\left(I\right)\subset Y$ is dense and since all
relevant spaces are sequence spaces (i.e.\@ they embed continuously
into the Hausdorff space $\mathbb{C}^{I}$, equipped with the product
topology), this implies 
\[
Y\hookrightarrow\ell_{u^{\left(k,p,2\right)}}^{2}\left(I\right)=\ell_{u^{\left(k,p,2\right)}}^{2^{\triangledown}}\left(I\right)
\]
for $0\leq k\leq\ell$. Since we also have $p\leq2$, Corollary \ref{cor:SufficientConditionsForSobolevEmbeddings}
yields $\mathcal{D}\left(\mathcal{Q},L^{p},Y\right)\overset{{\rm inj}.}{\hookrightarrow}W^{\ell,2}\left(\mathbb{R}^{d}\right)$.
\end{proof}

\section{Characterization of the embedding $\ell_{v}^{r}\left(I\right)\hookrightarrow\ell_{w}^{s}\left(I\right)$}

\label{sec:SimplifiedConditions}As we saw in Sections \ref{sec:SufficientConditions}
and \ref{sec:NecessaryConditions}, to decide existence of the embedding
\[
\mathcal{D}\left(\mathcal{Q},L^{p},Y\right)\hookrightarrow W^{k,q}\left(\mathbb{R}^{d}\right),
\]
we have to decide whether an embedding of the form $Y\hookrightarrow\ell_{u}^{r}\left(I\right)$
for certain $r\in\left(0,\infty\right]$ and a certain weight $u=\left(u_{i}\right)_{i\in I}$
is valid. In the present section, we will simplify this problem for
the case $Y=\ell_{v}^{s}\left(I\right)$ to the point where only finiteness
of a single sequence space norm has to be decided.
\begin{lem}
\label{lem:SequenceSpaceEmbeddingCharacterization}Let $u=\left(u_{i}\right)_{i\in I}$
and $v=\left(v_{i}\right)_{i\in I}$ be positive weights on an index
set $I\neq\emptyset$ and let $p,q\in\left(0,\infty\right]$. Then,
the map
\[
\iota:\left(\ell_{0}\left(I\right),\left\Vert \cdot\right\Vert _{\ell_{v}^{p}}\right)\hookrightarrow\ell_{u}^{q}\left(I\right),\left(c_{i}\right)_{i\in I}\mapsto\left(c_{i}\right)_{i\in I}
\]
is bounded if and only if
\begin{equation}
\left(u_{i}/v_{i}\right)_{i\in I}\in\ell^{q\cdot\left(p/q\right)'}\left(I\right).\label{eq:SequenceSpacesEmbeddingConvenientCondition}
\end{equation}
Here, the exponent $q\cdot\left(p/q\right)'$ has to be calculated
according to the following convention%
\footnote{Apart from the convention noted here, we have $p/q=\infty$ if $p=\infty$
and $q<\infty$ and (as in the rest of the paper) $r'=\infty$ if
$r\leq1$. Finally, $\infty'=1$.%
}:
\[
q\cdot\left(p/q\right)'=\begin{cases}
\infty, & \text{if }q=\infty,\\
q\cdot\left(p/q\right)', & \text{if }q<\infty.
\end{cases}
\]

If equation (\ref{eq:SequenceSpacesEmbeddingConvenientCondition})
is fulfilled, we even have $\ell_{v}^{p}\left(I\right)\hookrightarrow\ell_{u}^{q}\left(I\right)$.\end{lem}
\begin{rem*}
In short, the last part of the lemma shows that it suffices to verify
an embedding on the (not necessarily dense) subspace $\ell_{v}^{p}\left(I\right)\cap\ell_{0}\left(I\right)$.
Thus, the restriction to finitely supported sequences in the necessary
conditions given in Theorem \ref{thm:SequenceSpaceEmbeddingIsNecessary}
is no essential restriction, at least if $Y$ is a weighted sequence
space.

Furthermore, note 
\begin{align}
q\cdot\left(p/q\right)'=\infty & \Longleftrightarrow q=\infty\text{ or }\left(q<\infty\text{ and }\left(p/q\right)'=\infty\right)\nonumber \\
 & \Longleftrightarrow q=\infty\text{ or }\left(q<\infty\text{ and }\frac{p}{q}\leq1\right)\nonumber \\
 & \Longleftrightarrow q=\infty\text{ or }\left(q<\infty\text{ and }p\leq q\right)\nonumber \\
 & \Longleftrightarrow p\leq q.\label{eq:SpecialExponentInfiniteCharacterization}
\end{align}
Finally, we have
\begin{equation}
\frac{1}{q\cdot\left(p/q\right)'}=\left(\frac{1}{q}-\frac{1}{p}\right)_{+},\label{eq:SpecialExponentReciprocal}
\end{equation}
where $x_{+}=\max\left\{ x,0\right\} $. Indeed, there are two cases:
\begin{casenv}
\item $\frac{1}{q}-\frac{1}{p}\leq0$, i.e.\@ $p\leq q$. As seen in equation
(\ref{eq:SpecialExponentInfiniteCharacterization}), this yields $q\cdot\left(p/q\right)'=\infty$
and hence
\[
\frac{1}{q\cdot\left(p/q\right)'}=0=\left(\frac{1}{q}-\frac{1}{p}\right)_{+}
\]
as claimed.
\item $\frac{1}{q}-\frac{1}{p}>0$, i.e.\@ $p>q$. By equation (\ref{eq:SpecialExponentInfiniteCharacterization})
again, this yields $q\cdot\left(p/q\right)'<\infty$ and in particular
$\left(p/q\right)'<\infty$. Hence,
\[
\frac{1}{q\cdot\left(p/q\right)'}=\frac{1}{q}\cdot\left(1-\frac{1}{p/q}\right)=\frac{1}{q}\cdot\left(1-\frac{q}{p}\right)=\frac{1}{q}-\frac{1}{p}=\left(\frac{1}{q}-\frac{1}{p}\right)_{+}.
\]

\end{casenv}
The two properties from equations (\ref{eq:SpecialExponentInfiniteCharacterization})
and (\ref{eq:SpecialExponentReciprocal}) will be used repeatedly
in Section \ref{sec:Applications} for concrete applications of our
embedding results.\end{rem*}
\begin{proof}[Proof of Lemma \ref{lem:SequenceSpaceEmbeddingCharacterization}]
We first establish the implication ``$\Leftarrow$''. To this end,
it suffices to show $\ell_{v}^{p}\left(I\right)\hookrightarrow\ell_{u}^{q}\left(I\right)$.
For the proof, we first make the following general observation: Hölder's
inequality 
\[
\left\Vert \left(x_{i}y_{i}\right)_{i\in I}\right\Vert _{\ell^{1}}=\sum_{i\in I}\left|x_{i}y_{i}\right|\leq\left\Vert \left(x_{i}\right)_{i\in I}\right\Vert _{\ell^{p}}\cdot\left\Vert \left(y_{i}\right)_{i\in I}\right\Vert _{\ell^{p'}}\quad,
\]
which is well known for $p\in\left[1,\infty\right]$ also holds for
$p\in\left(0,1\right)$, since in this case, we have the norm-decreasing
embedding $\ell^{p}\hookrightarrow\ell^{1}$ and hence
\[
\sum_{i\in I}\left|x_{i}y_{i}\right|\leq\left\Vert \left(x_{i}\right)_{i\in I}\right\Vert _{\ell^{1}}\cdot\left\Vert \left(y_{i}\right)_{i\in I}\right\Vert _{\ell^{\infty}}\leq\left\Vert \left(x_{i}\right)_{i\in I}\right\Vert _{\ell^{p}}\cdot\left\Vert \left(y_{i}\right)_{i\in I}\right\Vert _{\ell^{p'}}\quad.
\]

Thus, for $c=\left(c_{i}\right)_{i\in I}\in\ell_{v}^{p}\left(I\right)$
and $q<\infty$, we have
\begin{align*}
\left\Vert c\right\Vert _{\ell_{u}^{q}} & =\left\Vert \left(u_{i}\cdot c_{i}\right)_{i\in I}\right\Vert _{\ell^{q}}\\
 & =\left\Vert \left(\left(u_{i}\cdot c_{i}\right)^{q}\right)_{i\in I}\right\Vert _{\ell^{1}}^{1/q}\\
 & =\left\Vert \left(\left(v_{i}c_{i}\right)^{q}\cdot\left(\frac{u_{i}}{v_{i}}\right)^{q}\right)_{i\in I}\right\Vert _{\ell^{1}}^{1/q}\\
 & \leq\left[\left\Vert \left(\left(v_{i}c_{i}\right)^{q}\right)_{i\in I}\right\Vert _{\ell^{p/q}}\cdot\left\Vert \left(\left(u_{i}/v_{i}\right)^{q}\right)_{i\in I}\right\Vert _{\ell^{\left(p/q\right)'}}\right]^{1/q}\\
 & =\left\Vert \left(v_{i}c_{i}\right)_{i\in I}\right\Vert _{\ell^{p}}\cdot\left\Vert \left(u_{i}/v_{i}\right)_{i\in I}\right\Vert _{\ell^{q\cdot\left(p/q\right)'}}\\
 & =\left\Vert \left(u_{i}/v_{i}\right)_{i\in I}\right\Vert _{\ell^{q\cdot\left(p/q\right)'}}\cdot\left\Vert c\right\Vert _{\ell_{v}^{p}}<\infty.
\end{align*}
If otherwise $q=\infty$, we have $q\cdot\left(p/q\right)'=\infty$
and hence for each $i\in I$
\begin{align*}
u_{i}\cdot\left|c_{i}\right| & =\frac{u_{i}}{v_{i}}\cdot v_{i}\left|c_{i}\right|\\
 & \leq\left\Vert \left(u_{i}/v_{i}\right)_{i\in I}\right\Vert _{\ell^{q\cdot\left(p/q\right)'}}\cdot\left\Vert \left(v_{i}\left|c_{i}\right|\right)_{i\in I}\right\Vert _{\ell^{p}}\\
 & =\left\Vert \left(u_{i}/v_{i}\right)_{i\in I}\right\Vert _{\ell^{q\cdot\left(p/q\right)'}}\cdot\left\Vert c\right\Vert _{\ell_{v}^{p}}<\infty.
\end{align*}
Thus, we have shown $\ell_{v}^{p}\left(I\right)\hookrightarrow\ell_{u}^{q}\left(I\right)$
in all possible cases.

It remains to show ``$\Rightarrow$''. To this end, let us first
prove $u/v\in\ell^{\infty}\left(I\right)$ with $\left(u/v\right)_{i}=u_{i}/v_{i}$
for $i\in I$. Indeed, for arbitrary $i\in I$, we have
\[
u_{i}=\left\Vert \delta_{i}\right\Vert _{\ell_{u}^{q}}\leq\left\Vert \iota\right\Vert \cdot\left\Vert \delta_{i}\right\Vert _{\ell_{v}^{p}}=\left\Vert \iota\right\Vert \cdot v_{i},
\]
so that we get $0\leq u_{i}/v_{i}\leq\left\Vert \iota\right\Vert $
and hence $u/v\in\ell^{\infty}\left(I\right)$.

Thus, we can assume in the following that $\alpha:=q\cdot\left(p/q\right)'<\infty$
and thus $\beta:=\left(p/q\right)'<\infty$, which means $p/q>1$,
i.e.\@ $\infty\geq p>q$. Now, let $\emptyset\neq I_{0}\subset I$
be an arbitrary finite subset and define
\[
c_{i}:=\frac{\left(u_{i}/v_{i}\right)^{\beta}}{u_{i}}\qquad\text{ for }i\in I_{0}
\]
and $c_{i}=0$ for $i\in I\setminus I_{0}$. Note $c=\left(c_{i}\right)_{i\in I}\in\ell_{0}\left(I\right)$.
Because of $q<\infty$ and $\alpha=q\cdot\beta$, we have
\begin{align*}
\left\Vert \frac{u}{v}\cdot\chi_{I_{0}}\right\Vert _{\ell^{\alpha}}^{\beta} & =\left[\sum_{i\in I_{0}}\left(u_{i}/v_{i}\right)^{q\cdot\beta}\right]^{1/q}\\
 & =\left[\sum_{i\in I_{0}}\left(u_{i}c_{i}\right)^{q}\right]^{1/q}\\
 & =\left\Vert c\right\Vert _{\ell_{u}^{q}}\\
 & \leq\left\Vert \iota\right\Vert \cdot\left\Vert c\right\Vert _{\ell_{v}^{p}}\\
 & =\left\Vert \iota\right\Vert \cdot\left\Vert \left(\left(\frac{u_{i}}{v_{i}}\right)^{\beta}\cdot\frac{v_{i}}{u_{i}}\right)_{i\in I_{0}}\right\Vert _{\ell^{p}}.
\end{align*}
Now, let us first assume $p<\infty$. In this case, we have $1<p/q<\infty$,
so that $\beta=\left(p/q\right)'$ satisfies $\beta\in\left(1,\infty\right)$
and 
\[
\beta-1=\frac{1}{1/\beta}-1=\frac{1}{1-\frac{1}{p/q}}-1=\frac{1}{1-\frac{q}{p}}-1=\frac{1}{\frac{p-q}{p}}-1=\frac{p}{p-q}-1=\frac{q}{p-q}.
\]
We claim $p\left(\beta-1\right)=\alpha$. Indeed, this is equivalent
to
\begin{align*}
\frac{1}{p}\frac{1}{\beta-1}\overset{!}{=}\frac{1}{\alpha}=\frac{1}{q}\cdot\frac{1}{\left(p/q\right)'}=\frac{1}{q}\cdot\left(1-\frac{1}{p/q}\right) & \Longleftrightarrow\frac{1}{p}\cdot\frac{p-q}{q}\overset{!}{=}\frac{1}{q}\cdot\left(1-\frac{q}{p}\right)\\
 & \Longleftrightarrow\frac{p-q}{pq}\overset{!}{=}\frac{p-q}{pq}
\end{align*}
which is a tautology. Hence,
\[
\left\Vert \left(\left(u_{i}/v_{i}\right)^{\beta}\cdot v_{i}/u_{i}\right)_{i\in I_{0}}\right\Vert _{\ell^{p}}=\left\Vert \left(\left(u/v\right)_{i}^{\beta-1}\right)_{i\in I_{0}}\right\Vert _{\ell^{p}}=\left\Vert u/v\right\Vert _{\ell^{p\left(\beta-1\right)}\left(I_{0}\right)}^{\beta-1}=\left\Vert u/v\right\Vert _{\ell^{\alpha}\left(I_{0}\right)}^{\beta-1}.
\]
Altogether, we have thus shown
\[
\left\Vert \frac{u}{v}\cdot\chi_{I_{0}}\right\Vert _{\ell^{\alpha}}^{\beta}\leq\left\Vert \iota\right\Vert \cdot\left\Vert \frac{u}{v}\cdot\chi_{I_{0}}\right\Vert _{\ell^{\alpha}}^{\beta-1},
\]
so that rearranging yields $\left\Vert \frac{u}{v}\cdot\chi_{I_{0}}\right\Vert _{\ell^{\alpha}}\leq\left\Vert \iota\right\Vert $.
Note that this used finiteness of $\left\Vert \frac{u}{v}\cdot\chi_{I_{0}}\right\Vert _{\ell^{\alpha}}$
which holds, since $I_{0}\subset I$ is finite. But since $I_{0}\subset I$
was an arbitrary finite subset and because of $\alpha<\infty$, we
get $\left\Vert u/v\right\Vert _{\ell^{\alpha}}\leq\left\Vert \iota\right\Vert <\infty$.

It remains to consider the case $\alpha<\infty$, but $p=\infty$.
In this case, we have $p/q=\infty$ and hence $\beta=\left(p/q\right)'=1$,
which finally yields $\alpha=q\cdot\beta=q$. In this case, let again
$I_{0}\subset I$ be an arbitrary finite subset and define $c_{i}:=v_{i}^{-1}$
for $i\in I_{0}$ and $c_{i}=0$ for $i\in I\setminus I_{0}$. Then
\[
\left\Vert \frac{u}{v}\cdot\chi_{I_{0}}\right\Vert _{\ell^{\alpha}}=\left\Vert \left(u_{i}\cdot c_{i}\right)_{i\in I}\right\Vert _{\ell^{q}}=\left\Vert c\right\Vert _{\ell_{u}^{q}}\leq\left\Vert \iota\right\Vert \left\Vert c\right\Vert _{\ell_{v}^{p}}=\left\Vert \iota\right\Vert \left\Vert \chi_{I_{0}}\right\Vert _{\ell^{\infty}}\leq\left\Vert \iota\right\Vert .
\]
As above, we conclude $\left\Vert u/v\right\Vert _{\ell^{\alpha}}\leq\left\Vert \iota\right\Vert <\infty$.
\end{proof}
We can now state a simplified version of our embedding results.
\begin{cor}
\label{cor:SimplifiedSobolevEmbedding}Let $\mathcal{Q}=\left(Q_{i}\right)_{i\in I}=\left(T_{i}Q_{i}'+b_{i}\right)_{i\in I}$
be a tight regular covering of the open set $\emptyset\neq\mathcal{O}\subset\mathbb{R}^{d}$.

Let $n\in\mathbb{N}_{0}$ and $p,q,r\in\left(0,\infty\right]$ and
let $u=\left(u_{i}\right)_{i\in I}$ be a $\mathcal{Q}$-moderate
weight. For $t\in\left(0,\infty\right]$, define the weight $w^{\left(t\right)}=\left(\smash{w_{i}^{\left(t\right)}}\right)_{i\in I}$
by
\[
w_{i}^{\left(t\right)}:=\left|\det T_{i}\right|^{\frac{1}{p}-\frac{1}{t}}\cdot\left(1+\left|b_{i}\right|^{n}+\left\Vert T_{i}\right\Vert ^{n}\right)\qquad\text{ for }i\in I.
\]
Then the following hold:
\begin{enumerate}
\item If $p\leq q$ and if
\[
\frac{w^{\left(q\right)}}{u}\in\ell^{q^{\triangledown}\cdot\left(r/q^{\triangledown}\right)'}\left(I\right),
\]
then all assumptions of Corollary \ref{cor:SufficientConditionsForSobolevEmbeddings}
(with $k=n$) are satisfied. In particular,
\[
\mathcal{D}\left(\mathcal{Q},L^{p},\ell_{u}^{r}\right)\hookrightarrow W^{n,q}\left(\mathbb{R}^{d}\right).
\]
If $q=\infty$, then $\mathcal{D}\left(\mathcal{Q},L^{p},\ell_{u}^{r}\right)\hookrightarrow C_{b}^{n}\left(\mathbb{R}^{d}\right)$.
\item Conversely, if
\[
\mathcal{D}\left(\mathcal{Q},L^{p},\ell_{u}^{r}\right)\hookrightarrow W^{n,q}\left(\mathbb{R}^{d}\right),
\]
then $p\leq q$ and the following hold:

\begin{enumerate}
\item We have
\[
\frac{w^{\left(q\right)}}{u}\in\ell^{q\cdot\left(r/q\right)'}\left(I\right).
\]

\item If $q=\infty$, then
\[
\frac{w^{\left(q\right)}}{u}\in\ell^{r'}\left(I\right).
\]

\item If $I_{0}\subset I$ satisfies $\sup_{i\in I_{0}}\left\Vert T_{i}^{-1}\right\Vert <\infty$,
then we have the following:

\begin{enumerate}
\item \label{enu:SpecialKhinchinConditionWithoutWeight}If $q\in\left(0,\infty\right)$,
then
\[
\frac{w^{\left(p\right)}}{u}\in\ell^{2\cdot\left(r/2\right)'}\left(I_{0}\right).
\]

\item \label{enu:SpecialKhinchinConditionWithWeight}If $q\in\left[2,\infty\right)$,
then
\[
\frac{w^{\left(2\right)}}{u}\in\ell^{2\cdot\left(r/2\right)'}\left(I_{0}\right).\qedhere
\]

\end{enumerate}
\end{enumerate}
\end{enumerate}
\end{cor}
\begin{proof}

\begin{enumerate}
\item To avoid confusion with the weights from the present corollary, let
us write $v^{\ast},w^{\ast}$ for the weights $v,w$ as defined in
Corollary \ref{cor:SufficientConditionsForSobolevEmbeddings} (with
$k=n$). With this notation, we have $w_{i}^{\left(q\right)}\geq v_{i}^{\ast}$
and $w_{i}^{\left(q\right)}\geq w_{i}^{\ast}$ for all $i\in I$.
Now, the assumption $w^{\left(q\right)}/u\in\ell^{q^{\triangledown}\cdot\left(r/q^{\triangledown}\right)'}\left(I\right)$
implies by Lemma \ref{lem:SequenceSpaceEmbeddingCharacterization}
that 
\[
\ell_{u}^{r}\left(I\right)\hookrightarrow\ell_{w^{\left(q\right)}}^{q^{\triangledown}}\left(I\right)\hookrightarrow\ell_{v^{\ast}}^{q^{\triangledown}}\left(I\right)\cap\ell_{w^{\ast}}^{q^{\triangledown}}\left(I\right).
\]
Since we also have $p\leq q$, all assumptions of Corollary \ref{cor:SufficientConditionsForSobolevEmbeddings}
are satisfied.
\item Assume $\mathcal{D}\left(\mathcal{Q},L^{p},\ell_{u}^{r}\right)\hookrightarrow W^{n,q}\left(\mathbb{R}^{d}\right)$.
By Definition \ref{def:EmbeddingDefinition} and Theorem \ref{thm:NecessaryExponentRelation},
this implies $p\leq q$. Furthermore, it is easy to see that the given
embedding implies boundedness of the maps $\iota_{\alpha}$ from Theorem
\ref{thm:SequenceSpaceEmbeddingIsNecessary} for every $\alpha\in\mathbb{N}_{0}^{d}$
with $\left|\alpha\right|\leq n$. Thus, by applying Theorem \ref{thm:SequenceSpaceEmbeddingIsNecessary}
(with $Y=\ell_{u}^{r}\left(I\right)$ and $k=0$ or $k=n$, respectively),
we get 
\[
\ell_{u}^{r}\left(I\right)\cap\ell_{0}\left(I\right)\hookrightarrow\ell_{u^{\left(0,p,q\right)}}^{q}\left(I\right)\qquad\text{ and }\qquad\ell_{u}^{r}\left(I\right)\cap\ell_{0}\left(I\right)\hookrightarrow\ell_{u^{\left(n,p,q\right)}}^{q}\left(I\right)
\]
for the weights $u^{\left(0,p,q\right)},u^{\left(n,p,q\right)}$ from
that theorem. Using Lemma \ref{lem:SequenceSpaceEmbeddingCharacterization},
we derive
\[
\frac{u^{\left(0,p,q\right)}}{u}\in\ell^{q\cdot\left(r/q\right)'}\left(I\right)\qquad\text{ and }\qquad\frac{u^{\left(n,p,q\right)}}{u}\in\ell^{q\cdot\left(r/q\right)'}\left(I\right).
\]
But clearly, $w_{i}^{\left(q\right)}\leq u_{i}^{\left(0,p,q\right)}+u_{i}^{\left(n,p,q\right)}$
for all $i\in I$, so that $w^{\left(q\right)}/u\in\ell^{q\cdot\left(r/q\right)'}\left(I\right)$
as claimed.

In case of $q=\infty$, since all of the maps $\iota_{\alpha}$ (with
$\left|\alpha\right|\leq n$) from Theorem \ref{thm:SequenceSpaceEmbeddingIsNecessary}
are bounded, we get (with $Y=\ell_{u}^{r}\left(I\right)$ and with
$k=0$ or $k=n$, respectively) that
\[
\ell_{u}^{r}\left(I\right)\cap\ell_{0}\left(I\right)\hookrightarrow\ell_{u^{\left(0,p,q\right)}}^{1}\left(I\right)\qquad\text{ and }\qquad\ell_{u}^{r}\left(I\right)\cap\ell_{0}\left(I\right)\hookrightarrow\ell_{u^{\left(n,p,q\right)}}^{1}\left(I\right)
\]
for the weights $u^{\left(0,p,q\right)},u^{\left(n,p,q\right)}$ from
Theorem \ref{thm:SequenceSpaceEmbeddingIsNecessary}. Using Lemma
\ref{lem:SequenceSpaceEmbeddingCharacterization}, as well as $1\cdot\left(r/1\right)'=r'$,
we get $w^{\left(q\right)}/u\in\ell^{r'}\left(I\right)$ using essentially
the same arguments as above.

Finally if $I_{0}\subset I$ satisfies $M:=\sup_{i\in I_{0}}\left\Vert T_{i}^{-1}\right\Vert <\infty$,
then the third part of Theorem \ref{thm:SequenceSpaceEmbeddingIsNecessary}
(with $Y=\ell_{u}^{r}\left(I\right)$ and $k=0$ or $k=n$, respectively)
yields
\[
\ell_{u}^{r}\left(I\right)\cap\ell_{0}\left(I_{0}\right)\hookrightarrow\ell_{u^{\left(0,p,2\right)}}^{2}\left(I_{0}\right)\qquad\text{ and }\qquad\ell_{u}^{r}\left(I\right)\cap\ell_{0}\left(I_{0}\right)\hookrightarrow\ell_{u^{\left(n,p,2\right)}}^{2}\left(I_{0}\right)\qquad\text{ if }q\in\left[2,\infty\right),
\]
as well as
\[
\ell_{u}^{r}\left(I\right)\cap\ell_{0}\left(I_{0}\right)\hookrightarrow\ell_{u^{\left(0,p,p\right)}}^{2}\left(I_{0}\right)\qquad\text{ and }\qquad\ell_{u}^{r}\left(I\right)\cap\ell_{0}\left(I_{0}\right)\hookrightarrow\ell_{u^{\left(n,p,p\right)}}^{2}\left(I_{0}\right)\qquad\text{ if }q\in\left(0,\infty\right).
\]
Using Lemma \ref{lem:SequenceSpaceEmbeddingCharacterization} and
the estimate $w_{i}^{\left(t\right)}\leq u_{i}^{\left(0,p,t\right)}+u_{i}^{\left(n,p,t\right)}$
for all $i\in I$, we get $w^{\left(p\right)}/u\in\ell^{2\cdot\left(r/2\right)'}\left(I_{0}\right)$
if $q\in\left(0,\infty\right)$ and $w^{\left(2\right)}/u\in\ell^{2\cdot\left(r/2\right)'}\left(I_{0}\right)$
if $q\in\left[2,\infty\right)$.\qedhere

\end{enumerate}
\end{proof}

\section{Embeddings into ${\rm BV}\left(\mathbb{R}^{d}\right)$}

\label{sec:EmbeddingsIntoBV}In this short section, we make the (perhaps
surprising) observation that a decomposition space embeds into 
\[
{\rm BV}^{k}\left(\mathbb{R}^{d}\right):=\left\{ f:\mathbb{R}^{d}\to\mathbb{C}\with f\in L^{1}\left(\mathbb{R}^{d}\right)\text{ and }\forall\alpha\in\mathbb{N}_{0}^{d}\setminus\left\{ 0\right\} \text{ with }\left|\alpha\right|\leq k:\quad\partial^{\alpha}f\text{ is a finite measure}\right\} 
\]
if and only if it embeds into $W^{k,1}\left(\mathbb{R}^{d}\right)$.
Here, ${\rm BV}^{k}\left(\mathbb{R}^{d}\right)$ is equipped with
the norm
\[
\left\Vert f\right\Vert _{{\rm BV}^{k}}:=\left\Vert f\right\Vert _{L^{1}}+\sum_{\alpha\in\mathbb{N}_{0}^{d}\setminus\left\{ 0\right\} }\left\Vert \partial^{\alpha}f\right\Vert _{{\rm TV}},
\]
where $\left\Vert \mu\right\Vert _{{\rm TV}}:=\left|\mu\right|\left(\mathbb{R}^{d}\right)$
denotes the total variation norm of a (finite) Borel measure $\mu$
on $\mathbb{R}^{d}$.
\begin{cor}
\label{cor:EmbeddingIntoBVEquivalentToSobolevEmbedding}Let $\mathcal{Q}=\left(Q_{i}\right)_{i\in I}=\left(T_{i}Q_{i}'+b_{i}\right)_{i\in I}$
be a tight regular covering of the open set $\emptyset\neq\mathcal{O}\subset\mathbb{R}^{d}$.
Let $n\in\mathbb{N}$ and $p\in\left(0,\infty\right]$ and let $Y\leq\mathbb{C}^{I}$
be a $\mathcal{Q}$-regular sequence space.

If $Y\cap\ell_{0}\left(I\right)\subset Y$ is dense or if $Y=\ell_{u}^{r}\left(I\right)$
for some $r\in\left(0,\infty\right]$ and a $\mathcal{Q}$-moderate
weight $u=\left(u_{i}\right)_{i\in I}$, then
\[
\mathcal{D}\left(\mathcal{Q},L^{p},Y\right)\hookrightarrow{\rm BV}^{n}\left(\mathbb{R}^{d}\right)
\]
holds if and only if
\[
\mathcal{D}\left(\mathcal{Q},L^{p},Y\right)\hookrightarrow W^{n,1}\left(\mathbb{R}^{d}\right)
\]
is true. In this case, we even have $p\leq1$ and $Y\hookrightarrow\ell_{v}^{1}\left(I\right)$
with
\[
v_{i}:=\left|\det T_{i}\right|^{\frac{1}{p}-1}\cdot\left(1+\left|b_{i}\right|^{n}+\left\Vert T_{i}\right\Vert ^{n}\right),
\]
as well as $\mathcal{D}\left(\mathcal{Q},L^{p},Y\right)\overset{{\rm inj}.}{\hookrightarrow}W^{n,1}\left(\mathbb{R}^{d}\right)$.\end{cor}
\begin{rem*}
Note that the corollary remains true (with the same proof) even if
the definition of ${\rm BV}^{k}\left(\mathbb{R}^{d}\right)$ is changed
such that the elements of ${\rm BV}^{k}\left(\mathbb{R}^{d}\right)$
are only required to be finite measures instead of $L^{1}$ functions.\end{rem*}
\begin{proof}
Note that $\theta:W^{n,1}\left(\mathbb{R}^{d}\right)\hookrightarrow{\rm BV}^{n}\left(\mathbb{R}^{d}\right),f\mapsto f$
is an isometric embedding. This makes the implication ``$\Leftarrow$''
trivial.

For ``$\Rightarrow$'', note that by assumption, there is a bounded
linear map $\iota:\mathcal{D}\left(\mathcal{Q},L^{p},Y\right)\to{\rm BV}^{n}\left(\mathbb{R}^{d}\right)$
which satisfies $\iota f=f$ for all $f\in\mathcal{S}_{\mathcal{O}}^{p,Y}$.
But $\mathcal{S}_{\mathcal{O}}^{p,Y}\subset\mathcal{S}\left(\mathbb{R}^{d}\right)\subset W^{n,1}\left(\mathbb{R}^{d}\right)$,
so that $\tilde{\iota}:\mathcal{S}_{\mathcal{O}}^{p,Y}\to W^{n,1}\left(\mathbb{R}^{d}\right),f\mapsto f$
is well-defined and bounded since 
\[
\left\Vert \tilde{\iota}f\right\Vert _{W^{n,1}}=\left\Vert \theta\tilde{\iota}f\right\Vert _{{\rm BV}^{n}}=\left\Vert f\right\Vert _{{\rm BV}^{n}}=\left\Vert \iota f\right\Vert _{{\rm BV}^{n}}\leq\left\Vert \iota\right\Vert \cdot\left\Vert f\right\Vert _{\mathcal{D}\left(\mathcal{Q},L^{p},Y\right)}.
\]

By Theorem \ref{thm:NecessaryExponentRelation}, this implies $p\leq1$.
Furthermore, boundedness of $\tilde{\iota}$ easily implies that each
of the maps $\iota_{\alpha}$ (with $\left|\alpha\right|\leq n$)
from Theorem \ref{thm:SequenceSpaceEmbeddingIsNecessary} are bounded.
Thus, two applications of this theorem (for $k=0$ and $k=n$, respectively),
show
\[
Y\cap\ell_{0}\left(I\right)\hookrightarrow\ell_{u^{\left(0,p,1\right)}}^{1}\left(I\right)=\ell_{u^{\left(0,p,1\right)}}^{1^{\triangledown}}\left(I\right)\qquad\text{ and }\qquad Y\cap\ell_{0}\left(I\right)\hookrightarrow\ell_{u^{\left(n,p,1\right)}}^{1}\left(I\right)=\ell_{u^{\left(n,p,1\right)}}^{1^{\triangledown}}\left(I\right),
\]
with $u^{\left(0,p,1\right)}$ and $u^{\left(n,p,1\right)}$ as in
Theorem \ref{thm:SequenceSpaceEmbeddingIsNecessary}. Here, we also
used $1^{\triangledown}=1$.

But obviously $v_{i}\leq u_{i}^{\left(0,p,1\right)}+u_{i}^{\left(n,p,1\right)}$
for all $i\in I$, so that we get $Y\cap\ell_{0}\left(I\right)\hookrightarrow\ell_{v}^{1^{\triangledown}}\left(I\right)=\ell_{v}^{1}\left(I\right)$.
There are now two cases:
\begin{casenv}
\item If $Y\cap\ell_{0}\left(I\right)\leq Y$ is dense, we get $Y\hookrightarrow\ell_{v}^{1}\left(I\right)$,
since $Y$ and $\ell_{v}^{1}\left(I\right)$ both embed continuously
into the Hausdorff space $\mathbb{C}^{I}$, so that the unique continuous
extension of the embedding $Y\cap\ell_{0}\left(I\right)\hookrightarrow\ell_{v}^{1}\left(I\right)$
has to be given by the identity.
\item If $Y=\ell_{u}^{r}\left(I\right)$ for some $r\in\left(0,\infty\right]$,
then Lemma \ref{lem:SequenceSpaceEmbeddingCharacterization} yields
$Y\hookrightarrow\ell_{v}^{1}\left(I\right)$.
\end{casenv}
Since we have $Y\hookrightarrow\ell_{v}^{1}\left(I\right)$ in each
case and since we also have $p\leq1$ as seen above, Corollary \ref{cor:SufficientConditionsForSobolevEmbeddings}
implies $\mathcal{D}\left(\mathcal{Q},L^{p},Y\right)\overset{\text{inj.}}{\hookrightarrow}W^{n,1}\left(\mathbb{R}^{d}\right)$,
as desired.
\end{proof}

\section{Applications}

\label{sec:Applications}In this section, we apply our general, simplified
embedding results from Section \ref{sec:SimplifiedConditions} to
a large collection of examples, namely to
\begin{enumerate}
\item homogeneous and inhomogeneous Besov spaces,
\item $\alpha$-modulation spaces,
\item shearlet smoothness spaces,
\item shearlet-type coorbit spaces and
\item coorbit spaces of the diagonal group.\end{enumerate}
\begin{example}
\label{exa:HomogeneousBesovSpaces}(homogeneous Besov spaces)

Homogeneous Besov spaces can be obtained as decomposition spaces with
respect to a certain dyadic covering of $\mathcal{O}:=\mathbb{R}^{d}\setminus\left\{ 0\right\} $.
More precisely, the covering is given by $\mathcal{Q}=\left(Q_{n}\right)_{n\in\mathbb{Z}}=\left(T_{n}Q+b_{n}\right)_{n\in\mathbb{Z}}$
with $b_{n}:=0$ and $T_{n}:=2^{n}\cdot{\rm id}$ for all $n\in\mathbb{Z}$,
where $Q:=B_{4}\left(0\right)\setminus\overline{B_{1/4}\left(0\right)}$.
It is easy to see that $P:=B_{2}\left(0\right)\setminus\overline{B_{1/2}\left(0\right)}$
is compactly contained in $Q$ and that $\mathbb{R}^{d}\setminus\left\{ 0\right\} =\bigcup_{n\in\mathbb{Z}}T_{n}P$.
Finally, $x\in Q_{n}\cap Q_{m}$ implies $2^{n-2}\leq\left|x\right|\leq2^{m+2}$
and thus $n\leq m+4$. By symmetry, we arrive at $\left|n-m\right|\leq4$,
so that on the one hand $\left|n^{\ast}\right|\leq9$ and on the other
hand
\[
\sup_{n\in\mathbb{Z}}\sup_{m\in n^{\ast}}\left\Vert T_{n}^{-1}T_{m}\right\Vert =\sup_{n\in\mathbb{Z}}\sup_{m\in n^{\ast}}2^{m-n}\leq2^{4},
\]
so that $\mathcal{Q}$ is indeed a structured admissible covering
of $\mathbb{R}^{d}\setminus\left\{ 0\right\} $. By Theorem \ref{thm:StructuredAdmissibleCoveringsAreRegular},
this implies that $\mathcal{Q}$ is a tight regular covering of $\mathbb{R}^{d}\setminus\left\{ 0\right\} $.
Thus, the assumptions regarding $\mathcal{Q}$ in Corollary \ref{cor:SimplifiedSobolevEmbedding}
are satisfied.

Using the covering $\mathcal{Q}$, the usual homogeneous Besov spaces
are (up to certain identifications) given by
\[
\dot{\mathcal{B}}_{s}^{p,r}\left(\mathbb{R}^{d}\right)=\mathcal{D}\left(\mathcal{Q},L^{p},\ell_{u}^{r}\right)
\]
for $u=u^{\left(s\right)}=\left(2^{sn}\right)_{n\in\mathbb{Z}}$ and
$p,r\in\left(0,\infty\right]$, as well as $s\in\mathbb{R}$.

Let $k\in\mathbb{N}_{0}$ and $q\in\left(0,\infty\right]$. We are
interested in whether an embedding of the form $\dot{\mathcal{B}}_{s}^{p,r}\left(\mathbb{R}^{d}\right)\hookrightarrow W^{k,q}\left(\mathbb{R}^{d}\right)$
holds. To this end, note that the weight $v:=w^{\left(q\right)}$
from Corollary \ref{cor:SimplifiedSobolevEmbedding} (with $n=k$)
is in this case given by
\[
v_{n}=\left|\det T_{n}\right|^{\frac{1}{p}-\frac{1}{q}}\cdot\left(1+\left|b_{n}\right|^{k}+\left\Vert T_{n}\right\Vert ^{k}\right)\asymp2^{nd\left(\frac{1}{p}-\frac{1}{q}\right)}\left(1+2^{nk}\right)
\]
for $n\in\mathbb{Z}$. Hence,
\[
\frac{v_{n}}{u_{n}}\asymp2^{n\left[d\left(\frac{1}{p}-\frac{1}{q}\right)-s\right]}\left(1+2^{nk}\right)=2^{n\left[d\left(\frac{1}{p}-\frac{1}{q}\right)-s\right]}+2^{n\left[d\left(\frac{1}{p}-\frac{1}{q}\right)+k-s\right]}.
\]
Now, note that a weight of the form $\left(2^{\alpha n}\right)_{n\in\mathbb{Z}}$
is unbounded (and thus not contained in any space $\ell^{\theta}\left(\mathbb{Z}\right)$
for $\theta\in\left(0,\infty\right]$) as soon as $\alpha\neq0$.
For $\alpha=0$, the weight is constant and thus contained in $\ell^{\infty}\left(\mathbb{Z}\right)$,
but in no space $\ell^{\theta}\left(\mathbb{Z}\right)$ with $0<\theta<\infty$.

Altogether, we get for $\theta\in\left(0,\infty\right]$ that
\begin{align*}
\frac{v}{u}\in\ell^{\theta}\left(\mathbb{Z}\right) & \Longleftrightarrow\left(2^{n\left[d\left(\frac{1}{p}-\frac{1}{q}\right)-s\right]}\right)_{n\in\mathbb{Z}}\in\ell^{\theta}\left(\mathbb{Z}\right)\text{ and }\left(2^{n\left[d\left(\frac{1}{p}-\frac{1}{q}\right)+k-s\right]}\right)_{n\in\mathbb{Z}}\in\ell^{\theta}\left(\mathbb{Z}\right)\\
 & \Longleftrightarrow\theta=\infty\text{ and }s=d\left(\frac{1}{p}-\frac{1}{q}\right)\text{ and }s=d\left(\frac{1}{p}-\frac{1}{q}\right)+k,
\end{align*}
which can only hold for $k=0$. Thus, Corollary \ref{cor:SimplifiedSobolevEmbedding}
implies that $\dot{\mathcal{B}}_{s}^{p,r}\left(\mathbb{R}^{d}\right)\hookrightarrow W^{k,q}\left(\mathbb{R}^{d}\right)$
can only hold for $k=0$.

For $k=0$, note that Corollary \ref{cor:SimplifiedSobolevEmbedding}
implies that $\dot{\mathcal{B}}_{s}^{p,r}\left(\mathbb{R}^{d}\right)\hookrightarrow L^{q}\left(\mathbb{R}^{d}\right)=W^{0,q}\left(\mathbb{R}^{d}\right)$
holds as soon as we have $p\leq q$ and
\begin{align*}
\frac{v}{u}\in\ell^{q^{\triangledown}\cdot\left(r/q^{\triangledown}\right)'}\left(\mathbb{Z}\right) & \Longleftrightarrow q^{\triangledown}\cdot\left(r/q^{\triangledown}\right)'=\infty\text{ and }s=d\left(\frac{1}{p}-\frac{1}{q}\right)\\
\left(\text{by eq. }\eqref{eq:SpecialExponentInfiniteCharacterization}\right) & \Longleftrightarrow r\leq q^{\triangledown}\text{ and }s=d\left(\frac{1}{p}-\frac{1}{q}\right).
\end{align*}
By Corollary \ref{cor:SimplifiedSobolevEmbedding}, for $q\in\left(0,2\right]\cup\left\{ \infty\right\} $,
this condition is also necessary for $\dot{\mathcal{B}}_{s}^{p,r}\left(\mathbb{R}^{d}\right)\hookrightarrow L^{q}\left(\mathbb{R}^{d}\right)$
to hold. For $q\in\left(2,\infty\right)$, however, Corollary \ref{cor:SimplifiedSobolevEmbedding}
(together with similar considerations) only implies that
\[
r\leq q\text{ and }s=d\left(\frac{1}{p}-\frac{1}{q}\right)
\]
is a necessary condition for $\dot{\mathcal{B}}_{s}^{p,r}\left(\mathbb{R}^{d}\right)\hookrightarrow L^{q}\left(\mathbb{R}^{d}\right)$
to hold.

This raises the question of what happens for $q^{\triangledown}<r\leq q$,
i.e.\@ if the necessary condition is fulfilled, but the sufficient
condition is not. In general, I do not know if the embedding $\dot{\mathcal{B}}_{s}^{p,r}\left(\mathbb{R}^{d}\right)\hookrightarrow L^{q}\left(\mathbb{R}^{d}\right)$
holds in this case. But at least for $p=q\in\left(2,\infty\right)$,
we can say slightly more: In this case, $s=0$. Furthermore, for $I_{0}:=\mathbb{N}_{0}$,
we have $\sup_{i\in I_{0}}\left\Vert T_{i}^{-1}\right\Vert =\sup_{n\in\mathbb{N}_{0}}2^{-n}=1<\infty$,
so that part \ref{enu:SpecialKhinchinConditionWithoutWeight} of Corollary
\ref{cor:SimplifiedSobolevEmbedding} is applicable. Hence,
\[
\left(3\right)_{n\in\mathbb{N}_{0}}=\left(\frac{3}{2^{sn}}\right)_{n\in\mathbb{N}_{0}}=\frac{w^{\left(p\right)}}{u}\in\ell^{2\cdot\left(r/2\right)'}\left(I_{0}\right)=\ell^{2\cdot\left(r/2\right)'}\left(\mathbb{N}_{0}\right),
\]
which can only hold for $2\cdot\left(r/2\right)'=\infty$, i.e.\@
for $r\leq2$. Hence, for $p\in\left(2,\infty\right)$, we see
that $\dot{\mathcal{B}}_{s}^{p,r}\left(\mathbb{R}^{d}\right)\hookrightarrow L^{p}\left(\mathbb{R}^{d}\right)$
can only hold for $s=0$ and $r\in\left(0,2\right]$ and does hold
for $r\in\left(0,q^{\triangledown}\right]$. We have thus reduced
the ``uncertain'' region from $\left(q^{\triangledown},q\right]$
to $\left(q^{\triangledown},2\right]$.
\end{example}

\begin{example}
\label{exa:InhomogeneousBesovSpaces}(inhomogeneous Besov spaces)

Inhomogeneous Besov spaces are defined using a similar covering than
the homogeneous Besov spaces from the previous example. The only difference
is that the sets $Q_{n}$ for $n<0$, i.e.\@ the ``small'' sets
of the homogeneous dyadic covering are replaced by a single ball which
covers the low frequencies (including zero).

Precisely, we will use the covering $\mathcal{Q}=\left(Q_{n}\right)_{n\in\mathbb{N}_{0}}=\left(T_{n}Q_{n}'+b_{n}\right)_{n\in\mathbb{N}_{0}}$
with $T_{n}=2^{n}\cdot{\rm id}$ and $b_{n}:=0$, as well as $Q_{n}':=B_{4}\left(0\right)\setminus\overline{B_{1/4}\left(0\right)}$
for $n\in\mathbb{N}$. Finally, for $n=0$, we set 
\[
T_{0}:={\rm id}\qquad\text{ and }\qquad b_{0}:=0\qquad\text{ as well as }\qquad Q_{0}':=B_{2}\left(0\right).
\]
Note that $\mathcal{Q}$ is indeed a semi-structured admissible covering,
since $x\in Q_{n}\cap Q_{m}$ for $n,m\in\mathbb{N}$ implies $2^{n-2}\leq\left|x\right|\leq2^{n+2}$
and likewise $2^{m-2}\leq\left|x\right|\leq2^{m+2}$. Thus,
\[
2^{n-2}\leq2^{m+2}\qquad\text{ and }\qquad2^{m-2}\leq2^{n+2},
\]
which yields $n-4\leq m\leq n+4$, i.e.\@ $n^{\ast}\subset\left[\left\{ n-4,\dots,n+4\right\} \cup\left\{ 0\right\} \right]\cap\mathbb{N}_{0}$
and hence $\left|n^{\ast}\right|\leq10$ for all $n\in\mathbb{N}$.
Furthermore, $\left\Vert T_{n}^{-1}T_{m}\right\Vert =2^{m-n}\leq2^{4}$
and $\left\Vert T_{m}^{-1}T_{n}\right\Vert =2^{n-m}\leq2^{4}$.

Likewise, for $x\in Q_{0}\cap Q_{m}$ with $m\in\mathbb{N}$, we get
\[
2^{m-2}\leq\left|x\right|\leq2=2^{1},
\]
i.e.\@ $m\leq3$ and thus $0^{\ast}\subset\left\{ 0,\dots,3\right\} $,
which implies $\left|0^{\ast}\right|\leq4$. Furthermore, $\left\Vert T_{0}^{-1}T_{m}\right\Vert =2^{m}\leq2^{4}$
and $\left\Vert T_{m}^{-1}T_{0}\right\Vert =2^{-m}\leq1$.

Finally, let $P_{n}':=B_{2}\left(0\right)\setminus\overline{B_{1/2}\left(0\right)}$
for $n\in\mathbb{N}$ and $P_{0}':=B_{3/2}\left(0\right)$. Then $P_{n}'$
is open with $\overline{P_{n}'}\subset Q_{n}'$ for all $n\in\mathbb{N}_{0}$.
Finally, 
\[
\bigcup_{n\in\mathbb{N}}T_{n}P_{n}'=\bigcup_{n\in\mathbb{N}}\left(B_{2^{n+1}}\left(0\right)\setminus\overline{B_{2^{n-1}}\left(0\right)}\right)\supset\mathbb{R}^{d}\setminus\overline{B_{1}\left(0\right)}
\]
and hence $\bigcup_{n\in\mathbb{N}_{0}}T_{n}P_{n}'=\mathbb{R}^{d}$.
Thus, $\mathcal{Q}$ satisfies all assumptions of Theorem \ref{thm:StructuredAdmissibleCoveringsAreRegular},
so that $\mathcal{Q}$ is a tight regular covering of $\mathbb{R}^{d}$.

With the dyadic covering $\mathcal{Q}$, the usual inhomogeneous Besov
spaces are given (up to trivial identifications) for $p,r\in\left(0,\infty\right]$
and $s\in\mathbb{R}$ by
\[
\mathcal{B}_{s}^{p,r}\left(\mathbb{R}^{d}\right)=\mathcal{D}\left(\mathcal{Q},L^{p},\ell_{u}^{r}\right)
\]
for $u=u^{\left(s\right)}=\left(2^{sn}\right)_{n\in\mathbb{N}_{0}}$.
Now, let $q\in\left(0,\infty\right]$ and $k\in\mathbb{N}_{0}$ be
arbitrary. The weight $v:=w^{\left(q\right)}$ from Corollary \ref{cor:SimplifiedSobolevEmbedding}
(with $n=k$) is given by
\[
v_{n}=\left|\det T_{n}\right|^{\frac{1}{p}-\frac{1}{q}}\cdot\left(1+\left|b_{n}\right|^{k}+\left\Vert T_{n}\right\Vert ^{k}\right)\asymp2^{dn\left(\frac{1}{p}-\frac{1}{q}\right)}\cdot\left(1+2^{nk}\right)\asymp2^{dn\left(\frac{1}{p}-\frac{1}{q}\right)}\cdot2^{nk}=2^{n\left(k+d\left(\frac{1}{p}-\frac{1}{q}\right)\right)}
\]
for $n\in\mathbb{N}_{0}$. Due to the exponential nature of the weights
$u,v$, we have
\[
\frac{v}{u}\asymp\left(2^{n\left(k-s+d\left(\frac{1}{p}-\frac{1}{q}\right)\right)}\right)_{n\in\mathbb{N}_{0}}\in\ell^{\theta}\left(\mathbb{N}_{0}\right)\;\Longleftrightarrow\;\begin{cases}
k-s+d\left(\frac{1}{p}-\frac{1}{q}\right)\leq0, & \text{if }\theta=\infty,\\
k-s+d\left(\frac{1}{p}-\frac{1}{q}\right)<0, & \text{if }\theta<\infty.
\end{cases}
\]

Recall from equation (\ref{eq:SpecialExponentInfiniteCharacterization})
that $q\cdot\left(p/q\right)'=\infty$ if and only if $p\leq q$.
Thus, Corollary \ref{cor:SimplifiedSobolevEmbedding} shows that we
have $\mathcal{B}_{s}^{p,r}\left(\mathbb{R}^{d}\right)\hookrightarrow W^{k,q}\left(\mathbb{R}^{d}\right)$
as soon as $p\leq q$ and
\begin{equation}
\frac{v}{u}\in\ell^{q^{\triangledown}\cdot\left(r/q^{\triangledown}\right)'}\;\Longleftrightarrow\;\begin{cases}
s\geq k+d\left(\frac{1}{p}-\frac{1}{q}\right), & \text{if }r\leq q^{\triangledown},\\
s>k+d\left(\frac{1}{p}-\frac{1}{q}\right), & \text{if }r>q^{\triangledown}.
\end{cases}\label{eq:InhomogeneousBesovSufficient}
\end{equation}
Conversely, Corollary \ref{cor:SimplifiedSobolevEmbedding} shows
for $q\in\left(0,2\right]\cup\left\{ \infty\right\} $ that the above
conditions are also necessary for existence of the embedding $\mathcal{B}_{s}^{p,r}\left(\mathbb{R}^{d}\right)\hookrightarrow W^{k,q}\left(\mathbb{R}^{d}\right)$.

For the analysis in case of $q\in\left(2,\infty\right)$, instead
of Corollary \ref{cor:SimplifiedSobolevEmbedding}, we employ the
``Besov detour'' as described in Subsection \ref{sub:ComparisonAndBesovDetour}.
This will show that the sufficient criterion from the present paper
lacks sharpness for $q\in\left(2,\infty\right)$. Indeed, by equation
(\ref{eq:SobolevBesovInclusionHard}), we have
\[
\mathcal{B}_{k}^{q,2}\left(\mathbb{R}^{d}\right)=\mathcal{B}_{k}^{q,\min\left\{ q,2\right\} }\left(\mathbb{R}^{d}\right)\hookrightarrow F_{k}^{q,2}\left(\mathbb{R}^{d}\right)=W^{k,q}\left(\mathbb{R}^{d}\right),
\]
so that $\mathcal{B}_{s}^{p,r}\left(\mathbb{R}^{d}\right)\hookrightarrow W^{k,q}\left(\mathbb{R}^{d}\right)$
holds as soon as $\mathcal{B}_{s}^{p,r}\left(\mathbb{R}^{d}\right)\hookrightarrow\mathcal{B}_{k}^{q,2}\left(\mathbb{R}^{d}\right)$.
By (the proof of) \cite[Proposition 2.4]{HanWangAlphaModulationEmbeddings}
(with $\alpha=1$), this holds as soon as $p\leq q$ and
\begin{equation}
\begin{cases}
s\geq k+d\left(\frac{1}{p}-\frac{1}{q}\right), & \text{if }r\leq2,\\
s>k+d\left(\frac{1}{p}-\frac{1}{q}\right), & \text{if }r>2.
\end{cases}\label{eq:InhomogeneousBesovSufficientImproved}
\end{equation}
Note that this condition is strictly weaker than the sufficient condition
in equation (\ref{eq:InhomogeneousBesovSufficient}).

At least for $p=q\in\left(2,\infty\right)$, this relaxed sufficient
condition is sharp: If $\mathcal{B}_{s}^{p,r}\left(\mathbb{R}^{d}\right)\hookrightarrow W^{k,p}\left(\mathbb{R}^{d}\right)$,
then part \ref{enu:SpecialKhinchinConditionWithoutWeight} of Corollary
\ref{cor:SimplifiedSobolevEmbedding} (with $I_{0}=I=\mathbb{N}_{0}$
and $\sup_{i\in I_{0}}\left\Vert T_{i}^{-1}\right\Vert =\sup_{n\in\mathbb{N}_{0}}2^{-n}=1<\infty$)
yields
\[
\left(\frac{2^{kn}}{2^{sn}}\right)_{n\in\mathbb{N}_{0}}\asymp\frac{w^{\left(p\right)}}{u}\in\ell^{2\cdot\left(r/2\right)'}\left(I_{0}\right),
\]
which is easily seen to be equivalent to condition (\ref{eq:InhomogeneousBesovSufficientImproved}),
since $p=q$. I do not know if the condition (\ref{eq:InhomogeneousBesovSufficientImproved})
is also necessary for existence of the embedding if $p<q$.
\end{example}

\begin{example}
\label{exa:AlphaModulationSpaces}($\alpha$-modulation spaces)

In \cite[Theorem 2.6]{BorupNielsenPsiDOsOnMultivariateAlphaModulationSpaces},
it was shown that for $0\leq\alpha<1$ and $d\in\mathbb{N}$, there
is some $r_{1}=r_{1}\left(\alpha,d\right)>0$ such that the family
\[
\mathbb{O}^{\left(\alpha\right)}:=\mathbb{O}_{r}^{\left(\alpha\right)}:=\left(O_{k}^{\left(\alpha\right)}\right)_{k\in\mathbb{Z}^{d}\setminus\left\{ 0\right\} }:=\left(B_{r\left|k\right|^{\alpha_{0}}}\left(\left|k\right|^{\alpha_{0}}k\right)\right)_{k\in\mathbb{Z}^{d}\setminus\left\{ 0\right\} }=\left(T_{k}Q+b_{k}\right)_{k\in\mathbb{Z}^{d}\setminus\left\{ 0\right\} }
\]
with $\alpha_{0}:=\frac{\alpha}{1-\alpha}$, $Q:=B_{r}\left(0\right)$
and
\[
T_{k}:=\left|k\right|^{\alpha_{0}}\cdot{\rm id}\qquad\text{ as well as }\qquad b_{k}:=\left|k\right|^{\alpha_{0}}k
\]
for $k\in\mathbb{Z}^{d}\setminus\left\{ 0\right\} $ is an admissible
covering of $\mathbb{R}^{d}$ for each $r>r_{1}$. In \cite[Theorem 6.1.3]{VoigtlaenderPhDThesis},
it was furthermore shown that this covering is indeed a structured
admissible covering and hence -- by Theorem \ref{thm:StructuredAdmissibleCoveringsAreRegular}
-- a tight regular covering of $\mathbb{R}^{d}$.

For $p,s\in\left(0,\infty\right]$ and $\gamma\in\mathbb{R}$, the
$\alpha$-modulation space with integrability exponents $p,s$ and
and weight exponent $\gamma$ is given (up to certain identifications)
by
\[
M_{\gamma,\alpha}^{p,s}\left(\mathbb{R}^{d}\right)=\mathcal{D}\left(\mathbb{O}^{\left(\alpha\right)},L^{p},\ell_{u^{\left(\gamma,\alpha\right)}}^{s}\right),
\]
where $u_{k}^{\left(\gamma,\alpha\right)}:=\left|k\right|^{\gamma/\left(1-\alpha\right)}$.
That the weight $u^{\left(\gamma,\alpha\right)}$ is indeed $\mathbb{O}^{\left(\alpha\right)}$-moderate
is shown in \cite[Lemma 6.1.2]{VoigtlaenderPhDThesis}.

Now, for $n\in\mathbb{N}_{0}$ and $q\in\left(0,\infty\right]$, we
are interested in existence of an embedding
\[
M_{\gamma,\alpha}^{p,s}\left(\mathbb{R}^{d}\right)\hookrightarrow W^{n,q}\left(\mathbb{R}^{d}\right).
\]
Note that the relevant weight $v:=w^{\left(q\right)}$ from Corollary
\ref{cor:SimplifiedSobolevEmbedding} is given by
\begin{align*}
v_{k} & =\left|\det T_{k}\right|^{\frac{1}{p}-\frac{1}{q}}\cdot\left(1+\left|b_{k}\right|^{n}+\left\Vert T_{k}\right\Vert ^{n}\right)\\
 & =\left|k\right|^{d\alpha_{0}\left(\frac{1}{p}-\frac{1}{q}\right)}\cdot\left(1+\left(\left|k\right|^{\alpha_{0}+1}\right)^{n}+\left|k\right|^{n\alpha_{0}}\right)\\
 & \asymp\left|k\right|^{d\alpha_{0}\left(\frac{1}{p}-\frac{1}{q}\right)}\cdot\left|k\right|^{\frac{n}{1-\alpha}}\\
 & =\left|k\right|^{\frac{1}{1-\alpha}\left(n+\alpha d\left(\frac{1}{p}-\frac{1}{q}\right)\right)}.
\end{align*}
Here, we used $\left|k\right|\geq1$ for $k\in\mathbb{Z}^{d}\setminus\left\{ 0\right\} $,
as well as $0\leq\alpha_{0}<1+\alpha_{0}=1+\frac{\alpha}{1-\alpha}=\frac{1}{1-\alpha}$.
We conclude
\[
\frac{v_{k}}{u_{k}^{\left(\gamma,\alpha\right)}}\asymp\left|k\right|^{\frac{1}{1-\alpha}\left(n-\gamma+\alpha d\left(\frac{1}{p}-\frac{1}{q}\right)\right)}
\]
for $k\in\mathbb{Z}^{d}\setminus\left\{ 0\right\} $.

Thus, for $\theta\in\left(0,\infty\right)$, we have
\[
\left\Vert \frac{v}{u^{\left(\gamma,\alpha\right)}}\right\Vert _{\ell^{\theta}}^{\theta}\asymp\sum_{k\in\mathbb{Z}^{d}\setminus\left\{ 0\right\} }\left|k\right|^{\frac{\theta}{1-\alpha}\left(n-\gamma+\alpha d\left(\frac{1}{p}-\frac{1}{q}\right)\right)},
\]
which is finite if and only if
\begin{align}
 & \frac{\theta}{1-\alpha}\left(n-\gamma+\alpha d\left(\frac{1}{p}-\frac{1}{q}\right)\right)<-d\nonumber \\
\Longleftrightarrow & n-\gamma+\alpha d\left(\frac{1}{p}-\frac{1}{q}\right)<\frac{d\left(\alpha-1\right)}{\theta}\nonumber \\
\Longleftrightarrow & \gamma>n+\alpha d\left(\frac{1}{p}-\frac{1}{q}\right)-\frac{d\left(\alpha-1\right)}{\theta}\nonumber \\
\Longleftrightarrow & \gamma>n+\alpha d\left(\frac{1}{p}-\frac{1}{q}-\frac{1}{\theta}\right)+\frac{d}{\theta}.\label{eq:AlphaModulationEmbeddingIntermediate}
\end{align}
For the remaining case $\theta=\infty$, it is easy to see $v/u^{\left(\gamma,\alpha\right)}\in\ell^{\theta}\left(\mathbb{Z}^{d}\setminus\left\{ 0\right\} \right)$
if and only if
\begin{align*}
 & \frac{1}{1-\alpha}\left(n-\gamma+\alpha d\left(\frac{1}{p}-\frac{1}{q}\right)\right)\leq0\\
\Longleftrightarrow & n-\gamma+\alpha d\left(\frac{1}{p}-\frac{1}{q}\right)\leq0\\
\Longleftrightarrow & \gamma\geq n+\alpha d\left(\frac{1}{p}-\frac{1}{q}\right),
\end{align*}
which is precisely the same condition as in equation (\ref{eq:AlphaModulationEmbeddingIntermediate}),
except that the strict inequality is replaced by a non-strict one.

Now, Corollary \ref{cor:SimplifiedSobolevEmbedding} shows that the
embedding $M_{\gamma,\alpha}^{p,s}\left(\mathbb{R}^{d}\right)\hookrightarrow W^{n,q}\left(\mathbb{R}^{d}\right)$
is valid as soon as $p\leq q$ and
\[
\frac{v}{u^{\left(\gamma,\alpha\right)}}\in\ell^{q^{\triangledown}\cdot\left(s/q^{\triangledown}\right)'}\left(\mathbb{Z}^{d}\setminus\left\{ 0\right\} \right).
\]
There are now two cases:
\begin{casenv}
\item If $s\leq q^{\triangledown}$, then equation (\ref{eq:SpecialExponentInfiniteCharacterization})
shows $q^{\triangledown}\cdot\left(s/q^{\triangledown}\right)'=\infty$,
so that our considerations above imply that $v/u^{\left(\gamma,\alpha\right)}\in\ell^{q^{\triangledown}\cdot\left(s/q^{\triangledown}\right)'}\left(\mathbb{Z}^{d}\setminus\left\{ 0\right\} \right)$
is equivalent to
\[
\gamma\geq n+\alpha d\left(\frac{1}{p}-\frac{1}{q}\right)=n+\alpha d\left(\frac{1}{p}-\frac{1}{q}\right)+d\left(1-\alpha\right)\left(\frac{1}{q^{\triangledown}}-\frac{1}{s}\right)_{+}.
\]

\item If $s>q^{\triangledown}$, then equation (\ref{eq:SpecialExponentInfiniteCharacterization})
shows $q^{\triangledown}\cdot\left(s/q^{\triangledown}\right)'<\infty$
and equation (\ref{eq:SpecialExponentReciprocal}) yields
\[
\frac{1}{q^{\triangledown}\cdot\left(s/q^{\triangledown}\right)'}=\left(\frac{1}{q^{\triangledown}}-\frac{1}{s}\right)_{+}.
\]
By equation (\ref{eq:AlphaModulationEmbeddingIntermediate}), we thus
see that $v/u^{\left(\gamma,\alpha\right)}\in\ell^{q^{\triangledown}\cdot\left(s/q^{\triangledown}\right)'}\left(\mathbb{Z}^{d}\setminus\left\{ 0\right\} \right)$
is equivalent to
\begin{align*}
\gamma & >n+\alpha d\left(\frac{1}{p}-\frac{1}{q}-\left(\frac{1}{q^{\triangledown}}-\frac{1}{s}\right)_{+}\right)+d\left(\frac{1}{q^{\triangledown}}-\frac{1}{s}\right)_{+}\\
 & =n+\alpha d\left(\frac{1}{p}-\frac{1}{q}\right)+d\left(1-\alpha\right)\left(\frac{1}{q^{\triangledown}}-\frac{1}{s}\right)_{+}.
\end{align*}

\end{casenv}
Altogether, we have shown that $M_{\gamma,\alpha}^{p,s}\left(\mathbb{R}^{d}\right)\hookrightarrow W^{n,q}\left(\mathbb{R}^{d}\right)$
holds as soon as $p\leq q$ and
\begin{equation}
\begin{cases}
\gamma\geq n+d\left[\alpha\left(\frac{1}{p}-\frac{1}{q}\right)+\left(1-\alpha\right)\left(\frac{1}{q^{\triangledown}}-\frac{1}{s}\right)_{+}\right], & \text{if }s\leq q^{\triangledown},\\
\gamma>n+d\left[\alpha\left(\frac{1}{p}-\frac{1}{q}\right)+\left(1-\alpha\right)\left(\frac{1}{q^{\triangledown}}-\frac{1}{s}\right)_{+}\right], & \text{if }s>q^{\triangledown}.
\end{cases}\label{eq:AlphaModulationEasySufficient}
\end{equation}
In case of $q\in\left(0,2\right]\cup\left\{ \infty\right\} $, Corollary
\ref{cor:SimplifiedSobolevEmbedding} also shows that this condition
is necessary for existence of the embedding.

For the analysis in case of $q\in\left(2,\infty\right)$, we again
use the ``Besov detour'' as described in Subsection \ref{sub:ComparisonAndBesovDetour}.
If we have $M_{\gamma,\alpha}^{p,s}\left(\mathbb{R}^{d}\right)\hookrightarrow W^{n,q}\left(\mathbb{R}^{d}\right)$,
then equation (\ref{eq:SobolevBesovInclusionHard}) yields
\[
M_{\gamma,\alpha}^{p,s}\left(\mathbb{R}^{d}\right)\hookrightarrow W^{n,q}\left(\mathbb{R}^{d}\right)\overset{q\in\left(1,\infty\right)}{=}F_{n}^{q,2}\left(\mathbb{R}^{d}\right)\hookrightarrow\mathcal{B}_{n}^{q,\max\left\{ q,2\right\} }\left(\mathbb{R}^{d}\right)=\mathcal{B}_{n}^{q,q}\left(\mathbb{R}^{d}\right).
\]
But by \cite[Theorem 6.2.8]{VoigtlaenderPhDThesis}, the embedding
$M_{\gamma,\alpha}^{p,s}\left(\mathbb{R}^{d}\right)\hookrightarrow\mathcal{B}_{n}^{q,q}\left(\mathbb{R}^{d}\right)$
holds if and only if $p\leq q$ and
\begin{equation}
\begin{cases}
\gamma\geq n+d\left[\alpha\left(\frac{1}{p}-\frac{1}{q}\right)+\left(1-\alpha\right)\left(\frac{1}{q^{\triangledown}}-\frac{1}{s}\right)_{+}\right], & \text{if }s\leq q,\\
\gamma>n+d\left[\alpha\left(\frac{1}{p}-\frac{1}{q}\right)+\left(1-\alpha\right)\left(\frac{1}{q^{\triangledown}}-\frac{1}{s}\right)_{+}\right], & \text{if }s>q.
\end{cases}\label{eq:AlphaModulationBesovSufficient}
\end{equation}
Similarly, for $q\in\left(2,\infty\right)$, the same theorem and
equation (\ref{eq:SobolevBesovInclusionHard}) also show that we have
\[
M_{\gamma,\alpha}^{p,s}\left(\mathbb{R}^{d}\right)\hookrightarrow\mathcal{B}_{n}^{q,2}\left(\mathbb{R}^{d}\right)=\mathcal{B}_{n}^{q,\min\left\{ q,2\right\} }\left(\mathbb{R}^{d}\right)\hookrightarrow F_{n}^{q,2}\left(\mathbb{R}^{d}\right)=W^{n,q}\left(\mathbb{R}^{d}\right)
\]
if $p\leq q$ and
\[
\begin{cases}
\gamma\geq n+d\left[\alpha\left(\frac{1}{p}-\frac{1}{q}\right)+\left(1-\alpha\right)\left(\frac{1}{q^{\triangledown}}-\frac{1}{s}\right)_{+}\right], & \text{if }s\leq2,\\
\gamma>n+d\left[\alpha\left(\frac{1}{p}-\frac{1}{q}\right)+\left(1-\alpha\right)\left(\frac{1}{q^{\triangledown}}-\frac{1}{s}\right)_{+}\right], & \text{if }s>2.
\end{cases}
\]
Thus, we obtain a complete characterization for $q\in\left(0,2\right]\cup\left\{ \infty\right\} $
and for $q\in\left(2,\infty\right)$ if $s\notin\left(2,q\right]$.
To obtain this last result, however, we had to resort to \cite{VoigtlaenderPhDThesis},
instead of the results in this paper.

I do not know if the condition (\ref{eq:AlphaModulationBesovSufficient})
is sharp \emph{in general}. But for $\alpha=0$ and $p=q\in\left(2,\infty\right)$,
the characterization given in \cite{KobayashiSugimotoModulationSobolevInclusion}
shows that $M_{\gamma,0}^{p,s}\left(\mathbb{R}^{d}\right)\hookrightarrow W^{n,p}\left(\mathbb{R}^{d}\right)$
holds if and only if condition (\ref{eq:AlphaModulationBesovSufficient})
(with $p=q$ and $\alpha=0$) is satisfied.
\end{example}

\begin{example}
\label{exa:ShearletSmoothnessSpaces}(Shearlet smoothness spaces)

Shearlet smoothness spaces have first been introduced by Labate et
al.\@ in \cite{Labate_et_al_Shearlet}. These spaces are defined
using a suitable covering of $\mathbb{R}^{2}$. This covering is (slightly
modified) given by (cf.\@ \cite[Definition 6.4.1]{VoigtlaenderPhDThesis})
\[
\mathcal{S}=\left(S_{i}\right)_{i\in I}=\left(T_{i}Q+b_{i}\right)_{i\in I},
\]
where $I=I_{0}\cup\left\{ 0\right\} $ with
\[
I_{0}=\left\{ \left(n,m,\varepsilon,\delta\right)\in\mathbb{N}_{0}\times\mathbb{Z}\times\left\{ \pm1\right\} \times\left\{ 0,1\right\} \with\:\left|m\right|\leq2^{n}\right\} 
\]
and $T_{0}:=4\cdot{\rm id}$, as well as $b_{0}:=\left(\begin{smallmatrix}-4\\
0
\end{smallmatrix}\right)$. Furthermore,
\[
Q:=U_{\left(-1,1\right)}^{\left(\frac{1}{3},3\right)}:=\left\{ \left(\begin{matrix}x\\
y
\end{matrix}\right)\in\left(\frac{1}{3},3\right)\times\mathbb{R}\with\frac{y}{x}\in\left(-1,1\right)\right\} 
\]
and
\[
T_{n,m,\varepsilon,\delta}:=\varepsilon\left(\begin{matrix}0 & 1\\
1 & 0
\end{matrix}\right)^{\delta}\cdot\left(\begin{matrix}2^{2n} & 0\\
2^{n}m & 2^{n}
\end{matrix}\right)\qquad\text{ as well as }\qquad b_{n,m,\varepsilon,\delta}:=0
\]
for $\left(n,m,\varepsilon,\delta\right)\in I_{0}$.

This covering is a slight modification of the one given by Labate
et al.\@, since the covering constructed in \cite{Labate_et_al_Shearlet}
is \emph{not} a structured admissible covering, contrary to the statement
of \cite[Proposition 4.1]{Labate_et_al_Shearlet}. In fact, the covering
given in \cite{Labate_et_al_Shearlet} does not admit an associated
partition of unity, since the interiors of the sets fail to cover
all of $\mathbb{R}^{2}$. The covering given here is indeed a structured
admissible covering, cf.\@ \cite[Lemma 6.4.2]{VoigtlaenderPhDThesis}.

Now, given $p,r\in\left(0,\infty\right]$ and $\beta\in\mathbb{R}$,
the shearlet smoothness space with parameters $p,r,\beta$ is given
by
\[
\mathcal{S}_{\beta}^{p,r}\left(\mathbb{R}^{2}\right)=\mathcal{D}\left(\mathcal{S},L^{p},\ell_{u^{\left(\beta\right)}}^{r}\right)
\]
with $u_{0}^{\left(\beta\right)}:=1$ and $u_{n,m,\varepsilon,\delta}^{\left(\beta\right)}=2^{2n\beta}$
for $\left(n,m,\varepsilon,\delta\right)\in I_{0}$. We are interested
in existence of the embedding $\mathcal{S}_{\beta}^{p,r}\left(\mathbb{R}^{2}\right)\hookrightarrow W^{k,q}\left(\mathbb{R}^{2}\right)$
for $q\in\left(0,\infty\right]$ and $k\in\mathbb{N}_{0}$. The relevant
weight from Corollary \ref{cor:SimplifiedSobolevEmbedding} (with
$n=k$) is given by
\begin{align*}
v_{i}:=w_{i}^{\left(q\right)} & =\left|\det T_{i}\right|^{\frac{1}{p}-\frac{1}{q}}\cdot\left(1+\left|b_{i}\right|^{k}+\left\Vert T_{i}\right\Vert ^{k}\right)\\
 & \asymp2^{n\left[2k+3\left(\frac{1}{p}-\frac{1}{q}\right)\right]}.
\end{align*}
for $i=\left(n,m,\varepsilon,\delta\right)\in I_{0}$. Here, we used
$\left|m\right|\leq2^{n}$ and $n\geq0$ for $i\in I_{0}$, so that
we get
\[
\left\Vert T_{i}\right\Vert =\left\Vert \left(\begin{matrix}2^{2n} & 0\\
2^{n}m & 2^{n}
\end{matrix}\right)\right\Vert \asymp\max\left\{ 2^{2n},2^{n}\left|m\right|,2^{n}\right\} =2^{2n}\geq1
\]
and $\left|\det T_{i}\right|=2^{3n}$, as well as $b_{i}=0$ for $i=\left(n,m,\varepsilon,\delta\right)\in I_{0}$.
Altogether, we get
\[
\frac{v_{i}}{u_{i}^{\left(\beta\right)}}\asymp2^{n\left[2k-2\beta+3\left(\frac{1}{p}-\frac{1}{q}\right)\right]}\text{ for }i=\left(n,m,\varepsilon,\delta\right)\in I_{0}.
\]

Since the single term for $i=0$ is irrelevant for membership of $v/u^{\left(\beta\right)}$
in $\ell^{\theta}\left(I\right)$ for $\theta\in\left(0,\infty\right)$,
we get
\begin{align*}
\frac{v}{u^{\beta}}\in\ell^{\theta}\left(I\right) & \Longleftrightarrow\left\Vert \left(\frac{v_{i}}{u_{i}^{\left(\beta\right)}}\right)_{i\in I_{0}}\right\Vert _{\ell^{\theta}}^{\theta}<\infty\\
 & \Longleftrightarrow\sum_{n=0}^{\infty}\sum_{m=-2^{n}}^{2^{n}}2^{\theta n\left[2k-2\beta+3\left(\frac{1}{p}-\frac{1}{q}\right)\right]}<\infty\\
 & \Longleftrightarrow\sum_{n=0}^{\infty}2^{n\left[1+\theta\left(2k-2\beta+3\left(\frac{1}{p}-\frac{1}{q}\right)\right)\right]}<\infty\:,
\end{align*}
which is equivalent to
\begin{align*}
 & 1+\theta\left(2k-2\beta+3\left(\frac{1}{p}-\frac{1}{q}\right)\right)<0\\
\Longleftrightarrow & 2k-2\beta+3\left(\frac{1}{p}-\frac{1}{q}\right)<-\frac{1}{\theta}\\
\Longleftrightarrow & k+\frac{3}{2}\left(\frac{1}{p}-\frac{1}{q}\right)+\frac{1}{2\theta}<\beta.
\end{align*}
In case of $\theta=\infty$, we see that $v/u^{\left(\beta\right)}\in\ell^{\theta}\left(I\right)$
if and only if $v/u^{\left(\beta\right)}$ is bounded if and only
if
\begin{align*}
 & 2k-2\beta+3\left(\frac{1}{p}-\frac{1}{q}\right)\leq0\\
\Longleftrightarrow & k+\frac{3}{2}\left(\frac{1}{p}-\frac{1}{q}\right)\leq\beta.
\end{align*}

Now, we are in a position to apply Corollary \ref{cor:SimplifiedSobolevEmbedding},
which shows that $\mathcal{S}_{\beta}^{p,r}\left(\mathbb{R}^{2}\right)\hookrightarrow W^{k,q}\left(\mathbb{R}^{2}\right)$
holds if we have $p\leq q$ and $\frac{v}{u^{\left(\beta\right)}}\in\ell^{q^{\triangledown}\cdot\left(r/q^{\triangledown}\right)'}\left(I\right)$.
There are now two cases:
\begin{casenv}
\item If $r\leq q^{\triangledown}$, then equation (\ref{eq:SpecialExponentInfiniteCharacterization})
shows $q^{\triangledown}\cdot\left(r/q^{\triangledown}\right)'=\infty$,
so that our considerations above imply that $v/u^{\left(\beta\right)}\in\ell^{q^{\triangledown}\cdot\left(r/q^{\triangledown}\right)'}\left(I\right)$
is equivalent to
\[
\beta\geq k+\frac{3}{2}\left(\frac{1}{p}-\frac{1}{q}\right)=k+\frac{3}{2}\left(\frac{1}{p}-\frac{1}{q}\right)+\frac{1}{2}\left(\frac{1}{q^{\triangledown}}-\frac{1}{r}\right)_{+}.
\]

\item If $r>q^{\triangledown}$, then equation (\ref{eq:SpecialExponentInfiniteCharacterization})
shows $q^{\triangledown}\cdot\left(r/q^{\triangledown}\right)'<\infty$
and equation (\ref{eq:SpecialExponentReciprocal}) yields
\[
\frac{1}{q^{\triangledown}\cdot\left(r/q^{\triangledown}\right)'}=\left(\frac{1}{q^{\triangledown}}-\frac{1}{r}\right)_{+}.
\]
By our considerations above, $v/u^{\left(\beta\right)}\in\ell^{q^{\triangledown}\cdot\left(r/q^{\triangledown}\right)'}\left(I\right)$
is thus equivalent to
\[
\beta>k+\frac{3}{2}\left(\frac{1}{p}-\frac{1}{q}\right)+\frac{1}{2}\left(\frac{1}{q^{\triangledown}}-\frac{1}{r}\right)_{+}.
\]

\end{casenv}
Altogether, we have shown that $\mathcal{S}_{\beta}^{p,r}\left(\mathbb{R}^{2}\right)\hookrightarrow W^{k,q}\left(\mathbb{R}^{2}\right)$
holds as soon as $p\leq q$ and
\[
\begin{cases}
\beta\geq k+\frac{3}{2}\left(\frac{1}{p}-\frac{1}{q}\right)+\frac{1}{2}\left(\frac{1}{q^{\triangledown}}-\frac{1}{r}\right)_{+}, & \text{if }r\leq q^{\triangledown}\\
\beta>k+\frac{3}{2}\left(\frac{1}{p}-\frac{1}{q}\right)+\frac{1}{2}\left(\frac{1}{q^{\triangledown}}-\frac{1}{r}\right)_{+}, & \text{if }r>q^{\triangledown}.
\end{cases}
\]
By Corollary \ref{cor:SimplifiedSobolevEmbedding}, for $q\in\left(0,2\right]\cup\left\{ \infty\right\} $,
these conditions are also necessary for existence of the embedding.

In fact, also for $q\in\left(2,\infty\right)$, our sufficient condition
turns out to be reasonably sharp. To see this, we again use the ``Besov
detour'', in conjunction with a result from \cite{VoigtlaenderPhDThesis}:
If the embedding $\mathcal{S}_{\beta}^{p,r}\left(\mathbb{R}^{2}\right)\hookrightarrow W^{k,q}\left(\mathbb{R}^{2}\right)$
holds, then equation (\ref{eq:SobolevBesovInclusionHard}) shows (recall
$q\in\left(2,\infty\right)\subset\left(1,\infty\right)$)
\[
\mathcal{S}_{\beta}^{p,r}\left(\mathbb{R}^{2}\right)\hookrightarrow W^{k,q}\left(\mathbb{R}^{2}\right)=F_{k}^{q,2}\left(\mathbb{R}^{2}\right)\hookrightarrow\mathcal{B}_{k}^{q,\max\left\{ q,2\right\} }\left(\mathbb{R}^{2}\right)=\mathcal{B}_{k}^{q,q}\left(\mathbb{R}^{2}\right).
\]
Using \cite[Theorem 6.4.3]{VoigtlaenderPhDThesis}, this implies that
$p\leq q$ and
\begin{equation}
\begin{cases}
\beta\geq k+\frac{3}{2}\left(\frac{1}{p}-\frac{1}{q}\right)+\frac{1}{2}\left(\frac{1}{q^{\triangledown}}-\frac{1}{r}\right)_{+} & \text{if }r\leq q,\\
\beta>k+\frac{3}{2}\left(\frac{1}{p}-\frac{1}{q}\right)+\frac{1}{2}\left(\frac{1}{q^{\triangledown}}-\frac{1}{r}\right)_{+}, & \text{if }r>q.
\end{cases}\label{eq:ShearletSmoothnessNecessaryBesov}
\end{equation}
Conversely, as in the case of $\alpha$-modulation spaces, the ``Besov
detour'', in conjunction with \cite[Theorem 6.4.3]{VoigtlaenderPhDThesis},
also shows that we have $\mathcal{S}_{\beta}^{p,r}\left(\mathbb{R}^{2}\right)\hookrightarrow W^{k,q}\left(\mathbb{R}^{2}\right)$
as soon as $p\leq q$ and
\[
\begin{cases}
\beta\geq k+\frac{3}{2}\left(\frac{1}{p}-\frac{1}{q}\right)+\frac{1}{2}\left(\frac{1}{q^{\triangledown}}-\frac{1}{r}\right)_{+} & \text{if }r\leq2,\\
\beta>k+\frac{3}{2}\left(\frac{1}{p}-\frac{1}{q}\right)+\frac{1}{2}\left(\frac{1}{q^{\triangledown}}-\frac{1}{r}\right)_{+}, & \text{if }r>2.
\end{cases}
\]

Thus, even for $q\in\left(2,\infty\right)$, the only difference
between this improved sufficient condition and the improved necessary
condition (\ref{eq:ShearletSmoothnessNecessaryBesov}) from above
is that the sufficient condition requires a strict inequality, while
the necessary criterion only yields a non-strict inequality. Furthermore,
this difference only occurs for $r\in\left(2,q\right]$.

Again, I do not know if the necessary condition (\ref{eq:ShearletSmoothnessNecessaryBesov})
is sharp in general. But in view of the results from \cite{KobayashiSugimotoModulationSobolevInclusion}
for modulation spaces (cf.\@ the previous example), I conjecture
(at least for $p=q\in\left(2,\infty\right)$) that the necessary condition
(\ref{eq:ShearletSmoothnessNecessaryBesov}) is also sufficient for
existence of the embedding $\mathcal{S}_{\beta}^{p,r}\left(\mathbb{R}^{2}\right)\hookrightarrow W^{k,q}\left(\mathbb{R}^{2}\right)$.
\end{example}

\begin{example}
\label{exa:ShearletCoorbitSpaces}(Shearlet coorbit spaces)

For $c\in\mathbb{R}$, let 
\[
H^{\left(c\right)}:=\left\{ \varepsilon\left(\begin{matrix}a & b\\
0 & a^{c}
\end{matrix}\right)\with a\in\left(0,\infty\right),\: b\in\mathbb{R},\:\varepsilon\in\left\{ \pm1\right\} \right\} ,
\]
denote the \textbf{Shearlet type group with parameter $c$} (as in
equation (\ref{eq:ShearletTypeGroup})). In \cite[Corollary 6.3.5 and Theorem 4.6.4]{VoigtlaenderPhDThesis},
it was shown that the coorbit space
\[
{\rm Co}\left(L_{v}^{p,r}\left(\mathbb{R}^{2}\rtimes H^{\left(c\right)}\right)\right)
\]
for $p,r\in\left(0,\infty\right]$ and a moderate weight $v:H^{\left(c\right)}\to\left(0,\infty\right)$,
is canonically isomorphic to a certain decomposition space $\mathcal{D}\left(\mathcal{S}^{\left(c\right)},L^{p},\ell_{u^{\left(r\right)}}^{r}\right)$,
where the weight $u^{\left(r\right)}$ depends on $r$ and on the
weight $v$. Here, the coorbit space is formed with respect to the
\textbf{quasi-regular representation} of $\mathbb{R}^{2}\rtimes H^{\left(c\right)}$,
which acts by translations and ($L^{2}$ normalized) dilations on
$L^{2}\left(\mathbb{R}^{2}\right)$. Furthermore, the space $L_{v}^{p,r}\left(\mathbb{R}^{2}\rtimes H^{\left(c\right)}\right)$
is a mixed, weighted Lebesgue space, with (quasi)-norm given by
\[
\left\Vert f\right\Vert _{L_{v}^{p,r}}:=\left[\int_{H^{\left(c\right)}}\left(v\left(h\right)\cdot\left\Vert f\left(\cdot,h\right)\right\Vert _{L^{p}\left(\mathbb{R}^{d}\right)}\right)^{r}\frac{{\rm d}h}{\left|\det h\right|}\right]^{1/r},
\]
with the usual modifications for $r=\infty$. Finally, the weight
$v:H^{\left(c\right)}\to\left(0,\infty\right)$ is called \textbf{moderate}
if it satisfies $v\left(xyz\right)\leq v_{0}\left(x\right)v\left(y\right)v_{0}\left(z\right)$
for all $x,y,z\in H^{\left(c\right)}$ and some locally bounded, submultiplicative
weight $v_{0}:H^{\left(c\right)}\to\left(0,\infty\right)$.

More precisely, the structured admissible covering $\mathcal{S}^{\left(c\right)}=\left(T_{i}^{\left(c\right)}Q\right)_{i\in I}$
of $\mathcal{O}=\mathbb{R}^{\ast}\times\mathbb{R}$ is defined as
follows (cf.\@ \cite[Corollary 6.3.5]{VoigtlaenderPhDThesis}): We
have $I=\mathbb{Z}^{2}\times\left\{ \pm1\right\} $ and
\[
T_{n,m,\varepsilon}^{\left(c\right)}=\varepsilon\cdot\left(\begin{matrix}2^{n} & 0\\
0 & 2^{nc}
\end{matrix}\right)\cdot\left(\begin{matrix}1 & 0\\
m & 1
\end{matrix}\right)\text{ for }\left(n,m,\varepsilon\right)\in I,
\]
as well as
\[
Q=\left\{ \left(\begin{matrix}x\\
y
\end{matrix}\right)\in\left(\frac{1}{2},2\right)\times\mathbb{R}\with\frac{y}{x}\in\left(-1,1\right)\right\} .
\]
Finally, the weight $u^{\left(r\right)}$ is given by
\[
u_{n,m,\varepsilon}^{\left(r\right)}=2^{-n\left(1+c\right)\left(\frac{1}{2}-\frac{1}{r}\right)}\cdot v\left(\left(T_{n,m,\varepsilon}^{\left(c\right)}\right)^{-T}\right)
\]
for $\left(n,m,\varepsilon\right)\in I$.

For the specific choice $c=\frac{1}{2}$ and weights $v=v^{\left(s\right)}$
of the form 
\[
v^{\left(s\right)}\left(\varepsilon\cdot\left(\begin{matrix}a & b\\
0 & a^{1/2}
\end{matrix}\right)\right)=a^{s},
\]
the coorbit space ${\rm Co}\left(L_{v^{\left(s\right)}}^{p,p}\left(\mathbb{R}^{2}\rtimes H^{\left(1/2\right)}\right)\right)$
was studied in \cite{Dahlke_etal_sh_coorbit1,Dahlke_etal_sh_coorbit2}.
In the second of those papers, the authors developed embeddings of
a \emph{subspace} of ${\rm Co}\left(L_{v^{\left(s\right)}}^{p,p}\left(\mathbb{R}^{2}\rtimes H^{\left(1/2\right)}\right)\right)$
into a \emph{sum} of two (homogeneous) Besov spaces. This is a very
different result than what we are interested in, namely establishing
an embedding
\[
{\rm Co}\left(L_{v}^{p,r}\left(\mathbb{R}^{2}\rtimes H^{\left(c\right)}\right)\right)\hookrightarrow W^{k,q}\left(\mathbb{R}^{2}\right)
\]
of the \emph{whole} coorbit space into a \emph{single Sobolev} space.

For our setting, it will turn out to be natural to consider weights
of the form $v=v^{\left(\alpha,\beta\right)}$ with
\[
v^{\left(\alpha,\beta\right)}:H^{\left(c\right)}\to\left(0,\infty\right),A=\varepsilon\left(\begin{matrix}a & b\\
0 & a^{c}
\end{matrix}\right)\mapsto a^{\alpha}\cdot\left\Vert A^{-T}\right\Vert ^{\beta}.
\]
Up to a slight transform in the parameters $\alpha,\beta$, this weight
coincides with the one considered in \cite[equation (6.3.15) and Theorem 6.3.11]{VoigtlaenderPhDThesis}
and is thus moderate.

With this choice, the weight $u^{\left(r\right)}$ from above is given
by
\[
u_{n,m,\varepsilon}^{\left(r\right)}=2^{-n\left(1+c\right)\left(\frac{1}{2}-\frac{1}{r}\right)}\cdot2^{-n\alpha}\cdot\left\Vert T_{n,m,\varepsilon}^{\left(c\right)}\right\Vert ^{\beta}.
\]
Furthermore, the weight $v:=w^{\left(q\right)}$ from Corollary \ref{cor:SimplifiedSobolevEmbedding}
satisfies
\[
v_{i}=\left|\det T_{i}^{\left(c\right)}\right|^{\frac{1}{p}-\frac{1}{q}}\cdot\left(1+\left|b_{i}\right|^{k}+\left\Vert T_{i}^{\left(c\right)}\right\Vert ^{k}\right)\text{ for }i\in I=\mathbb{Z}^{2}\times\left\{ \pm1\right\} ,
\]
so that the quotient $w=\left(w_{i}\right)_{i\in I}$ of the two weights
satisfies
\begin{align}
w_{n,m,\varepsilon} & :=\frac{v_{n,m,\varepsilon}}{u_{n,m,\varepsilon}^{\left(r\right)}}\nonumber \\
 & =2^{n\left[\alpha+\left(1+c\right)\left(\frac{1}{2}-\frac{1}{r}+\frac{1}{p}-\frac{1}{q}\right)\right]}\cdot\left(1+\left\Vert T_{n,m,\varepsilon}^{\left(c\right)}\right\Vert ^{k}\right)\cdot\left\Vert T_{n,m,\varepsilon}^{\left(c\right)}\right\Vert ^{-\beta}\nonumber \\
 & =2^{n\left[\alpha+\left(1+c\right)\gamma\right]}\cdot\left(\left\Vert T_{n,m,\varepsilon}^{\left(c\right)}\right\Vert ^{-\beta}+\left\Vert T_{n,m,\varepsilon}^{\left(c\right)}\right\Vert ^{-\left(\beta-k\right)}\right)\label{eq:ShearletCoorbitMainWeight}
\end{align}
with $\gamma:=\frac{1}{2}-\frac{1}{r}+\frac{1}{p}-\frac{1}{q}$. As
a preparation for the application of Corollary \ref{cor:SimplifiedSobolevEmbedding},
we want to characterize the condition $w\in\ell^{\theta}\left(I\right)$
in terms of $\alpha,\beta\in\mathbb{R}$ and $\theta\in\left(0,\infty\right]$.

To this end, note
\[
\left\Vert T_{n,m,\varepsilon}^{\left(c\right)}\right\Vert =\left\Vert \left(\begin{matrix}2^{n} & 0\\
2^{nc}m & 2^{nc}
\end{matrix}\right)\right\Vert \asymp2^{n}+2^{nc}+2^{nc}\left|m\right|.
\]
In particular $\left\Vert T_{0,m,\varepsilon}^{\left(c\right)}\right\Vert \asymp1+\left|m\right|$
and hence $w_{0,m,\varepsilon}\asymp\left(1+\left|m\right|\right)^{-\beta}+\left(1+\left|m\right|\right)^{-\left(\beta-k\right)}$.
Thus, if we have $w\in\ell^{\theta}\left(I\right)\subset\ell^{\infty}\left(I\right)$,
then $\beta-k\geq0$, i.e.\@ $\beta\geq k\geq0$.

In summary, our problem reduces to characterizing the condition $w^{\left(a,b\right)}\in\ell^{\theta}\left(\mathbb{Z}^{2}\right)$
for
\[
w_{n,m}^{\left(a,b\right)}:=2^{an}\cdot\left\Vert T_{n,m,1}^{\left(c\right)}\right\Vert ^{-b},
\]
where $a\in\mathbb{R}$ and $b\geq0$ are arbitrary. For this, we
distinguish two cases:
\begin{casenv}
\item $c\geq1$. Define 
\begin{align*}
M_{1} & :=\left\{ \left(n,m\right)\in\mathbb{Z}^{2}\with n\geq0\text{ and }m\neq0\right\} ,\\
M_{2} & :=\left\{ \left(n,m\right)\in\mathbb{Z}^{2}\with n\geq0\text{ and }m=0\right\} ,\\
M_{3} & :=\left\{ \left(n,m\right)\in\mathbb{Z}^{2}\with n<0\text{ and }\left|m\right|\geq\left\lceil 2^{n\left(1-c\right)}\right\rceil \right\} ,\\
M_{4} & :=\left\{ \left(n,m\right)\in\mathbb{Z}^{2}\with n<0\text{ and }\left|m\right|\leq\left\lceil 2^{n\left(1-c\right)}\right\rceil -1\right\} .
\end{align*}
We now distinguish the four subcases corresponding to $\left(n,m\right)\in M_{i}$
for $i\in\underline{4}$.

\begin{enumerate}
\item For $\left(n,m\right)\in M_{1}$, note $cn\geq n$ and $\left|m\right|\geq1$,
so that $2^{nc}\left|m\right|\geq2^{nc}\geq2^{n}$. All in all, this
yields $\left\Vert T_{n,m,\varepsilon}^{\left(c\right)}\right\Vert \asymp2^{nc}\left|m\right|$.
\item For $\left(n,m\right)\in M_{2}$, note again $cn\geq n$ and hence
$2^{nc}\geq2^{n}\geq0=2^{nc}\left|m\right|$. All in all, this yields
$\left\Vert T_{n,m,\varepsilon}^{\left(c\right)}\right\Vert \asymp2^{nc}$.
\item For $\left(n,m\right)\in M_{3}$, we have $\left|m\right|\geq\left\lceil 2^{n\left(1-c\right)}\right\rceil \geq2^{n\left(1-c\right)}$
and thus $2^{nc}\left|m\right|\geq2^{n}\geq2^{nc}$, where the last
step used that $n<0$ and $c\geq1$ imply $cn\leq n$. All in all,
we get $\left\Vert T_{n,m,\varepsilon}^{\left(c\right)}\right\Vert \asymp2^{nc}\left|m\right|$.
\item For $\left(n,m\right)\in M_{4}$, we have $\left|m\right|\!\leq\!\left\lceil 2^{n\left(1-c\right)}\right\rceil -1$.
Hence, $\left|m\right|\!\leq2^{n\left(1-c\right)}$, which implies
$2^{nc}\left|m\right|\!\leq2^{n}$. Since $n<0$ and $c\geq1$, we
also have $cn\leq n$ and thus $2^{nc}\leq2^{n}$. Altogether, this
yields $\left\Vert T_{n,m,\varepsilon}^{\left(c\right)}\right\Vert \asymp2^{n}$.
\end{enumerate}

In summary, we have shown
\[
w_{n,m}^{\left(a,b\right)}=2^{an}\cdot\left\Vert T_{n,m,\varepsilon}^{\left(c\right)}\right\Vert ^{-b}\asymp\begin{cases}
2^{an}\cdot\left(2^{nc}\left|m\right|\right)^{-b}=2^{n\left(a-bc\right)}\left|m\right|^{-b}, & \text{if }\left(n,m\right)\in M_{1},\\
2^{an}\cdot\left(2^{nc}\right)^{-b}=2^{n\left(a-bc\right)}, & \text{if }\left(n,m\right)\in M_{2},\\
2^{an}\cdot\left(2^{nc}\left|m\right|\right)^{-b}=2^{n\left(a-bc\right)}\cdot\left|m\right|^{-b}, & \text{if }\left(n,m\right)\in M_{3},\\
2^{an}\cdot\left(2^{n}\right)^{-b}=2^{n\left(a-b\right)}, & \text{if }\left(n,m\right)\in M_{4}.
\end{cases}
\]
Based on this asymptotic behaviour, we can now characterize finiteness
of $\left\Vert w^{\left(a,b\right)}\right\Vert _{\ell^{\infty}\left(M_{i}\right)}$
for each $i\in\underline{4}$.
\begin{enumerate}
\item On $M_{1}$, we see that 
\begin{align*}
\left\Vert w^{\left(a,b\right)}\right\Vert _{\ell^{\infty}\left(M_{1}\right)} & \asymp\left\Vert \left(2^{n\left(a-bc\right)}\cdot\left|m\right|^{-b}\right)_{n\geq0,m\neq0}\right\Vert _{\ell^{\infty}}\\
\left(b\geq0,\,\left|m\right|\geq1\right) & =\left\Vert \left(2^{n\left(a-bc\right)}\right)_{n\geq0}\right\Vert _{\ell^{\infty}}
\end{align*}
is finite if and only if $a-bc\leq0$.
\item The norm $\left\Vert w^{\left(a,b\right)}\right\Vert _{\ell^{\infty}\left(M_{2}\right)}\asymp\left\Vert \left(2^{n\left(a-bc\right)}\right)_{n\geq0}\right\Vert _{\ell^{\infty}}$
is finite iff the condition $a-bc\leq0$ from the preceding case
is satisfied.
\item On $M_{3}$, we see that
\begin{align*}
\left\Vert w^{\left(a,b\right)}\right\Vert _{\ell^{\infty}\left(M_{3}\right)} & \asymp\left\Vert \left(2^{n\left(a-bc\right)}\cdot\left|m\right|^{-b}\right)_{n<0,\,\left|m\right|\geq\left\lceil 2^{n\left(1-c\right)}\right\rceil }\right\Vert _{\ell^{\infty}}\\
\left(b\geq0\right) & =\left\Vert \left(2^{n\left(a-bc\right)}\cdot\left\lceil 2^{n\left(1-c\right)}\right\rceil ^{-b}\right)_{n<0}\right\Vert _{\ell^{\infty}}\\
 & \overset{\left(\ast\right)}{\asymp}\left\Vert \left(2^{n\left(a-bc\right)}\cdot2^{-nb\left(1-c\right)}\right)_{n<0}\right\Vert _{\ell^{\infty}}\\
 & =\left\Vert \left(2^{-n\left(b-a\right)}\right)_{n<0}\right\Vert _{\ell^{\infty}}
\end{align*}
is finite if and only if $b-a\leq0$. To justifty the step marked
with $\left(\ast\right)$, note that we have $n\left(1-c\right)\geq0$
for $n<0$ and thus $2^{n\left(1-c\right)}\geq1$, which easily yields
\[
2^{n\left(1-c\right)}\leq\left\lceil 2^{n\left(1-c\right)}\right\rceil \leq2^{n\left(1-c\right)}+1\leq2\cdot2^{n\left(1-c\right)}.
\]

\item On $M_{4}$, we see that
\[
\left\Vert w^{\left(a,b\right)}\right\Vert _{\ell^{\infty}\left(M_{4}\right)}\asymp\left\Vert \left(2^{n\left(a-b\right)}\right)_{n<0,\:\left|m\right|\leq\left\lceil 2^{n\left(1-c\right)}\right\rceil -1}\right\Vert _{\ell^{\infty}}=\left\Vert \left(2^{-n\left(b-a\right)}\right)_{n<0}\right\Vert _{\ell^{\infty}}
\]
is finite if and only if $b-a\leq0$.
\end{enumerate}

In summary, we see (for $b\geq0$) that $\left\Vert w^{\left(a,b\right)}\right\Vert _{\ell^{\infty}\left(\mathbb{Z}^{2}\right)}$
is finite if and only if we have $b\leq a\leq bc$.

\noindent Now, we characterize the condition $\left\Vert w^{\left(a,b\right)}\right\Vert _{\ell^{\theta}\left(\mathbb{Z}^{2}\right)}<\infty$
for $b\geq0$ in terms of $a,b$ and $\theta\in\left(0,\infty\right)$.
As above, we distinguish four cases:
\begin{enumerate}
\item On $M_{1}$, we see that
\begin{align*}
\left\Vert w^{\left(a,b\right)}\right\Vert _{\ell^{\theta}\left(M_{1}\right)}^{\theta} & =\sum_{n=0}^{\infty}\left[\sum_{m\in\mathbb{Z}\setminus\left\{ 0\right\} }\left(w_{n,m}^{\left(a,b\right)}\right)^{\theta}\right]\\
 & \asymp\left(\sum_{n=0}^{\infty}2^{n\theta\left(a-bc\right)}\right)\cdot\left(\sum_{m=1}^{\infty}m^{-\theta b}\right)
\end{align*}
is finite if and only if $a-bc<0$ and $\theta b>1$.
\item On $M_{2}$, we see that
\[
\left\Vert w^{\left(a,b\right)}\right\Vert _{\ell^{\theta}\left(M_{2}\right)}^{\theta}\asymp\sum_{n=0}^{\infty}2^{n\theta\left(a-bc\right)}
\]
is finite if and only if $a-bc<0$, which is already implied by finiteness
of $\left\Vert w^{\left(a,b\right)}\right\Vert _{\ell^{\theta}\left(M_{1}\right)}$.
\item On $M_{3}$, we see that (since we can assume that $b\theta>1$ by
finiteness of $\left\Vert w^{\left(a,b\right)}\right\Vert _{\ell^{\theta}\left(M_{1}\right)}$),
finiteness of
\begin{align*}
\left\Vert w^{\left(a,b\right)}\right\Vert _{\ell^{\theta}\left(M_{3}\right)}^{\theta} & \asymp\sum_{n=-\infty}^{-1}\sum_{\left|m\right|\geq\left\lceil 2^{n\left(1-c\right)}\right\rceil }\left(w_{n,m}^{\left(a,b\right)}\right)^{\theta}\\
 & \asymp\sum_{n=1}^{\infty}\left(2^{-n\theta\left(a-bc\right)}\sum_{m=\left\lceil 2^{-n\left(1-c\right)}\right\rceil }^{\infty}m^{-b\theta}\right)\\
\left(\text{since }b\theta>1\right) & \overset{\left(\ast\right)}{\asymp}\sum_{n=1}^{\infty}\left(2^{-n\theta\left(a-bc\right)}\cdot\left\lceil 2^{-n\left(1-c\right)}\right\rceil ^{1-b\theta}\right)\\
\left(\text{since }n\left(c-1\right)\geq0\right) & \asymp\sum_{n=1}^{\infty}\left(2^{-n\theta\left(a-bc\right)}\cdot2^{n\left(c-1\right)\left(1-b\theta\right)}\right)\\
 & =\sum_{n=1}^{\infty}\left(2^{n\left(c+\theta\left(b-a\right)-1\right)}\right)
\end{align*}
is equivalent (under the assumption $\left\Vert w^{\left(a,b\right)}\right\Vert _{\ell^{\theta}\left(M_{1}\right)}<\infty$)
to
\[
c+\theta\left(b-a\right)-1<0.
\]
Here, at $\left(\ast\right)$, we used (for $m_{0}=\left\lceil 2^{n\left(c-1\right)}\right\rceil \in\mathbb{N}$)
the estimate
\begin{align}
\sum_{m=m_{0}}^{\infty}m^{\varrho} & =\sum_{m=m_{0}}^{\infty}\int_{m}^{m+1}m^{\varrho}\,{\rm d}x\nonumber \\
 & \overset{\left(\dagger\right)}{\asymp_{\varrho}}\sum_{m=m_{0}}^{\infty}\int_{m}^{m+1}x^{\varrho}\,{\rm d}x\nonumber \\
 & =\int_{m_{0}}^{\infty}x^{\varrho}\,{\rm d}x\nonumber \\
 & =\begin{cases}
\frac{x^{1+\varrho}}{1+\varrho}\bigg|_{m_{0}}^{\infty}=\frac{m_{0}^{1+\varrho}}{-\left(1+\varrho\right)}\asymp_{\varrho}\; m_{0}^{1+\varrho}, & \text{if }\varrho<-1,\\
\infty, & \text{if }\varrho\geq-1,
\end{cases}\label{eq:PowerSeriesAsymptotic}
\end{align}
which is justified, since inside the integral at $\left(\dagger\right)$,
we have $1\leq m_{0}\leq m\leq x\leq m+1$ and thus $m\leq x\leq m+1\leq2m$,
which finally yields $m^{\varrho}\asymp_{\varrho}x^{\varrho}$ for
arbitrary $\varrho\in\mathbb{R}$.
\item On $M_{4}$, the norm
\begin{align*}
\left\Vert w^{\left(a,b\right)}\right\Vert _{\ell^{\theta}\left(M_{4}\right)}^{\theta} & \asymp\sum_{n=-\infty}^{-1}\sum_{m=-\left\lceil 2^{n\left(1-c\right)}\right\rceil +1}^{\left\lceil 2^{n\left(1-c\right)}\right\rceil -1}2^{n\theta\left(a-b\right)}\\
 & =\sum_{\ell=1}^{\infty}2^{-\ell\theta\left(a-b\right)}\left(2\cdot\left\lceil 2^{-\ell\left(1-c\right)}\right\rceil -1\right)\\
 & \overset{\left(\ast\right)}{\asymp}\sum_{\ell=1}^{\infty}2^{-\ell\theta\left(a-b\right)}2^{-\ell\left(1-c\right)}\\
 & =\sum_{\ell=1}^{\infty}2^{-\ell\left[1-c+\theta\left(a-b\right)\right]}
\end{align*}
is finite if and only if $1-c+\theta\left(a-b\right)>0$. Here, the
step marked with $\left(\ast\right)$ used that we have for $\ell\in\mathbb{N}$
that $-\ell\left(1-c\right)\geq0$ and hence $2^{-\ell\left(1-c\right)}\geq1$,
which implies
\[
\qquad\qquad\qquad\quad2^{-\ell\left(1-c\right)}\leq\left\lceil 2^{-\ell\left(1-c\right)}\right\rceil \leq2\!\cdot\!\left\lceil 2^{-\ell\left(1-c\right)}\right\rceil -1\leq2\!\cdot\!\left\lceil 2^{-\ell\left(1-c\right)}\right\rceil \leq2\!\cdot\!\left(2^{-\ell\left(1-c\right)}+1\right)\leq4\cdot2^{-\ell\left(1-c\right)}.
\]

\end{enumerate}

\noindent Altogether, we have shown for $\theta\in\left(0,\infty\right)$
and $b\geq0$ that $\left\Vert w^{\left(a,b\right)}\right\Vert _{\ell^{\theta}\left(\mathbb{Z}^{2}\right)}$
is finite if and only if we have
\begin{align*}
 & b\theta>1\text{ and }a-bc<0\text{ and }1-c+\theta\left(a-b\right)>0\\
\Longleftrightarrow & b\theta>1\text{ and }a\in\left(b+\frac{c-1}{\theta},\: cb\right).
\end{align*}

\item $c\in\left(-\infty,1\right)$. Here, we use the modified sets
\begin{align*}
M_{1} & :=\left\{ \left(n,m\right)\in\mathbb{Z}^{2}\with n\geq0\text{ and }\left|m\right|\leq\left\lceil 2^{\left(1-c\right)n}\right\rceil \right\} ,\\
M_{2} & :=\left\{ \left(n,m\right)\in\mathbb{Z}^{2}\with n\geq0\text{ and }\left|m\right|\geq\left\lceil 2^{\left(1-c\right)n}\right\rceil +1\right\} ,\\
M_{3} & :=\left\{ \left(n,m\right)\in\mathbb{Z}^{2}\with n<0\text{ and }m\neq0\right\} ,\\
M_{4} & :=\left\{ \left(n,m\right)\in\mathbb{Z}^{2}\with n<0\text{ and }m=0\right\} .
\end{align*}
Completely analogous to the previous case, one can show
\[
w_{n,m}^{\left(a,b\right)}=2^{an}\cdot\left\Vert T_{n,m,\varepsilon}^{\left(c\right)}\right\Vert ^{-b}\asymp\begin{cases}
2^{an}\cdot2^{-bn}=2^{n\left(a-b\right)}, & \text{if }\left(n,m\right)\in M_{1},\\
2^{an}\cdot\left(2^{cn}\left|m\right|\right)^{-b}=2^{n\left(a-bc\right)}\left|m\right|^{-b}, & \text{if }\left(n,m\right)\in M_{2},\\
2^{an}\cdot\left(2^{nc}\left|m\right|\right)^{-b}=2^{n\left(a-bc\right)}\left|m\right|^{-b}, & \text{if }\left(n,m\right)\in M_{3},\\
2^{an}\cdot\left(2^{nc}\right)^{-b}=2^{n\left(a-bc\right)}, & \text{if }\left(n,m\right)\in M_{4}.
\end{cases}
\]
Using this simplified expression for $w^{\left(a,b\right)}$, one
can characterize finiteness of $\left\Vert w^{\left(a,b\right)}\right\Vert _{\ell^{\infty}\left(M_{i}\right)}$
for $i\in\underline{4}$ as in the previous case. The results are
as follows:

\begin{enumerate}
\item On $M_{1}$,  $\left\Vert w^{\left(a,b\right)}\right\Vert _{\ell^{\infty}\left(M_{1}\right)}$
is finite if and only if $a-b\leq0$.
\item On $M_{2}$,  $\left\Vert w^{\left(a,b\right)}\right\Vert _{\ell^{\infty}\left(M_{2}\right)}$
is finite if and only if $a-b\leq0$.
\item On $M_{3}$,  $\left\Vert w^{\left(a,b\right)}\right\Vert _{\ell^{\infty}\left(M_{3}\right)}$
is finite if and only if $a-bc\geq0$.
\item On $M_{4}$,  $\left\Vert w^{\left(a,b\right)}\right\Vert _{\ell^{\infty}\left(M_{4}\right)}$
is finite if and only if $a-bc\geq0$.
\end{enumerate}

\noindent Altogether, we see that $\left\Vert w^{\left(a,b\right)}\right\Vert _{\ell^{\infty}\left(\mathbb{Z}^{2}\right)}$
is finite if and only if we have $bc\leq a\leq b$.

Finally, one can characterize the condition $\left\Vert w^{\left(a,b\right)}\right\Vert _{\ell^{\theta}\left(\mathbb{Z}^{2}\right)}<\infty$
for $b\geq0$ in terms of $a,b$ and $\theta\in\left(0,\infty\right)$,
as for $c\geq1$. The results are as follows:
\begin{enumerate}
\item On $M_{1}$, $\left\Vert w^{\left(a,b\right)}\right\Vert _{\ell^{\theta}\left(M_{1}\right)}^{\theta}$
is finite if and only if $1-c+\theta\left(a-b\right)<0$.
\item On $M_{2}$,  $\left\Vert w^{\left(a,b\right)}\right\Vert _{\ell^{\theta}\left(M_{2}\right)}^{\theta}$
is finite if and only if we have $b\theta>1$ and $\theta\left(a-b\right)+1-c<0$.
\item On $M_{3}$,  $\left\Vert w^{\left(a,b\right)}\right\Vert _{\ell^{\theta}\left(M_{3}\right)}^{\theta}$
is finite if and only if we have $b\theta>1$ and $\theta\left(bc-a\right)<0$,
which is equivalent to $bc-a<0$.
\item On $M_{4}$,  $\left\Vert w^{\left(a,b\right)}\right\Vert _{\ell^{\theta}\left(M_{4}\right)}^{\theta}$
is finite if and only if we have $\theta\left(bc-a\right)<0$, i.e.\@
if and only if $bc-a<0$.
\end{enumerate}

\noindent Altogether, we see for $\theta\in\left(0,\infty\right)$
and $b\geq0$ that $\left\Vert w^{\left(a,b\right)}\right\Vert _{\ell^{\theta}\left(\mathbb{Z}^{2}\right)}$
is finite if and only if we have
\begin{align*}
 & b\theta>1\text{ and }bc-a<0\text{ and }1-c+\theta\left(a-b\right)<0\\
\Longleftrightarrow & b\theta>1\text{ and }a\in\left(cb,\: b+\frac{c-1}{\theta}\right).
\end{align*}

\end{casenv}

Finally, recall that we wanted to characterize finiteness of $\left\Vert w\right\Vert _{\ell^{\theta}\left(I\right)}$
with $w$ as in equation (\ref{eq:ShearletCoorbitMainWeight}), i.e.\@
with 
\[
w_{n,m,\varepsilon}\asymp w_{n,m}^{\left(\alpha+\left(1+c\right)\gamma,\:\beta\right)}+w_{n,m}^{\left(\alpha+\left(1+c\right)\gamma,\;\beta-k\right)}.
\]
Using the characterization that we just obtained, we see that $\left\Vert w\right\Vert _{\ell^{\theta}\left(I\right)}$
is finite if and only if the following holds:
\begin{equation}
\begin{cases}
\beta\overset{!}{\geq}k\text{ and }\beta\overset{!}{\leq}\alpha+\left(1+c\right)\gamma\overset{!}{\leq}c\left(\beta-k\right), & \text{if }c\geq1\text{ and }\theta=\infty,\\
\beta\overset{!}{\geq}k\text{ and }\max\left\{ c\beta,c\left(\beta-k\right)\right\} \overset{!}{\leq}\alpha+\left(1+c\right)\gamma\overset{!}{\leq}\beta-k, & \text{if }c<1\text{ and }\theta=\infty,\\
\beta\overset{!}{>}k+\frac{1}{\theta}\text{ and }\beta+\frac{c-1}{\theta}\overset{!}{<}\alpha+\left(1+c\right)\gamma\overset{!}{<}c\left(\beta-k\right), & \text{if }c\geq1\text{ and }\theta<\infty,\\
\beta\overset{!}{>}k+\frac{1}{\theta}\text{ and }\max\left\{ c\beta,c\left(\beta-k\right)\right\} \overset{!}{<}\alpha+\left(1+c\right)\gamma\overset{!}{<}\beta-k+\frac{c-1}{\theta}, & \text{if }c<1\text{ and }\theta<\infty.
\end{cases}\label{eq:ShearletCoorbitMainCondition}
\end{equation}
Here, we used $\gamma=\frac{1}{2}-\frac{1}{r}+\frac{1}{p}-\frac{1}{q}$,
cf.\@ equation (\ref{eq:ShearletCoorbitMainWeight}). Note
\[
\max\left\{ c\beta,c\left(\beta-k\right)\right\} =\begin{cases}
c\beta, & \text{if }c\geq0,\\
c\left(\beta-k\right), & \text{if }c<0.
\end{cases}
\]

Finally, we are in a position to apply Corollary \ref{cor:SimplifiedSobolevEmbedding}.
A sufficient condition for existence of the embedding 
\begin{equation}
{\rm Co}\left(L_{v^{\left(\alpha,\beta\right)}}^{p,r}\left(\mathbb{R}^{2}\rtimes H^{\left(c\right)}\right)\right)\cong\mathcal{D}\left(\mathcal{S}^{\left(c\right)},L^{p},\ell_{u^{\left(r\right)}}^{r}\right)\hookrightarrow W^{k,q}\left(\mathbb{R}^{2}\right)\label{eq:ShearletCoorbitEmbedding}
\end{equation}
is that $p\leq q$ and $w\in\ell^{q^{\triangledown}\cdot\left(r/q^{\triangledown}\right)'}\left(\mathbb{Z}^{2}\times\left\{ \pm1\right\} \right)$.
In view of equations (\ref{eq:ShearletCoorbitMainCondition}), (\ref{eq:SpecialExponentInfiniteCharacterization})
and (\ref{eq:SpecialExponentReciprocal}), the second part of this
condition is equivalent to the following:
\begin{equation}
\begin{cases}
\beta\geq k\text{ and }\beta\leq\alpha+\left(1+c\right)\gamma\leq c\left(\beta-k\right), & \text{if }c\geq1\text{ and }r\leq q^{\triangledown},\\
\beta\geq k\text{ and }\max\left\{ c\beta,c\left(\beta-k\right)\right\} \leq\alpha+\left(1+c\right)\gamma\leq\beta-k, & \text{if }c<1\text{ and }r\leq q^{\triangledown},\\
\beta>k+\frac{1}{q^{\triangledown}}-\frac{1}{r}\text{ and }\beta+\left(c-1\right)\left(\frac{1}{q^{\triangledown}}-\frac{1}{r}\right)<\alpha+\left(1+c\right)\gamma<c\left(\beta-k\right), & \text{if }c\geq1\text{ and }r>q^{\triangledown},\\
\beta>k+\frac{1}{q^{\triangledown}}-\frac{1}{r}\text{ and }\max\left\{ c\beta,c\left(\beta-k\right)\right\} <\alpha+\left(1+c\right)\gamma<\beta-k+\left(c-1\right)\left(\frac{1}{q^{\triangledown}}-\frac{1}{r}\right), & \text{if }c<1\text{ and }r>q^{\triangledown}.
\end{cases}\label{eq:ShearletMainConditionSimplified}
\end{equation}
Furthermore, Corollary \ref{cor:SimplifiedSobolevEmbedding} shows
for $q\in\left(0,2\right]\cup\left\{ \infty\right\} $ that these
conditions are also necessary for existence of the embedding (\ref{eq:ShearletCoorbitEmbedding}).
Finally, a necessary (but (probably) not sufficient) condition for
existence of the embedding in case of $q\in\left(2,\infty\right)$
is obtained by replacing $q^{\triangledown}$ by $q$ everywhere in
the sufficient condition (\ref{eq:ShearletMainConditionSimplified})
(and by additionally requiring $p\leq q$).

In the present setting, I refrain from invoking the ``Besov detour''
for $q\in\left(2,\infty\right)$, since the necessary and sufficient
conditions for existence of the embedding $\mathcal{D}\left(\mathcal{S}^{\left(c\right)},L^{p},\ell_{u}^{r}\right)\hookrightarrow\mathcal{B}_{n}^{q,s}\left(\mathbb{R}^{2}\right)$
from \cite[Theorem 6.3.12]{VoigtlaenderPhDThesis} lack sharpness
for $q\in\left(2,\infty\right)$, since the covering $\mathcal{S}^{\left(c\right)}$
is not moderate%
\footnote{See \cite[Definition 3.3.1]{VoigtlaenderPhDThesis} for the definition
of moderateness of one covering with respect to another one.%
} with respect to the dyadic ``Besov'' covering. Thus, the results
in \cite{VoigtlaenderPhDThesis} do not yield an improvement in this
case.

Before closing this example, I want to draw attention to a special
case of the characterization from above: The starting point for the
present paper was the following question posed by Holger Rauhut:
\begin{quote}
When is there an embedding of a shearlet coorbit space into a BV space?
\end{quote}
In view of Corollary \ref{cor:EmbeddingIntoBVEquivalentToSobolevEmbedding},
this is equivalent to asking when ${\rm Co}\left(L_{v^{\left(\alpha,\beta\right)}}^{p,r}\left(\mathbb{R}^{2}\rtimes H^{\left(1/2\right)}\right)\right)\hookrightarrow W^{1,1}\left(\mathbb{R}^{2}\right)$
is true. This amounts to choosing $k=q=1$ and $c=\frac{1}{2}$. Because
of $q=1\in\left(0,2\right]$, we thus obtain a \emph{complete solution}
to this question: The desired embedding holds \emph{if and only if}
we have $p\leq1$ and
\[
\begin{cases}
\beta\geq1\text{ and }\frac{\beta}{2}\leq\alpha+\frac{3}{2}\left(\frac{1}{p}-\frac{1}{r}-\frac{1}{2}\right)\leq\beta-1, & \text{if }r\leq1,\\
\beta>2-\frac{1}{r}\text{ and }\frac{\beta}{2}<\alpha+\frac{3}{2}\left(\frac{1}{p}-\frac{1}{r}-\frac{1}{2}\right)<\beta-1-\frac{1}{2}\left(1-\frac{1}{r}\right), & \text{if }r>1.
\end{cases}
\]
Slightly simplified, this last condition reads
\begin{equation}
\begin{cases}
\beta\geq1\text{ and }\frac{\beta}{2}+\frac{3}{4}\leq\alpha+\frac{3}{2}\left(\frac{1}{p}-\frac{1}{r}\right)\leq\beta-\frac{1}{4}, & \text{if }r\leq1,\\
\beta>2-\frac{1}{r}\text{ and }\frac{\beta}{2}+\frac{3}{4}<\alpha+\frac{3}{2}\left(\frac{1}{p}-\frac{1}{r}\right)<\beta-\frac{3}{4}+\frac{1}{2r}, & \text{if }r>1.
\end{cases}\label{eq:ShearletBVEmbeddingSimplified}
\end{equation}
Note that we have $\left(\beta-\frac{1}{4}\right)-\left(\frac{\beta}{2}+\frac{3}{4}\right)=\frac{\beta}{2}-1\geq0$
if and only if $\beta\geq2$. For $r\leq1$, this shows that the condition
 in equation (\ref{eq:ShearletBVEmbeddingSimplified}) can be satisfied
(with a suitable choice of $\alpha\in\mathbb{R}$) if and only if
$\beta\geq2$. Likewise, $\left(\beta-\frac{3}{4}+\frac{1}{2r}\right)-\left(\frac{\beta}{2}+\frac{3}{4}\right)=\frac{\beta}{2}-\frac{3}{2}+\frac{1}{2r}>0$
if and only if $\beta>3-\frac{1}{r}$. This shows that for $r>1$,
the condition in equation (\ref{eq:ShearletBVEmbeddingSimplified})
can be satisfied (for a suitable choice of $\alpha\in\mathbb{R}$)
if and only if $\beta>3-\frac{1}{r}$.

As concluding remarks, we note the following:
\begin{enumerate}
\item The condition $\beta\geq k$ is always necessary for the embedding
${\rm Co}\left(L_{v^{\left(\alpha,\beta\right)}}^{p,r}\left(\mathbb{R}^{2}\rtimes H^{\left(c\right)}\right)\right)\hookrightarrow W^{k,q}\left(\mathbb{R}^{2}\right)$.
Thus, using the weight $v^{\left(s\right)}$ as in \cite{Dahlke_etal_sh_coorbit1,Dahlke_etal_sh_coorbit2}
(which amounts to choosing $\beta=0$), one can \emph{never} obtain
an embedding into a Sobolev space $W^{k,q}\left(\mathbb{R}^{2}\right)$
of \emph{positive} smoothness $k\geq1$.
\item Using the approach in the present paper, we were able to handle arbitrary
$c\in\mathbb{R}$, whereas in \cite[Theorem 6.3.11]{VoigtlaenderPhDThesis},
only the range $c\in\left(0,1\right]$ was considered, since in the
remaining range, the dyadic ``Besov covering'' is \emph{incompatible}
with the covering $\mathcal{S}^{\left(c\right)}$, i.e.\@ neither
is subordinate to the other one. This shows that, in addition to being
easier to apply, the results in the present paper can be employed
in some settings in which the findings from \cite{VoigtlaenderPhDThesis}
are not applicable.\qedhere
\end{enumerate}
\end{example}

\begin{example}
\label{exa:DiagonalCoorbitSpaces}(Coorbit spaces of the diagonal
group)

In this example, we consider coorbit spaces of the \textbf{diagonal
group} $D\leq{\rm GL}\left(\mathbb{R}^{d}\right)$ for $d\in\mathbb{N}$,
given by
\[
D=\left\{ {\rm diag}\left(a_{1},\dots,a_{d}\right)=\left(\begin{matrix}a_{1}\\
 & \ddots\\
 &  & a_{d}
\end{matrix}\right)\with a_{1},\dots,a_{d}\in\mathbb{R}^{\ast}=\mathbb{R}\setminus\left\{ 0\right\} \right\} .
\]
The general setup for the definition of coorbit spaces with respect
to $D$ is as in the case of shearlet-type coorbit spaces (cf.\@
Example \ref{exa:ShearletCoorbitSpaces}).

It is easy to see that the dual action $D\times\mathbb{R}^{d}\to\mathbb{R}^{d},\left(h,\xi\right)\mapsto h^{-T}\xi$
has the unique open(!) orbit $\mathcal{O}=\left(\mathbb{R}^{\ast}\right)^{d}$,
on which it acts with trivial (and hence compact) stabilizers. Thus,
$D$ is an admissible dilation group in the sense of \cite{FuehrVoigtlaenderCoorbitSpacesAsDecompositionSpaces}
and \cite[Chapter 4]{VoigtlaenderPhDThesis}.

Now, let $I:=\mathbb{Z}^{d}\times\left\{ \pm1\right\} ^{d}$ and define
\[
A_{k,\varepsilon}:={\rm diag}\left(\varepsilon_{1}\cdot2^{k_{1}},\dots,\varepsilon_{d}\cdot2^{k_{d}}\right)\in D\qquad\text{ for }\left(k,\varepsilon\right)\in I.
\]
We claim that $\left(A_{i}\right)_{i\in I}$ is well-spread in $D$.
To this end, note that 
\[
U_{0}:=D\cap\left\{ A\in\mathbb{R}^{d\times d}\with\left\Vert A-{\rm id}\right\Vert <\frac{1}{2}\right\} 
\]
is a unit neighborhood in $D$. Thus, there is a (smaller) unit neighborhood
$U\subset D$ with $U=U^{-1}$ and $UU\subset U_{0}$.

We claim that $\left(A_{i}\right)_{i\in I}$ is $U$-separated. Indeed,
it is easy to see $A_{k,\varepsilon}^{-1}A_{\ell,\delta}=A_{\varepsilon\delta,\ell-k}$
and hence 
\begin{align*}
\left\Vert A_{k,\varepsilon}^{-1}A_{\ell,\delta}-{\rm id}\right\Vert  & =\left\Vert {\rm diag}\left(\left(\varepsilon\delta\right)_{1}\cdot2^{\ell_{1}-k_{1}}-1,\dots,\left(\varepsilon\delta\right)_{d}\cdot2^{\ell_{d}-k_{d}}-1\right)\right\Vert \\
 & =\max\left\{ \left|1-\left(\varepsilon\delta\right)_{i}\cdot2^{\ell_{i}-k_{i}}\right|\with i\in\underline{d}\right\} .
\end{align*}
Now, if $\varepsilon_{i}\neq\delta_{i}$ for some $i\in\underline{d}$,
this implies $\left(\varepsilon\delta\right)_{i}\cdot2^{\ell_{i}-k_{i}}<0$
and thus $\left\Vert A_{k,\varepsilon}^{-1}A_{\ell,\delta}-{\rm id}\right\Vert >1>\frac{1}{2}$.
If otherwise $\varepsilon_{i}=\delta_{i}$ for all $i\in\underline{d}$,
then $\left\Vert A_{k,\varepsilon}^{-1}A_{\ell,\delta}-{\rm id}\right\Vert <\frac{1}{2}$
implies $2^{\ell_{i}-k_{i}}\in B_{1/2}\left(1\right)=\left(\frac{1}{2},\frac{3}{2}\right)\subset\left(\frac{1}{2},2\right)$
and hence $\ell_{i}=k_{i}$ for all $i\in\underline{d}$. We have
thus shown $\left\Vert A_{k,\varepsilon}^{-1}A_{\ell,\delta}-{\rm id}\right\Vert \geq\frac{1}{2}$
if $\left(k,\varepsilon\right)\neq\left(\ell,\delta\right)$. But
$A_{k,\varepsilon}U\cap A_{\ell,\delta}U\neq\emptyset$ implies $A_{k,\varepsilon}^{-1}A_{\ell,\delta}\in UU^{-1}=UU\subset U_{0}$
and hence $\left\Vert A_{k,\varepsilon}^{-1}A_{\ell,\delta}-{\rm id}\right\Vert <\frac{1}{2}$,
i.e.\@ $\left(k,\varepsilon\right)=\left(\ell,\delta\right)$. Thus,
$\left(A_{i}\right)_{i\in I}$ is indeed $U$-separated.

Now, since $D$ is homeomorphic to $\left(\mathbb{R}^{\ast}\right)^{d}$,
it is easy to see that 
\[
V:=\left\{ {\rm diag}\left(a_{1},\dots,a_{d}\right)\with a_{1},\dots,a_{d}\in\left[\frac{1}{2},2\right]\right\} \subset D
\]
is a compact unit neighborhood. But $\left(A_{i}\right)_{i\in I}$
is $V$-dense and thus well-spread in $D$. Indeed, for arbitrary
$A={\rm diag}\left(a_{1},\dots,a_{d}\right)\in D$, we can choose
\[
\varepsilon_{i}:={\rm sgn}\left(a_{i}\right)\in\left\{ \pm1\right\} \qquad\text{ and }\qquad k_{i}:=\left\lfloor \log_{2}\left(\left|a_{i}\right|\right)\right\rfloor \in\mathbb{Z}\qquad\text{ for }i\in\underline{d}.
\]
With this choice, we easily see $-\log_{2}\left(\left|a_{i}\right|\right)\leq-k_{i}\leq1-\log_{2}\left(\left|a_{i}\right|\right)$
and thus $1\leq2^{-k_{i}}\left|a_{i}\right|\leq2$ for all $i\in\underline{d}$,
which implies
\[
A_{k,\varepsilon}^{-1}A={\rm diag}\left(2^{-k_{1}}\left|a_{1}\right|,\dots,2^{-k_{d}}\left|a_{d}\right|\right)\in V
\]
and hence $A\in A_{k,\varepsilon}V$.

Now, let $P:=\left(\frac{2}{3},\frac{3}{2}\right)^{d}$ and $Q:=\left(\frac{1}{2},2\right)^{d}$.
Then $\overline{P}\subset Q$ and it is easy to see $\bigcup_{i\in I}A_{i}^{-T}P=\left(\mathbb{R}^{\ast}\right)^{d}=\mathcal{O}$.
Thus, $\mathcal{P}:=\left(A_{i}^{-T}P\right)_{i\in I}$ and $\mathcal{Q}:=\mathcal{Q}_{D}:=\left(A_{i}^{-T}Q\right)_{i\in I}$
are both coverings of $\mathcal{O}$ induced by $D$ (cf.\@ \cite[Definition 4.4.3]{VoigtlaenderPhDThesis})
and hence semi-structured admissible coverings of $\mathcal{O}$,
by \cite[Lemma 4.4.4]{VoigtlaenderPhDThesis}. Consequently, $\mathcal{Q}$
is a structured admissible covering of $\mathcal{O}$ and in particular
a tight regular covering and thus also an $L^{p}$-decomposition covering
of $\mathcal{O}$ for all $p\in\left(0,\infty\right]$, cf.\@ Theorem
\ref{thm:StructuredAdmissibleCoveringsAreRegular} and Corollary \ref{cor:RegularPartitionsAreBAPUs}.

Now, by \cite[Theorem 4.6.4]{VoigtlaenderPhDThesis}, we get (up to
canonical identifications) that
\[
{\rm Co}\left(L_{v}^{p,r}\left(\mathbb{R}^{d}\rtimes D\right)\right)=\mathcal{D}\left(\mathcal{Q}_{D},L^{p},\ell_{u^{\left(r\right)}}^{r}\right)
\]
for every moderate weight $v:D\to\left(0,\infty\right)$, where the
weight $u^{\left(r\right)}$ is given by
\[
u_{i}^{\left(r\right)}:=\left|\det A_{i}\right|^{\frac{1}{2}-\frac{1}{r}}\cdot v\left(A_{i}\right)\qquad\text{ for }i\in I.
\]
See Example \ref{exa:ShearletCoorbitSpaces} for the definition of
a moderate weight and note that each submultiplicative, continuous
weight is moderate. Furthermore, \cite[Lemma 4.2.7]{VoigtlaenderPhDThesis}
shows that the set of submultiplicative, continuous weights is closed
under multiplication and addition and also under taking maximums and
inversions, where the inversion of a weight $v:G\to\left(0,\infty\right)$
is given by $v^{\vee}\left(x\right):=v\left(x^{-1}\right)$.

Now, we describe the class of weights $v:D\to\left(0,\infty\right)$
that we will consider in this example: For $\ell\in\underline{d}$,
the weight 
\[
v_{1}^{\left(\ell\right)}:D\to\left(0,\infty\right),{\rm diag}\left(a_{1},\dots,a_{d}\right)\mapsto\left|a_{\ell}\right|
\]
is continuous and (sub)multiplicative. This yields submultiplicativity
of the weight $v_{2}^{\left(\ell\right)}:=\max\left\{ 1,\smash{v_{1}^{\left(\ell\right)}}\right\} $,
which satisfies
\[
v_{2}^{\left(\ell\right)}\left({\rm diag}\left(a_{1},\dots,a_{d}\right)\right)=\begin{cases}
\left|a_{\ell}\right|, & \text{if }\left|a_{\ell}\right|\geq1,\\
1, & \text{if }\left|a_{\ell}\right|\leq1.
\end{cases}
\]
But since $v_{2}^{\left(\ell\right)}$ is a submultiplicative weight
on $D$, the same is true of $v_{3}^{\left(\ell\right)}:D\to\left(0,\infty\right),A\mapsto v_{2}^{\left(\ell\right)}\left(A^{-1}\right)$
and thus also of
\[
v^{\left(\ell,\alpha_{1},\alpha_{2}\right)}:D\to\left(0,\infty\right),A={\rm diag}\left(a_{1},\dots,a_{d}\right)\mapsto\left(v_{2}^{\left(\ell\right)}\left(A\right)\right)^{\alpha_{1}}\cdot\left(v_{3}^{\left(\ell\right)}\left(A\right)\right)^{-\alpha_{2}}=\begin{cases}
\left|a_{\ell}\right|^{\alpha_{1}}, & \text{if }\left|a_{\ell}\right|\geq1,\\
\left|a_{\ell}\right|^{\alpha_{2}}, & \text{if }\left|a_{\ell}\right|<1,
\end{cases}
\]
for arbitrary $\alpha_{1},\alpha_{2}\in\mathbb{R}$. Finally, for
$\alpha,\beta\in\mathbb{R}^{d}$, the weight 
\[
v^{\left(\alpha,\beta\right)}:D\to\left(0,\infty\right),A\mapsto\prod_{\ell=1}^{d}v^{\left(\ell,\alpha_{\ell},\beta_{\ell}\right)}\left(A\right)
\]
is submultiplicative as a product of submultiplicative weights.

Now, we are interested in existence of an embedding
\begin{equation}
{\rm Co}\left(L_{v^{\left(\alpha,\beta\right)}}^{p,r}\left(\mathbb{R}^{d}\rtimes D\right)\right)=\mathcal{D}\left(\mathcal{Q}_{D},L^{p},\ell_{u^{\left(r\right)}}^{r}\right)\hookrightarrow W^{k,q}\left(\mathbb{R}^{d}\right)\label{eq:DiagonalCoorbitDesiredEmbedding}
\end{equation}
for $k\in\mathbb{N}_{0}$ and $q\in\left(0,\infty\right]$, where
\begin{align*}
u_{k,\varepsilon}^{\left(r\right)} & =\left|\det A_{k,\varepsilon}\right|^{\frac{1}{2}-\frac{1}{r}}\cdot v^{\left(\alpha,\beta\right)}\left(A_{k,\varepsilon}\right)\\
 & =\prod_{\ell=1}^{d}\begin{cases}
2^{k_{\ell}\left(\alpha_{\ell}+\frac{1}{2}-\frac{1}{r}\right)}, & \text{if }k_{\ell}\geq0,\\
2^{k_{\ell}\left(\beta_{\ell}+\frac{1}{2}-\frac{1}{r}\right)}, & \text{if }k_{\ell}<0.
\end{cases}
\end{align*}
In the present setting, the weight $v:=w^{\left(q\right)}$ from Corollary
\ref{cor:SimplifiedSobolevEmbedding} is given by
\begin{align*}
v_{k,\varepsilon} & =\left|\det T_{k,\varepsilon}\right|^{\frac{1}{p}-\frac{1}{q}}\cdot\left(1+\left|b_{k,\varepsilon}\right|^{n}+\left\Vert T_{k,\varepsilon}\right\Vert ^{n}\right)\\
 & \asymp\left|\det A_{k,\varepsilon}\right|^{\frac{1}{q}-\frac{1}{p}}\cdot\left(1+\left\Vert A_{k,\varepsilon}^{-1}\right\Vert ^{n}\right)\\
 & =\left(1+\max\left\{ 2^{-nk_{\ell}}\with\ell\in\underline{d}\right\} \right)\cdot\prod_{\ell=1}^{d}2^{k_{\ell}\left(\frac{1}{q}-\frac{1}{p}\right)}\:,
\end{align*}
where it is important to note $T_{i}=A_{i}^{-T}$ for $i\in I$, since
$\mathcal{Q}_{D}=\left(A_{i}^{-T}Q\right)_{i\in I}$.

Thus, the relevant quotient $\frac{v}{u^{\left(r\right)}}$ satisfies
\begin{align}
\frac{v_{k,\varepsilon}}{u_{k,\varepsilon}^{\left(r\right)}} & \asymp\left(1+\max\left\{ 2^{-nk_{\ell}}\with\ell\in\underline{d}\right\} \right)\cdot\prod_{\ell=1}^{d}\begin{cases}
2^{k_{\ell}\left(\gamma-\alpha_{\ell}\right)}, & \text{if }k_{\ell}\geq0,\\
2^{k_{\ell}\left(\gamma-\beta_{\ell}\right)}, & \text{if }k_{\ell}<0
\end{cases}\nonumber \\
 & \asymp\left(1+\sum_{\ell=1}^{d}2^{-nk_{\ell}}\right)\cdot v^{\left(\gamma_{e}-\alpha,\,\gamma_{e}-\beta\right)}\left({\rm diag}\left(2^{k_{1}},\dots,2^{k_{d}}\right)\right)\nonumber \\
 & =v^{\left(\gamma_{e}-\alpha,\:\gamma_{e}-\beta\right)}\left({\rm diag}\left(2^{k_{1}},\dots,2^{k_{d}}\right)\right)+\sum_{\ell=1}^{d}v^{\left(\gamma_{e}-\alpha-ne_{\ell},\:\gamma_{e}-\beta-ne_{\ell}\right)}\left({\rm diag}\left(2^{k_{1}},\dots,2^{k_{d}}\right)\right),\label{eq:DiagonalGroupMainWeight}
\end{align}
with $\gamma:=\frac{1}{q}-\frac{1}{p}+\frac{1}{r}-\frac{1}{2}$ and
$\gamma_{e}:=\gamma\cdot\left(1,\dots,1\right)\in\mathbb{R}^{d}$.

Now, due to the exponential nature and due to the product structure
of the weight 
\[
w^{\left(\alpha,\beta\right)}\left(k_{1},\dots,k_{d}\right):=v^{\left(\alpha,\beta\right)}\left({\rm diag}\left(2^{k_{1}},\dots,2^{k_{d}}\right)\right)=\prod_{\ell=1}^{d}\begin{cases}
2^{k_{\ell}\alpha_{\ell}}, & \text{if }k_{\ell}\geq0,\\
2^{k_{\ell}\beta_{\ell}}, & \text{if }k_{\ell}<0,
\end{cases}
\]
we see
\begin{equation}
w^{\left(\alpha,\beta\right)}\in\ell^{\theta}\left(\mathbb{Z}^{d}\right)\Longleftrightarrow\begin{cases}
\alpha_{\ell}\leq0\text{ and }\beta_{\ell}\geq0\text{ for all }\ell\in\underline{d}, & \text{if }\theta=\infty,\\
\alpha_{\ell}<0\text{ and }\beta_{\ell}>0\text{ for all }\ell\in\underline{d}, & \text{if }\theta<\infty.
\end{cases}\label{eq:DiagonalGroupGeneralCharacterization}
\end{equation}

For brevity, let us write $a\leq b$ for $a,b\in\mathbb{R}^{d}$ if
$a_{\ell}\leq b_{\ell}$ for all $\ell\in\underline{d}$. The notation
$a<b$ is defined analogously. In view of equation (\ref{eq:DiagonalGroupMainWeight}),
we see
\begin{align*}
 & \frac{v}{u^{\left(r\right)}}\in\ell^{\theta}\left(\mathbb{Z}^{d}\times\left\{ \pm1\right\} ^{d}\right)\\
\Longleftrightarrow & \begin{cases}
\gamma_{e}-\alpha\leq0\text{ and }\gamma_{e}-\beta\geq0\text{ and }\gamma_{e}-\alpha-ne_{\ell}\leq0\text{ and }\gamma_{e}-\beta-ne_{\ell}\geq0\text{ for all }\ell\in\underline{d}, & \text{if }\theta=\infty,\\
\gamma_{e}-\alpha<0\text{ and }\gamma_{e}-\beta>0\text{ and }\gamma_{e}-\alpha-ne_{\ell}<0\text{ and }\gamma_{e}-\beta-ne_{\ell}>0\text{ for all }\ell\in\underline{d}, & \text{if }\theta<\infty
\end{cases}\\
\Longleftrightarrow & \begin{cases}
\alpha_{\ell}\geq\gamma\text{ and }\beta_{\ell}\leq\gamma-n\text{ for all }\ell\in\underline{d}, & \text{if }\theta=\infty,\\
\alpha_{\ell}>\gamma\text{ and }\beta_{\ell}<\gamma-n\text{ for all }\ell\in\underline{d}, & \text{if }\theta<\infty.
\end{cases}
\end{align*}
Now, Corollary \ref{cor:SimplifiedSobolevEmbedding} (in conjunction
with equation (\ref{eq:SpecialExponentInfiniteCharacterization}))
shows that the embedding (\ref{eq:DiagonalCoorbitDesiredEmbedding})
holds if we have $p\leq q$ and
\begin{equation}
\begin{cases}
\alpha_{\ell}\geq\frac{1}{q}-\frac{1}{p}+\frac{1}{r}-\frac{1}{2}\text{ and }\beta_{\ell}\leq\frac{1}{q}-\frac{1}{p}+\frac{1}{r}-\frac{1}{2}-n\text{ for all }\ell\in\underline{d}, & \text{if }r\leq q^{\triangledown},\\
\alpha_{\ell}>\frac{1}{q}-\frac{1}{p}+\frac{1}{r}-\frac{1}{2}\text{ and }\beta_{\ell}<\frac{1}{q}-\frac{1}{p}+\frac{1}{r}-\frac{1}{2}-n\text{ for all }\ell\in\underline{d}, & \text{if }r>q^{\triangledown}.
\end{cases}\label{eq:DiagonalGroupSimplified}
\end{equation}
For $q\in\left(0,2\right]\cup\left\{ \infty\right\} $, these conditions
are also necessary for existence of the embedding (\ref{eq:DiagonalCoorbitDesiredEmbedding}).
For $q\in\left(2,\infty\right)$, a necessary condition for (\ref{eq:DiagonalCoorbitDesiredEmbedding})
is that $p\leq q$ and that equation (\ref{eq:DiagonalGroupSimplified})
holds, with $q^{\triangledown}$ replaced by $q$ throughout.

In summary, due to the exponential nature of the weight $v^{\left(\alpha,\beta\right)}$,
our criteria are again reasonably sharp, even for $q\in\left(2,\infty\right)$.
The only difference between sufficient and necessary criteria is that
for $q\in\left(2,\infty\right)$ and $r\in\left(q^{\triangledown},q\right]$,
the sufficient condition requires a \emph{strict} inequality, while
the necessary condition only yields a non-strict estimate. We again
refrain from using the ``Besov detour'', since the methods from
\cite{VoigtlaenderPhDThesis} are also not sharp, since the covering
$\mathcal{Q}=\mathcal{Q}_{D}$ is not moderate with respect to the
dyadic ``Besov'' covering.\end{example}
\begin{rem}[Concluding remarks]
The examples in this section demonstrate the vast generality of our
embedding results: Using a uniform and easily applicable method, we
were able to give sufficient conditions and necessary conditions for
embeddings of largely different decomposition spaces into Sobolev
spaces. Furthermore, for $q\in\left(0,2\right]\cup\left\{ \infty\right\} $,
we still achieve the best of two (usually) contradictory properties:
generality \emph{and} sharpness. In fact, we can completely characterize
the existence of the desired embeddings.

Only for $q\in\left(2,\infty\right)$, we pay for our generality:
In this case, specialized results like those in \cite{KobayashiSugimotoModulationSobolevInclusion}
are superior to our criteria. Nevertheless, to my knowledge, the
criteria developed in this paper are currently the best ones known
for large classes of spaces; in particular for shearlet-type coorbit
spaces and coorbit spaces with respect to the diagonal group.
\end{rem}

\section*{Acknowledgements}

I would like to thank Hartmut Führ and Hans Feichtinger for fruitful
discussions and inspiring comments. Furthermore, I thank Holger Rauhut
for posing the question whether shearlet coorbit spaces embed into
BV spaces. In a sense, this question caused me to write this paper.

This research was funded by the Excellence Initiative of the German
federal and state governments, and by the German Research Foundation
(DFG), under the contract FU 402/5-1.

\bibliographystyle{plain}
\bibliography{Bibliography}

\end{document}